\documentclass[10pt,leqno,twoside]{amsart}

\usepackage{enumerate}

\usepackage{amssymb}
\usepackage[english]{babel}
\usepackage{amssymb,amsthm,amsmath,eucal,mathrsfs}
\usepackage{bm}

\usepackage{amsmath}
\usepackage{amsfonts}
\usepackage{amsmath}
\usepackage{amssymb}
\usepackage{graphicx}
\usepackage{lscape}
\usepackage{amstext}
\usepackage{amsthm}
\usepackage{color}
\usepackage{float}
\usepackage{mathrsfs}
\usepackage{epsfig}
\usepackage{url}
\usepackage{fancyhdr}
\usepackage{pspicture}
\usepackage{graphicx}

\usepackage{cite}

\usepackage{amsmath,verbatim}

\usepackage{amsthm}

\usepackage{amssymb}

\usepackage{amsfonts}

\usepackage{dsfont}

\usepackage{hyperref}

\newtheorem{theorem}{Theorem}[section]

\newtheorem{lemma}[theorem]{Lemma}

\newtheorem{proposition}[theorem]{Proposition}

\theoremstyle{definition}

\newtheorem{remark}[theorem]{Remark}

\numberwithin{equation}{section}

\renewcommand{\Im}{{\ensuremath{\mathrm{Im\,}}}} 

\renewcommand{\Re}{{\ensuremath{\mathrm{Re\,}}}} 

\renewcommand{\div}{\mathrm{div}\,}    

\newcommand\restr[2]{{
  \left.\kern-\nulldelimiterspace 
  #1 
  \vphantom{\big|} 
  \right|_{#2} 
  }}

\usepackage{eucal}

\title[The Electromagnetism via Dielectric Nanoparticles]{The Electromagnetic Waves Generated by Dielectric Nanoparticles}

\author[Cao, Ghandriche and Sini]{Xinlin Cao $^*$ Ahcene Ghandriche  $^{**}$ and Mourad Sini$^{\ddag}$}
\thanks{$^*$ RICAM, Austrian Academy of Sciences, Altenbergerstrasse 69, A-4040, Linz, Austria. Email: xinlin.cao@ricam.oeaw.ac.at. This author is supported by the Austrian Science Fund (FWF): P 32660}
\thanks{$^{**}$ RICAM, Austrian Academy of Sciences, Altenbergerstrasse 69, A-4040, Linz, Austria. Email: ahcene.ghandriche@ricam.oeaw.ac.at. This author is supported by the Austrian Science Fund (FWF): P 30756-NBL}
\thanks{$^{\ddag}$ RICAM, Austrian Academy of Sciences, Altenbergerstrasse 69, A-4040, Linz, Austria. Email: mourad.sini@oeaw.ac.at. This author is partially supported by the Austrian Science Fund (FWF): P 30756-NBL and P 32660}

\begin{document}


\allowdisplaybreaks

\begin{abstract}
We estimate the electromagnetic fields generated by a cluster of dielectric nanoparticles embedded into a background made of a vacuum. The dielectric nanoparticles are small scaled but enjoy high contrast of their relative permittivity. Such scales/contrasts can be ensured using the Lorentz model with incident frequencies chosen appropriately close to the undamped resonance (appearing in the Lorentz model). Under certain ratio between their size and contrast, these nanoparticles generate resonances, called dielectric resonances. These resonances are characterized and computed via the spectrum of the electric Newtonian operator, stated on the support of nanoparticles, projected on the space of divergence-free fields with vanishing boundary normal components. We characterize the dominant field generated by a cluster of such dielectric-resonating nanoparticles. In this point-interaction approximation, the nanoparticles can be distributed to occupy volume-like domains or low dimensional hypersurfaces where periodicity is not required. The form of these approximations suggests that the effective electromagnetic medium, equivalent to the cluster of such nanoparticles, is a perturbation of the magnetic permeability and not the electric permittivity. The cluster can be tuned such that the equivalent permeability has positive or negative values (while the permittivity stays unchanged).
\end{abstract}


\maketitle
\section{Introduction}\label{prelimilary}

\subsection{Background}
 In the recent years, in the engineering community, there was a high gain of interests in studying electromagnetic wave propagating in resonating dielectric nanostructures \cite{Bohren-Huffmann, PD-L-K, J-J:2016, K-M-B-K-L:2016}. These media are composed of the homogeneous background in which we inject nano-scaled particles. In addition to be small scaled, these particles enjoy  high values of their relative index of refraction with a relatively small $Q-$factor (i.e. the ratio between the absorption and diffusion coefficients is very small). These properties allow them to enjoy interesting and useful optical properties. The smallness of the $Q-$factor allow them to be less loosy as compared with other types of nanoparticles as the plasmonic ones. As such, they are more suited to be used in imaging techniques that require remote measurements, for instance. The high contrasts of the refraction index allow them to resonate at certain scales size/contrast (or at certain ranges of incident frequencies). Such resonance's effects are very attractive in both imaging and material sciences. They are also potentially applicable for the design of highly nonlinear material (i.e. the Kerker effect).   
\bigskip

Our interest in this topic is to understand the interaction between the light and the nanostructure. Precisely, we want to estimate the perturbations of the used incident field due to the presence of these nanostructures. 
Looking at the problem under this angle, we wish to derive the dominating terms in the expansion of the scattered field taking into account the whole structure of the composite, namely the (potentially high) number of the nanoparticles, their sizes and the high contrasts of the related indices of refraction. The derived close form of the dominating term will allow us to tune the structure, at will, so that the equivalent material will enjoy needed properties as sign changing and nonlinearity of the effective electric or magnetic susceptibility. 

\bigskip

In this work, we focus on deriving these approximations for general shapes. This follows the lines of our previous works on point-interaction approximations (called also the Foldy-Lax approximation) for the electromagnetic fields, \cite{Ali, Ali-2, Cao-Sini}.  In \cite{Ali, Ali-2}, we derived the close form of the dominating term in the case that the contrasts of the permittivity and the permeability of the particles are moderate. This dominating term is given in a form of a superposition of electric poles and magnetic dipoles with attached weights. These attached weights (that encode the structure of the composite) are vectors that are computable by solving an invertible algebraic system. The analysis is based on the related Lippmann-Schwinger system of equations coupling the electric and magnetic fields. This Lippmann-Schwinger system is defined through an operator that couples the electric vector Newtonian operator and the Magnetization operator. As the contrasts of the particles are moderate, both the two operators play equivalent roles. That is why the derived approximation in \cite{Ali, Ali-2} involves contributions from the electric and magnetic polarization tensor with equivalent roles. The equivalent media generated by such moderately contrasting nanoparticles have been derived in \cite{Cao-Sini} that confirms that both the electric permittivity and the magnetic permeability are perturbed. However, in the case of dielectric nanoparticles, as in our current work, the permittivity has high contrast while the permeability stays moderate (or unchanged). In this case, the effect of the vector Newtonian operator dominates the one of the Magnetization operator. In addition, under critical scales between the sizes and the contrasts, of the particles, one can excite a subfamily of the eigenvalues of the Newtonian operator. These are related to the projection of this Newtonian operator on the subspace of divergence free fields with zero normal components. This makes the analysis more involved and subtle than that in \cite{Ali, Ali-2}. However, this is worth the efforts. Indeed, we show that even though the contrasts of the nanoparticles are due to the permittivity, the dominating field is of magnetic type.
In other words, the point-interaction approximation, due to the presence of the whole cluster of dielectric nanoparticles, suggests that the effective medium will be a perturbation of the magnetic permeability and not the electric permittivity. The justification of such effective medium, in different scenarios including the Moir\'e metamaterials, see \cite{Wu-Zheng}, will be done in a forthcoming work.  
\bigskip

Wave propagation in resonating media is highly
attractive and, in the very recent years, we witness a rapid growth of the number of published works
in the mathematical community (in addition to the relatively large engineering literature). 
These resonances occur if the particles are small-sized and enjoy high or negative contrasts of their materials. In effect, they are related to the spectrum of the volumetric operator (Newtonian potential operator) or surface operators (Neumann-Poincar\'e operator or related Magnetization operator) appearing in the resolvent operator modeling the wave propagation. 
With such contrasts of the materials, one can excite these resonances and enhance the values of scattering coefficients (or generally the polarization tensors). Such enhancements are useful in many applications as in the effective medium theory.  In this direction, let us cite the following works \cite{Ammari_2019, ACP-2, AFGLH, AFLYH, AH} for the acoustic propagation.  As far as
the effective medium theory is concerned, few results are known for the electromagnetism
with highly contrasting or negative permittivity or permeability. We can cite \cite{Allaire-1998, Bouchitte-Schweizer, BBM, Chen-Lipton-2013, Cher-Cooper, Cher-Ersh-Kise, BBF, L-S:2016, Schweizer:2017, Suslina-2019} who assume periodicity and derive the equivalent coefficients, via the homogenization theory, for dielectric
nanoparticles. The results provided in the current work is a contribution to fill in this gap.

\subsection{Main results}

Let $D$ be a bounded and Lipschitz-regular domain in $\mathbb{R}^3$. We assume that $D$ is a non-magnetic material meaning that the permeability $\mu$ is constant everywhere and equals to the one of the vaccum $\mu_0$.
The electric permittivity is equal to the one of the vaccum outside $D$, $\epsilon_0$, but inside, we assume it to have different values  $\epsilon_p$. We denote by the relative permittivity and permeability as $\epsilon_r:=\frac{\epsilon}{\epsilon_0}$ and $\mu_r:=\frac{\mu}{\mu_0}$. With such notations, we have $\epsilon_r=1$ outside $D$ while $\mu_r=1$ in the whole space $\mathbb{R}^3$.

The electromagnetic scattering of time-harmonic plane waves from the body $D$ reads as follows

\begin{equation}\label{U}
\left \{
\begin{array}{llrr}
Curl(E^{T}) - i \, k \, H^{T} = 0 \, \quad & \mbox{in } \mathbb{R}^3 \\
\\
Curl(H^{T}) + i \, k \, \varepsilon_{r} \, E^{T} = 0 \, \quad & \mbox{in } \mathbb{R}^3,
\end{array} 
\right.
\end{equation}
where the total field $(E^T, H^{T})$ is of the form $(E^T:=E^{Inc} +E^s,H^T:=H^{Inc} +H^s)$ with and incident plane wave $(E^{Inc},H^{Inc})$ of the form
\begin{equation}\notag
E^{Inc}(x,\theta) = \theta^{\perp} \exp\left(i \, k \; \theta \cdot x \right)\quad\mbox{and}\quad
H^{Inc}(x,\theta) = \big( \theta^{\perp} \times \theta \big) \; \exp\big(i \, k \; \theta \cdot x \big),
\end{equation}
and the scattered field $(E^s, H^s)$ satisfies the Silver-M\"{u}ller radiation conditions (SMRC) at infinity:

\begin{equation}\notag
	\sqrt{\mu_0\epsilon_0^{-1}}H^{s}(x)\times \frac{x}{|x|}-E^s(x)=O(\frac{1}{|x|^2}).
\end{equation}



This problem is well posed in appropriate Sobolev spaces, see \cite{colton2019inverse, Mitrea}, and we have the behaviors

\begin{equation}\label{def-far}
E^s(x)=\frac{e^{i k|x|}}{|x|}\left(E^\infty(\hat{x})+O(|x|^{-1})\right),\quad \mbox{as}\quad |x|\rightarrow \infty,
\end{equation}
and
\begin{equation}\notag
	H^s(x)=\frac{e^{i k|x|}}{|x|}\left(H^\infty(\hat{x})+O(|x|^{-1})\right),\quad \mbox{as}\quad |x|\rightarrow \infty,
\end{equation}
where $(E^\infty(\hat{x}), H^\infty(\hat{x}))$ is the corresponding electromagnetic far field pattern of \eqref{U} in the propagation direction $\hat{x}:=\frac{x}{|x|}$.
\bigskip


Next, we present the needed assumptions on the model \eqref{U} to derive our results.
\vspace*{10px}


	\uppercase\expandafter{\romannumeral1}. \emph{Assumptions on the cluster of particles}\label{\romannumeral1}. 
	Suppose that each component $D_m$ of $D$ is of the form $D_m=a {B}_m+{z}_m$, $m=1,\cdots, \aleph$, which is characterized by the parameter $a>0$ and the location ${z}_m$. We denote
$ a:=\max\limits_{1\leq m \leq M}\mathrm{diam}(D_m),\quad
	d:=\min\limits_{{1\leq m,j\leq M}\atop{m\neq j}} d_{mj}:=\min\limits_{{1\leq m,j\leq M}\atop{m\neq j}} \mathrm{dist}(D_m, D_j).$ We take 
	\begin{equation}\label{cluster}
	\aleph =O(d^{-3}) \mbox{ and } d\sim a^t
	\end{equation} with the nonnegative parameter $t\leq1$. For simplicity of the exposition, we assume that the shapes of $B_m$'s are the same  and we denote $B:=B_m$.  The domain $B$ is a bounded Lipschitz domain that contains the origin.
	\bigskip
	
	\uppercase\expandafter{\romannumeral2}. \emph{ Assumption on the shape of $B$.}\label{\romannumeral2} Define the vector Newtonian operator $N^0$. Denote $e_n^{(1)}$ as the corresponding eigenfunctions over the subspace $\mathbb{H}_{0}(\div=0)$. Since $
	\mathbb{H}_{0}\left( \div = 0 \right) \equiv Curl \left( \mathbb{H}_{0}\left( Curl \right) \cap \mathbb{H}\left( \div = 0 \right) \right),$
	see for instance \cite{{amrouche1998vector}}, then we have 
\begin{equation}\label{pre-cond}
e_{n_0}^{(1)}= Curl (\phi_{n_0}) \,\, \mbox{ with } \,\, \nu\times \phi_{n_{0}}=0 \,\, \text{and} \,\, \div(\phi_{n_{0}})=0.
\end{equation}	
We assume that, for a certain $n_0$, 
\begin{equation*}
\int_{B} \phi_{n_{0}}(y) \, dy \neq 0.
\end{equation*}
	
	\uppercase\expandafter{\romannumeral3}.\emph{Assumptions on the permittivity and permeability of each particle.}\label{\romannumeral2} In order to investigate the electromagnetic scattering of dielectric nanoparticles with high contrast electric permittivity parameter, we assume that for a constant $\varsigma$,
	\begin{equation}\label{contrast-epsilon}
	\eta:=\epsilon_{r} - 1=\varsigma \; a^{-2},\; \mbox{ with } \;\; \varsigma <1\; \footnote{The number of particles that can be distributed in a given bounded (3D) domain is $\aleph =O(d^{-3})$. The condition on $\varsigma$, i.e. $\varsigma<1$, is needed only if we distribute the maximum number of particles, i.e. $\aleph \sim d^{-3}$. In addition, $\varsigma$ can be complex and in this case we need $\boldsymbol{\Re} (\varsigma) <1$ with $\boldsymbol{\Im} (\varsigma) \ll 1$, see Remark \ref{RE-Lorentz-2} for an example.}\; ~~ a \ll 1,
	\end{equation}
	and the magnetic permeability $\mu_r$ to be moderate, namely $\mu_r=1$.
	\bigskip
	
\uppercase\expandafter{\romannumeral4}. \emph{ Assumption on the used incident frequency $k$.}\label{\romannumeral4} There exists a positive constant $c_0$ such that
\begin{equation}\label{condition-on-k}
			1 \, - \, k^2 \, \eta \, a^2 \, \lambda_{n_{0}}^{(1)} \, = \, \pm \; c_0\; a^h,\; ~~ a \ll 1,
		\end{equation}
		where $\lambda_{n_0}^{(1)}$ is the eigenvalue corresponding to $e_{n_0}^{(1)}$.

\vspace*{10px}
\begin{remark}\label{RE-Lorentz-1}
The conditions (\ref{contrast-epsilon}) and (\ref{condition-on-k}) can be derived from the Lorentz model by choosing appropriate incident frequency $k$. Indeed, we recall the Lorentz model for the relative permittivity
$\label{Lorentz model}
\epsilon_r=1-\dfrac{k_\mathrm{p}^2}{k_0^2-k^2-ik\gamma}
$
where $k_\mathrm{p}$ is the plasmonic resonance, $k_0$ is the undamped resonance and $\gamma$ is the damping frequency. We also recall the eigenvalues $\lambda_{n_{0}}^{(1)}$ of the Newtonian operator stated on $B$, and the contrast parameter $\eta=\epsilon_r-1$. 
If the frequency of the incident wave $k$ is chosen to be real and $k^2$ close to the undamped resonance frequency $k^2_0$ with 
\begin{equation*}
k^2-k_0^2 = \frac{k_{p}^{2} \, a^2 \,  \lambda_{n_0}^{(1)} \, k_0^2 }{\left(1 \mp \boldsymbol{\Re}(c_{0}) \, a^{h} \right) \pm \dfrac{\boldsymbol{\Im}^{2}(c_{0}) \, a^{h}}{\left(1 \mp \boldsymbol{\Re}(c_{0}) \, a^{h} \right)} - k_{p}^{2} \, a^{2} \, \lambda_{n_0}^{(1)}} = k_{p}^{2} \, a^2 \,  \lambda_{n_0}^{(1)} \, k_0^2 \, \left[ 1 + \mathcal{O}(a^{h}) \right]
\end{equation*}
and 
\begin{equation*}
\gamma k = \pm \frac{\boldsymbol{\Im}(c_{0}) \, a^{2} \, (k^{2} - k^{2}_{0})}{\left(1 \mp \boldsymbol{\Re}(c_{0}) \, a^{h} \right)} = \pm \, \boldsymbol{\Im}(c_{0}) \, a^{4} \, \lambda_{n_0}^{(1)} \, k_0^2 \, k_p^2 \left[ 1 + \mathcal{O}(a^{h}) \right],
\end{equation*} 
then we obtain $\boldsymbol{\Re}\Big(\eta\Big) = a^{-2}\Big(\lambda_{n_0}^{(1)} \, k_0^2 \Big)^{-1} + \mathcal{O}\left( a^{2} \right)$ and $\boldsymbol{\Im}\Big(\eta\Big) = \mp \, \boldsymbol{\Im}\Big( c_{0} \Big) \, \Big( \lambda_{n_0}^{(1)} \, k_0^2 \Big)^{-1} + \mathcal{O}\left( a^{4} \right)$.  With these choices, we see that $\frac{k^2}{k_0^2}\sim 1$. In addition, we derive the needed relation $1-k^2\eta a^2 \lambda_{n_{0}}^{(1)}=\pm \; c_0\; a^h,\; ~~ a \ll 1$ with a given value of $c_{0} \in \mathbb{C}$. In practice, an additional condition on the sign of $\boldsymbol{\Im} (c_0)$ might be needed to ensure the non-negativity of the damping, i.e. $\boldsymbol{\Im } (\epsilon_{r})\geq 0$. In particular, if $\boldsymbol{\Im} (c_0) =0$ then $\gamma=0$ and then we get the Drude's model.
\end{remark}

\begin{remark}\label{RE-Lorentz-2}
From the previous remark we deduce that $\varsigma$, given in $(\ref{contrast-epsilon})$,
can be approximated by 
\begin{equation*}
\varsigma = \frac{1}{\lambda_{n_{0}}^{(1)} \; k_{0}^{2} \; \left(1 \pm i \, \boldsymbol{\Im}\Big( c_{0} \Big) \, a^{2} \right) \, \left(1 + \mathcal{O}\left( a^{h} \right) \right)} = \frac{1}{\lambda_{n_{0}}^{(1)} \; k_{0}^{2}} +  \mathcal{O}\left( a^{h} \right),
\end{equation*}
and the condition $\boldsymbol{\Re} (\varsigma) < 1$ will be fulfilled as long as $1 < \lambda_{n_{0}}^{(1)} \; k_{0}^{2}$ with $a\ll 1$. If $\boldsymbol{\Im}(c_0)=0$ then $\varsigma$ is real valued.
\end{remark}

Based on the above conditions, we are now in a position to state our main result.

\begin{theorem}\label{main-1}
	Let the conditions (\uppercase\expandafter{\romannumeral1}, \uppercase\expandafter{\romannumeral2}, \uppercase\expandafter{\romannumeral3}, \uppercase\expandafter{\romannumeral4}) on the electromagnetic scattering problem \eqref{U}, by the multiple particles, $D_{1},\cdots,D_{\aleph}$, be satisfied. For $t$ and $h$ in $[0, 1]$ such that 
\begin{equation}\label{conditions-t-h}
3-3t-h \ge 0,~~~ \mbox{ and } ~~~ \frac{9}{11}<h<1, \footnote{The limit value $h=1$ can be handled at the expense of assuming the constant $c_0$ appearing in \ref{condition-on-k} to be large enough. The lower bound is only sufficient and not optimal.}
\end{equation}
the far field of the scattered wave admits the following expansion 
		\begin{equation}\label{approximation-E}
		E^\infty(\hat{x}) = - \, i \, k \, \eta \,  \sum_{m=1}^{\aleph} \, e^{i \, k \, \hat{x} \cdot z_{m}} \hat{x}\times Q_m+\mathcal{O}\left(a^{\frac{h}{3}}\right),
		\end{equation}
 where 
		$\left( Q_m \right)_{m=1,\cdots, \aleph}$ is the vector solution to the algebraic system
		\begin{equation}\label{add1}
		 Q_{m}-\eta \, k^2 \, a^{5-h} \, \sum_{j=1 \atop j \neq m}^\aleph {\bf{P}}_0 \cdot \Upsilon_k(z_{m}, z_j) \cdot {Q}_j
		=i \, k \, a^{5-h} \; {\bf{P}}_0 \cdot  H^{Inc}(z_{m}),
		\end{equation}
		where $\Upsilon_k$ is the dyadic Green's function given by \eqref{dyadicG} and ${\bf{P}}_0$ is the polarization matrix defined by
\begin{equation}\label{defP0}
{\bf{P}}_0 := \sum_{m} \int_{B}\phi_{n_{0},m}(y)\,dy\otimes\int_{B}\phi_{n_{0},m}(y)\,dy, 
\end{equation}		
where $\phi_{n_{0},m}$ fulfills 
\begin{equation*}
e_{n_{0},m}^{(1)} = Curl\left( \phi_{n_{0},m} \right), \,\, \div\left( \phi_{n_{0},m} \right) = 0, \,\, \nu \times \phi_{n_{0},m} = 0 \,\, \text{and} \,\, N\left( e_{n_{0},m}^{(1)} \right) = \lambda_{n_{0}}^{(1)} \, e_{n_{0},m}^{(1)}.
\end{equation*}
In particular, \eqref{add1} is invertible under the condition that
		\begin{equation}\notag
			\frac{ k^2 |\eta| a^5}{d^3 \left\vert 1 - k^2 \, \eta \, a^2 \, \lambda_{n_{0}}^{(1)} \right\vert}\left\lVert {\bf{P}}_0 \right\rVert <1.
		\end{equation}

\end{theorem}

We observe from (\ref{approximation-E})-(\ref{add1}) that if the number of the particles behaves like $\aleph\sim d^{-3}(\sim a^{-3t})$, with $3-3t-h = 0$, and the particles are distributed in a bounded domain $\Omega_{eff}$, then, at least formally, the electric field is equivalent to the one generated by an effective medium given by $(\epsilon_{eff}:=\epsilon_0, \mu_{eff}:=\mu_0+\mu_{pert}\chi_{\Omega})$ where $\mu_{pert}$ is related to the cluster of particles and it can be positive or negative (depending on the chosen sign in (\ref{condition-on-k})). In particular, we see that even though the dielectric nanoparticles are merely generated by the contrasts of their permittivity (and not their permeability), the effective medium is a perturbation of the permeability and not the permittivity. 
\bigskip

Before closing this introduction, let us mention that quite recently, after closing the content of the current work, we got aware of the work \cite{Ammari-Li-Zou-2}, where the authors derived similar approximation formulas in the case of a single dielectric nanoparticle. 
We note that the error term in (\ref{approximation-E}), of the order $\mathcal{O}\left(a^{\frac{h}{3}}\right)$, takes into account the number of particles of the order $\aleph =O(a^{-3t})$ with $t$ satisfying (\ref{conditions-t-h}). For the case of a single particle, this error term is smaller and it is dominated by the first term in (\ref{approximation-E}), with $\aleph =1$. This can be seen at the end of Section \ref{sec-proof-main} with more details.
\bigskip

The rest of this paper is organized as follows. In Section \ref{sec-pre}, we introduce some preliminaries including the Lippmann-Schwinger system of equations for the solution to \eqref{U} and the needed decomposition  of the vector $\mathbb{L}^2$ space via the subspaces $\mathbb{H}_0(\div=0)$, $\mathbb{H}_0(Curl=0)$ and $\nabla \mathcal{H}armonic$. Then we present some needed a-priori estimates of the projections of the solution on these three subspaces, respectively, which will be proved later in Section \ref{sec-proof-prior}. In Section \ref{sec-la-system}, we first investigate the general expressions of the linear algebraic system and the corresponding invertibility conditions. Then we formulate the precise form of the linear algebraic systems associated to the projection of the total wave onto the two subspaces $\mathbb{H}_0(\div=0)$ and $\nabla \mathcal{H}armonic$ of $\mathbb{L}^2$, respectively. Detailed proofs will be showed later in Section \ref{sec-proof-la}. Section \ref{sec-proof-main} is devoted to show the rigorous proof of our main theorem, in particular the Foldy-Lax approximation of the far-field. Finally, in Section \ref{sec-proof-prior} and Section \ref{sec-proof-la}, we give the detailed proofs of the a-prior estimates introduced in Section \ref{sec-pre} and the construction of the linear algebraic systems in Section \ref{sec-la-system}, respectively.

\section{Some preliminary knowledge and a-prior estimates.}\label{sec-pre}


\subsection{Decomposition of $\mathbb{L}^2$.}

The following direct sum provide a useful decomposition of $\mathbb{L}^{2}$ (see \cite{Dautry-Lions}, page 314) 
\begin{equation}\label{L2-decomposition}
\mathbb{L}^{2}  = \mathbb{H}_{0}\left(\div=0 \right) \overset{\perp}{\oplus} \mathbb{H}_{0}\left(Curl=0 \right) \overset{\perp}{\oplus} \nabla \mathcal{H}armonic
\end{equation}   
where 
\begin{eqnarray*}
	\mathbb{H}_{0}\left(\div=0 \right) &:=& \left\lbrace E \in \left( \mathbb{L}^{2}(D)\right)^{3}, \, \div E = 0, \, \nu \cdot E = 0 \, \; \text{on} \;\, \partial D \right\rbrace ,\\
	\mathbb{H}_{0}\left(Curl =0 \right) &:=& \left\lbrace E \in \left( \mathbb{L}^{2}(D)\right)^{3}, \, Curl \, E = 0, \, \nu \times E = 0 \, \; \text{on} \;\, \partial D \right\rbrace,
\end{eqnarray*}
and 
\begin{equation*}
\nabla \mathcal{H}armonic := \left\lbrace E: \; E = \nabla \psi, \, \psi \in \mathbb{H}^{1}(D), \, \Delta\psi=0 \right\rbrace.
\end{equation*}
From the decomposition \eqref{L2-decomposition}, 
we define $\overset{1}{\mathbb{P}}, \overset{2}{\mathbb{P}}$ and $\overset{3}{\mathbb{P}}$ to be the natural projectors as follows
\begin{equation}\label{project}
	\overset{1}{\mathbb{P}} := \mathbb{L}^{2} \longrightarrow  \mathbb{H}_{0}\left(\div=0 \right), \;\;\;
	\overset{2}{\mathbb{P}} := \mathbb{L}^{2} \longrightarrow  \mathbb{H}_{0}\left(Curl = 0 \right) \;\; \text{and} \;\;
	\overset{3}{\mathbb{P}} := \mathbb{L}^{2} \longrightarrow  \nabla \mathcal{H}armonic.
\end{equation}

\subsection{Lippmann-Schwinger integral formulation of the solutions.} 

Recall the Green's function for the Helmholtz operator
\begin{equation}\notag
	\Phi_{k}(x,y)=\frac{1}{4\pi}\frac{e^{i k |x-y|}}{|x-y|},\quad x\neq y,
\end{equation}
and the corresponding dyadic Green's function
\begin{equation}\label{dyadicG}
	\Upsilon_k(x,y):=\underset{x}{Hess}\Phi_{k}(x,y)+k^2 \Phi_{k}(x,y){I}, \quad x\neq y.
\end{equation}
For any vector function $F$, define the Newtonian potential operator $N^k$ and the Magnetization operator $\nabla M^k$ as follows:
\begin{equation}\label{N-M opera}
N^k(F)(x):=\int_{D} \Phi_{k}(x,y)F(y)\,dy \quad \text{and} \quad \nabla M^k(F)(x):=\underset{x}{\nabla}\int_{D}\underset{y}{\nabla}\Phi_{k}(x,y)\cdot F(y)\,dy.
\end{equation}


The solution to \eqref{U} of the integral form can be formulated as the following proposition.

\begin{proposition}\label{prop-LS}
	The solution to the electromagnetic scattering problem \eqref{U} satisfies
	\begin{equation}\label{L-S eq}
		E^T(x) - \eta \, \int_{D}\Upsilon_k(x,y) \cdot E^T(y)\,dy=E^{Inc}(x),\quad x\in D,
	\end{equation}
	where $\eta := \epsilon_{r} - 1$ is the contrast of the electric permittivity, and $\Upsilon_k(x,y)$ is the dyadic Green's function defined by \eqref{dyadicG}. Equivalently, we have
	\begin{equation}\label{LS eq2}
		E^T(x) + \eta \, \nabla M^k(E^T)(x) - k^2 \, \eta \, N^{k}(E^T)(x) = E^{Inc}(x),\quad x\in D,
	\end{equation}
	by virtue of the Newtonian operator $N^k$ and the Magnetization operator $\nabla M^k$ given by \eqref{N-M opera}.
\end{proposition}

The proposition can be proved by utilizing the Stratton-Chu formula directly, see \cite[Theorem 6.1]{colton2019inverse} for more detailed discussions.

In addition, in \eqref{dyadicG}, we have
\begin{eqnarray}\label{expansion-of-Hess}
\nonumber
\underset{x}{Hess} \Phi_{k}(x,y) &=& \underset{x}{Hess} \Phi_{0}(x,y)  + \underset{x}{Hess} (\Phi_{k}-\Phi_{0})(x,y) \\ \nonumber
&=& \underset{x}{Hess} \Phi_{0}(x,y)  + \frac{1}{4 \pi} \; \sum_{n \geq 1} \frac{(ik)^{n+1}}{(n+1)!} \; \underset{x}{Hess} \left( \left\Vert x - y \right\Vert^{n} \right) \\ \nonumber
&=& \underset{x}{Hess} \Phi_{0}(x,y) - \frac{k^{2}}{2} \, \Phi_{0}(x,y) \, I_{3} - \frac{i k^{3}}{24 \pi} \, I_{3} + \frac{k^{2}}{2} \, \Phi_{0}(x,y) \, \frac{A(x,y)}{\Vert x-y  \Vert^{2}} \\  
&+& \frac{1}{4 \pi} \; \sum_{n \geq 3} \frac{(ik)^{n+1}}{(n+1)!} \; \underset{x}{Hess} \left( \left\Vert x - y \right\Vert^{n} \right), 
\end{eqnarray}
where $A$ is the matrix given by $A(x,y) := \left( x - y \right) \otimes \left( x - y \right).$
Then, we can write the Magnetization operator $\nabla M^k$ with the help of \eqref{expansion-of-Hess} as 
\begin{eqnarray}\label{expansion-gradMk}
\nonumber
\nabla M^{k}(F)(x)  &=& \nabla M^{0}(F)(x) +  \frac{k^{2}}{2} \, N^{0}(F)(x) + \frac{i k^{3}}{12 \pi} \int_{D} F(y) dy - \frac{k^{2}}{2} \int_{D} \Phi_{0}(x,y) \frac{A(x,y)\cdot F(y)}{\Vert x -y \Vert^{2}}  dy \\ 
&-& \frac{1}{4 \pi} \; \sum_{n \geq 3} \frac{(ik)^{n+1}}{(n+1)!} \; \int_{D} \; \underset{x}{Hess} \left( \left\Vert x - y \right\Vert^{n} \right) \cdot F(y) \, dy .
\end{eqnarray}
Similarly, we write an analogous formula for the Newtonian potential operator $N^k$ as
\begin{equation}\label{expansion-Nk}
N^{k}(F)(x) =  N^{0}(F)(x) + \frac{ik}{4\pi} \int_{D} F(y) dy + \frac{1}{4 \pi} \; \sum_{n \geq 1} \frac{(ik)^{n+1}}{(n+1)!} \; \int_{D} \; \left\Vert x - y \right\Vert^{n}  F(y) \, dy .
\end{equation} 
For shortness reason, we use the notation $\nabla M$ instead of $\nabla M^0$ and $N$ instead of $N^0$ in the subsequent analyses.

 
 Thanks to Green's formulas, and recalling the $\mathbb{L}^2$-space decomposition \eqref{L2-decomposition}, we have
\begin{equation}\label{grad-M-1st-2nd}
\forall \; E \in \mathbb{H}_{0}\left(\div=0 \right) , \;\, \nabla M(E) = 0 \, \quad \text{and} \quad \forall \; E \in \mathbb{H}_{0}\left(Curl=0 \right) , \;\, \nabla M(E) = E,
\end{equation}
and other nice properties for the Magnetization operator, such as self-adjointness, positivity, spectrum, boundedness ($\left\Vert \nabla M \right\Vert = 1$) , invariance of $\nabla \mathcal{H}armonic$, etc., can be found in \cite{Raevskii1994} and \cite{10.2307/2008286}.

\subsection{A-prior Estimates.}\label{subsec-eigen} Based on the decomposition (\eqref{L2-decomposition}), we present here some necessary a-prior estimates derived from the Lippmann-Schwinger equation \eqref{LS eq2}, which play an important role in the proof of our main results. The proof of the propositions and lemmas in this subsection shall be given later in Section \ref{sec-proof-prior}.

Following the notations in \eqref{project}, we denote 
\begin{enumerate} \label{note-eigen}
	\item[$\ast$] $\left( \lambda^{(1)}_{n};e^{(1)}_{n} \right)$ as the eigensystem of the operator $N$ projected on  the subspace $\mathbb{H}_{0}(\div = 0)$,\\
	\item[$\ast$] $\left( \lambda^{(2)}_{n};e^{(2)}_{n} \right)$ as the eigensystem of the operator $N$ projected on the subspace $\mathbb{H}_{0}(Curl = 0)$,\\
	\item[$\ast$] $\left( \lambda^{(3)}_{n};e^{(3)}_{n} \right)$ as the eigensystem of the operator $\nabla M$ on the subspace $\nabla \mathcal{H}armonic$.
\end{enumerate}
For the existence and the construction of $\left( \lambda^{(j)}_{n};e^{(j)}_{n} \right)_{n \in \mathbb{N}},  j=1,2,3$, we refer to Section 5 of \cite{GS}. \\ Now, suppose $E^T$ solves \eqref{U} with the integral formulation \eqref{LS eq2}. Then the projection of $E^T$ onto three subspaces $\mathbb{H}_0(\div=0)$, $\mathbb{H}_0(Curl=0)$ and $\nabla \mathcal{H}armonic$ can be respectively represented as $\overset{1}{\mathbb{P}}(E^T)$, $\overset{2}{\mathbb{P}}(E^T)$ and $\overset{3}{\mathbb{P}}(E^T)$.
Then we have the following estimates.

\vspace*{5px}

\subsubsection{Estimation for one particle.}

\begin{proposition}\label{es-oneP}
	Under Assumptions (\uppercase\expandafter{\romannumeral1}, \uppercase\expandafter{\romannumeral2}, \uppercase\expandafter{\romannumeral3}, \uppercase\expandafter{\romannumeral4}), consider the electromagnetic scattering problem \eqref{U} with only one distributed particle. Let $k$ fulfill
	\begin{equation}\label{choice-k-1st-regime}
		k^2 :=  \frac{1 \mp c_{0} \, a^h}{\eta \, a^2 \, \lambda_{n_{0}}^{(1)}} \, \sim 1,
	\end{equation}
	where $\lambda_{n_{0}}^{(1)}$ is an eigenvalue of the Newtonian potential operator in the subspace $\mathbb{H}_0(\div=0)$. Then for $h<2$, there holds
	\begin{equation}\label{es-oneE}
		\lVert \tilde{E}^T\rVert_{\mathbb{L}^2(B)}=\mathcal{O}(a^{1-h}),
	\end{equation}
	 and in particular,
	\begin{equation}\label{*add1}
	 \overset{2}{\mathbb{P}}\left(\tilde{E}^T\right)=0,
	\end{equation}
	where $\tilde{E}^T$ is the total wave of $E^T$ after scaling from $D$ to $B$.
\end{proposition}

The detailed proof of Proposition \ref{es-oneP} can be seen in Subsection \ref{subsec-51}. Based on the above a-priori estimate for one particle, we also investigate the estimation of the solution in the presence of multiple particles. 

\vspace*{5px}
\subsubsection{Estimation for several particles.}

\begin{proposition}\label{lem-es-multi}
	Under Assumptions (\uppercase\expandafter{\romannumeral1}, \uppercase\expandafter{\romannumeral2}, \uppercase\expandafter{\romannumeral3}, \uppercase\expandafter{\romannumeral4}), consider the electromagnetic scattering problem \eqref{U} for multiple particles $D_j$, $j=1,2,\cdots, \aleph$. Then for $t$ and $h$ satisfying \begin{equation}\label{conditions-t-h-section-2}
3-3t-h \ge 0,~~~ \mbox{ and } ~~~ \frac{9}{11}<h<1,
\end{equation} there holds the estimation
	\begin{equation}\label{max-P1P3}
		\max_j\left\lVert \overset{1}{\mathbb{P}}\left(\tilde{E}_j^T\right)\right\rVert^2_{\mathbb{L}^2(B)}\lesssim a^{2-2h}\quad\mbox{and}\quad
		\max_j\left\lVert \overset{3}{\mathbb{P}}\left(\tilde{E}_j^T\right)\right\rVert^2_{\mathbb{L}^2(B)}\lesssim a^{3+h}.
	\end{equation}
\end{proposition}
From \eqref{max-P1P3}, we expect that for $j=1, 2, \cdots, \aleph$, $\overset{1}{\mathbb{P}}\left(\tilde{E}^T_j\right)$ dominates.
To clear the structure of the paper, we also postpone the proof of Proposition \ref{lem-es-multi} to Subsection \ref{subsec-52}.

\subsubsection{Estimation for the scattering coefficient}
\bigskip

Recall the solution to the electromagnetic scattering problem with the form \eqref{LS eq2}. Let $W$ be the solution to 
	\begin{equation}\label{AI1}
	\left(I + \eta \, \nabla M^{-k} - k^{2} \, \eta \, N^{-k} \right)\left( W \right)(x) = \mathcal{P}(x,z), \quad x \in D,
	\end{equation}
	where $\nabla M^{-k}$ and $N^{-k}$ are the adjoint operators to $\nabla M^k$ and $N^k$ introduced in \eqref{N-M opera}, respectively, and $\mathcal{P}(x, z)$ is the matrix form of the vector $(x-z)$, which can be expressed by
	\begin{equation}\notag
		\mathcal{P}(x, z)=\left(\begin{array}{c}
		(x-z)_1 I_3\\ 
		(x-z)_2 I_3\\ 
		(x-z)_3 I_3
		\end{array} \right),
	\end{equation} 
    with $(x-z)_i$, $i=1, 2, 3$, being the corresponding components. Define the scattering coefficient as
	\begin{equation}\label{def-scoeff}
	\mathcal{C} := \int_{D} W(x) \, dx = \int_{D} \left(I + \eta \nabla M^{-k} - k^{2} \, \eta \, N^{-k} \right)^{-1}\left( \mathcal{P}(\cdot,z)\right)(x) \, dx.
	\end{equation}
	Then there holds the following estimates of $\mathcal{C}$.
	
	\begin{proposition}\label{prop-scoeff}
		The scattering coefficient $\mathcal{C}$ defined by \eqref{def-scoeff} satisfies
		\begin{equation}\label{*add0}
			\mathcal{C}  \sim a^{4-h} \, \sum_{n} \, \frac{1}{\left( 1 + \eta \, \lambda_{n}^{(3)} \right)} \, \langle \overset{1}{\mathbb{P}}\left(N\left(e_{n}^{(3)}\right)\right) ; e^{(1)}_{n_{0}} \rangle  \langle \mathcal{P}(\cdot,0); e^{(1)}_{n_0} \rangle \otimes \int_{B} e^{(3)}_{n}(x) \, dx = \mathcal{O}\left( a^{6-h} \right),
		\end{equation}
		where $\lambda_{n}^{(3)}$, $e_{n_0}^{(1)}$ and $e_n^{(3)}$ are of the eigensystem introduced in Subsection \ref{subsec-eigen}. Moreover, aftering scaling from $D$ to $B$, the solution $W$ to \eqref{AI1} satisfies
		\begin{equation}\label{prop-es-W}
			\left\lVert \overset{1}{\mathbb{P}}\left(\tilde{W}\right)\right\rVert_{\mathbb{L}^2(B)}=\mathcal{O}(a^{1-h}),\quad 
			\left\lVert \overset{2}{\mathbb{P}}\left(\tilde{W}\right)\right\rVert_{\mathbb{L}^2(B)}=\mathcal{O}(a^{3}),\quad
			\left\lVert\overset{3}{\mathbb{P}}\left(\tilde{W}\right)\right\rVert_{\mathbb{L}^2(B)}=\mathcal{O}(a^{3}).
		\end{equation}
	\end{proposition}
We refer to Subsection \ref{subsec-53} for the proof of Proposition \ref{prop-scoeff}.

\section{Linear algebraic system and correponding invertibility discussions.}\label{sec-la-system}

In this section, we first present a linear algebraic system of general form and study the corresponding invertibility condition. Then, by projecting the solution $E$ onto the two subspaces $\mathbb{H}_0(\div=0)$ and $\nabla \mathcal{H}armonic$, we give the precise formulation of the linear algebraic systerms with respect to $\overset{1}{\mathbb{P}}\left(E^T\right)$ and $\overset{3}{\mathbb{P}}\left(E^T\right)$, respectively. The estimates for the significant polarization tensors appearing in the linear algebraic systems related to those two subspaces are also investigated. These expressions and estimations are essentially related to our main result of the Foldy-Lax approximation to the far-field.

\subsection{General linear algebraic system and corresponding invertibility condition.}

Recall that $D_j$, $j=1,2,\cdots, \aleph$, are the cluster of particles introduced in Assumption 	 (\uppercase\expandafter{\romannumeral1}, \uppercase\expandafter{\romannumeral2}, \uppercase\expandafter{\romannumeral3}, \uppercase\expandafter{\romannumeral4}).
Denote by $[P_{D_j}]$, $j=1,2, \cdots, \aleph$, the polarization tensors, which carry the information of the geometry of the scatterer and the material parameters for the electromagnetic scattering problem \eqref{U}.  

\begin{proposition}\label{pro-general la}
	Let $J_j$, $j=1, 2, \cdots, \aleph$, be a sequence of vectors. The linear algebraic system 
	\begin{equation}\label{eq-gene-la}
		Q_{j_0} - \eta \, \sum_{j = 1\atop j\neq j_0}^\aleph [P_{D_{j_0}}] \cdot \Upsilon_k(z_{j_0}, z_j) \cdot Q_j = J_{j_0}, \quad j_0=1, 2, \cdots, \aleph,
	\end{equation}
	is invertible under the condition that
	\begin{equation}\label{invert-condi-gene}
		|\eta| \, \max_{j=1, ..., \aleph}\left\lVert[P_{D_j}]\right\rVert_{\mathbb{L}^{\infty}(\Omega)} \, d^{-3} \, <1.
	\end{equation}
\end{proposition}
We refer to Section \ref{sec-proof-la} for the detailed proof of this proposition. On the basis of Proposition \ref{pro-general la}, next, we present two concrete forms of the linear algebraic systems with respect to $\overset{1}{\mathbb{P}}(E^T_j)$ and $\overset{3}{\mathbb{P}}(E^T_j)$, $j=1, 2, \cdots, \aleph$, which shall play an essential role in the study of our main approximation result for the far-field.

\subsection{Construction of the linear algebraic system for $\overset{1}{\mathbb{P}}(E^T)$.}

Recall that
\begin{equation}\label{Hdiv-curl}
	\mathbb{H}_0(\div=0)=Curl(\mathbb{H}_0(Curl)\cap \mathbb{H}(\div=0)).
\end{equation}
For $j=1,2, \cdots, \aleph$, $\overset{1}{\mathbb{P}}(E^T_j)$ denotes the projection of the total wave $E^T$ onto the subspace $\mathbb{H}_0(\div=0)$ with respect to the particle $D_j$. Then from \eqref{Hdiv-curl}, we can write
\begin{equation}\label{def-F_j}
	\overset{1}{\mathbb{P}}(E^T_j)=Curl (F_j)\quad\mbox{with}\quad \nu\times F_j=0,\ \div(F_j)=0.
\end{equation}
Setting
\begin{equation}\label{Qj1}
	Q_j:=\int_{D_{j}}F_j(y)\,dy,
\end{equation} 
we can construct the following precise form of the linear algebraic system with respect to $\overset{1}{\mathbb{P}}(E^T_j)$, with the help of the Lippmann-Schwinger equation \eqref{LS eq2}.

\begin{proposition}\label{prop-la-1}
	Under Assumption (	\uppercase\expandafter{\romannumeral1}, 	\uppercase\expandafter{\romannumeral2}, 	\uppercase\expandafter{\romannumeral3}, 	\uppercase\expandafter{\romannumeral4}), there holds the following linear algebraic systerm associated with the electromagnetic scattering problem \eqref{U} related to $\overset{1}{\mathbb{P}}(E^T_{j_0})$, $j_0=1,2, \cdots, \aleph$, as
	\begin{equation}\label{la-1}
		Q_{j_0} - \eta \, k^2 \, a^{5-h} \, \sum_{j=1 \atop j \neq j_{0}}^\aleph {\bf{P}}_0 \cdot \Upsilon_k(z_{j_0}, z_j) \cdot Q_j \, = i \, k \, a^{5-h} \, {\bf{P}}_0 \cdot H^{Inc}(z_{j_0})+ Error,
	\end{equation} 
	where  ${\bf{P}}_0$ is defined by $(\ref{defP0})$ and 
	\begin{equation}\label{la-error}
		Error := \mathcal{O}\left(a^{9-2h}d^{-4} \right) + \mathcal{O}\left( a^{6-h}  \right).
	\end{equation}
	Moreover, the linear algebraic systerm \eqref{la-1} is invertible if
\begin{equation}\label{inver-con-1}
	\frac{k^2 \, |\eta| a^5}{d^3 \, \left\vert 1 - k^2 \, \eta \, a^2 \, \lambda_{n_{0}}^{(1)} \right\vert}\left\lVert {\bf{P}}_0 \right\rVert \,\,  <1.
\end{equation}
\end{proposition}

It is easy to see from \eqref{eq-gene-la} and \eqref{la-1} that in Proposition \ref{prop-la-1}, we take
\begin{equation}\notag
	[P_{D_{j_0}}]:=k^2 \, a^{5-h} \, {\bf{P}}_0 \quad \text{and} \quad J_{j_0}:= i \, k \, a^{5-h} \, {\bf{P}}_0 \cdot H^{Inc}(z_{j_0})+Error.
\end{equation}
The detailed discussions on the construction of Proposition \ref{prop-la-1} can be found in Section 6. In a similar manner, in $\nabla\mathcal{H}armonic$, we also have the following form of the linear algebraic systerm with respect to $\overset{3}{\mathbb{P}}(E^T_j)$, $j=1, 2, \cdots, \aleph$.

\subsection{Construction of the linear algebraic system for $\overset{3}{\mathbb{P}}(E^T)$.}

We set here
\begin{equation}\notag
	Q_j:=\int_{D_{j}}\overset{3}{\mathbb{P}}\left(E^T_j\right)(y)\,dy.
\end{equation}

\begin{proposition}\label{prop-la-2}
	Under Assumption (	\uppercase\expandafter{\romannumeral1}, 	\uppercase\expandafter{\romannumeral2}, 	\uppercase\expandafter{\romannumeral3}, 	\uppercase\expandafter{\romannumeral4}), there holds the following linear algebraic system associated with the electromagnetic problem \eqref{U} with respect to $\overset{3}{\mathbb{P}}(E^T_{j_0})$, $j_0=1, 2,\cdots, \aleph$, as
	\begin{eqnarray}\label{la-2}
Q_{j_0} - \eta  \, \sum_{j = 1\atop j\neq j_0}^\aleph {\bf{P}}_{1} \cdot \Upsilon_k(z_{j_0}, z_j) \cdot Q_j &=&  {\bf{P}}_{1} \cdot E^{Inc}_{j_0}(z_{j_0}) 
		+ \mathcal{O}\left( a^{\min(6,8-h-3t)} \right), 
	\end{eqnarray}
where  ${\bf{P}}_{1}$ is the polarization matrix defined by
\begin{equation}\label{DefP1S}
{\bf{P}}_{1} := a^{3} \, \sum_{n} \frac{1}{1 + \eta \, \lambda_{n}^{(3)}} \int_{B}e_n^{(3)}(y)\,dy\otimes\int_{B}e_n^{(3)}(y)\,dy = a^{3} \, \sum_{n} \frac{1}{1 + \eta \, \lambda_{n}^{(3)}} \, \sum_{m} \int_{B}e_{n,m}^{(3)}(y)\,dy\otimes\int_{B}e_{n,m}^{(3)}(y)\,dy, 
\end{equation}		
where $e_{n,m}^{(3)}$ is the eigenvalue such that $\nabla M \left( e_{n,m}^{(3)} \right) = \lambda_{n}^{(3)} \, e_{n,m}^{(3)}$. Moreover, the linear algebraic systerm \eqref{la-2} is invertible if
	\begin{equation}\label{inver-con-2}
		\frac{|\eta| \, a^3}{ d^3} \, \left\lVert \sum_{n} \frac{1}{1 + \eta \, \lambda_{n}^{(3)}} \int_{B}e_n^{(3)}(y)\,dy\otimes\int_{B}e_n^{(3)}(y)\,dy \right\rVert  < 1.
	\end{equation}
\end{proposition}

Combining with \eqref{eq-gene-la} and \eqref{la-2}, we notice that in Proposition \ref{prop-la-2}, we take $[P_{D_{j_0}}]:= {\bf{P}}_{1}$ and 
\begin{equation*}
	J_{j_0} := {\bf{P}}_{1} \cdot E^{Inc}_{j_0}(z_{j_0})  + \mathcal{O}\left(a^{\min(6,8-h-3t)}\right).
\end{equation*}
The proof of Proposition \ref{prop-la-2} can be seen in Section \ref{sec-proof-la}.

\bigskip
\begin{remark}
Indeed, the linear algebraic systems $(\ref{la-1})$ and $(\ref{la-2})$ can be seen as two particular cases of the following general form 
\begin{equation}\label{GAS}
\left[
\mathbf{I}+
\left(
\begin{array}{c|c}
\bm{B_{11}} & \bm{B_{12}} \\ \hline
\bm{B_{21}} & \bm{B_{22}}
\end{array}
\right)\right]
 \cdot \begin{pmatrix} \int_{D_{1}} F_{1}(y) dy \\
\vdots \\
\int_{D_{\aleph}} F_{\aleph}(y) dy
\\ \hline
\int_{D_{1}}\overset{3}{\mathbb{P}}\left(E^T_1\right)(y)\,dy \\
\vdots \\
\int_{D_{\aleph}}\overset{3}{\mathbb{P}}\left(E^T_\aleph\right)(y)\,dy
\end{pmatrix} = \begin{pmatrix}
i \, k \, a^{5-h} \, {\bf{P}}_0 \cdot H^{Inc}(z_{1}) \\
\vdots \\
i \, k \, a^{5-h} \, {\bf{P}}_0 \cdot H^{Inc}(z_{\aleph}) \\ \hline
{\bf{P}}_1 \cdot E^{Inc}(z_{1}) \\ 
\vdots \\
{\bf{P}}_1 \cdot E^{Inc}(z_{\aleph})
\end{pmatrix} + Error,
\end{equation}    
where 
\begin{eqnarray*}
\bm{B_{11}} &:=& - \eta \, k^2 \, a^{5-h} \, 
\begin{pmatrix}
0 &  {\bf{P}}_0 \cdot \Upsilon_k(z_{1}, z_{2}) & \cdots &  {\bf{P}}_0 \cdot \Upsilon_k(z_{1}, z_{\aleph})\\
{\bf{P}}_0 \cdot  \Upsilon_k(z_{2}, z_{1}) & 0 & \cdots &  {\bf{P}}_0 \cdot \Upsilon_k(z_{2}, z_{\aleph}) \\
\vdots & \vdots & \ddots & \vdots \\
{\bf{P}}_0 \cdot \Upsilon_k(z_{\aleph}, z_{1}) & {\bf{P}}_0 \cdot \Upsilon_k(z_{\aleph}, z_{2}) & \cdots & 0
\end{pmatrix} ,\\
&& \\
&& \\
\bm{B_{22}} &:=& - \eta \,  
\begin{pmatrix}
0 &  {\bf{P}}_1 \cdot \Upsilon_k(z_{1}, z_{2}) & \cdots &  {\bf{P}}_1 \cdot \Upsilon_k(z_{1}, z_{\aleph})\\
{\bf{P}}_1 \cdot  \Upsilon_k(z_{2}, z_{1}) & 0 & \cdots &  {\bf{P}}_1 \cdot \Upsilon_k(z_{2}, z_{\aleph}) \\
\vdots & \vdots & \ddots & \vdots \\
{\bf{P}}_1 \cdot \Upsilon_k(z_{\aleph}, z_{1}) & {\bf{P}}_1 \cdot \Upsilon_k(z_{\aleph}, z_{2}) & \cdots & 0
\end{pmatrix}
 \end{eqnarray*}
and $\mathbf{I}$ is the identity matrix. The linear algebraic system $(\ref{GAS})$ is the one corresponding to the Lippmann-Schwinger equation given by: 
\begin{equation}\label{LSERemark}
		E^T(x) + \eta \, \nabla M^k(E^T)(x) - k^2 \, \eta \, N^{k}(E^T)(x) = E^{Inc}(x),\quad x\in D,
\end{equation}
seeing $(\ref{LS eq2})$. Since there are two operators presented in \eqref{LSERemark}, namely the Magnetization operator and the Newtonian operator, we distinguish the following two cases in terms of the choice of the incident frequencies to enhance the scattered electromagnetic fields. 
\begin{enumerate}
\item[]
\item \emph{On the case when the incident frequencies are chosen to generate plasmonic resonances}. In this case, we can prove that the field corresponding to the Magnetization operator $\nabla M(\cdot)$ is the  dominating one in equation $(\ref{LSERemark})$, and consequently, the matrix block $\bm{B_{22}}$ dominates the other blocks. Hence, $(\ref{GAS})$ will be reduced to $(\ref{la-2})$. The analysis of this case for a single plasmonic nano-particle was the object of \cite{GS}. 
\item[]
\item \emph{On the case when the incident frequencies are chosen to generate dielectric resonances}. In this case, we can prove that the field generated by the Newtonian operator $N(\cdot)$ is the  dominating one in equation $(\ref{LSERemark})$, and consequently, the matrix block $\bm{B_{11}}$ dominates the other blocks. Hence, $(\ref{GAS})$ will be reduced to $(\ref{la-1})$. The goal of the current work is to analyze this case in details.
\item[]
\end{enumerate}
Furthermore, regardless of the chosen resonances, we can prove the smallness of the two matrices $\bm{B_{12}}$ and $\bm{B_{21}}$, seeing the proof of Proposition \ref{prop-la-1} and Proposition \ref{prop-la-2}.  
\end{remark}

\section{Proof of Theorem \ref{main-1}.}\label{sec-proof-main}

With all the necessary propositions presented in the previous sections, we give the proof of our main result in this section as follows.

\begin{proof}[Proof of Theorem \ref{main-1}]
	We have
	\begin{eqnarray*}
		E_{\infty}(\hat{x}) &=& \eta \, \left( I - \hat{x} \otimes \hat{x} \right) \; \int_{D} e^{i \, k \, \hat{x} \cdot y} \; E^T(y) \; dy, \quad D = \overset{\aleph}{\underset{m=1}{\cup}} D_{m} \\
		&=& \eta \, \left( I - \hat{x} \otimes \hat{x} \right) \; \sum_{m=1}^{\aleph} \int_{D_{m}} e^{i \, k \, \hat{x} \cdot y} \; E_{m}^T(y) \; dy, 
	\end{eqnarray*}
	where  $E_{m}^T(\cdot) := {E^T(\cdot)_{|}}_{D_{m}}$ and $D_{m} = z_{m} + a \, B$. By developing near the center, we obtain 
	\begin{eqnarray*}
		E_{\infty}(\hat{x}) 
		&=& \eta \, \left( I - \hat{x} \otimes \hat{x} \right) \; \ \sum_{m=1}^{\aleph}  \; \int_{D_{m}} \left[  e^{i \, k \, \hat{x} \cdot z_{m}} + i \, k \, \, e^{i \, k \, \hat{x} \cdot z_{m}} \, (y-z_{m}) \cdot \hat{x} \right]  \; E^T_{m}(y) \; dy \\
		&-& \eta \, \left( I - \hat{x} \otimes \hat{x} \right) \; \ \sum_{m=1}^{\aleph}  \; \int_{D_{m}} \frac{k^2}{2}  \left( (y-z_{m}) \cdot \hat{x} \right)^{2} \, \int_{0}^{1} (1-t) \, e^{i \, k \, \hat{x} \cdot \left( z_{m} + t (y-z_{m}) \right)} \, dt  \; E^T_{m}(y) \; dy \\
		&=& \eta \, \left( I - \hat{x} \otimes \hat{x} \right)  \;\sum_{m=1}^{\aleph} \, e^{i \, k \, \hat{x} \cdot z_{m}} \left[ \int_{D_{m}}  E^T_{m}(y) \, dy + i \, k  \, \int_{D_{m}} \hat{x} \cdot (y-z_{m}) \; E^T_{m}(y) \, dy \right] \\
		&+& \mathcal{O}\left( a^{\frac{3}{2}} \, \aleph \,\max_m \left\Vert E_m^T \right\Vert_{\mathbb{L}^{2}(D_m)} \right).
	\end{eqnarray*}
	Based on Proposition \ref{es-oneP}, we split $E^T_{m}$ as $E^T_{m} := \overset{1}{\mathbb{P}}\left(E^T_{m}\right) + \overset{3}{\mathbb{P}}\left(E^T_{m}\right)$ and we plug it into the previous equation to obtain  
	\begin{eqnarray*}
		E_{\infty}(\hat{x}) & = & \eta \, \left( I - \hat{x} \otimes \hat{x} \right) \, \sum_{m=1}^{\aleph} \, e^{i \, k \, \hat{x} \cdot z_{m}}  \left[ \int_{D_{m}} \overset{3}{\mathbb{P}}\left(E^T_{m}\right)(y) \, dy + i \, k  \, \int_{D_{m}} \hat{x} \cdot (y-z_{m}) \; \overset{1}{\mathbb{P}}\left(E^T_{m}\right)(y) \, dy \right] \\
		& + & \eta \, \left( I - \hat{x} \otimes \hat{x} \right) \, \sum_{m=1}^{\aleph} \, e^{i \, k \, \hat{x} \cdot z_{m}}  \left[ i \, k  \, \int_{D_{m}} \hat{x} \cdot (y-z_{m}) \; \overset{3}{\mathbb{P}}\left(E^T_{m}\right)(y) \, dy  \right] + \mathcal{O}\left( a^{\frac{3}{2}} \, \aleph \,\max_m \left\Vert E^T_m \right\Vert_{\mathbb{L}^{2}(D_m)} \right).
	\end{eqnarray*}
	Since from \eqref{max-P1P3} of Proposition \ref{lem-es-multi}, we have that
	\begin{equation}\notag
	\max_m \left\Vert E^T_m \right\Vert_{\mathbb{L}^{2}(D_m)}=a^{\frac{3}{2}}\max_m \left\Vert \tilde{E}^T_m \right\Vert_{\mathbb{L}^{2}(B)}\lesssim a^{\frac{3}{2}}\cdot a^{1-h}=\mathcal{O}\left(a^{\frac{5}{2}-h}\right),
	\end{equation}
	in which $\left\Vert E^T_m \right\Vert_{\mathbb{L}^{2}(D_m)}$ is dominated by $\left\Vert \overset{1}{\mathbb{P}}\left(E^T_m\right) \right\Vert_{\mathbb{L}^{2}(D_m)}$. Combining with the Cauchy-Schwarz inequality and the scaling property, we can derive
	\begin{eqnarray}\label{mid2}
	E_{\infty}(\hat{x}) & = & \eta \, \left( I - \hat{x} \otimes \hat{x} \right) \, \sum_{m=1}^{\aleph} \, e^{i \, k \, \hat{x} \cdot z_{m}}  \left[ \int_{D_{m}} \overset{3}{\mathbb{P}}\left(E^T_{m}\right)(y) \, dy + i \, k  \, \int_{D_{m}} \hat{x} \cdot (y-z_{m}) \; \overset{1}{\mathbb{P}}\left(E^T_{m}\right)(y) \, dy \right] \notag\\
	& + & \mathcal{O}\left(a^2 \, \aleph \, \max_{m} \left\lVert\overset{3}{\mathbb{P}}\left(\tilde{E}^T_m\right)\right\rVert_{\mathbb{L}^2(B)}\right) + \mathcal{O}\left( a \right),
	\end{eqnarray}
	where we use the fact that $\aleph=O(d^{-3})\sim a^{-3t}$ with $t\leq 1-\frac{h}{3}$, in $\mathbb{R}^3$, which is imposed by the invertibility condition of the linear algebraic systems. In particular, using Proposition \ref{lem-es-multi} again, for $\frac{9}{11}<h<1$, we can further simplify \eqref{mid2} as
	\begin{equation}\notag
	E_{\infty}(\hat{x})  =  \eta \, \left( I - \hat{x} \otimes \hat{x} \right) \, \sum_{m=1}^{\aleph} \, e^{i \, k \, \hat{x} \cdot z_{m}}  \left[ \int_{D_{m}} \overset{3}{\mathbb{P}}\left(E^T_{m}\right)(y) \, dy + i \, k  \, \int_{D_{m}} \hat{x} \cdot (y-z_{m}) \; \overset{1}{\mathbb{P}}\left(E^T_{m}\right)(y) \, dy \right]+\mathcal{O}\left( a \right).
	\end{equation}
	By \eqref{def-F_j}, we know that there holds
	\begin{eqnarray}\label{mid3}
	\nonumber
	E_{\infty}(\hat{x}) &=& \eta \, \left( I - \hat{x} \otimes \hat{x} \right) \, \sum_{m=1}^{\aleph} \, e^{i \, k \, \hat{x} \cdot z_{m}} \int_{D_{m}} \overset{3}{\mathbb{P}}\left(E^T_{m}\right)(y) \, dy \\ &-& \eta \, i \,  k\left( I - \hat{x} \otimes \hat{x} \right)\sum_{m=1}^{\aleph} \, e^{i \, k \, \hat{x} \cdot z_{m}} \, \hat{x}\times \int_{D_{m}} F_m(y)\,dy +\mathcal{O}\left( a \right)\notag\\
	&=& \eta \, \left( I - \hat{x} \otimes \hat{x} \right) \, \sum_{m=1}^{\aleph} \, e^{i \, k \, \hat{x} \cdot z_{m}} \int_{D_{m}} \overset{3}{\mathbb{P}}\left(E^T_{m}\right)(y) \, dy-\eta \, i \,  k\sum_{m=1}^{\aleph} \, e^{i \, k \, \hat{x} \cdot z_{m}} \, \hat{x}\times \int_{D_{m}} F_m(y)\,dy +\mathcal{O}\left( a \right),
	\end{eqnarray}
	where the vector $\left(\int_{D_{m}}\overset{3}{\mathbb{P}}\left(E^T_{m}\right)(y) \, dy\right)_{m=1,\cdots, \aleph}$ is the solution to \eqref{la-2}. Indeed, in \eqref{la-2}, by denoting
	\begin{equation}\label{Sourceterm}
	\mathcal{S}_{ource, j_0}:=  {\bf{P}}_{1}  \cdot E^{Inc,j_{0}}(z_{j_{0}}),
	\end{equation}
	it yields,
	\begin{equation}\label{AYM}
	\int_{D_{j_{0}}} \overset{3}{\mathbb{P}}\left(E^{T}_{j_{0}}\right)(x) \, dx  - \eta \, \sum_{j=1 \atop j \neq j_{0}}^{\aleph} {\bf{P}}_{1} \cdot  \Upsilon_{k}(z_{j_{0}},z_{j}) \cdot \int_{D_{j}} \overset{3}{\mathbb{P}}\left(E^{T}_j\right)(y)\, dy = \mathcal{S}_{ource, j_0}+\mathcal{O}(a^{\min(6,8-h-3t)}).
	\end{equation} 
	Under the invertibility condition \eqref{inver-con-2}, in Proposition \ref{prop-la-2}, now we can give the estimation for the first term in \eqref{mid3} on the basis of \eqref{AYM} as follows.
\begin{eqnarray*}
L_{1} &:=& \eta \, \left( I - \hat{x} \otimes \hat{x} \right) \, \sum_{m=1}^{\aleph} \, e^{i \, k \, \hat{x} \cdot z_{m}} \int_{D_{m}} \overset{3}{\mathbb{P}}\left(E^T_{m}\right)(y) \, dy \\
\left\vert L_{1} \right\vert & \lesssim & \left|\eta \, \left( I - \hat{x} \otimes \hat{x} \right) \, \sum_{m=1}^{\aleph} \, e^{i \, k \, \hat{x} \cdot z_{m}} \mathcal{S}_{ource, m}\right|  \lesssim | \eta | \sum_{m=1}^{\aleph}|\mathcal{S}_{ource, m}|,
\end{eqnarray*}
using the definition of $\mathcal{S}_{ource,m}$ in $(\ref{Sourceterm})$, we obtain that
\begin{equation}\label{es-inver-la1}
\left| L_{1} \right| \lesssim \left\vert \eta \right\vert \, \left\vert {\bf{P}}_{1} \right\vert \, \aleph \overset{(\ref{DefP1S})}{=} \mathcal{O}\left( \left\vert \eta \right\vert \, a^{5} \, \aleph \right)=\mathcal{O}\left(a^h\right), 
\end{equation}
	where we use the fact that $d\sim a^t$, $t\leq 1-\frac{h}{3}$ and $\aleph = \mathcal{O}\left( d^{-3} \right)$, seeing $(\ref{cluster})$.

	Therefore, equation \eqref{mid3} becomes, 
	\begin{equation}\notag
	E_{\infty}(\hat{x})=- i \, k \, \eta \,  \sum_{m=1}^{\aleph} \, e^{i \, k \, \hat{x} \cdot z_{m}}  \left[   \hat{x} \times \int_{D_{m}} F_{m}(y) \, dy \right]+\mathcal{O}\left(a^h\right)=-i \, k \, \eta \,  \sum_{m=1}^{\aleph} \, e^{i \, k \, \hat{x} \cdot z_{m}} \hat{x}\times Q_m+\mathcal{O}\left(a^h\right),
	\end{equation}
	following the notation \eqref{Qj1} presented in Proposition \ref{prop-la-1},
	where the vector $\left(Q_m \right)_{m=1,\cdots,\aleph}$ is the solution to the linear algebraic system \eqref{la-1}.
	
	Consider the linear algebraic system 
	\begin{equation}\label{la-1-plus}
	\hat Q_{m} - \eta \, k^2 \, a^{5-h} \, \sum_{j=1 \atop j \neq m}^\aleph {\bf{P}}_{0} \cdot \Upsilon_k(z_{m}, z_j) \cdot \hat{Q}_{j}
	= i \, k \, a^{5-h} \, {\bf{P}}_{0} \cdot \hat H^{Inc}(z_{m}).
	\end{equation}
	Substracting \eqref{la-1-plus} from \eqref{la-1}, it is clear that there holds
	\begin{equation}\label{diff-Q}
	(Q_{m}-\hat Q_{m}) - \eta \, k^2 \, a^{5-h}\, \sum_{j=1 \atop j \neq m}^\aleph {\bf{P}}_{0} \cdot \Upsilon_k(z_{m}, z_j) \cdot (Q_j-\hat{Q}_j)
	= i \, k \, a^{5-h} \, {\bf{P}}_{0} \cdot ( H^{Inc}(z_{m})-\hat H^{Inc}(z_{m}))+Error,
	\end{equation}
	where $\hat H^{Inc}(z_{m}) = - \, i \, k^{-1} \, a^{h-5} \, {\bf{P}}_{0}^{-1} \cdot Error$ and $Error$ possesses the expression \eqref{la-error}.
	
	Combining with the right hand side of \eqref{la-1} as well as the invertibility condition \eqref{inver-con-1}, we can know by direct verification that for any $m=1, 2, \cdots, \aleph$, $\hat{Q}_m$ fulfills
	\begin{equation*}
	\left(\sum_{m = 1}^\aleph|\hat{Q}_m|^2\right)^\frac{1}{2}\leq 
	\frac{k \, a^{5-h} \left\lVert {\bf{P}}_{0} \right\rVert}{1-	\dfrac{k^2 \, \left\vert \eta \right\vert a^5}{d^3 |1 - k^2 \eta a^2 \lambda_{n_{0}}^{(1)}|}\left\lVert {\bf{P}}_{0} \right\rVert }\left(\sum_{j=1}^\aleph|\hat{H}^{Inc}(z_m)|^2\right)^{\frac{1}{2}},
	\end{equation*}
	which indicates that in \eqref{diff-Q}, we have
	\begin{align}\notag
	\left(\sum_{m = 1}^\aleph|Q_{m}-\hat{Q}_{m}|^2\right)^\frac{1}{2}&\leq 
	\frac{k \, a^{5-h} \, \left\lVert {\bf{P}}_{0}\right\rVert}{1-	\dfrac{k^2 \, |\eta| \, a^5}{d^3 \left\vert 1 - k^2 \, \eta \, a^2 \, \lambda_{n_{0}}^{(1)} \right\vert}\left\lVert {\bf{P}}_{0} \right\rVert} \, \left(a^{2h-10} \, k^{-2} \, \left\lVert {\bf{P}}_{0}^{-1}\right\rVert^2\sum_{m=1}^\aleph|Error|^2\right)^{\frac{1}{2}}\notag\\
	&\leq \frac{\left\lVert {\bf{P}}_{0}\right\rVert \, \left\lVert {\bf{P}}_{0}^{-1}\right\rVert}{1-	\dfrac{ k^2 \, |\eta| \, a^5}{d^3 \,  \left\vert 1 - k^2 \, \eta \, a^2 \, \lambda_{n_{0}}^{(1)} \right\vert}\left\lVert {\bf{P}}_{0} \right\rVert} \, \aleph^{\frac{1}{2}} \, |Error|.\notag
	\end{align}
	Combining with the a-priori estimates \eqref{max-P1P3} given in Proposition \ref{lem-es-multi}, by direct calculations, we can deduce that
	\begin{align}\notag
	E_\infty(\hat{x})&=-i \, k \, \eta \,  \sum_{m=1}^{\aleph} \, e^{i \, k \, \hat{x} \cdot z_{m}} \hat{x}\times Q_{m}+\mathcal{O}\left(a^h\right)\notag\\
	&=-i \, k \, \eta \,  \sum_{m=1}^{\aleph} \, e^{i \, k \, \hat{x} \cdot z_{m}} \hat{x}\times \hat Q_{m}-i \, k \, \eta \,  \sum_{m=1}^{\aleph} \, e^{i \, k \, \hat{x} \cdot z_{m}} \hat{x}\times (Q_{m}-\hat{Q}_{m})+\mathcal{O}\left(a^h\right)\notag\\
	&=-i \, k \, \eta \,  \sum_{m=1}^{\aleph} \, e^{i \, k \, \hat{x} \cdot z_{m}} \hat{x}\times \hat Q_m+\mathcal{O}\left(|\eta|\left(\sum_{m=1}^\aleph|e^{i \, k \, \hat{x} \cdot z_{m}}|^2\right)^{\frac{1}{2}}\cdot\left(\sum_{m=1}^\aleph|Q_m-\hat{Q}_m|^2\right)^{\frac{1}{2}}\right)+\mathcal{O}\left(a^h\right)\notag\\
	&=-i \, k \, \eta \,  \sum_{m=1}^{\aleph} \, e^{i \, k \, \hat{x} \cdot z_{m}} \hat{x}\times \hat Q_m+\mathcal{O}(a^{\frac{h}{3}})\notag
	\end{align}
	by taking $t\leq 1-\frac{h}{3}$ for $\frac{9}{11}<h<1$, which completes the proof.
\end{proof}

\section{Proof of the a-priori estimates of Section \ref{sec-pre}.}\label{sec-proof-prior}

\subsection{Proof of Proposition \ref{es-oneP}.}\label{subsec-51}
	Recall the Lippmann-Schwinger equation in $\mathbb{L}^{2}(D)$ presented in Proposition \ref{prop-LS} that
	\begin{equation*}
	E^{T} + \eta \, \nabla M^{k}(E^{T}) - k^{2}  \, \eta \, N^{k}(E^{T}) = E^{Inc}, \quad \mbox{in }D.
	\end{equation*}
	After scaling to $B$, we obtain
	\begin{equation}\label{Gc}
	\tilde{E}^{T} + \eta \, \nabla M^{ka}(\tilde{E}^{T}) - k^{2}  \, \eta \, a^{2} \, N^{ka}(\tilde{E}^{T}) = \tilde{E}^{Inc}, \quad \mbox{in } B.
	\end{equation}
	To study $(\ref{Gc})$ in $\mathbb{L}^{2}$, we project it onto each subspace by writing $\tilde{E}^{T}$ over $\left( e^{(i)}_{n} \right)_{n \in \mathbb{N}}, i=1,2,3.$ 
	\begin{enumerate}
		\item Taking the $\mathbb{L}^{2}(B)$-inner product with respect to $e^{(1)}_{n}$, then  
		\begin{equation*}
		\langle \tilde{E}^{T}, e^{(1)}_{n} \rangle + \eta \, \langle \nabla M^{ka}(\tilde{E}^{T}),e^{(1)}_{n} \rangle - k^{2}  \, \eta \, a^{2} \, \langle N^{ka}(\tilde{E}^{T}),e^{(1)}_{n} \rangle = \langle \tilde{E}^{Inc},e^{(1)}_{n} \rangle.
		\end{equation*}
		By using the fact that $\nabla M^{ka}$ is vanishing on $\mathbb{H}_{0}(\div=0)$, we obtain 
		\begin{eqnarray*}
			\nonumber
			\langle \tilde{E}^{T}, e^{(1)}_{n} \rangle \left( 1 - k^{2}  \, \eta \, a^{2} \, \lambda^{(1)}_{n} \right)  &=& \langle \tilde{E}^{Inc},e^{(1)}_{n} \rangle + k^{2} \, \eta \, a^{2} \, \langle \left( N^{ka}- N^{0}\right) (\tilde{E}^{T}),e^{(1)}_{n} \rangle \\ \nonumber
			& \overset{(\ref{expansion-Nk})}{=} & \langle \tilde{E}^{Inc},e^{(1)}_{n} \rangle + k^{2} \, \eta \, a^{2} \, \frac{i \, k \, a}{4 \, \pi} \, \langle \int_{B} \tilde{E}^{T}(y) \, dy ,  e^{(1)}_{n} \rangle  \\
			&+& k^{2} \, \eta \, a^{2} \, \frac{1}{4 \, \pi} \, \sum_{j \geq 1} \left( i k a \right)^{j+1} \langle \int_{B} \frac{\left\vert \cdot - y \right\vert^{j}}{(j+1)!} \tilde{E}^{T}(y) \, dy , e^{(1)}_{n} \rangle. 
		\end{eqnarray*}
		From the mutual orthogonality of the decomposed three subspaces introduced in \eqref{L2-decomposition}, since constant vectors are in $ \nabla\mathcal{H}armonic$, we deduce that $\int_{B} e^{(1)}_{n}(y) \, dy = 0, \; \forall \; n \, \in \mathbb{N}$.
		Then, 
		\begin{equation*}\label{LGSMc}
		\langle \tilde{E}^{T}, e^{(1)}_{n} \rangle   = \frac{1}{\left( 1 - k^{2}  \, \eta \, a^{2} \, \lambda^{(1)}_{n} \right)} \left[ \langle \tilde{E}^{Inc},e^{(1)}_{n} \rangle +  \frac{k^{2} \, \eta \, a^{2}}{4 \, \pi} \, \sum_{j \geq 1} \left( i k a \right)^{j+1} \langle \int_{B} \frac{\left\vert \cdot - y \right\vert^{j}}{(j+1)!} \tilde{E}^{T}(y) \, dy , e^{(1)}_{n} \rangle \right]. 
		\end{equation*}

Taking the modulus, there holds
		\begin{eqnarray*}
			\left\vert \langle \tilde{E}^{T}, e^{(1)}_{n} \rangle \right\vert  & \lesssim &  \frac{1}{\left\vert 1 - k^{2}  \, \eta \, a^{2} \, \lambda^{(1)}_{n} \right\vert} \left[ \left\vert \langle \tilde{E}^{Inc},e^{(1)}_{n} \rangle \right\vert + a^{2} \, \left( \sum_{j \geq 1}  \left\vert \langle \int_{B} \frac{\left\vert \cdot - y \right\vert^{j}}{(j+1)!} \, \tilde{E}^{T}(y) \, dy, e^{(1)}_{n} \rangle \right\vert^{2} \right)^{\frac{1}{2}} \right] ,\\
		\end{eqnarray*}
		This implies,  
		\begin{eqnarray}\label{12120111}
		\nonumber
		\left\Vert \overset{1}{\mathbb{P}}\left( \tilde{E}^{T} \right) \right\Vert^{2}_{\mathbb{L}^{2}(B)} &:=&  \sum_{n} \left\vert \langle  \tilde{E}^{T}, e^{(1)}_{n} \rangle \right\vert^{2} \\ \nonumber
		& \lesssim & \sum_{n} \frac{\left\vert \langle \tilde{E}^{Inc},e^{(1)}_{n} \rangle \right\vert^{2}}{\left\vert 1 - k^{2}  \, \eta \, a^{2} \, \lambda^{(1)}_{n} \right\vert^{2}}   + a^{4-2h} \,\sum_{n}  \sum_{j \geq 1}  \left\vert \langle \int_{B} \frac{\left\vert \cdot - y \right\vert^{j}}{(j+1)!} \, \tilde{E}^{T}(y) \, dy, e^{(1)}_{n} \rangle \right\vert^{2}  \\
		& \lesssim & \sum_{n} \frac{\left\vert \langle \tilde{E}^{Inc},e^{(1)}_{n} \rangle \right\vert^{2}}{\left\vert 1 - k^{2}  \, \eta \, a^{2} \, \lambda^{(1)}_{n} \right\vert^{2}}    + a^{4-2h}  \, \left\Vert \tilde{E}^{T} \right\Vert^{2}_{\mathbb{L}^{2}(B)}.
		\end{eqnarray}
		\item Taking the $\mathbb{L}^{2}(B)$-inner product with respect to $e^{(2)}_{n}$, then \\
		\begin{equation*}
		\langle \tilde{E}^{T} , e_{n}^{(2)} \rangle + \eta \, \langle \nabla M^{ka}(\tilde{E}^{T}) , e_{n}^{(2)} \rangle - k^{2}  \, \eta \, a^{2} \, \langle N^{ka}(\tilde{E}^{T}), e_{n}^{(2)} \rangle  = \langle \tilde{E}^{Inc}, e_{n}^{(2)} \rangle.
		\end{equation*}
		Passing to the adjoint operators for both $N^{ka}$ and $\nabla M^{ka}$, then 
		\begin{equation}\label{aze}
		\langle \tilde{E}^{T} , e_{n}^{(2)} \rangle + \eta \, \langle \tilde{E}^{T} , \nabla M^{-ka}\left( e_{n}^{(2)} \right) \rangle - k^{2}  \, \eta \, a^{2} \, \langle \tilde{E}^{T}, N^{-ka}\left( e_{n}^{(2)} \right) \rangle  = \langle \tilde{E}^{Inc}, e_{n}^{(2)} \rangle.
		\end{equation}
		Since, for $x \in B$, we have
		\begin{equation*}
		\nabla M^{-ka}\left( e_{n}^{(2)} \right)(x) := \underset{x}{\nabla} \int_{B} \underset{y}{\nabla} \Phi_{-ka}(x,y) \cdot  e_{n}^{(2)}(y) \, dy = - \underset{x}{\nabla} \underset{x}{\div} \int_{B}  \Phi_{-ka}(x,y)  e_{n}^{(2)}(y) \, dy = - \underset{x}{\nabla} \underset{x}{\div} N^{-ka}\left(e_{n}^{(2)}\right)(x).
		\end{equation*}
		Using the identity that $\Delta + Curl \circ Curl = \nabla \, \div$, we get
		\begin{eqnarray*}
			\nabla M^{-ka}\left( e_{n}^{(2)} \right)  &=& - \left( \underset{x}{\Delta} + \underset{x}{Curl} \circ \underset{x}{Curl} \right) N^{-ka}\left(e_{n}^{(2)}\right) \\
			&=& (-ka)^{2} N^{-ka}\left(e_{n}^{(2)}\right) + e_{n}^{(2)} -  \underset{x}{Curl} \left( N^{-ka}\left(Curl \, e_{n}^{(2)} \right) - SL^{-ka}\left( \nu \times e_{n}^{(2)} \right) \right),
		\end{eqnarray*}
	where $SL^{-ka}$ is the Single Layer operator 
	\begin{equation}\notag
		SL^{-ka}(\nu\times e_n^{(2)})(\cdot) := \int_{\partial B}\Phi_{-ka}(\cdot, y)\cdot(\nu\times e_n^{(2)})(y)\,d\sigma(y).
	\end{equation}
		Since $Curl \left( e_{n}^{(2)} \right) = 0$ in the domain and $\nu \times e_{n}^{(2)} = 0$ on the boundary, the previous equation will be reduced to  
		\begin{equation}\label{grad-M-k-2nd-subspace}
		\nabla M^{-ka}\left( e_{n}^{(2)} \right)  = (-ka)^{2} N^{-ka}\left(e_{n}^{(2)}\right) + e_{n}^{(2)}.
		\end{equation}
		By plugging this into equation $(\ref{aze})$, we obtain 
		\begin{equation*}
		\langle \tilde{E}^{T} , e_{n}^{(2)} \rangle \left( 1 + \eta \right) = \langle \tilde{E}^{Inc}, e_{n}^{(2)} \rangle.
		\end{equation*}
		Since $\tilde{E}^{Inc} \in \mathbb{H}\left( \div = 0 \right) = \left( \mathbb{H}_{0}\left( \div = 0 \right) \oplus \nabla \mathcal{H}armonic \right) \perp \mathbb{H}_{0}\left( Curl = 0 \right)$, then 
		\begin{equation*}
		\langle \tilde{E}^{T} , e_{n}^{(2)} \rangle = \left( 1 + \eta \right)^{-1}  \langle \tilde{E}^{Inc}, e_{n}^{(2)} \rangle = 0,
		\end{equation*}
		and
		\begin{equation}\label{Exact2ndproj}
		\left\Vert \overset{2}{\mathbb{P}}\left( \tilde{E}^{T} \right) \right\Vert^{2}_{\mathbb{L}^{2}(B)} := \sum_{n} \left\vert \langle  \tilde{E}^{T}, e^{(2)}_{n} \rangle \right\vert^{2} = 0,
		\end{equation}
		which proves \eqref{*add1}.
		\item Taking the $\mathbb{L}^{2}(B)$-inner product with respect to $e^{(3)}_{n}$, then 
		\begin{eqnarray*}
			\langle \tilde{E}^{T}, e^{(3)}_{n} \rangle + \eta \, \langle \nabla M^{ka}(\tilde{E}^{T}),e^{(3)}_{n} \rangle - k^{2}  \, \eta \, a^{2} \, \langle N^{ka}(\tilde{E}^{T}),e^{(3)}_{n} \rangle = \langle \tilde{E}^{Inc},e^{(3)}_{n} \rangle.
		\end{eqnarray*} 
		Since on this subspace, we have $\nabla M(e^{(3)}_{n}) = \lambda^{(3)}_{n} e^{(3)}_{n}$, then we obtain that 
		\begin{eqnarray}\label{qdfsg}
		\nonumber
		\langle \tilde{E}^{T}, e^{(3)}_{n} \rangle \left( 1 + \eta  \, \lambda^{(3)}_{n} \right) &=& \langle \tilde{E}^{Inc},e^{(3)}_{n} \rangle + k^{2} \, \eta \, a^{2} \, \langle N^{ka}(\tilde{E}^{T}),e^{(3)}_{n} \rangle  
		- \eta \, \langle \left(\nabla M^{ka} - \nabla M \right)\left( \tilde{E}^{T} \right),  e^{(3)}_{n} \rangle \\
		\langle \tilde{E}^{T}, e^{(3)}_{n} \rangle  &=& \frac{1}{\left( 1 + \eta  \, \lambda^{(3)}_{n} \right)} \left[ \langle \tilde{E}^{Inc},e^{(3)}_{n} \rangle + \eta \, \langle \left(  k^{2} \, a^{2} \, N^{ka} - \nabla M^{ka} + \nabla M \right)\left( \tilde{E}^{T} \right),  e^{(3)}_{n} \rangle \right].
		\end{eqnarray}
		Using $(\ref{expansion-gradMk})$ and $(\ref{expansion-Nk})$, we rewrite the second term on the right hand side of the previous equation  as
		\begin{eqnarray}\label{DefOperatorS}
		\nonumber
\mathcal{S}\left( \tilde{E}^{T} \right) &:=& \left( k^{2} \, a^{2} \, N^{ka} - \nabla M^{ka}  + \nabla M \right) \left( \tilde{E}^{T} \right) \\ \nonumber &=&  \frac{\left( ka \right)^{2}}{2} \, N\left( \tilde{E}^{T} \right)  + \frac{\left( ka \right)^{2}}{2} \int_{B} \Phi_{0}(\cdot,y) \frac{A(\cdot,y)\cdot \tilde{E}^{T}(y)}{\Vert x -y \Vert^{2}}  dy +  \frac{i \left( ka \right)^{3}}{6 \pi} \int_{B} \tilde{E}^{T}(y) dy \\ &+& \frac{1}{4 \pi} \; \sum_{j \geq 3} \left( ika \right)^{j+1} \; \int_{B} \; \frac{\underset{x}{Hess} \left( \left\Vert x - y \right\Vert^{j} \right)}{(j+1)!} \cdot \tilde{E}^{T}(y) \, dy +   \frac{k^{2}a^{2}}{4 \pi} \; \sum_{j \geq 1} (ika)^{j+1} \; \int_{B} \; \frac{\left\Vert \cdot - y \right\Vert^{j}}{(j+1)!}  \tilde{E}^{T}(y) \, dy.
		\end{eqnarray}   
		Then equation $(\ref{qdfsg})$ becomes,  
		\begin{eqnarray}\label{Coeff-Scale-E-3rd}
		\nonumber
		\langle \tilde{E}^{T}, e^{(3)}_{n} \rangle  &=& \frac{\langle \tilde{E}^{Inc},e^{(3)}_{n} \rangle }{\left( 1 + \eta  \, \lambda^{(3)}_{n} \right)} + \frac{\eta}{\left( 1 + \eta  \, \lambda^{(3)}_{n} \right)}  \Bigg[ \frac{\left( ka \right)^{2}}{2} \, \langle N\left( \tilde{E}^{T} \right) ,  e^{(3)}_{n} \rangle  \\ \nonumber && \qquad \qquad \qquad + \frac{\left( ka \right)^{2}}{2} \, \langle \int_{B} \Phi_{0}(\cdot,y) \frac{A(\cdot,y)\cdot \tilde{E}^{T}(y)}{\Vert x -y \Vert^{2}}  dy , e^{(3)}_{n} \rangle  + \frac{i \left( ka \right)^{3}}{6 \pi} \langle  \int_{B} \tilde{E}^{T}(y) dy , e^{(3)}_{n} \rangle \\ \nonumber && \qquad \qquad \qquad + \frac{1}{4 \pi} \; \sum_{j \geq 3} \left( ika \right)^{j+1} \langle \int_{B} \; \frac{\underset{x}{Hess} \left( \left\Vert x - y \right\Vert^{j} \right)}{(j+1)!} \cdot \tilde{E}^{T}(y) \, dy ,  e^{(3)}_{n} \rangle \\ && \qquad \qquad \qquad + \frac{k^{2}a^{2}}{4 \pi} \; \sum_{j \geq 1} (ika)^{j+1} \langle \int_{B} \; \frac{\left\Vert \cdot - y \right\Vert^{j}}{(j+1)!}  \tilde{E}^{T}(y) \, dy , e^{(3)}_{n} \rangle \Bigg],
		\end{eqnarray}
		and 
		\begin{eqnarray*}
			\left\vert \langle \tilde{E}^{T}, e^{(3)}_{n} \rangle \right\vert^{2} &\lesssim & \frac{\left\vert \langle \tilde{E}^{Inc},e^{(3)}_{n} \rangle \right\vert^{2}}{\left\vert 1 + \eta  \, \lambda^{(3)}_{n} \right\vert^{2}} +   \Bigg[a^{4} \left\vert \langle N\left( \tilde{E}^{T} \right) ,  e^{(3)}_{n} \rangle \right\vert^{2} + a^{4} \left\vert \langle \int_{B} \Phi_{0}(\cdot,y) \frac{A(\cdot,y)\cdot \tilde{E}^{T}(y)}{\Vert x -y \Vert^{2}}  dy ,  e^{(3)}_{n} \rangle  \right\vert^{2} \\ &+&  a^{6} \left\vert \langle  \int_{B} \tilde{E}^{T}(y) dy , e^{(3)}_{n} \rangle \right\vert^{2} + a^{8} \sum_{j \geq 3}  \left\vert \langle \int_{B} \; \frac{\underset{x}{Hess} \left( \left\Vert \cdot - y \right\Vert^{j} \right)}{(j+1)!} \cdot \tilde{E}^{T}(y) \, dy ,  e^{(3)}_{n} \rangle \right\vert^{2}  \\ && \qquad \qquad \qquad + a^{8} \;  \sum_{j \geq 1}  \left\vert \langle \int_{B} \; \frac{\left\Vert \cdot - y \right\Vert^{j}}{(j+1)!}  \tilde{E}^{T}(y) \, dy , e^{(3)}_{n} \rangle \right\vert^{2} \Bigg].
		\end{eqnarray*}
		Taking the series with respect to $n$ on the both sides and using the continuity of the Newtonian potential operator, we deduce that 
		\begin{eqnarray}\label{Hsah}
		\nonumber
		\left\Vert \overset{3}{\mathbb{P}}\left( \tilde{E}^{T} \right) \right\Vert^{2}_{\mathbb{L}^{2}(B)} &\lesssim & a^{4} \; \left\Vert \overset{3}{\mathbb{P}}\left( \tilde{E}^{Inc} \right) \right\Vert^{2}_{\mathbb{L}^{2}(B)} +   \Bigg[a^{4} \left\Vert  \tilde{E}^{T}  \right\Vert^{2}_{\mathbb{L}^{2}(B)} + a^{4} \left\Vert  \tilde{E}^{T}  \right\Vert^{2}_{\mathbb{L}^{2}(B)} + a^{6} \left\Vert  \tilde{E}^{T}  \right\Vert^{2}_{\mathbb{L}^{2}(B)} \\ \nonumber
		&+&   a^{8} \left\Vert  \tilde{E}^{T}  \right\Vert^{2}_{\mathbb{L}^{2}(B)} \sum_{j \geq 3}  \int_{B} \int_{B}  \frac{\left\vert \underset{x}{Hess} \left( \left\Vert x - y \right\Vert^{j} \right) \right\vert^{2} }{ \left( (j+1)! \right)^{2}}  dy dx  \\ &+&  a^{8} \; \left\Vert  \tilde{E}^{T}  \right\Vert^{2}_{\mathbb{L}^{2}(B)} \sum_{j \geq 1}  \int_{B}  \int_{B} \; \frac{\left\Vert x - y \right\Vert^{2j}}{\left( (j+1)! \right)^{2}} \, dy dx \Bigg].
		\end{eqnarray}
		We prove that the two series, on $j$ index, are converging. For the first one, we have\footnote{Straightforward computations allow us to deduce that 
			\begin{equation*}
			\underset{x}{Hess} \left( \left\Vert x - y \right\Vert^{j} \right) = j \, \left\Vert x - y \right\Vert^{j-2} \, I + j \, (j-2) \, \left\Vert x - y \right\Vert^{j-4} (x-y) \otimes (x-y). 
			\end{equation*}} 
		\begin{equation}\label{Series1}
		S_{1} := \sum_{j \geq 3}  \int_{B} \int_{B}  \frac{\left\vert \underset{x}{Hess} \left( \left\Vert x - y \right\Vert^{j} \right) \right\vert^{2} }{ \left( (j+1)! \right)^{2}}  dy dx \leq \left\vert B \right\vert^{2} \, \sum_{j \geq 3}  \frac{j^{2} \, \left( diam(B)^{2}\right)^{(j-2)}}{ \left( (j+1)! \right)^{2}} < +\infty,
		\end{equation} 
		and for the second one, we have 
		\begin{eqnarray}\label{Series2}
		S_{2} & := & \underset{j \geq 1}{\sum}  \int_{B}  \int_{B} \; \frac{\left\Vert x - y \right\Vert^{2j}}{\left( (j+1)! \right)^{2}} \, dy dx \leq  \vert B \vert^{2} \, \underset{j \geq 1}{\sum}  \; \frac{\left(diam(B)^{2}\right)^{j}}{(j+1)! (j+1)!} < +\infty,
		\end{eqnarray}
		where $diam(B)$ is the diameter of the domain $B$.\, \footnote{Recall that $\underset{n}{\sum} a_{n}$ converge if $\underset{n}{\lim} \; ( a_{n+1} / a_{n}) \overset{n \rightarrow +\infty}{\longrightarrow} 0$. } Equation $(\ref{Hsah})$ becomes, 
		\begin{equation}\label{LZ01N}
		\left\Vert \overset{3}{\mathbb{P}}\left( \tilde{E}^{T} \right) \right\Vert^{2}_{\mathbb{L}^{2}(B)} \lesssim  a^{4} \; \left\Vert \overset{3}{\mathbb{P}}\left( \tilde{E}^{Inc} \right) \right\Vert^{2}_{\mathbb{L}^{2}(B)} +   a^{4} \left\Vert  \tilde{E}^{T}  \right\Vert^{2}_{\mathbb{L}^{2}(B)}.
		\end{equation}
	\end{enumerate}
	Now, by adding $(\ref{12120111})$,$(\ref{Exact2ndproj})$ and  $(\ref{LZ01N})$, we can obtain that for 
	\begin{equation}\label{h1}
	h<2,
	\end{equation}
	there holds the following formula
	\begin{eqnarray*}
		\left\Vert \tilde{E}^{T}  \right\Vert^{2}_{\mathbb{L}^{2}(B)} &:=& \sum_{j=1}^{3} \left\Vert \overset{j}{\mathbb{P}}\left( \tilde{E}^{T} \right) \right\Vert^{2}_{\mathbb{L}^{2}(B)}, \\ 
		\left\Vert \tilde{E}^{T}  \right\Vert^{2}_{\mathbb{L}^{2}(B)} & \lesssim & \sum_{n} \frac{\left\vert \langle \tilde{E}^{Inc},e^{(1)}_{n} \rangle \right\vert^{2}}{\left\vert 1 - k^{2}  \, \eta \, a^{2} \, \lambda^{(1)}_{n} \right\vert^{2}}   +  a^{4} \; \left\Vert \overset{3}{\mathbb{P}}\left( \tilde{E}^{Inc} \right) \right\Vert^{2}_{\mathbb{L}^{2}(B)} 
\end{eqnarray*}
Since $\left\Vert \overset{3}{\mathbb{P}}\left( \tilde{E}^{Inc} \right) \right\Vert_{\mathbb{L}^{2}(B)} \sim 1$, we deduce: 
\begin{equation*}
	\left\Vert \tilde{E}^{T}  \right\Vert^{2}_{\mathbb{L}^{2}(B)}	 \lesssim  a^{-2h} \, \left\vert \langle \tilde{E}^{Inc},e^{(1)}_{n_{0}} \rangle \right\vert^{2} + \sum_{n \neq n_{0}} \left\vert \langle \tilde{E}^{Inc},e^{(1)}_{n} \rangle \right\vert^{2} + a^{4}.
	\end{equation*}
	To compute $\langle \tilde{E}^{Inc},e^{(1)}_{n} \rangle$, we write $e^{(1)}_{n} = Curl \left( \phi_{n} \right)$, then with integration by parts, we have
	\begin{equation*}
	\langle \tilde{E}^{Inc},e^{(1)}_{n} \rangle = \langle \tilde{E}^{Inc}, Curl\left( \phi_{n} \right) \rangle = \langle Curl\left( \tilde{E}^{Inc} \right), \phi_{n} \rangle = i \, k \, a \langle  \tilde{H}^{Inc}, \phi_{n} \rangle \sim a.
	\end{equation*}
	Then, 
	\begin{equation*}
		\left\Vert \tilde{E}^{T}  \right\Vert^{2}_{\mathbb{L}^{2}(B)}  \lesssim  a^{2-2h} \, \left\vert \langle \tilde{H}^{Inc}, \phi_{n_{0}} \rangle \right\vert^{2} + a^{2} \, \sum_{n \neq n_{0}} \left\vert \langle \tilde{H}^{Inc}, \phi_{n} \rangle \right\vert^{2} +  a^{4} 
		\lesssim  a^{2-2h} + a^{2}  + a^{4}. 
	\end{equation*}
	Therefore we deduce that 
	\begin{equation}\label{A priori estimate Electric Field}
	\left\Vert \tilde{E}^{T}  \right\Vert_{\mathbb{L}^{2}(B)} \sim a^{1-h}. 
	\end{equation}

This ends the proof of Proposition \ref{es-oneP}.


\subsection{Proof of Proposition \ref{lem-es-multi}.}\label{subsec-52}

Before presenting the proof of Proposition \ref{lem-es-multi}, to write short formulas, we set  the following volume integral operator: 
\begin{equation*}
\boldsymbol{T}_{k} :=  \left(I + \eta \, \nabla M^{k} - k^{2} \, \eta \, N^{k} \right).
\end{equation*}
Since the operators $\nabla M^{-k}$ and $N^{-k}$ are the adjoint operators of $\nabla M^{k}$ and $N^{k}$, respectively, we deduce that $\boldsymbol{T}_{- \, k}$ is the adjoint operator of $\boldsymbol{T}_{k}$. 
In addition, we state the following lemma.
\begin{lemma}\label{AI4}
Recall that $e_n^{(1)}$ is the eigenfunction given by Subsection \ref{subsec-eigen}. Then there holds
	\begin{equation}\label{equa-lemmma-21}
	\boldsymbol{T}_{- \, k \, a}^{-1}\left( e_{n}^{(1)} \right) =  \frac{1}{\left(1 - k^{2} \, \eta \, a^{2} \, \lambda_{n}^{(1)} \right)} e_{n}^{(1)} + R_{n},  
	\end{equation}
	where
	\begin{equation}\label{Def-Rn}
	R_{n} := \frac{1 \mp c_0 a^{h}}{4 \pi \,\lambda_{n_{0}}^{(1)} \, \left(1 - k^{2} \, \eta \, a^{2} \, \lambda_{n}^{(1)} \right)} \sum_{\ell \geq 1} (-ika)^{\ell+1}  \boldsymbol{T}_{- \, k \, a}^{-1}  \int_{B} \; \frac{\left\Vert \cdot - y \right\Vert^{\ell}}{(\ell+1)!}  e_{n}^{(1)}(y) \, dy. 
	\end{equation}
\end{lemma}
\begin{proof}
	\begin{eqnarray*}
\boldsymbol{T}_{- \, k \, a}\left( e_{n}^{(1)} \right) & = & \left( 1 - k^{2} \, \eta \, a^{2} \, \lambda_{n}^{(1)} \right) e_{n}^{(1)} - k^{2} \, \eta \, a^{2} \left( N^{-ka} - N \right)\left(e_{n}^{(1)}\right) \\
		& \overset{(\ref{expansion-Nk})}{=} &  \left( 1 - k^{2} \, \eta \, a^{2} \, \lambda_{n}^{(1)} \right) e_{n}^{(1)} - \frac{k^{2} \, \eta \, a^{2}}{4 \pi} \; \sum_{\ell \geq 1} \frac{(-ika)^{\ell+1}}{(\ell+1)!} \; \int_{B} \; \left\Vert \cdot - y \right\Vert^{\ell}  e_{n}^{(1)}(y) \, dy \\
		\\
		& \overset{(\ref{choice-k-1st-regime})}{=} &  \left( 1 - k^{2} \, \eta \, a^{2} \, \lambda_{n}^{(1)} \right) e_{n}^{(1)} - \frac{1 \mp c_0 \, a^{h}}{4 \pi \, \lambda_{n_{0}}^{(1)}} \; \sum_{\ell \geq 1} \frac{(-ika)^{\ell+1}}{(\ell+1)!} \; \int_{B} \; \left\Vert \cdot - y \right\Vert^{\ell}  e_{n}^{(1)}(y) \, dy.
	\end{eqnarray*}
	The remainder of the proof consists in taking the inverse of the operator appearing on the left hand side. 
\end{proof}

Now we prove Proposition \ref{lem-es-multi} as follows.
Recall the Lippmann-Schwinger equation \eqref{LS eq2}
\begin{eqnarray*}
	E^{T}(x) + \eta \, \sum_{j=1}^{\aleph} \underset{x}{\nabla}M^{k}(E_{j}^{T})(x) - k^{2}  \, \eta \, \sum_{j=1}^{\aleph} \, N^{k}(E_{j}^{T})(x) &=& E^{Inc}(x), \quad x \in D = \overset{\aleph}{\underset{j=1}{\cup}} D_{j}.
\end{eqnarray*}
By Proposition \ref{es-oneP}, from \eqref{Exact2ndproj}, we can write $E_{j}^{T} = \overset{1}{\mathbb{P}}\left( E_{j}^{T} \right) + \overset{3}{\mathbb{P}}\left( E_{j}^{T} \right), \; j=1,\cdots,\aleph$. We split the study into two parts. 
\begin{enumerate}
	\item[$i)$] \textit{Estimation of} $\underset{j}{\max}  \left\Vert  \overset{1}{\mathbb{P}}\left(\tilde{E}_{j}^{T} \right) \right\Vert^{2}$ \\
	For $x \in D_{j_{0}}$, we have
	\begin{equation}\label{*add3}
\boldsymbol{T}_{k} \, \left( E_{j_{0}}^{T} \right)(x) + \eta \, \sum_{j \neq j_{0}}^{\aleph} \left( \underset{x}{\nabla}M^{k} - k^{2} \, N^{k} \right) (E_{j}^{T})(x) = E^{Inc}_{j_{0}}(x). 
	\end{equation}
	Taking the inverse operator of $\boldsymbol{T}_{k}$ on the both sides of \eqref{*add3}, we get 
	\begin{eqnarray*}
		E_{j_{0}}^{T}(x) &+& \eta \, \boldsymbol{T}_{k}^{-1} \sum_{j \neq j_{0}}^{\aleph} \left(\underset{x}{\nabla}M^{k} - k^{2} \, N^{k} \right)\left(E_{j}^{T}\right)(x)  = \boldsymbol{T}_{k}^{-1}\left(E^{Inc}_{j_{0}}\right)(x),\\
		\overset{1}{\mathbb{P}}\left(E_{j_{0}}^{T}\right)(x)  &-& \eta \,\sum_{j \neq j_{0}}^{\aleph} \, \boldsymbol{T}_{k}^{-1}  \int_{D_{j}} \Upsilon(x,y) \cdot \overset{1}{\mathbb{P}}\left(E_{j}^{T}\right)(y) \, dy   \\ 
		&=&\boldsymbol{T}_{k}^{-1}\left(E^{Inc}_{j_{0}}\right)(x)-\overset{3}{\mathbb{P}}\left(E_{j_{0}}^{T}\right)(x) + \eta \,\sum_{j \neq j_{0}}^{\aleph} \, \boldsymbol{T}_{k}^{-1}  \int_{D_{j}} \Upsilon(x,y) \cdot \overset{3}{\mathbb{P}}\left(E_{j}^{T}\right)(y) \, dy.  
	\end{eqnarray*}
	Now, we have 
	\begin{eqnarray*}
		\int_{D_{j}} \Upsilon(x,y) \cdot \overset{1}{\mathbb{P}}\left(E_{j}^{T}\right)(y) \, dy &:=& \left( -  \, \underset{x}{\nabla} M^{k} + k^{2}  \, N^{k}  \right) \left( \overset{1}{\mathbb{P}}\left(E_{j}^{T}\right) \right)(x) = k^{2} \; N^{k} \left( \overset{1}{\mathbb{P}}\left(E_{j}^{T}\right) \right)(x) \\
		&=& k^{2} \; \int_{D_{j}}  \Phi_{k}(x,y)  \overset{1}{\mathbb{P}}\left(E_{j}^{T}\right)(y) \; dy. 
	\end{eqnarray*}
	Expanding\footnote{Recall the Taylor expansion for a function of several variables: 
		\begin{equation*}
		f(x+h) = f(x) + \sum_{k=1}^{n} \frac{1}{k!} \bm{d}^{k}f_{x}\left( h^{[k]} \right) + \int_{0}^{1} \frac{(1-t)^{n}}{n!} \, \bm{d}^{n+1}f_{x+th}\left( h^{[n+1]} \right) \, dt.  
		\end{equation*}
	} near the center $z_{j}$ and using the fact that $\int_{D_{j}} \overset{1}{\mathbb{P}}\left(E_{j}^{T}\right)(y) \; dy = 0$, we obtain
	\begin{eqnarray*}
&&		\int_{D_{j}} \Upsilon(x,y) \cdot \overset{1}{\mathbb{P}}\left(E_{j}^{T}\right)(y) \, dy =  k^{2} \underset{y}{\nabla}\left(\Phi_{k} I \right)(z_{j_{0}} ,z_{j})\cdot  \int_{D_{j}} \mathcal{P}(y;z_{j})\cdot   \overset{1}{\mathbb{P}}\left(E_{j}^{T}\right)(y) \; dy \\ &+& k^{2} \; \int_{D_{j}} \int_{0}^{1} (1-t) (y-z_{j})^{\perp}\cdot \underset{y}{Hess}\left(\Phi_{k} \right)(z_{j_{0}},z_{j}+t(y-z_{j}))\cdot (y-z_{j}) dt  \; \overset{1}{\mathbb{P}}\left(E_{j}^{T}\right)(y) \; dy \\ &+& Err(x,j_{0},j,\overset{1}{\mathbb{P}}\left(E_{j}^{T}\right)),  
	\end{eqnarray*}
	where
	\begin{eqnarray}\label{LMERR1}
	\nonumber
	&& Err(x,j_{0},j,\overset{1}{\mathbb{P}}\left(E_{j}^{T}\right)) := \\ \nonumber &-& k^{2} \int_{D_{j}} \int_{0}^{1}  \underset{y}{Hess}\left(\Phi_{k} \right)(z_{j_{0}}+t(x-z_{j_{0}}),z_{j}) \cdot (x-z_{j_{0}}) \, dt \cdot (y-z_{j})  \overset{1}{\mathbb{P}}\left(E_{j}^{T}\right)(y) \; dy \\ \nonumber
	&+& k^{2} \; \int_{D_{j}} \int_{0}^{1} (1-t) (y-z_{j})^{\perp}\cdot  \underset{x}{\nabla}\left( \underset{y}{Hess}\left(\Phi_{k} \right)\right)(z_{j_{0}},z_{j}+t(y-z_{j}))  \cdot \mathcal{P}(x,z_{j_{0}}) \cdot (y-z_{j}) dt  \; \overset{1}{\mathbb{P}}\left(E_{j}^{T}\right)(y) \; dy \\ \nonumber
	&+& k^{2} \; \int_{D_{j}} \int_{0}^{1} (1-t) (y-z_{j})^{\perp} \\ \nonumber &\cdot& \left[\int_{0}^{1} \int_{0}^{1} \underset{x}{\bm{d}}\left(\underset{x}{\nabla}\left( \underset{y}{Hess}\left(\Phi_{k} \right)\right)\right)(z_{j_{0}}+ \rho \, s(x-z_{j_{0}}),z_{j}+t(y-z_{j})) \cdot (x-z_{j_{0}})^{[1]} \, d\rho \cdot \mathcal{P}(x,z_{j_{0}}) \, ds  \right] \\ && \qquad \qquad \qquad \qquad \qquad \qquad \qquad \qquad \cdot (y-z_{j}) dt  \; \overset{1}{\mathbb{P}}\left(E_{j}^{T}\right)(y) \; dy,
	\end{eqnarray}
	and
	\begin{eqnarray*}
		\int_{D_{j}} \Upsilon(x,y) \cdot \overset{3}{\mathbb{P}}\left(E_{j}^{T}\right)(y) \, dy &=& \Upsilon(z_{j_{0}} ,z_{j}) \cdot \int_{D_{j}} \overset{3}{\mathbb{P}}\left(E_{j}^{T}\right)(y) \, dy \\ 
		& + & \int_{D_{j}} \int_{0}^{1} \underset{y}{\nabla}\left( \Upsilon \right)(z_{j_{0}},z_{j}+t(y-z_{j})) \cdot \mathcal{P}(y,z_{j}) dt \cdot \overset{3}{\mathbb{P}}\left(E_{j}^{T}\right)(y) \, dy \\ &+& Err(x,j_{0},j,\overset{3}{\mathbb{P}}\left(E_{j}^{T}\right)),
	\end{eqnarray*}
	where
	\begin{eqnarray}\label{LMERR3}
	\nonumber
	Err(x,j_{0},j,\overset{3}{\mathbb{P}}\left(E_{j}^{T}\right)) &=&  \int_{0}^{1} \underset{x}{\nabla}\left( \Upsilon \right)(z_{j_{0}}+t(x-z_{j_{0}}),z_{j}) \cdot \mathcal{P}(x,z_{j_{0}}) \, dt \cdot \int_{D_{j}}  \overset{3}{\mathbb{P}}\left(E_{j}^{T}\right)(y) \, dy \\ \nonumber
	&+&  \int_{D_{j}} \int_{0}^{1} \left[\int_{0}^{1} \underset{x}{\nabla}\left(\underset{y}{\nabla} \left( \Upsilon \right) \right)(z_{j_{0}}+s(x-z_{j_{0}}),z_{j}+t(y-z_{j})) \cdot \mathcal{P}(x,z_{j_{0}}) ds \right] \\
	&& \qquad \qquad \qquad \cdot \mathcal{P}(y,z_{j}) dt \cdot \overset{3}{\mathbb{P}}\left(E_{j}^{T}\right)(y) \, dy .
	\end{eqnarray}
	Then, 
	\begin{eqnarray*}
		\overset{1}{\mathbb{P}}\left(E_{j_{0}}^{T}\right)(x)  &-& \eta \,k^{2} \, \sum_{j \neq j_{0}}^{\aleph} \, \boldsymbol{T}_{k}^{-1}  \underset{y}{\nabla}\left(\Phi_{k} I \right)(z_{j_{0}},z_{j})\cdot  \int_{D_{j}} \mathcal{P}(y;z_{j})\cdot   \overset{1}{\mathbb{P}}\left(E_{j}^{T}\right)(y) \; dy  \\ 
		&-& \eta \,k^{2} \, \sum_{j \neq j_{0}}^{\aleph} \, \boldsymbol{T}_{k}^{-1}    \int_{D_{j}} \int_{0}^{1} (1-t)(y-z_{j})^{\perp} \cdot \underset{y}{Hess}\left( \Phi_{k} \right)(z_{j_{0}},z_{j}+t(y-z_{j}))\cdot (y-z_{j})  dt \cdot   \overset{1}{\mathbb{P}}\left(E_{j}^{T}\right)(y) \; dy  \\
		&=&\boldsymbol{T}_{k}^{-1}\left(E^{Inc,j_{0}}\right)(x)-\overset{3}{\mathbb{P}}\left(E_{j_{0}}^{T}\right)(x) + \eta \,\sum_{j \neq j_{0}}^{\aleph} \, \boldsymbol{T}_{k}^{-1}   \Upsilon(z_{j_{0}},z_{j}) \cdot \int_{D_{j}}  \overset{3}{\mathbb{P}}\left(E_{j}^{T}\right)(y) \, dy \\
		&+& \eta \,\sum_{j \neq j_{0}}^{\aleph} \, \boldsymbol{T}_{k}^{-1}    \int_{D_{j}} \int_{0}^{1} \underset{y}{\nabla}\left(\Upsilon \right)(z_{j_{0}},z_{j}+t(y-z_{j})) \cdot \mathcal{P}(y,z_{j}) dt \cdot \overset{3}{\mathbb{P}}\left(E_{j}^{T}\right)(y) \, dy \\
		&+& \eta \,\sum_{j \neq j_{0}}^{\aleph} \boldsymbol{T}_{k}^{-1}\left[Err(x,j_{0},j,\overset{1}{\mathbb{P}}\left(E_{j}^{T}\right)) + Err(x,j_{0},j,\overset{3}{\mathbb{P}}\left(E_{j}^{T}\right)) \right].  
	\end{eqnarray*} 
	Scaling to $B$, 
	taking the inner product with respect to $ e_{n}^{(1)}(\cdot)$ and using the adjoint  operator of $\boldsymbol{T}_{k \, a}$, i.e. $\boldsymbol{T}^{\star}_{k \, a} = \boldsymbol{T}_{- \, k \, a}$, we obtain  
	\begin{eqnarray}\label{Sys-Equa-HA}
	\nonumber
\langle \overset{1}{\mathbb{P}}\left(\tilde{E}_{j_{0}}^{T}\right), e_{n}^{(1)} \rangle &-& \eta \,k^{2} \,a^{4} \sum_{j \neq j_{0}}^{\aleph} \langle   \underset{y}{\nabla}\left(\Phi_{k} I \right)(z_{j_{0}}, z_{j})\cdot  \int_{B} \mathcal{P}(y;0)\cdot \overset{1}{\mathbb{P}}\left(\tilde{E}_{j}^{T}\right)(y) \; dy, \boldsymbol{T}_{- \, k \, a}^{-1}\left(e_{n}^{(1)} \right) \rangle \\ \nonumber 
	&-& \eta \,k^{2} \,a^{5} \sum_{j \neq j_{0}}^{\aleph} \, \langle   \int_{B} \int_{0}^{1} (1-t) y^{\perp} \cdot \underset{y}{Hess}\left( \Phi_{k} \right)(z_{j_{0}},z_{j}+t a \, y) \cdot y  dt \cdot   \overset{1}{\mathbb{P}}\left(\tilde{E}_{j}^{T}\right)(y) \; dy, \boldsymbol{T}_{- \, k \, a}^{-1}\left( e_{n}^{(1)} \right) \rangle \\ \nonumber
	&=&\langle \tilde{E}^{Inc}_{j_{0}},\boldsymbol{T}_{- \, k \, a}^{-1}\left( e_{n}^{(1)} \right) \rangle + \eta \,a^{3} \, \sum_{j \neq j_{0}}^{\aleph} \, \langle  \Upsilon(z_{j_{0}}, z_{j}) \cdot \int_{B}  \overset{3}{\mathbb{P}}\left(\tilde{E}_{j}^{T}\right)(y) \, dy, \boldsymbol{T}_{- \, k \, a}^{-1}\left( e_{n}^{(1)} \right) \rangle \\ \nonumber
	&+& \eta \, a^{4} \, \sum_{j \neq j_{0}}^{\aleph} \langle \int_{B} \int_{0}^{1} \underset{y}{\nabla}\left(\Upsilon \right)(z_{j_{0}},z_{j}+t a  y) \cdot \mathcal{P}(y,0) dt \cdot \overset{3}{\mathbb{P}}\left(\tilde{E}_{j}^{T}\right)(y) \, dy, \boldsymbol{T}_{- \, k \, a}^{-1}\left( e_{n}^{(1)} \right) \rangle \\
	&+& \eta \,\sum_{j \neq j_{0}}^{\aleph} \langle \widetilde{Err}(x,j_{0},j,\overset{1}{\mathbb{P}}\left(E_{j}^{T}\right)) + \widetilde{Err}(x,j_{0},j,\overset{3}{\mathbb{P}}\left(E_{j}^{T}\right)),\boldsymbol{T}_{- \, k \, a}^{-1}\left( e_{n}^{(1)} \right) \rangle.  
	\end{eqnarray} 
	Using $(\ref{equa-lemmma-21})$ of Lemma \ref{AI4}, we rewrite $(\ref{Sys-Equa-HA})$ as 
	\begin{eqnarray*}
		\langle \overset{1}{\mathbb{P}}\left(\tilde{E}_{j_{0}}^{T}\right), e_{n}^{(1)} \rangle & = &  \frac{\eta \, k^{2} \, a^{4}}{\left( 1 - k^{2} \, \eta \, a^{2} \, \lambda_{n}^{(1)} \right)} \sum_{j \neq j_{0}}^{\aleph}  \langle  \underset{x}{\nabla}\left(\Phi_{k} I \right)(z_{j_{0}},z_{j})\cdot  \int_{B} \mathcal{P}(y;0)\cdot \overset{1}{\mathbb{P}}\left(\tilde{E}_{j}^{T}\right)(y) \; dy, e_{n}^{(1)}  \rangle  \\
		&+& \frac{\eta \,k^{2} \,a^{5}}{\left(1 - k^{2} \, \eta \,a^{2} \, \lambda^{(1)}_{n} \right)} \sum_{j \neq j_{0}}^{\aleph} \, \langle   \int_{B} \int_{0}^{1} (1-t)y^{\perp} \cdot \underset{y}{Hess}\left( \Phi_{k} \right)(z_{j_{0}},z_{j}+t a \, y) \cdot y  dt \cdot   \overset{1}{\mathbb{P}}\left(\tilde{E}_{j}^{T}\right)(y) \; dy,  e_{n}^{(1)} \rangle \\
		&+& \eta \,k^{2} \,a^{5} \sum_{j \neq j_{0}}^{\aleph} \, \langle   \int_{B} \int_{0}^{1}(1-t) y^{\perp} \cdot \underset{y}{Hess}\left( \Phi_{k} \right)(z_{j_{0}},z_{j}+t a \, y) \cdot y  dt \cdot   \overset{1}{\mathbb{P}}\left(\tilde{E}_{j}^{T}\right)(y) \; dy, R_{n} \rangle \\
		&+&\frac{\langle \tilde{E}^{Inc}_{j_{0}},e_{n}^{(1)} \rangle}{\left( 1 - k^{2} \, \eta \, a^{2} \, \lambda_{n}^{(1)} \right)} + \frac{a^{3} \, \eta}{\left( 1 - k^{2} \, \eta \, a^{2} \, \lambda_{n}^{(1)} \right)} \sum_{j \neq j_{0}}^{\aleph}  \langle  \Upsilon(z_{j_{0}} ,z_{j}) \cdot \int_{B}  \overset{3}{\mathbb{P}}\left(\tilde{E}_{j}^{T}\right)(y) \, dy,  e_{n}^{(1)}  \rangle \\ 
		&+& a^{3} \, \eta \sum_{j \neq j_{0}}^{\aleph} \langle  \Upsilon(z_{j_{0}},z_{j}) \cdot \int_{B}  \overset{3}{\mathbb{P}}\left(\tilde{E}_{j}^{T}\right)(y) \, dy,  R_{n} \rangle  \\
		&+& a^{4} \, \eta \,k^{2} \sum_{j \neq j_{0}}^{\aleph}  \langle \underset{x}{\nabla}\left(\Phi_{k} I \right)(z_{j_{0}},z_{j})\cdot  \int_{B} \mathcal{P}(y;0)\cdot \overset{1}{\mathbb{P}}\left(\tilde{E}_{j}^{T}\right)(y) \; dy,  R_{n} \rangle  \\ 
		&+& \frac{\eta \, a^{4}}{\left(1 - k^{2} \, \eta \,a^{2} \lambda_{n}^{(1)} \right)}  \sum_{j \neq j_{0}}^{\aleph} \langle \int_{B} \int_{0}^{1} \underset{y}{\nabla}\left(\Upsilon \right)(z_{j_{0}},z_{j}+t a  y) \cdot \mathcal{P}(y,0) dt \cdot \overset{3}{\mathbb{P}}\left(\tilde{E}_{j}^{T}\right)(y) \, dy,  e_{n}^{(1)} \rangle \\
		&+& \eta \, a^{4} \, \sum_{j \neq j_{0}}^{\aleph} \langle \int_{B} \int_{0}^{1} \underset{y}{\nabla}\left(\Upsilon \right)(z_{j_{0}},z_{j}+t a  y) \cdot \mathcal{P}(y,0) dt \cdot \overset{3}{\mathbb{P}}\left(\tilde{E}_{j}^{T}\right)(y) \, dy, R_{n} \rangle \\
		&+& \frac{ \eta}{\left( 1 - k^{2} \, \eta \, a^{2} \, \lambda_{n}^{(1)} \right)}  \,\sum_{j \neq j_{0}}^{\aleph}   \langle \widetilde{Err}(x,j_{0},j,\overset{1}{\mathbb{P}}\left(E_{j}^{T}\right)) + \widetilde{Err}(x,j_{0},j,\overset{3}{\mathbb{P}}\left(E_{j}^{T}\right)), e_{n}^{(1)}  \rangle  \\
		&+& \eta  \sum_{j \neq j_{0}}^{\aleph}  \langle \widetilde{Err}(x,j_{0},j,\overset{1}{\mathbb{P}}\left(E_{j}^{T}\right)) + \widetilde{Err}(x,j_{0},j,\overset{3}{\mathbb{P}}\left(E_{j}^{T}\right)),R_{n} \rangle + \langle \tilde{E}^{Inc}_{j_{0}},R_{n} \rangle. 
	\end{eqnarray*} 
	As the eigenfunctions $e_{n}^{(1)}, n=1, 2,...$, which are in $\mathbb{H}_{0}\left(\div=0 \right)$,  are orthogonal to the constant vectors, which are in $\nabla \mathcal{H}armonic$, then we deduce that
	\begin{eqnarray*}
		\langle  \Upsilon(z_{j_{0}},z_{j}) \cdot \int_{B}  \overset{3}{\mathbb{P}}\left(\tilde{E}_{j}^{T}\right)(y) \, dy, e_{n}^{(1)} \rangle &=& 0, \\
		\langle  \underset{x}{\nabla}\left(\Phi_{k} I \right)(z_{j_{0}},z_{j})\cdot  \int_{B} \mathcal{P}(y;0)\cdot \overset{1}{\mathbb{P}}\left(\tilde{E}_{j}^{T}\right)(y) \; dy, e_{n}^{(1)}  \rangle &=& 0, \\
		\langle \int_{B} \int_{0}^{1} \underset{y}{\nabla}\left(\Upsilon \right)(z_{j_{0}},z_{j}+t a  y) \cdot \mathcal{P}(y,0) dt \cdot \overset{3}{\mathbb{P}}\left(\tilde{E}_{j}^{T}\right)(y) \, dy;  e_{n}^{(1)} \rangle &=& 0, \\
		\langle   \int_{B} \int_{0}^{1} y^{\perp} \cdot \underset{y}{Hess}\left( \Phi_{k} \right)(z_{j_{0}},z_{j}+t a \, y) \cdot y  dt \cdot   \overset{1}{\mathbb{P}}\left(\tilde{E}_{j}^{T}\right)(y) \; dy,  e_{n}^{(1)} \rangle & = & 0.
	\end{eqnarray*}
	Then, after taking the squared modulus, we get 
	\begin{eqnarray*}
		\left\vert \langle \overset{1}{\mathbb{P}}\left(\tilde{E}_{j_{0}}^{T}\right), e_{n}^{(1)} \rangle \right\vert^{2} & \lesssim & \frac{\left\vert \langle \tilde{E}^{Inc}_{j_{0}},e_{n}^{(1)} \rangle \right\vert^{2}}{\left\vert 1 - k^{2} \, \eta \, a^{2} \, \lambda_{n}^{(1)} \right\vert^{2}} + a^{2} \, \aleph \,  \sum_{j \neq j_{0}}^{\aleph} \left\vert \langle  \Upsilon(z_{j_{0}},z_{j}) \cdot \int_{B}  \overset{3}{\mathbb{P}}\left(\tilde{E}_{j}^{T}\right)(y) \, dy,  R_{n} \rangle \right\vert^{2} \\
		&+& a^{4} \, \aleph \, \sum_{j \neq j_{0}}^{\aleph} \left\vert \langle \underset{x}{\nabla}\left(\Phi_{k} I \right)(z_{j_{0}},z_{j})\cdot  \int_{B} \mathcal{P}(y;0)\cdot \overset{1}{\mathbb{P}}\left(\tilde{E}_{j}^{T}\right)(y) \; dy,  R_{n} \rangle \right\vert^{2} \\ 
		&+& \frac{\left\vert \eta \right\vert^{2}}{\left\vert 1 - k^{2} \, \eta \, a^{2} \, \lambda_{n}^{(1)} \right\vert^{2}}  \,\sum_{j \neq j_{0}}^{\aleph} \left\vert \langle \widetilde{Err}(x,j_{0},j,\overset{1}{\mathbb{P}}\left(E_{j}^{T}\right)) + \widetilde{Err}(x,j_{0},j,\overset{3}{\mathbb{P}}\left(E_{j}^{T}\right)), e_{n}^{(1)}  \rangle \right\vert^{2} \\
		&+& \frac{\left\vert \eta \right\vert^{2}}{\left\vert 1 - k^{2} \, \eta \, a^{2} \, \lambda_{n}^{(1)} \right\vert^{2}} \Bigg[  \,\sum_{k=1 \atop k \neq j_{0}}^{\aleph} \left\vert \langle \widetilde{Err}(x,j_{0},k,\overset{1}{\mathbb{P}}\left(E_{k}^{T}\right)) + \widetilde{Err}(x,j_{0},k,\overset{3}{\mathbb{P}}\left(E_{k}^{T}\right)), e_{n}^{(1)}  \rangle \right\vert  \\
		&& \qquad \, \qquad \, \qquad \, \qquad  \,  \cdot \sum_{i > k \atop i \neq j_{0}}^{\aleph} \left\vert \langle \widetilde{Err}(x,j_{0},i,\overset{1}{\mathbb{P}}\left(E_{i}^{T}\right)) + \widetilde{Err}(x,j_{0},i,\overset{3}{\mathbb{P}}\left(E_{i}^{T}\right)), e_{n}^{(1)}  \rangle \right\vert \Bigg] \\
		&+& \left\vert \eta \right\vert^{2} \, \aleph \, \sum_{j \neq j_{0}}^{\aleph} \left\vert \langle \widetilde{Err}(x,j_{0},j,\overset{1}{\mathbb{P}}\left(E_{j}^{T}\right)) + \widetilde{Err}(x,j_{0},j,\overset{3}{\mathbb{P}}\left(E_{j}^{T}\right)),R_{n} \rangle \right\vert^{2} + \left\vert \langle \tilde{E}^{Inc}_{j_{0}},R_{n} \rangle \right\vert^{2}\\
		&+& a^{6} \, \aleph \, \sum_{j\neq j_{0}}^{\aleph} \left\vert \langle \int_{B} \int_{0}^{1} (1-t) y^{\perp} \cdot \underset{y}{Hess}\left( \Phi_{k} \right)(z_{j_{0}};z_{j}+aty)\cdot y \, dt \cdot\overset{1}{\mathbb{P}}\left(\tilde{E}_{j}^{T}\right)(y)dy , R_{n} \rangle \right\vert^{2}\\
		&+& a^{4} \, \aleph \, \sum_{j\neq j_{0}}^{\aleph} \left\vert \langle \int_{B} \int_{0}^{1} \underset{y}{\nabla}\left( \Upsilon \right) (z_{j_{0}},z_{j}+tay) \cdot \mathcal{P}(y,0) \, dt \cdot \overset{3}{\mathbb{P}}\left(\tilde{E}_{j}^{T}\right)(y)dy , R_{n} \rangle \right\vert^{2}
	\end{eqnarray*} 
	Taking the series with respect to the index $n$, we obtain 
	\begin{eqnarray}\label{MLB}
	\nonumber
	\left\Vert \overset{1}{\mathbb{P}}\left(\tilde{E}_{j_{0}}^{T}\right) \right\Vert^{2} & \lesssim & \frac{\left\vert \langle \tilde{E}^{Inc}_{j_{0}},e_{n_{0}}^{(1)} \rangle \right\vert^{2}}{\left\vert 1 - k^{2} \, \eta \, a^{2} \, \lambda_{n_{0}}^{(1)} \right\vert^{2}} + \sum_{n \neq n_{0}} \, \frac{\left\vert \langle \tilde{E}^{Inc}_{j_{0}},e_{n}^{(1)} \rangle \right\vert^{2}}{\left\vert 1 - k^{2} \, \eta \, a^{2} \, \lambda_{n}^{(1)} \right\vert^{2}} \\ \nonumber
	&+& a^{2} \, \aleph \,  \sum_{j \neq j_{0}}^{\aleph} \sum_{n} \left\vert \langle  \Upsilon(z_{j_{0}},z_{j}) \cdot \int_{B}  \overset{3}{\mathbb{P}}\left(\tilde{E}_{j}^{T}\right)(y) \, dy,  R_{n} \rangle \right\vert^{2} \\ \nonumber
	&+& a^{4} \, \aleph \, \sum_{j \neq j_{0}}^{\aleph} \sum_{n} \left\vert \langle \underset{x}{\nabla}\left(\Phi_{k} I \right)(z_{j_{0}},z_{j})\cdot  \int_{B} \mathcal{P}(y;0)\cdot \overset{1}{\mathbb{P}}\left(\tilde{E}_{j}^{T}\right)(y) \; dy,  R_{n} \rangle \right\vert^{2} \\ \nonumber
	&+& \left\vert \eta \right\vert^{2} \sum_{j \neq j_{0}}^{\aleph} \sum_{n} \, \frac{1}{\left\vert 1 - k^{2} \, \eta \, a^{2} \, \lambda_{n}^{(1)} \right\vert^{2}} \, \left\vert \langle \widetilde{Err}(x,j_{0},j,\overset{1}{\mathbb{P}}\left(E_{j}^{T}\right)) + \widetilde{Err}(x,j_{0},j,\overset{3}{\mathbb{P}}\left(E_{j}^{T}\right)), e_{n}^{(1)}  \rangle \right\vert^{2} \\ \nonumber
	&+& \left\vert \eta \right\vert^{2}  \sum_{n} \, \frac{1}{\left\vert 1 - k^{2} \, \eta \, a^{2} \, \lambda_{n}^{(1)} \right\vert^{2}} \, \Bigg[\sum_{k=1 \atop k \neq j_{0}}^{\aleph} \, \left\vert \langle \widetilde{Err}(x,j_{0},k,\overset{1}{\mathbb{P}}\left(E_{k}^{T}\right)) + \widetilde{Err}(x,j_{0},k,\overset{3}{\mathbb{P}}\left(E_{k}^{T}\right)), e_{n}^{(1)}  \rangle \right\vert \\ \nonumber
	&&  \qquad \; \qquad  \qquad \; \qquad \; \qquad \; \quad \cdot \sum_{i > k \atop i \neq j_{0}}^{\aleph} \left\vert \langle \widetilde{Err}(x,j_{0},i,\overset{1}{\mathbb{P}}\left(E_{i}^{T}\right)) + \widetilde{Err}(x,j_{0},i,\overset{3}{\mathbb{P}}\left(E_{i}^{T}\right)), e_{n}^{(1)}  \rangle \right\vert \Bigg] \\ \nonumber
	&+& \left\vert \eta \right\vert^{2} \, \aleph \, \sum_{j \neq j_{0}}^{\aleph} \sum_{n} \left\vert \langle \widetilde{Err}(x,j_{0},j,\overset{1}{\mathbb{P}}\left(E_{j}^{T}\right)) + \widetilde{Err}(x,j_{0},j,\overset{3}{\mathbb{P}}\left(E_{j}^{T}\right)),R_{n} \rangle \right\vert^{2} + \sum_{n} \left\vert \langle \tilde{E}^{Inc}_{j_{0}},R_{n} \rangle \right\vert^{2}\\ \nonumber
	&+& a^{6} \, \aleph \, \sum_{j\neq j_{0}}^{\aleph} \sum_{n} \left\vert \langle \int_{B} \int_{0}^{1} (1-t) y^{\perp} \cdot \underset{y}{Hess}\left( \Phi_{k} \right)(z_{j_{0}};z_{j}+aty)\cdot y \, dt \cdot\overset{1}{\mathbb{P}}\left(\tilde{E}_{j}^{T}\right)(y)dy , R_{n} \rangle \right\vert^{2}\\
	&+& a^{4} \, \aleph \, \sum_{j\neq j_{0}}^{\aleph}  \sum_{n} \left\vert \langle \int_{B} \int_{0}^{1} \underset{y}{\nabla}\left( \Upsilon \right)(z_{j_{0}},z_{j}+tay) \cdot \mathcal{P}(y,0) \, dt \cdot \overset{3}{\mathbb{P}}\left(\tilde{E}_{j}^{T}\right)(y)dy , R_{n} \rangle \right\vert^{2}
	\end{eqnarray}
	We set 
	\begin{equation*}
	I := \sum_{j \neq j_{0}}^{\aleph} \sum_{n} \, \frac{1}{\left\vert 1 - k^{2} \, \eta \, a^{2} \, \lambda_{n}^{(1)} \right\vert^{2}} \, \left\vert \langle \widetilde{Err}(x,j_{0},j,\overset{1}{\mathbb{P}}\left(E_{j}^{T}\right)) + \widetilde{Err}(x,j_{0},j,\overset{3}{\mathbb{P}}\left(E_{j}^{T}\right)), e_{n}^{(1)}  \rangle \right\vert^{2},
	\end{equation*}
	then, obviously, we have
	\begin{eqnarray}\label{DFD50}
	\nonumber
	|I|  & \lesssim &   a^{-2h} \;\sum_{j \neq j_{0}}^{\aleph} \sum_{n} \, \left\vert \langle \widetilde{Err}(x,j_{0},j,\overset{1}{\mathbb{P}}\left(E_{j}^{T}\right)), e_{n}^{(1)}  \rangle \right\vert^{2}  +a^{-2h} \;\sum_{j \neq j_{0}}^{\aleph} \sum_{n} \, \left\vert \langle \widetilde{Err}(x,j_{0},j,\overset{3}{\mathbb{P}}\left(E_{j}^{T}\right)), e_{n}^{(1)}  \rangle \right\vert^{2} \\
	& \lesssim & a^{-2h} \;\sum_{j \neq j_{0}}^{\aleph}  \left\Vert  \widetilde{Err}(\cdot , j_{0},j,\overset{1}{\mathbb{P}}\left(E_{j}^{T}\right)) \right\Vert^{2}_{\mathbb{L}^{2}(B)}  +a^{-2h} \;\sum_{j \neq j_{0}}^{\aleph} \, \left\Vert  \widetilde{Err}(\cdot ,j_{0},j,\overset{3}{\mathbb{P}}\left(E_{j}^{T}\right)) \right\Vert^{2}_{\mathbb{L}^{2}(B)}.
	\end{eqnarray}
	Next, to compute an upper bound for $I$, we split the computations into two parts. 
	\begin{enumerate}
		\item Computing $ \left\Vert  \widetilde{Err}(\cdot , j_{0},j,\overset{1}{\mathbb{P}}\left(E_{j}^{T}\right)) \right\Vert^{2}_{\mathbb{L}^{2}(B)} $. 
		\newline
		Firstly, we rewrite $Err(x,j_{0},j,\overset{1}{\mathbb{P}}\left(E_{j}^{T}\right))$, given by $(\ref{LMERR1})$, as\footnote{The difference between $(\ref{LMERR1})$ and $(\ref{SKIS})$ is merely technical. We keep the first term on the right hand side of \eqref{LMERR1} with the help of the Taylor expansion of $Hess \, N^{k}(\cdots)$, then we split  $Hess \, N^{k}$  into  $Hess \, N$ and $Hess \, \left(N^{k}-N\right)$.} 
		\begin{eqnarray}\label{SKIS}
		\nonumber
		&& Err(x,j_{0},j,\overset{1}{\mathbb{P}}\left(E_{j}^{T}\right)) := \\ \nonumber &-& k^{2} \int_{D_{j}} \int_{0}^{1} \underset{y}{Hess}\left(\Phi_{0} \right)(z_{j_{0}}+t(x-z_{j_{0}}),y) \cdot (x-z_{j_{0}}) \, dt \cdot (y-z_{j})  \overset{1}{\mathbb{P}}\left(E_{j}^{T}\right)(y) \; dy \\ \nonumber
		&-& k^{2} \int_{D_{j}} \int_{0}^{1} \underset{y}{Hess}\left(\Phi_{k} - \Phi_{0} \right)(z_{j_{0}}+t(x-z_{j_{0}}),y) \cdot (x-z_{j_{0}}) \, dt \cdot (y-z_{j})  \overset{1}{\mathbb{P}}\left(E_{j}^{T}\right)(y) \; dy \\ \nonumber
		&+& k^{2} \int_{D_{j}} \int_{0}^{1} \int_{0}^{1} \Bigg[ \underset{y}{\nabla}\left(\underset{y}{Hess}\left(\Phi_{k} \right)\right)(z_{j_{0}}+t(x-z_{j_{0}}),z_{j}+s(y-z_{j})) \cdot \mathcal{P}(y,z_{j}) \, ds \cdot (x-z_{j_{0}}) \, dt \Bigg] \\ \nonumber
	& &	\cdot  (y-z_{j})  \overset{1}{\mathbb{P}}\left(E_{j}^{T}\right)(y) \; dy \\ \nonumber
		&+& k^{2} \; \int_{D_{j}} \int_{0}^{1} (1-t) (y-z_{j})^{\perp}\cdot  \underset{x}{\nabla}\left( \underset{y}{Hess}\left(\Phi_{k} \right)\right)(z_{j_{0}},z_{j}+t(y-z_{j}))  \cdot \mathcal{P}(x,z_{j_{0}}) \cdot (y-z_{j}) dt  \; \overset{1}{\mathbb{P}}\left(E_{j}^{T}\right)(y) \; dy \\ \nonumber
		&+& k^{2} \; \int_{D_{j}} \int_{0}^{1} (1-t) (y-z_{j})^{\perp}\cdot \\ \nonumber && \left[\int_{0}^{1} \int_{0}^{1} \underset{x}{\bm{d}}\left(\underset{x}{\nabla}\left( \underset{y}{Hess}\left(\Phi_{k} \right)\right)\right)(z_{j_{0}}+ \rho \, s(x-z_{j_{0}}),z_{j}+t(y-z_{j})) \cdot (x-z_{j_{0}})^{[1]} \, d\rho \cdot \mathcal{P}(x,z_{j_{0}}) \, ds  \right] \\ && \qquad \qquad \qquad \qquad \qquad \qquad \qquad \qquad \cdot (y-z_{j}) dt  \; \overset{1}{\mathbb{P}}\left(E_{j}^{T}\right)(y) \; dy.
		\end{eqnarray}
		Secondly, for the shortness reason, we denote
		\begin{equation}\label{J1+...+J5}
		Err(\cdot ,j_{0},j,\overset{1}{\mathbb{P}}\left(E_{j}^{T}\right)) = \sum_{\ell = 1}^{5} J_{\ell}(\cdot)
		\end{equation} 
		and we compute an estimation of the $\mathbb{L}^{2}$-norm of each $J_{\ell}(\cdot), \; \ell =1,\cdots,5$. For this, we have: 
		\begin{enumerate}
			\item Estimation of 
			\begin{equation*}
			J_{1}(x) := - k^{2} \int_{D_{j}} \int_{0}^{1} \underset{y}{Hess}\left(\Phi_{0} \right)(z_{j_{0}}+t(x-z_{j_{0}}),y) \cdot (x-z_{j_{0}}) \, dt \cdot (y-z_{j})  \overset{1}{\mathbb{P}}\left(E_{j}^{T}\right)(y) \; dy.
			\end{equation*}
			For simplicity, we set $\beta_{j}$, a scalar function, to be an arbitrary component of the vector field $\overset{1}{\mathbb{P}}\left(E_{j}^{T}\right)$. We also set $J_{1, j}$ the corresponding $j$th component of $J_{1}$. Then,
			\begin{eqnarray*}
				J_{1, j}(x) &=& - k^{2} \int_{D_{j}} \int_{0}^{1} \underset{y}{Hess}\left(\Phi_{0} \right)(z_{j_{0}}+t(x-z_{j_{0}}),y) \cdot (x-z_{j_{0}}) \, dt \cdot (y-z_{j})  \beta_{j}(y) \; dy\\
				&=& - k^{2} \int_{D_{j}} \int_{0}^{1} \underset{y}{Hess}\left(\Phi_{0} \right)(z_{j_{0}}+t(x-z_{j_{0}}),y) \cdot (x-z_{j_{0}}) \, dt \cdot (y-z_{j})  \beta_{j}(y) \; \chi_{D_{j}}(y) \; dy\\
				&=& - k^{2} \int_{0}^{1}  \int_{D_{j}}  \underset{y}{Hess}\left(\Phi_{0} \right)(z_{j_{0}}+t(x-z_{j_{0}}),y) \cdot (y-z_{j})  \beta_{j}(y) \chi_{D_{j}}(y) \; dy \cdot (x-z_{j_{0}}) \, dt \\
				&=& - k^{2} \int_{0}^{1}  Hess \, N_{D_{j}}  \left( ( \cdot -z_{j})  \beta_{j}(\cdot) \chi_{D_{j}}(\cdot) \right) (z_{j_{0}}+t(x-z_{j_{0}}))  \cdot (x-z_{j_{0}}) \, dt.
			\end{eqnarray*}
			Scaling to $B$, we obtain 
			\begin{equation*}
			\tilde{J_{1, j}}(x) = - k^{2} \, a^{5} \, \int_{0}^{1}  Hess \, N_{B}  \left( \cdot  \tilde{\beta_{j}}(\cdot) \chi_{B}(\cdot) \right) (z_{j_{0}}+t \, a \, x)  \cdot x \, dt.
			\end{equation*}
			Now, taking the norm on the both sides, we get 
			\begin{eqnarray*}
				\left\Vert \tilde{J_{1, j}} \right\Vert_{\mathbb{L}^{2}(B)} & \lesssim & a^{5} \; \left\Vert \int_{0}^{1}  Hess \, N_{B}  \left(  \cdot  \tilde{\beta_{j}}(\cdot) \chi_{B}(\cdot) \right) (z_{j_{0}}+t a \, \cdot ) \, dt \right\Vert_{\mathbb{L}^{2}(B)}  \\
				& \lesssim & a^{5} \; \int_{0}^{1} \left\Vert   Hess \, N_{B}  \left(  \cdot \tilde{\beta_{j}}(\cdot) \chi_{B}(\cdot) \right) (z_{j_{0}}+t \, a \, \cdot)  \right\Vert_{\mathbb{L}^{2}(B)} \, dt \\
				& \lesssim & a^{5} \;  \left\Vert   Hess \, N_{B}  \left(  \cdot   \tilde{\beta_{j}}(\cdot) \chi_{B}(\cdot) \right)  \right\Vert_{\mathbb{L}^{2}(\mathbb{R}^{3})}.
			\end{eqnarray*}
			Due to Theorem 9.9, formula 9.28, page 230 of \cite{gilbarg2001elliptic}, we have  
			\begin{equation*}
			\left\Vert   Hess \, N_{B}  \left(  \cdot   \tilde{\beta_{j}}(\cdot) \chi_{B}(\cdot) \right)  \right\Vert_{\mathbb{L}^{2}(\mathbb{R}^{3})} = \left\Vert    \cdot \tilde{ \beta_{j}}(\cdot)  \right\Vert_{\mathbb{L}^{2}(B)}.
			\end{equation*}
			This implies, 
			\begin{equation*}
			\left\Vert \tilde{J_{1, j}} \right\Vert_{\mathbb{L}^{2}(B)}  \lesssim  a^{5} \;  \left\Vert   \tilde{\beta_{j}}  \right\Vert_{\mathbb{L}^{2}(B)}, 
			\end{equation*}
			Since $\beta_{j}$ is chosen as an arbitrarily component, we deduce that 
			\begin{equation}\label{CJ1}
			\left\Vert \tilde{J_{1}} \right\Vert_{\mathbb{L}^{2}(B)} = \mathcal{O}\left(a^{5} \; \left\Vert \overset{1}{\mathbb{P}}\left(\tilde{E}_{j}^{T}\right) \right\Vert_{\mathbb{L}^{2}(B)} \right).
			\end{equation}  
			\item Estimation of 
			\begin{equation*}
			J_{2}(x) := -  k^{2} \int_{D_{j}} \int_{0}^{1} \underset{y}{Hess}\left(\Phi_{k} - \Phi_{0} \right)(z_{j_{0}}+t(x-z_{j_{0}}),y) \cdot (x-z_{j_{0}}) \, dt \cdot (y-z_{j})  \overset{1}{\mathbb{P}}\left(E_{j}^{T}\right)(y) \; dy.
			\end{equation*}
			Recall the expansion $(\ref{expansion-of-Hess})$. Then we have 
				\begin{equation*}
				{Hess}\left(\Phi_{k} - \Phi_{0} \right)(x,y) = - \frac{k^{2}}{2} \Phi_{0}(x,y) \, I - i \frac{k^{3}}{24\pi} I_{3} + \frac{k^{2}}{2} \, \Phi_{0}(x,y) \, \frac{A(x,y)}{\Vert x-y  \Vert^{2}} +\frac{1}{4 \pi} \; \sum_{n \geq 3} \frac{(ik)^{n+1}}{(n+1)!} \; \underset{x}{Hess} \left( \left\Vert x - y \right\Vert^{n} \right).
				\end{equation*} 
Using this we deduce that the dominant term of $J_2(x)$ is 
			\begin{equation*}
			J_{2}(x) \sim \frac{k^{4}}{2} \int_{D_{j}} \int_{0}^{1}  \Phi_{0}(z_{j_{0}}+t(x-z_{j_{0}}),y) \, dt (x-z_{j_{0}})  \cdot (y-z_{j})  \overset{1}{\mathbb{P}}\left(E_{j}^{T}\right)(y) \; dy .
			\end{equation*}
		By taking the norm, we get 
			\begin{eqnarray}\label{CJ2}
			\nonumber
			\left\Vert J_{2} \right\Vert_{\mathbb{L}^{2}(D_{j_{0}})} & \lesssim &   \left\Vert  \int_{D_{j}} \int_{0}^{1}  \Phi_{0}(z_{j_{0}}+t(\cdot -z_{j_{0}}),y) \, dt (\cdot -z_{j_{0}})  \cdot (y-z_{j})  \overset{1}{\mathbb{P}}\left(E_{j}^{T}\right)(y) \; dy \right\Vert_{\mathbb{L}^{2}(D_{j_{0}})} \\
			& \lesssim &  a^{5} \; \left\vert z_{j} - z_{j_{0}} \right\vert^{-1} \; \left\Vert \overset{1}{\mathbb{P}}\left(E_{j}^{T}\right) \right\Vert_{\mathbb{L}^{2}(D_{j})}.
			\end{eqnarray}
			\item Estimation of 
			\begin{eqnarray}\label{CJ3}
			\nonumber
			J_{3} &:=& k^{2} \int_{D_{j}} \int_{0}^{1} \int_{0}^{1} \Bigg[ \underset{y}{\nabla}\left(\underset{y}{Hess}\left(\Phi_{k} \right)\right)(z_{j_{0}}+t(x-z_{j_{0}}),z_{j}+s(y-z_{j})) \cdot \mathcal{P}(y,z_{j}) \, ds   \\ \nonumber
			&& \qquad \qquad  \qquad \qquad \qquad \qquad \cdot (x-z_{j_{0}}) \, dt \Bigg] \cdot (y-z_{j})  \overset{1}{\mathbb{P}}\left(E_{j}^{T}\right)(y) \; dy \\ \nonumber
			\left\Vert J_{3} \right\Vert_{\mathbb{L}^{2}(D_{j_{0}})} & \lesssim & \Bigg\Vert \int_{D_{j}} \int_{0}^{1} \int_{0}^{1} \Bigg[ \underset{y}{\nabla}\left(\underset{y}{Hess}\left(\Phi_{k} \right)\right)(z_{j_{0}}+t(\cdot -z_{j_{0}}),z_{j}+s(y-z_{j})) \cdot \mathcal{P}(y,z_{j}) \, ds   \\ \nonumber
			&& \qquad \qquad  \qquad \qquad \qquad \qquad \cdot (\cdot -z_{j_{0}}) \, dt \Bigg] \cdot (y-z_{j})  \overset{1}{\mathbb{P}}\left(E_{j}^{T}\right)(y) \; dy \Bigg\Vert_{\mathbb{L}^{2}(D_{j_{0}})} \\
			& \lesssim & a^{6}  \; \left\vert z_{j} - z_{j_{0}} \right\vert^{-4} \; \left\Vert \overset{1}{\mathbb{P}}\left(E_{j}^{T}\right) \right\Vert_{\mathbb{L}^{2}(D_{j})}.
			\end{eqnarray}
			\item Estimation of 
			\begin{eqnarray*}
			\nonumber
			J_{4} &:=& k^{2} \; \int_{D_{j}} \int_{0}^{1} (1-t) (y-z_{j})^{\perp}\cdot \\ \nonumber && \qquad \qquad \underset{x}{\nabla}\left( \underset{y}{Hess}\left(\Phi_{k} \right)\right)(z_{j_{0}},z_{j}+t(y-z_{j}))  \cdot \mathcal{P}(x,z_{j_{0}}) \cdot (y-z_{j}) dt  \; \overset{1}{\mathbb{P}}\left(E_{j}^{T}\right)(y) \; dy.
\end{eqnarray*}
Taking the $\mathbb{L}^{2}(D_{j_{0}})-$norm on the both sides, we obtain that  
\begin{equation}\label{CJ4}			
\left\Vert J_{4} \right\Vert_{\mathbb{L}^{2}(D_{j_{0}})}	 \lesssim   a^{6}  \; \left\vert z_{j} - z_{j_{0}} \right\vert^{-4} \; \left\Vert \overset{1}{\mathbb{P}}\left(E_{j}^{T}\right) \right\Vert_{\mathbb{L}^{2}(D_{j})}.
			\end{equation}
			\item Estimation of 
			\begin{eqnarray}\label{CJ5}
			\nonumber
			&& J_{5} := k^{2} \; \int_{D_{j}} \int_{0}^{1} (1-t) (y-z_{j})^{\perp}\cdot \\ \nonumber && \left[\int_{0}^{1} \int_{0}^{1} \underset{x}{\bm{d}}\left(\underset{x}{\nabla}\left( \underset{y}{Hess}\left(\Phi_{k} \right)\right)\right)(z_{j_{0}}+ \rho \, s(x-z_{j_{0}}),z_{j}+t(y-z_{j})) \cdot (x-z_{j_{0}})^{[1]} \, d\rho \cdot \mathcal{P}(x,z_{j_{0}}) \, ds  \right] \\ \nonumber && \qquad \qquad \qquad \qquad \qquad  \qquad \cdot (y-z_{j}) dt  \; \overset{1}{\mathbb{P}}\left(E_{j}^{T}\right)(y) \; dy, \\ \nonumber
			&& \left\Vert J_{5} \right\Vert_{\mathbb{L}^{2}(D_{j})} \lesssim  \Bigg\Vert \int_{D_{j}} \int_{0}^{1} (1-t) (y-z_{j})^{\perp}\cdot \\ \nonumber && \left[\int_{0}^{1} \int_{0}^{1} \underset{x}{\bm{d}}\left(\underset{x}{\nabla}\left( \underset{y}{Hess}\left(\Phi_{k} \right)\right)\right)(z_{j_{0}}+ \rho \, s(\cdot -z_{j_{0}}),z_{j}+t(y-z_{j})) \cdot (\cdot -z_{j_{0}})^{[1]} \, d\rho \cdot \mathcal{P}(\cdot ,z_{j_{0}}) \, ds  \right] \\ \nonumber && \qquad \qquad \qquad \qquad \qquad  \qquad \cdot (y-z_{j}) dt  \; \overset{1}{\mathbb{P}}\left(E_{j}^{T}\right)(y) \; dy \Bigg\Vert_{\mathbb{L}^{2}(D_{j})} \\
			& & \qquad \qquad \qquad \lesssim  a^{7}  \; \left\vert z_{j} - z_{j_{0}} \right\vert^{-5} \; \left\Vert \overset{1}{\mathbb{P}}\left(E_{j}^{T}\right) \right\Vert_{\mathbb{L}^{2}(D_{j})}.
			\end{eqnarray}
		\end{enumerate}
		Therefore, after combining with $(\ref{CJ1}), (\ref{CJ2}), (\ref{CJ3}), (\ref{CJ4}), (\ref{CJ5})$ and $(\ref{J1+...+J5})$, we obtain that
		\begin{equation}\label{HN1} 
		\left\Vert \widetilde{Err}(\cdot ,j_{0},j,\overset{1}{\mathbb{P}}\left(E_{j}^{T}\right)) \right\Vert^{2}_{\mathbb{L}^{2}(B)} =  \mathcal{O}\left( \left(a^{10} \, d^{-2}_{j,j_{0}} + a^{12} \, d^{-8}_{j,j_{0}} \right) \; \left\Vert \overset{1}{\mathbb{P}}\left(\tilde{E}_{j}^{T}\right) \right\Vert^{2}_{\mathbb{L}^{2}(B)} \right).
		\end{equation}

		\item Computing $ \left\Vert \widetilde{Err}(\cdot ,j_{0},j,\overset{3}{\mathbb{P}}\left(E_{j}^{T}\right)) \right\Vert^{2}_{\mathbb{L}^{2}(B)} $. 
		\newline 
		Using the definition of $Err(\cdot ,j_{0},j,\overset{3}{\mathbb{P}}\left(E_{j}^{T}\right))$, see $(\ref{LMERR3})$, we divide it into two parts
		\begin{equation*}
		Err(\cdot ,j_{0},j,\overset{3}{\mathbb{P}}\left(E_{j}^{T}\right)) = J_{6} + J_{7}.
		\end{equation*}
		\begin{enumerate}
			\item Estimation of
			\begin{eqnarray*}
				J_{6} &:=& \int_{0}^{1} \underset{x}{\nabla}\left( \Upsilon \right)(z_{j_{0}}+t(x-z_{j_{0}}),z_{j}) \cdot \mathcal{P}(x,z_{j_{0}}) \, dt \cdot \int_{D_{j}}  \overset{3}{\mathbb{P}}\left(E_{j}^{T}\right)(y) \, dy \\
				\left\Vert J_{6} \right\Vert_{\mathbb{L}^{2}(D_{j_{0}})} &=& \left\Vert \int_{0}^{1} \underset{x}{\nabla}\left( \Upsilon \right)(z_{j_{0}}+t(\cdot -z_{j_{0}}),z_{j}) \cdot \mathcal{P}(\cdot ,z_{j_{0}}) \, dt \cdot \int_{D_{j}}  \overset{3}{\mathbb{P}}\left(E_{j}^{T}\right)(y) \, dy \right\Vert_{\mathbb{L}^{2}(D_{j_{0}})} \\
				& \lesssim & a^{4} \; \left\vert z_{j} - z_{j_{0}} \right\vert^{-4} \; \left\Vert   \overset{3}{\mathbb{P}}\left(E_{j}^{T} \right) \right\Vert_{\mathbb{L}^{2}(D_{j})}.
			\end{eqnarray*}
			\item Estimation of 
			\begin{eqnarray*}
				J_{7} &:=& \int_{D_{j}} \int_{0}^{1} \left[\int_{0}^{1} \underset{x}{\nabla}\left(\underset{y}{\nabla} \left( \Upsilon \right) \right)(z_{j_{0}}+s(x-z_{j_{0}}),z_{j}+t(y-z_{j})) \cdot \mathcal{P}(x,z_{j_{0}}) ds \right] \\&& \qquad \qquad \qquad \cdot \mathcal{P}(y,z_{j}) dt \cdot \overset{3}{\mathbb{P}}\left(E_{j}^{T}\right)(y) \, dy \\
				\left\Vert J_{7} \right\Vert_{\mathbb{L}^{2}(D_{j_{0}})} &=&  \Bigg\Vert \int_{D_{j}} \int_{0}^{1} \left[\int_{0}^{1} \underset{x}{\nabla}\left(\underset{y}{\nabla} \left( \Upsilon \right) \right)(z_{j_{0}}+s(\cdot -z_{j_{0}}),z_{j}+t(y-z_{j})) \cdot \mathcal{P}(\cdot ,z_{j_{0}}) ds \right] \\&& \qquad \qquad \qquad \cdot \mathcal{P}(y,z_{j}) dt \cdot \overset{3}{\mathbb{P}}\left(E_{j}^{T}\right)(y) \, dy \Bigg\Vert_{\mathbb{L}^{2}(D_{j_{0}})} \\
				& \lesssim & a^{5} \; \left\vert z_{j} - z_{j_{0}} \right\vert^{-5} \; \left\Vert   \overset{3}{\mathbb{P}}\left(E_{j}^{T} \right) \right\Vert_{\mathbb{L}^{2}(D_{j})}.
			\end{eqnarray*}
		\end{enumerate}
		Clearly, 
		\begin{equation*}
		\left\Vert Err(\cdot ,j_{0},j,\overset{3}{\mathbb{P}}\left(E_{j}^{T}\right)) \right\Vert^{2}_{\mathbb{L}^{2}(D_{j_{0}})} = \mathcal{O}\left( a^{8} \; \left\vert z_{j} - z_{j_{0}} \right\vert^{-8} \; \left\Vert   \overset{3}{\mathbb{P}}\left(E_{j}^{T} \right) \right\Vert^{2}_{\mathbb{L}^{2}(D_{j})} \right),
		\end{equation*}
		and after scaling, we obtain that
		\begin{equation}\label{HN2}
		\left\Vert \widetilde{Err}(\cdot ,j_{0},j,\overset{3}{\mathbb{P}}\left(E_{j}^{T}\right)) \right\Vert^{2}_{\mathbb{L}^{2}(B)} = \mathcal{O}\left( a^{8} \; \left\vert z_{j} - z_{j_{0}} \right\vert^{-8} \; \left\Vert   \overset{3}{\mathbb{P}}\left(\tilde{E}_{j}^{T} \right) \right\Vert^{2}_{\mathbb{L}^{2}(B)} \right).
		\end{equation}
	\end{enumerate}
	Going back to $(\ref{DFD50})$ and using the estimations $(\ref{HN1})$ and $(\ref{HN2})$, we can derive that
	\begin{eqnarray}\label{I-Error-term}
	\nonumber
	|I|  & \lesssim &   a^{-2h} \;\sum_{j \neq j_{0}}^{\aleph}  \left\Vert  \widetilde{Err}(\cdot , j_{0},j,\overset{1}{\mathbb{P}}\left(E_{j}^{T}\right)) \right\Vert^{2}_{\mathbb{L}^{2}(B)}  +a^{-2h} \;\sum_{j \neq j_{0}}^{\aleph} \, \left\Vert  \widetilde{Err}(\cdot ,j_{0},j,\overset{3}{\mathbb{P}}\left(E_{j}^{T}\right)) \right\Vert^{2}_{\mathbb{L}^{2}(B)} \\ \nonumber
	& \lesssim &   a^{10-2h}  \sum_{j \neq j_{0}}^{\aleph} \frac{1}{d^{2}_{j,j_{0}}} \left\Vert \overset{1}{\mathbb{P}}\left(\tilde{E}_{j}^{T}\right) \right\Vert^{2} + a^{12-2h}  \sum_{j \neq j_{0}}^{\aleph} \frac{1}{d^{8}_{j,j_{0}}} \left\Vert \overset{1}{\mathbb{P}}\left(\tilde{E}_{j}^{T}\right) \right\Vert^{2} + a^{8-2h} \;\sum_{j \neq j_{0}}^{\aleph} \frac{1}{d^{8}_{j,j_{0}}} \left\Vert \overset{3}{\mathbb{P}}\left(\tilde{E}_{j}^{T}\right) \right\Vert^{2} \\ 
	& \lesssim & \left( a^{10-2h} \, d^{-3} + a^{12-2h} d^{-8} \right) \max_{j}  \left\Vert \overset{1}{\mathbb{P}}\left(\tilde{E}_{j}^{T}\right) \right\Vert^{2} +a^{8-2h} d^{-8} \max_{j} \left\Vert \overset{3}{\mathbb{P}}\left(\tilde{E}_{j}^{T}\right) \right\Vert^{2}. 
	\end{eqnarray}
	We also need to estimate 
	\begin{eqnarray*}
		I^{\star}  &:=& \sum_{n} \, \frac{1}{\left\vert 1 - k^{2} \, \eta \, a^{2} \, \lambda_{n}^{(1)} \right\vert^{2}} \, \Bigg[\sum_{k=1 \atop k \neq j_{0}}^{\aleph} \, \left\vert \langle \widetilde{Err}(x,j_{0},k,\overset{1}{\mathbb{P}}\left(E_{k}^{T}\right)) + \widetilde{Err}(x,j_{0},k,\overset{3}{\mathbb{P}}\left(E_{k}^{T}\right)), e_{n}^{(1)}  \rangle \right\vert \\ \nonumber
		&&  \qquad \; \qquad  \qquad \; \qquad  \qquad  \cdot \sum_{i > k \atop i \neq j_{0}}^{\aleph} \left\vert \langle \widetilde{Err}(x,j_{0},i,\overset{1}{\mathbb{P}}\left(E_{i}^{T}\right)) + \widetilde{Err}(x,j_{0},i,\overset{3}{\mathbb{P}}\left(E_{i}^{T}\right)), e_{n}^{(1)}  \rangle \right\vert \Bigg] \\ \nonumber
		\left\vert I^{\star} \right\vert  & \lesssim & a^{-2h} \;   \Bigg[\sum_{k=1 \atop k \neq j_{0}}^{\aleph} \left( \left\Vert  \widetilde{Err}(x,j_{0},k,\overset{1}{\mathbb{P}}\left(E_{k}^{T}\right)) \right\Vert + \left\Vert \widetilde{Err}(x,j_{0},k,\overset{3}{\mathbb{P}}\left(E_{k}^{T}\right)) \right\Vert \right) \\ \nonumber
		&&  \quad  \qquad \cdot \sum_{i > k \atop i \neq j_{0}}^{\aleph} \left( \left\Vert \widetilde{Err}(x,j_{0},i,\overset{1}{\mathbb{P}}\left(E_{i}^{T}\right)) \right\Vert + \left\Vert \widetilde{Err}(x,j_{0},i,\overset{3}{\mathbb{P}}\left(E_{i}^{T}\right)) \right\Vert \right) \Bigg].
	\end{eqnarray*}
	By using $(\ref{HN1})$ and $(\ref{HN2})$, we obtain 
	\begin{eqnarray*}
		\left\vert I^{\star} \right\vert & \lesssim & a^{-2h} \;   \Bigg[\sum_{k=1 \atop k \neq j_{0}}^{\aleph} \left( a^{5} \, d^{-1}_{k,j_{0}} \, \left\Vert \overset{1}{\mathbb{P}}\left(\tilde{E}_{k}^{T}\right) \right\Vert + a^{6} \, d^{-4}_{k,j_{0}} \left\Vert \overset{1}{\mathbb{P}}\left(\tilde{E}_{k}^{T}\right) \right\Vert + a^{4} \, d^{-4}_{k,j_{0}} \left\Vert \overset{3}{\mathbb{P}}\left(\tilde{E}_{k}^{T}\right) \right\Vert \right)  \\ && \qquad \quad \; \cdot \sum_{i > k \atop i \neq j_{0}}^{\aleph} \left( a^{5} \, d^{-1}_{i,j_{0}} \,  \left\Vert \overset{1}{\mathbb{P}}\left(\tilde{E}_{i}^{T}\right) \right\Vert + a^{6} \, d^{-4}_{i,j_{0}} \left\Vert \overset{1}{\mathbb{P}}\left(\tilde{E}_{i}^{T}\right) \right\Vert + a^{4} \, d^{-4}_{i,j_{0}} \left\Vert \overset{3}{\mathbb{P}}\left(\tilde{E}_{i}^{T}\right) \right\Vert \right) \Bigg] \\ \nonumber
		& \lesssim & a^{-2h} \; \left(  \left( a^{5} \, d^{-3} + a^{6} \, d^{-4} \right) \,  \underset{i}{\max} \left\Vert \overset{1}{\mathbb{P}}\left(\tilde{E}_{i}^{T}\right) \right\Vert + a^{4} \, d^{-4} \underset{i}{\max} \left\Vert \overset{3}{\mathbb{P}}\left(\tilde{E}_{i}^{T}\right) \right\Vert \right)^{2} \\ \nonumber
		& \lesssim & a^{-2h} \; \left(  a^{10} \,d^{-6} \, \underset{i}{\max} \left\Vert \overset{1}{\mathbb{P}}\left(\tilde{E}_{i}^{T}\right) \right\Vert^{2} + a^{8} \, d^{-8} \underset{i}{\max} \left\Vert \overset{3}{\mathbb{P}}\left(\tilde{E}_{i}^{T}\right) \right\Vert^{2} \right). 
	\end{eqnarray*}
	Finally, 
	\begin{equation}\label{Istarformula}
	|I^{\star}| = \mathcal{O}\left( a^{10-2h} \; d^{-6} \; \underset{i}{\max} \left\Vert \overset{1}{\mathbb{P}}\left(\tilde{E}_{i}^{T}\right) \right\Vert^{2} + a^{8 - 2h} \, d^{-8} \underset{i}{\max} \left\Vert \overset{3}{\mathbb{P}}\left(\tilde{E}_{i}^{T}\right) \right\Vert^{2} \right).
	\end{equation}
	The analysis of the terms $\underset{n}{\sum} \left\vert \langle \cdots , R_{n} \rangle \right\vert^{2}$ is more technical. For an arbitrary vector field $F$, we have 
	\begin{eqnarray*}
		\langle F , R_{n} \rangle & \overset{(\ref{Def-Rn})}{=} & \frac{1 \mp c_{0} \, a^{h}}{4 \pi \,\lambda_{n_{0}}^{(1)} \left(1 - k^{2} \, \eta \, a^{2} \, \lambda_{n}^{(1)} \right)} \sum_{\ell \geq 1} (-ika)^{\ell+1} \langle F , \boldsymbol{T}_{- \, k \, a}^{-1}  \int_{B} \; \frac{\left\Vert \cdot - y \right\Vert^{\ell}}{(\ell+1)!}  e_{n}^{(1)}(y) \, dy \rangle \\
		& \simeq & \frac{1}{\left(1 - k^{2} \, \eta \, a^{2} \, \lambda_{n}^{(1)} \right)} \sum_{\ell \geq 1} (-ika)^{\ell+1} \langle \boldsymbol{T}_{k \, a}^{-1}(F) , \int_{B} \; \frac{\left\Vert \cdot - y \right\Vert^{\ell}}{(\ell+1)!}  e_{n}^{(1)}(y) \, dy \rangle \\
		& \simeq & \frac{1}{\left(1 - k^{2} \, \eta \, a^{2} \, \lambda_{n}^{(1)} \right)} \sum_{\ell \geq 1} (-ika)^{\ell+1} \langle \int_{B} \boldsymbol{T}_{k \, a}^{-1}(F)(x) \frac{\left\Vert \cdot - x \right\Vert^{\ell}}{(\ell+1)!} dx, e_{n}^{(1)} \rangle.
	\end{eqnarray*} 
	Taking the modulus, with the help of the Cauchy-Schwartz inequality, we obtain that
	\begin{equation*}
	\left\vert \langle F , R_{n} \rangle \right\vert^{2} \lesssim  \frac{a^{4}}{\left\vert 1 - k^{2} \, \eta \, a^{2} \, \lambda_{n}^{(1)} \right\vert^{2}} \sum_{\ell \geq 1} \left\vert \langle \int_{B} \boldsymbol{T}_{k \, a}^{-1}(F)(x) \frac{\left\Vert \cdot - x \right\Vert^{\ell}}{(\ell+1)!} dx, e_{n}^{(1)} \rangle \right\vert^{2}. 
	\end{equation*}
	Then, 
	\begin{eqnarray*}
\sum_{n} \left\vert \langle F , R_{n} \rangle \right\vert^{2}		& \lesssim & a^{4-2h} \, \sum_{\ell \geq 1} \, \left\Vert \int_{B} \boldsymbol{T}_{k \, a}^{-1}(F)(x) \frac{\left\Vert \cdot - x \right\Vert^{\ell}}{(\ell+1)!} dx \right\Vert^{2}_{\mathbb{L}^{2}(B)} \\
		& \lesssim & a^{4-2h} \, \left\Vert  \boldsymbol{T}_{k \, a}^{-1}(F) \right\Vert^{2}_{\mathbb{L}^{2}(B)} \sum_{\ell \geq 1} \, \int_{B} \int_{B}  \frac{\left(\left\Vert y - x \right\Vert^{2}\right)^{\ell}}{((\ell+1)!)^{2}} dx \, dy,
	\end{eqnarray*}
	where we have the convergence of the series $
	\underset{\ell \geq 1}{\sum} \, \int_{B} \int_{B}  \frac{\left(\left\Vert y - x \right\Vert^{2}\right)^{\ell}}{((\ell+1)!)^{2}} dx \, dy.$ 
	Since we are approaching an eigenvalue of the first family, the dominant part of  $\left\Vert  \boldsymbol{T}_{k \, a}^{-1}(F) \right\Vert^{2}_{\mathbb{L}^{2}(B)}$ is  $\left\Vert  \boldsymbol{T}_{k \, a}^{-1}\left(\overset{1}{\mathbb{P}}\left(F \right)\right) \right\Vert^{2}_{\mathbb{L}^{2}(B)}$ and then  
	\begin{equation*}
		\sum_{n} \left\vert \langle F , R_{n} \rangle \right\vert^{2}  \lesssim  a^{4-2h} \; \left\Vert  \boldsymbol{T}_{k \, a}^{-1}\left(\overset{1}{\mathbb{P}}\left(F \right)\right) \right\Vert^{2}_{\mathbb{L}^{2}(B)} \simeq  a^{4-2h} \; \sum_{\ell} \frac{\left\vert \langle  \overset{1}{\mathbb{P}}\left(F \right), e_{\ell}^{(1)} \rangle \right\vert^{2}}{\left\vert 1 - k^{2} \, \eta \, \lambda_{\ell}^{(1)} \, a^{2} \right\vert^{2}} \simeq  a^{4-4h} \; \left\vert \langle  \overset{1}{\mathbb{P}}\left(F \right), e_{n_{0}}^{(1)} \rangle \right\vert^{2}. 
	\end{equation*}
	Finally, 
	\begin{equation}\label{serie-Rn}
	\sum_{n} \left\vert \langle F , R_{n} \rangle \right\vert^{2} = \mathcal{O}\left( a^{4-4h} \; \left\vert \langle  \overset{1}{\mathbb{P}}\left(F \right), e_{n_{0}}^{(1)} \rangle \right\vert^{2} \right). 
	\end{equation}
	Remark that when $F$ or $\overset{1}{\mathbb{P}}\left(F \right)$ is a constant vector field, using the representation $e_{\ell}^{(1)} = Curl(\phi_{\ell}), \; \div(\phi_{\ell}) = 0, \nu \times \phi_{\ell} = 0$ and integration by parts, the series $
	\underset{\ell}{\sum} \left\vert 1 - k^{2} \, \eta \, \lambda_{\ell}^{(1)} \, a^{2} \right\vert^{-2} \left\vert \langle  \overset{1}{\mathbb{P}}\left(F \right), e_{\ell}^{(1)} \rangle \right\vert^{2}$
 is vanishing. In this case, consequently, $\underset{n}{\sum} \left\vert \langle F , R_{n} \rangle \right\vert^{2}$ will be zero. \\
	By utilizing $(\ref{I-Error-term}), (\ref{Istarformula})$ and $(\ref{serie-Rn})$, the equation $(\ref{MLB})$ takes the following form 
	\begin{eqnarray}\label{mid}
	\nonumber
	\left\Vert \overset{1}{\mathbb{P}}\left(\tilde{E}_{j_{0}}^{T}\right) \right\Vert^{2} & \lesssim & \frac{\left\vert \langle \tilde{E}^{Inc}_{j_{0}},e_{n_{0}}^{(1)} \rangle \right\vert^{2}}{\left\vert 1 - k^{2} \, \eta \, a^{2} \, \lambda_{n_{0}}^{(1)} \right\vert^{2}} + \sum_{n \neq n_{0}} \, \left\vert \langle \tilde{E}^{Inc}_{j_{0}},e_{n}^{(1)} \rangle \right\vert^{2}+ a^{4-4h}  \left\vert \langle \overset{1}{\mathbb{P}}\left( \tilde{E}^{Inc}_{j_{0}}\right), e^{(1)}_{n_{0}} \rangle \right\vert^{2} \notag\\ 
	&+&  a^{6-2h} \, d^{-6} \; \max_{j}  \left\Vert \overset{1}{\mathbb{P}}\left(\tilde{E}_{j}^{T}\right) \right\Vert^{2}  +  \, a^{4-2h} d^{-8} \max_{j}  \left\Vert \overset{3}{\mathbb{P}}\left(\tilde{E}_{j}^{T}\right) \right\Vert^{2} \notag \\
	&+& a^{-4h} \, \aleph \, \sum_{j \neq j_{0}}^{\aleph} \left\vert \langle \overset{1}{\mathbb{P}}\left( \widetilde{Err}(x,j_{0},j,\overset{1}{\mathbb{P}}\left(E_{j}^{T}\right)) + \widetilde{Err}(x,j_{0},j,\overset{3}{\mathbb{P}}\left(E_{j}^{T}\right))\right), e^{(1)}_{n_{0}} \rangle \right\vert^{2}. 
	\end{eqnarray}
	
		For the last term, we know that 
		\begin{align}\notag
		&\left\vert \langle \overset{1}{\mathbb{P}}\left( \widetilde{Err}(x,j_{0},j,\overset{1}{\mathbb{P}}\left(E_{j}^{T}\right)) + \widetilde{Err}(x,j_{0},j,\overset{3}{\mathbb{P}}\left(E_{j}^{T}\right))\right), e^{(1)}_{n_{0}} \rangle \right\vert^{2}\notag\\
		=&\ \left\vert \langle  \widetilde{Err}(x,j_{0},j,\overset{1}{\mathbb{P}}\left(E_{j}^{T}\right)) + \widetilde{Err}(x,j_{0},j,\overset{3}{\mathbb{P}}\left(E_{j}^{T}\right)), e^{(1)}_{n_{0}} \rangle \right\vert^{2}\notag\\
		\lesssim&\ \left\vert \langle  \widetilde{Err}(x,j_{0},j,\overset{1}{\mathbb{P}}\left(E_{j}^{T}\right)), e^{(1)}_{n_{0}} \rangle \right\vert^{2}+\left\vert \langle   \widetilde{Err}(x,j_{0},j,\overset{3}{\mathbb{P}}\left(E_{j}^{T}\right)), e^{(1)}_{n_{0}} \rangle \right\vert^{2},\notag
		\end{align}
		where we can see from \eqref{LMERR1} that
		\begin{equation}\notag
		\left\lvert\widetilde{Err}(x,j_{0},j,\overset{1}{\mathbb{P}}\left(E_{j}^{T}\right)) \right\rvert\lesssim a^5\int_{B}\int_{0}^1 \underset{y}{Hess}\left(\Phi_{ka} \right)(z_{j_0}+tax, z_j)\cdot x\,dt\cdot y\cdot \overset{1}{\mathbb{P}}\left(\tilde{E}_{j}^{T}\right)(y)\,dy.
		\end{equation}
		Then there holds
		\begin{align}\label{comp1}
		&\left\vert \langle  \widetilde{Err}(x,j_{0},j,\overset{1}{\mathbb{P}}\left(E_{j}^{T}\right)), e^{(1)}_{n_{0}} \rangle \right\vert^{2}\notag\\
		\lesssim & \ a^{10}\left\lvert\int_{B}e^{(1)}_{n_{0}}(x)\int_{B}\int_{0}^1\underset{y}{Hess}\left(\Phi_{ka} \right)(z_{j_0}+tax, z_j)\cdot x\,dt\cdot y\cdot \overset{1}{\mathbb{P}}\left(\tilde{E}_{j}^{T}\right)(y)\,dy\,dx\right\rvert ^2\notag\\
		\lesssim & \
		a^{10}\int_{B}\left\lvert\int_{B}\int_{0}^1\underset{y}{Hess}\left(\Phi_{ka} \right)(z_{j_0}+tax, z_j)\cdot x\,dt\cdot y\cdot \overset{1}{\mathbb{P}}\left(\tilde{E}_{j}^{T}\right)(y)\,dy\right\rvert^2\,dx\notag\\
		\lesssim& \
		a^{10}\int_{B} \left\lVert\int_{0}^1\underset{y}{Hess}\left(\Phi_{ka} \right)(z_{j_0}+tax, z_j)\,dt\right\rVert^2\cdot \left\lVert\overset{1}{\mathbb{P}}\left(\tilde{E}_{j}^{T}\right)\right\rVert^2\,dx\notag\\
		\lesssim & \
		a^{10}\frac{1}{\left|z_{j_0}-z_j\right|^6}\left\lVert\overset{1}{\mathbb{P}}\left(\tilde{E}_{j}^{T}\right)\right\rVert^2.
		\end{align}
		Similarly, from \eqref{LMERR3}, we deduce that
		\begin{align}\label{comp2}
		&\left\vert \langle  \widetilde{Err}(x,j_{0},j,\overset{3}{\mathbb{P}}\left(E_{j}^{T}\right)), e^{(1)}_{n_{0}} \rangle \right\vert^{2}\notag\\
		\lesssim& \
		a^8 \left\lvert\int_B\int_{0}^1\nabla\left( \Upsilon_{ka} \right)(z_{j_0}+tax, z_j)\cdot\mathcal{P}(x,0)\,dt\int_{B}\overset{3}{\mathbb{P}}\left(\tilde{E}_{j}^{T}\right)(y)\,d y\cdot e_{n_0}^{(1)}(x)\,d x\right\rvert^2\notag\\
		\lesssim& \
		a^8\frac{1}{|z_{j_0}-z_j|^8}\left\lVert\overset{3}{\mathbb{P}}\left(\tilde{E}_{j}^{T}\right)\right\rVert^2.
		\end{align}
		Combining with \eqref{comp1} and \eqref{comp2}, it is easy to see that in \eqref{mid}
		\begin{align}\label{comp3}
		&a^{-4h} \aleph \sum_{j \neq j_{0}}^\aleph \left\vert \langle \overset{1}{\mathbb{P}}\left( \widetilde{Err}(x,j_{0},j,\overset{1}{\mathbb{P}}\left(E_{j}^{T}\right)) + \widetilde{Err}(x,j_{0},j,\overset{3}{\mathbb{P}}\left(E_{j}^{T}\right))\right), e^{(1)}_{n_{0}} \rangle \right\vert^{2}\notag\\
		\lesssim & \
		a^{-4h} \aleph \sum_{j \neq j_{0}}^\aleph \left(a^{10}\frac{1}{|z_{j_0}-z_j|^6}\left\lVert\overset{1}{\mathbb{P}}\left(\tilde{E}_{j}^{T}\right)\right\rVert^2+a^8\frac{1}{|z_{j_0}-z_j|^8}\left\lVert\overset{3}{\mathbb{P}}\left(\tilde{E}_{j}^{T}\right)\right\rVert^2\right)\notag\\
		\lesssim & \
		a^{10-4h} \aleph d^{-6}\max_{j}\left\lVert\overset{1}{\mathbb{P}}\left(\tilde{E}_{j}^{T}\right)\right\rVert^2+a^{8-4h}\aleph d^{-8}\max_{j}\left\lVert\overset{3}{\mathbb{P}}\left(\tilde{E}_{j}^{T}\right)\right\rVert^2.
		\end{align}
		Now, we compare the orders of $a$ in \eqref{comp3} with those in the following terms
		\begin{equation}\label{term2}
		a^{6-2h} \, d^{-6} \; \max_{j}  \left\Vert \overset{1}{\mathbb{P}}\left(\tilde{E}_{j}^{T}\right) \right\Vert^{2}  +  \, a^{4-2h} d^{-8} \max_{j}  \left\Vert \overset{3}{\mathbb{P}}\left(\tilde{E}_{j}^{T}\right) \right\Vert^{2},
		\end{equation}
		and by direct computations, we derive that the condition 
		\begin{equation}\label{cond1}
		4-s-2h\geq0,
		\end{equation}
		is sufficient to ensure that \eqref{term2} is dominated. 

	For the other terms, we have 
	\begin{equation*}
	\langle \tilde{E}^{Inc}_{j_{0}},e_{n}^{(1)} \rangle = \langle \tilde{E}^{Inc}_{j_{0}}, Curl\left( \phi_{n} \right) \rangle = \langle Curl\left( \tilde{E}^{Inc}_{j_{0}} \right),  \phi_{n}  \rangle = i k \, a \, \langle  \tilde{H}^{Inc}_{j_{0}},  \phi_{n}  \rangle,
	\end{equation*} 
	and this implies  
		\begin{eqnarray*}
			\sum_{n \neq n_{0}} \, \left\vert \langle \tilde{E}^{Inc}_{j_{0}},e_{n}^{(1)} \rangle \right\vert^{2}+ a^{4-4h}  \left\vert \langle \overset{1}{\mathbb{P}}\left( \tilde{E}^{Inc}_{j_{0}}\right), e^{1}_{n_{0}} \rangle \right\vert^{2} & \sim & a^{2} + a^{6-4h} = \mathcal{O}(a^{\min( 2,6-4h)})=\mathcal{O}(a^{2}), 
		\end{eqnarray*} 
		for $h\leq 1$ by \eqref{cond1}. We end up with the following equation
	\begin{eqnarray*}
		\left\Vert \overset{1}{\mathbb{P}}\left(\tilde{E}_{j_{0}}^{T}\right) \right\Vert^{2}  & \lesssim &    a^{2-2h} \, \left\vert \langle \tilde{H}^{Inc}_{j_{0}},\phi_{n_{0}} \rangle \right\vert^{2} +  a^{6-2h} \, d^{-6} \, \max_{j}  \left\Vert \overset{1}{\mathbb{P}}\left(\tilde{E}_{j}^{T}\right) \right\Vert^{2} \\ &+& a^{4-2h} d^{-8} \max_{j}  \left\Vert \overset{3}{\mathbb{P}}\left(\tilde{E}_{j}^{T}\right) \right\Vert^{2}  + \mathcal{O}(a^{2}).
	\end{eqnarray*}
	Taking the $\max\left( \cdot \right)$ with respect to $j_{0}$ on the both sides, we obtain that
	\begin{eqnarray*}
		\max_{j_{0}} \left\Vert \overset{1}{\mathbb{P}}\left(\tilde{E}_{j_{0}}^{T}\right) \right\Vert^{2}  & \lesssim &  a^{2-2h} \, \max_{j_{0}} \left\vert \langle \tilde{H}^{Inc}_{j_{0}},\phi_{n_{0}} \rangle \right\vert^{2} +   a^{6-2h} \, d^{-6} \; \max_{j}  \left\Vert \overset{1}{\mathbb{P}}\left(\tilde{E}_{j}^{T}\right) \right\Vert^{2} \\ &+& a^{4-2h} d^{-8} \max_{j}  \left\Vert \overset{3}{\mathbb{P}}\left(\tilde{E}_{j}^{T}\right) \right\Vert^{2}  + \mathcal{O}( a^{2}).
	\end{eqnarray*}
	Assuming that $\; a^{6-2h} \, d^{-6} \, \sim  \, a^{6-2h-6t} \;\;$ is small enough or, the following sufficient condition 
	\begin{equation}\label{Cdt=0}
	3-h-3t > 0\;  
	\end{equation} 
	is satisfied\footnote{The limit condition $3-h-3t=0$ can also be considered, see Remark \ref{rem-cond1}.}, then,  
	\begin{equation}\label{MP1<MP3}
	\max_{j_{0}} \left\Vert \overset{1}{\mathbb{P}}\left(\tilde{E}_{j_{0}}^{T}\right) \right\Vert^{2}  \lesssim   a^{2-2h} \, \max_{j_{0}} \left\vert \langle \tilde{H}^{Inc}_{j_{0}},\phi_{n_{0}} \rangle \right\vert^{2} +   a^{4-2h} d^{-8} \max_{j}  \left\Vert \overset{3}{\mathbb{P}}\left(\tilde{E}_{j}^{T}\right) \right\Vert^{2}.
	\end{equation}
	
	
	The goal of the next part is to derive an analogous formula to $(\ref{MP1<MP3})$.  
	\item[$ii)$] \textit{Estimation of} $\underset{j}{\max}  \left\Vert  \overset{3}{\mathbb{P}}\left(\tilde{E}_{j}^{T} \right) \right\Vert^{2}$ \\
	For $x \in D_{j_{0}}$, we have 
	\begin{equation*}
	\left(I + \eta \, \underset{x}{\nabla} M^{k} \right) \, \left( E_{j_{0}}^{T} \right)(x) + \eta \, \sum_{j \neq j_{0}}^{\aleph} \underset{x}{\nabla}M^{k}(E_{j}^{T})(x) - k^{2}  \, \eta \, \sum_{j = 1}^{\aleph} \, N^{k}(E_{j}^{T})(x) = E^{Inc}_{j_{0}}(x).
	\end{equation*}
	Using the integration by parts, we deduce that  
	$\nabla M^{k}\left( \overset{1}{\mathbb{P}}\left( E_{j}^{T} \right) \right) = 0,\, j=1,\cdots,\aleph$. Then 
		\begin{eqnarray}\label{Att-Vienne} 
		\nonumber 
		\left(I + \eta \, \underset{x}{\nabla} M^{k} \right) \, \left( \overset{3}{\mathbb{P}}\left(E_{j_{0}}^{T} \right) \right)(x) &=& E^{Inc}_{j_{0}}(x) + k^{2}  \, \eta \,  N^{k}\left(\overset{1}{\mathbb{P}}\left(E_{j_{0}}^{T}\right) \right)(x) - \overset{1}{\mathbb{P}}\left( E_{j_{0}}^{T}\right)(x) + k^2 \, \eta \, N^k \left( \overset{3}{\mathbb{P}}\left(E_{j_{0}}^{T} \right) \right)(x) \\ 
		&+& k^{2}  \, \eta \, \sum_{j \neq j_{0}}^{\aleph} \int_{D_{j}} \Phi_{k}(x,y) \, I \cdot \overset{1}{\mathbb{P}}\left(E_{j}^{T}\right)(y) \, dy  + \eta \, \sum_{j \neq j_{0}}^{\aleph} \, \int_{D_{j}} \Upsilon_{k}(x,y) \cdot \overset{3}{\mathbb{P}}\left(E_{j}^{T}\right)(y) \, dy,
		\end{eqnarray} 
	where $\Upsilon_{k}(\cdot,\cdot)$ is given by $(\ref{dyadicG})$.
	By Taylor expansion near the center $z_{j}$, we obtain that
	\begin{eqnarray*}
		\int_{D_{j}} \Upsilon_{k}(x,y) \cdot \overset{3}{\mathbb{P}}\left(E_{j}^{T}\right)(y) \, dy &=& \Upsilon_{k}(z_{j_{0}},z_{j}) \cdot \int_{D_{j}} \overset{3}{\mathbb{P}}\left(E_{j}^{T}\right)(y) \, dy \\ &+& 
		\int_{0}^{1} \underset{x}{\nabla}\left(\Upsilon_{k}\right)(z_{j_{0}} + t (x - z_{j_{0}}),z_{j}) \cdot \mathcal{P}(x,z_{j_{0}}) dt \cdot \int_{D_{j}} \overset{3}{\mathbb{P}}\left(E_{j}^{T}\right)(y) \, dy \\ &+& \int_{D_{j}} \int_{0}^{1} \underset{y}{\nabla}\left(\Upsilon_{k}\right)(x,z_{j}+t(y-z_{j})) \cdot \mathcal{P}(y,z_{j}) \, dt \cdot \overset{3}{\mathbb{P}}\left(E_{j}^{T}\right)(y) \, dy,
	\end{eqnarray*} 
	and 
	\begin{eqnarray}\label{*add5}
		\int_{D_{j}} \Phi_{k}(x,y) \, I \cdot \overset{1}{\mathbb{P}}\left(E_{j}^{T}\right)(y) dy & = & \Phi_{k}(x,z_{j}) \, I \cdot \int_{D_{j}}  \overset{1}{\mathbb{P}}\left(E_{j}^{T}\right)(y) dy \\ 
		&+&  \int_{D_{j}} \int_{0}^{1} \underset{y}{\nabla}\left( \Phi_{k} \, I \right)(x,z_{j}+t(y-z_{j})) \cdot \mathcal{P}(y,z_{j}) dt \cdot \overset{1}{\mathbb{P}}\left(E_{j}^{T}\right)(y) dy.\notag
	\end{eqnarray}
	The first integral on the right hand side of \eqref{*add5} is vanishing and the second term by Taylor expansion give us that
	\begin{eqnarray*}
&&	\int_{D_{j}} \Phi_{k}(x,y) \, I \cdot \overset{1}{\mathbb{P}}\left(E_{j}^{T}\right)(y) dy \\ &=& - \int_{D_{j}} \int_{0}^{1} \int_{0}^{1} \underset{y}{Hess}\left(\Phi_{k}\right)(z_{j_{0}}+h(x-z_{j_{0}}),z_{j}+t(y-z_{j})) \cdot (x-z_{j_{0}}) \; dh \cdot (y - z_{j}) dt \; \overset{1}{\mathbb{P}}\left(E_{j}^{T}\right)(y) dy.
	\end{eqnarray*}
	Plugging this expansion into equation $(\ref{Att-Vienne})$, we can deduce that 
	\begin{eqnarray*}
		&& \left(I + \eta \, \underset{x}{\nabla} M^{k} \right) \, \left( \overset{3}{\mathbb{P}}\left(E_{j_{0}}^{T} \right) \right)(x) = E^{Inc}_{j_{0}}(x) + k^{2}  \, \eta \,  N^{k}\left(\overset{1}{\mathbb{P}}\left(E_{j_{0}}^{T}\right) \right)(x) + k^{2}  \, \eta \,  N^{k}\left(\overset{3}{\mathbb{P}}\left(E_{j_{0}}^{T}\right) \right)(x) -\overset{1}{\mathbb{P}}\left( E_{j_{0}}^{T}\right)(x)  \\
		&-& k^{2}  \, \eta \, \sum_{j \neq j_{0}}^{\aleph} \int_{D_{j}} \int_{0}^{1} \int_{0}^{1} \underset{y}{Hess}\left(\Phi_{k}\right)(z_{j_{0}}+h(x-z_{j_{0}}),z_{j}+t(y-z_{j})) \cdot (x-z_{j_{0}}) \, dh \cdot (y - z_{j}) dt \; \overset{1}{\mathbb{P}}\left(E_{j}^{T}\right)(y) \, dy \\  
		&+& \eta \, \sum_{j \neq j_{0}}^{\aleph} \Upsilon_{k}(z_{j_{0}},z_{j})  \cdot \int_{D_{j}}\overset{3}{\mathbb{P}}\left(E_{j}^{T}\right)(y) \, dy + \eta \, \sum_{j \neq j_{0}}^{\aleph} \, \int_{0}^{1} \underset{x}{\nabla}\left(\Upsilon_{k}\right)(z_{j_{0}} + t (x - z_{j_{0}}),z_{j}) \cdot \mathcal{P}(x,z_{j_{0}}) dt \cdot \int_{D_{j}} \overset{3}{\mathbb{P}}\left(E_{j}^{T}\right)(y) \, dy \\
		&+& \eta \, \sum_{j \neq j_{0}}^\aleph \int_{D_{j}} \int_{0}^{1} \underset{y}{\nabla}\left(\Upsilon_{k}\right)(x,z_{j}+t(y-z_{j})) \cdot \mathcal{P}(y,z_{j}) \, dt \cdot \overset{3}{\mathbb{P}}\left(E_{j}^{T}\right)(y) \, dy.
	\end{eqnarray*} 
	Scaling to the domain $B$, and by 
	taking the inner product with respect to $e^{(3)}_{n}(\cdot)$, we obtain that
	\begin{equation}\label{coeff-3rd-proj-sev-par}
	\langle  \overset{3}{\mathbb{P}}\left(\tilde{E}_{j_{0}}^{T} \right) , e^{(3)}_{n} \rangle = \frac{1}{(1 + \eta \, \lambda_{n}^{(3)})} \left[ \langle \tilde{E}^{Inc}_{j_{0}}, e^{(3)}_{n} \rangle + \eta \, a^{3} \, \sum_{j \neq j_{0}}^{\aleph} \langle \Upsilon_{k}(z_{j_{0}},z_{j})  \cdot \int_{B}\overset{3}{\mathbb{P}}\left(\tilde{E}_{j}^{T}\right)(y) \, dy, e^{(3)}_{n} \rangle + Error(n,j_{0}) \right],
	\end{equation}
	where 
	\begin{eqnarray*}
		Error(n,j_{0}) &:=&  k^{2}  \, \eta \, a^{2} \, \langle N^{ka}\left(\overset{1}{\mathbb{P}}\left(\tilde{E}_{j_{0}}^{T}\right) \right), e^{(3)}_{n} \rangle  - \eta \, \langle \left(\underset{x}{\nabla} M^{ka} -  \underset{x}{\nabla} M \right) \, \left( \overset{3}{\mathbb{P}}\left(\tilde{E}_{j_{0}}^{T} \right) \right), e^{(3)}_{n} \rangle \\ &+& k^{2}  \, \eta \, a^{2} \, \langle N^{ka}\left(\overset{3}{\mathbb{P}}\left(\tilde{E}_{j_{0}}^{T}\right) \right), e^{(3)}_{n} \rangle  \\ 
		&+&  \eta \,a^{4} \, \sum_{j \neq j_{0}}^{\aleph} \, \langle \int_{0}^{1} \underset{x}{\nabla}\left(\Upsilon_{ka}\right)(z_{j_{0}} + t \, a \,  \cdot ,z_{j}) \cdot \mathcal{P}(x,0) dt \cdot \int_{B} \overset{3}{\mathbb{P}}\left(\tilde{E}_{j}^{T}\right)(y) \, dy, e^{(3)}_{n} \rangle \\
		&-& k^{2}  \, \eta \, a^{5} \, \sum_{j \neq j_{0}}^{\aleph} \langle \int_{B} \int_{0}^{1} \int_{0}^{1} \underset{y}{Hess}\left(\Phi_{ka}\right)(z_{j_{0}}+h \, a \, \cdot ,z_{j}+t \, a \, y) \cdot x \, dh \cdot y \, dt \; \overset{1}{\mathbb{P}}\left(\tilde{E}_{j}^{T}\right)(y) \, dy, e^{(3)}_{n} \rangle \\  
		&+& \eta \, a^{4} \, \sum_{j \neq j_{0}}^\aleph \langle \int_{B} \int_{0}^{1} \underset{y}{\nabla}\left(\Upsilon_{ka}\right)(z_{j_{0}}+a \cdot ,z_{j}+t \, a \, y) \cdot \mathcal{P}(y,0) \, dt \cdot \overset{3}{\mathbb{P}}\left(\tilde{E}_{j}^{T}\right)(y) \, dy, e^{(3)}_{n} \rangle.
	\end{eqnarray*}
	We have 
	\begin{eqnarray*}
		k^{2}  \, \eta \, a^{2} \, \langle N^{ka}\left(\overset{1}{\mathbb{P}}\left(\tilde{E}_{j_{0}}^{T}\right) \right), e^{(3)}_{n} \rangle  &=&   k^{2}  \, \eta \, a^{2} \, \langle N\left(\overset{1}{\mathbb{P}}\left(\tilde{E}_{j_{0}}^{T}\right) \right), e^{(3)}_{n} \rangle + k^{2}  \, \eta \, a^{2} \, \langle \left( N^{ka}-N \right) \left(\overset{1}{\mathbb{P}}\left(\tilde{E}_{j_{0}}^{T}\right) \right), e^{(3)}_{n} \rangle \\ 
		&=& k^{2}  \, \eta \, a^{2} \,\sum_{\ell \geq 1} \langle \overset{1}{\mathbb{P}}\left(\tilde{E}_{j_{0}}^{T}\right) , e^{(1)}_{\ell} \rangle \, \lambda^{(1)}_{\ell} \,  \langle e_{\ell}^{(1)}, e^{(3)}_{n} \rangle  + k^{2}  \, \eta \, a^{2} \, \langle \left( N^{ka}-N \right) \left(\overset{1}{\mathbb{P}}\left(\tilde{E}_{j_{0}}^{T}\right) \right), e^{(3)}_{n} \rangle  \\ 
		& \overset{(\ref{choice-k-1st-regime})}{=}& 0 + \frac{1 \mp c_{0} \, a^{h}}{\lambda^{(1)}_{n_{0}}} \, \langle \left( N^{ka}-N \right) \left(\overset{1}{\mathbb{P}}\left(\tilde{E}_{j_{0}}^{T}\right) \right), e^{(3)}_{n} \rangle \\
		& \overset{(\ref{expansion-Nk})}{=}& \frac{1 \mp c_{0} \, a^{h}}{\lambda^{(1)}_{n_{0}}} \, \frac{ika}{4\pi} \langle \int_{B} \overset{1}{\mathbb{P}}\left(\tilde{E}_{j_{0}}^{T}\right)(y) \, dy, e^{(3)}_{n} \rangle \\ 
		& + & \frac{1 \mp c_{0} \, a^{h}}{4 \, \pi \, \lambda^{(1)}_{n_{0}}}  \, \sum_{\ell \geq 1} \frac{\left( ika \right)^{\ell + 1 }}{(\ell + 1)!} \langle \int_{B} \left\vert \cdot - y \right\vert^{\ell} \overset{1}{\mathbb{P}}\left(\tilde{E}_{j_{0}}^{T}\right)(y) \, dy , e^{(3)}_{n} \rangle \\
		& = & 0 + \frac{1 \mp c_{0} \, a^{h}}{4 \, \pi \, \lambda^{(1)}_{n_{0}}}  \, \sum_{\ell \geq 1} \frac{\left( ika \right)^{\ell + 1 }}{(\ell + 1)!} \langle \int_{B} \left\vert \cdot - y \right\vert^{\ell} \overset{1}{\mathbb{P}}\left(\tilde{E}_{j_{0}}^{T}\right)(y) \, dy , e^{(3)}_{n} \rangle. 
	\end{eqnarray*}
	Similarly, using $(\ref{choice-k-1st-regime})$ and $(\ref{expansion-gradMk})$, we rewrite
	\begin{eqnarray*}
		\eta \, \langle \left(\underset{x}{\nabla} M^{ka} -  \underset{x}{\nabla} M \right) \, \left( \overset{3}{\mathbb{P}}\left(\tilde{E}_{j_{0}}^{T} \right) \right), e^{(3)}_{n} \rangle &=& \frac{1 \mp c_{0} \, a^{h}}{2 \lambda^{(1)}_{n_{0}} }  \, \langle N \left( \overset{3}{\mathbb{P}}\left(\tilde{E}_{j_{0}}^{T} \right) \right), e^{(3)}_{n} \rangle + \eta \,\frac{i \, (ka)^{3}}{12 \, \pi}  \langle \int_{B} \overset{3}{\mathbb{P}}\left(\tilde{E}_{j_{0}}^{T}(y)\right) \, dy, e^{(3)}_{n} \rangle \\ 
		&-& \frac{1 \mp c_{0} \, a^{h}}{2 \lambda^{(1)}_{n_{0}} } \, \langle \int_{B} \Phi_{0}(\cdot,y) \frac{A(\cdot,y) \cdot \overset{3}{\mathbb{P}}\left(\tilde{E}_{j_{0}}^{T}(y)\right)}{\left\Vert \cdot - y \right\Vert^{2}} dy, e^{(3)}_{n} \rangle \\ 
		&-& \frac{\eta}{4\pi} \,\sum_{\ell \geq 3} \left(i \, k \, a \right)^{\ell+1} \langle  \int_{B} \frac{\underset{y}{Hess}\left( \left\vert \cdot - y \right\vert^{\ell} \right)}{(\ell +1 )!}  \cdot \overset{3}{\mathbb{P}}\left(\tilde{E}_{j_{0}}^{T}\right)(y) \, dy , e^{(3)}_{n} \rangle .
	\end{eqnarray*}
	Then, $ Error(n,j_{0})$ takes the form 
	\begin{eqnarray}\label{Error(n,j,j0)}
	\nonumber
	Error(n,j_{0}) &:=&  \frac{1 \mp c_{0} \, a^{h}}{4 \, \pi \, \lambda^{(1)}_{n_{0}}}  \, \sum_{\ell \geq 1} \frac{\left( ika \right)^{\ell + 1 }}{(\ell + 1)!} \langle \int_{B} \left\vert \cdot - y \right\vert^{\ell} \overset{1}{\mathbb{P}}\left(\tilde{E}_{j_{0}}^{T}\right)(y) \, dy , e^{(3)}_{n} \rangle  - \frac{1 \mp c_{0} \, a^{h}}{2 \lambda^{(1)}_{n_{0}} }  \, \langle N \left( \overset{3}{\mathbb{P}}\left(\tilde{E}_{j_{0}}^{T} \right) \right), e^{(3)}_{n} \rangle \\ \nonumber &-& \eta \,\frac{i \, (ka)^{3}}{12 \, \pi}  \langle \int_{B} \overset{3}{\mathbb{P}}\left(\tilde{E}_{j_{0}}^{T}(y)\right)(y) \, dy, e^{(3)}_{n} \rangle + \frac{1 \mp c_{0} \, a^{h}}{2 \lambda^{(1)}_{n_{0}} } \, \langle \int_{B} \Phi_{0}(\cdot,y) \frac{A(\cdot,y) \cdot \overset{3}{\mathbb{P}}\left(\tilde{E}_{j_{0}}^{T}(y)\right)}{\left\Vert \cdot - y \right\Vert^{2}} dy, e^{(3)}_{n} \rangle \\ \nonumber 
	&+& \frac{\eta}{4\pi} \,\sum_{\ell \geq 3} \left(i \, k \, a \right)^{\ell+1} \langle  \int_{B} \frac{\underset{y}{Hess}\left( \left\vert \cdot - y \right\vert^{\ell} \right)}{(\ell +1 )!}  \cdot \overset{3}{\mathbb{P}}\left(\tilde{E}_{j_{0}}^{T}\right)(y) \, dy , e^{(3)}_{n} \rangle +  k^{2}  \, \eta \, a^{2} \, \langle N^{ka}\left(\overset{3}{\mathbb{P}}\left(\tilde{E}^T_{j_0}\right)\right), e_n^{(3)} \rangle  \\ 
	\nonumber &+&  \eta \,a^{4} \, \sum_{j \neq j_{0}}^{\aleph} \, \langle \int_{0}^{1} \underset{x}{\nabla}\left(\Upsilon_{ka}\right)(z_{j_{0}} + t \, a \,  \cdot ,z_{j}) \cdot \mathcal{P}(x,0) dt \cdot \int_{B} \overset{3}{\mathbb{P}}\left(\tilde{E}_{j}^{T}\right)(y) \, dy, e^{(3)}_{n} \rangle \\ \nonumber
	&-& k^{2}  \, \eta \, a^{5} \, \sum_{j \neq j_{0}}^{\aleph} \langle \int_{B} \int_{0}^{1} \int_{0}^{1} \underset{y}{Hess}\left(\Phi_{ka}\right)(z_{j_{0}}+h \, a \, \cdot ,z_{j}+t \, a \, y) \cdot x \, dh \cdot y \, dt \; \overset{1}{\mathbb{P}}\left(\tilde{E}_{j}^{T}\right)(y) \, dy, e^{(3)}_{n} \rangle \\  
	&+& \eta \, a^{4} \, \sum_{j \neq j_{0}}^\aleph \langle \int_{B} \int_{0}^{1} \underset{y}{\nabla}\left(\Upsilon_{ka}\right)(z_{j_{0}}+a \cdot ,z_{j}+t \, a \, y) \cdot \mathcal{P}(y,0) \, dt \cdot \overset{3}{\mathbb{P}}\left(\tilde{E}_{j}^{T}\right)(y) \, dy, e^{(3)}_{n} \rangle.
	\end{eqnarray}
	Taking the squared modulus on the both sides of $(\ref{coeff-3rd-proj-sev-par})$, it yields
	\begin{eqnarray*}
		\left\vert \langle  \overset{3}{\mathbb{P}}\left(\tilde{E}_{j_{0}}^{T} \right) , e^{(3)}_{n} \rangle \right\vert^{2} & \lesssim & \left\vert \eta \right\vert^{-2} \Bigg[ \left\vert \langle \tilde{E}^{Inc}_{j_{0}}, e^{(3)}_{n} \rangle \right\vert^{2} +  a^{2} \, \sum_{j \neq j_{0}}^{\aleph} \left\vert \langle \Upsilon_{k}(z_{j_{0}},z_{j})  \cdot \int_{B}\overset{3}{\mathbb{P}}\left(\tilde{E}_{j}^{T}\right)(y) \, dy, e^{(3)}_{n} \rangle \right\vert^{2}  \\
		& + & a^{2} \, \sum_{j=1 \atop j \neq j_{0}}^{\aleph} \left\vert \langle \Upsilon_{k}(z_{j_{0}},z_{j})  \cdot \int_{B}\overset{3}{\mathbb{P}}\left(\tilde{E}_{j}^{T}\right)(y) \, dy, e^{(3)}_{n} \rangle \right\vert \, \sum_{i > j \atop i \neq j_{0}}^{\aleph} \left\vert \langle \Upsilon_{k}(z_{j_{0}},z_{i})  \cdot \int_{B}\overset{3}{\mathbb{P}}\left(\tilde{E}_{i}^{T}\right)(y) \, dy, e^{(3)}_{n} \rangle \right\vert \\ &+& \left\vert Error(n,j_{0}) \right\vert^{2} \Bigg]. 
	\end{eqnarray*} 
	Taking the series with respect to $n$, we can further deduce that
	\begin{eqnarray*} 
		\left\Vert  \overset{3}{\mathbb{P}}\left(\tilde{E}_{j_{0}}^{T} \right) \right\Vert^{2} & \lesssim & \left\vert \eta \right\vert^{-2} \Bigg[ \left\Vert \tilde{E}^{Inc}_{j_{0}} \right\Vert^{2} + a^{2}  \, \sum_{j \neq j_{0}}^{\aleph} \left\vert  \Upsilon_{k}(z_{j_{0}},z_{j})  \cdot \int_{B}\overset{3}{\mathbb{P}}\left(\tilde{E}_{j}^{T}\right)(y) \, dy \right\vert^{2} + \sum_{n} \, \left\vert Error(n,j_{0}) \right\vert^{2} \\
		&+& a^{2} \, \sum_{j=1 \atop j \neq j_{0}}^{\aleph} \left\vert  \Upsilon_{k}(z_{j_{0}},z_{j})  \cdot \int_{B}\overset{3}{\mathbb{P}}\left(\tilde{E}_{j}^{T}\right)(y) \, dy \right\vert \, \sum_{i > j \atop i \neq j_{0}}^{\aleph} \left\vert  \Upsilon_{k}(z_{j_{0}},z_{i})  \cdot \int_{B}\overset{3}{\mathbb{P}}\left(\tilde{E}_{i}^{T}\right)(y) \, dy \right\vert \Bigg]. 
	\end{eqnarray*} 
	Using the fact that $\Upsilon_{k}(z_{j_{0}},z_{j}) \simeq d^{-3}_{j,j_{0}}$, it is easy to get the following upper bound
	\begin{equation*}
	\sum_{j \neq j_{0}}^{\aleph} \left\vert  \Upsilon_{k}(z_{j_{0}},z_{j})  \cdot \int_{B}\overset{3}{\mathbb{P}}\left(\tilde{E}_{j}^{T}\right)(y) \, dy \right\vert^{2} \lesssim \sum_{j \neq j_{0}}^{\aleph} \frac{1}{d^{6}_{j,j_{0}}}\left\Vert \overset{3}{\mathbb{P}}\left(\tilde{E}_{j}^{T}\right) \right\Vert^{2} \lesssim \; d^{-6} \; \max_{j} \left\Vert \overset{3}{\mathbb{P}}\left(\tilde{E}_{j}^{T}\right) \right\Vert^{2},
	\end{equation*}
	and thus
	\begin{eqnarray*}
		I^{\star \star} &:=& \sum_{j=1 \atop j \neq j_{0}}^{\aleph} \left\vert  \Upsilon_{k}(z_{j_{0}},z_{j})  \cdot \int_{B}\overset{3}{\mathbb{P}}\left(\tilde{E}_{j}^{T}\right)(y) dy \right\vert  \sum_{i > j \atop i \neq j_{0}}^{\aleph} \left\vert  \Upsilon_{k}(z_{j_{0}},z_{i})  \cdot \int_{B}\overset{3}{\mathbb{P}}\left(\tilde{E}_{i}^{T}\right)(y)  dy \right\vert \\ & \leq & \sum_{j=1 \atop j \neq j_{0}}^{\aleph} d^{-3}_{j,j_{0}} \left\Vert \overset{3}{\mathbb{P}}\left(\tilde{E}_{j}^{T}\right) \right\Vert \, \sum_{i > j \atop i \neq j_{0}}^{\aleph} d^{-3}_{i,j_{0}} \left\Vert  \overset{3}{\mathbb{P}}\left(\tilde{E}_{i}^{T}\right) \right\Vert \simeq d^{-6} \; \left\vert \log(d) \right\vert^{2} \; \underset{j}{\max} \left\Vert  \overset{3}{\mathbb{P}}\left(\tilde{E}_{j}^{T}\right) \right\Vert^{2}.
	\end{eqnarray*}
	Therefore, 
	\begin{equation}\label{Norm-Proj-3rd}
	\left\Vert  \overset{3}{\mathbb{P}}\left(\tilde{E}_{j_{0}}^{T} \right) \right\Vert^{2}  \lesssim  \left\vert \eta \right\vert^{-2}  \left\Vert \tilde{E}^{Inc}_{j_{0}} \right\Vert^{2} + \, a^{6} \, \left\vert \log(d) \right\vert^{2} \; d^{-6} \; \max_{j} \left\Vert \overset{3}{\mathbb{P}}\left(\tilde{E}_{j}^{T}\right) \right\Vert^{2}  +  \left\vert \eta \right\vert^{-2} \, \sum_{n} \, \left\vert Error(n,j_{0}) \right\vert^{2} 
	\end{equation} 
	Taking the $\underset{j_{0}}{\max}\left( \cdot \right)$ on the both sides of \eqref{Norm-Proj-3rd}, we have
	\begin{equation*}
	\max_{j_{0}}  \left\Vert  \overset{3}{\mathbb{P}}\left(\tilde{E}_{j_{0}}^{T} \right) \right\Vert^{2}  \lesssim  \left\vert \eta \right\vert^{-2}  \max_{j_{0}} \left\Vert \tilde{E}^{Inc}_{j_{0}} \right\Vert^{2} + \, a^{6} \, \left\vert \log(d) \right\vert^{2}  \; d^{-6} \; \max_{j} \left\Vert \overset{3}{\mathbb{P}}\left(\tilde{E}_{j}^{T}\right) \right\Vert^{2} +  \left\vert \eta \right\vert^{-2} \, \max_{j_{0}} \sum_{n} \, \left\vert Error(n,j_{0}) \right\vert^{2}.
	\end{equation*} 
	Remark that the coefficient $\; a^{6} \, \left\vert \log(d) \right\vert^{2} \, d^{-6} \sim a^{6-6t} \; \left\vert \log(a) \right\vert^{2} $ is small, then 
	we obtain
	\begin{equation}\label{3rd-proj=inc+err}
	\max_{j_{0}}  \left\Vert  \overset{3}{\mathbb{P}}\left(\tilde{E}_{j_{0}}^{T} \right) \right\Vert^{2}  \lesssim  \left\vert \eta \right\vert^{-2}  \max_{j_{0}} \left\Vert \tilde{E}^{Inc}_{j_{0}} \right\Vert^{2} +  \left\vert \eta \right\vert^{-2} \, \max_{j_{0}} \sum_{n} \, \left\vert Error(n,j_{0}) \right\vert^{2}.
	\end{equation}
	To finish the estimation of $(\ref{3rd-proj=inc+err})$, we need to estimate $\underset{j_{0}}{\max} \underset{n}{\sum} \, \left\vert Error(n,j_{0}) \right\vert^{2}$. For this, taking the squared modulus on the both sides of $(\ref{Error(n,j,j0)})$, we get 
	\begin{eqnarray*}
		\left\vert Error(n,j_{0}) \right\vert^{2} & \lesssim &  a^{4} \sum_{\ell \geq 1}  \left\vert \langle \int_{B} \frac{\left\vert \cdot - y \right\vert^{\ell}}{(\ell + 1)!} \overset{1}{\mathbb{P}}\left(\tilde{E}_{j_{0}}^{T}\right)(y) \, dy , e^{(3)}_{n} \rangle \right\vert^{2}  + \left\vert \langle N \left( \overset{3}{\mathbb{P}}\left(\tilde{E}_{j_{0}}^{T} \right) \right), e^{(3)}_{n} \rangle \right\vert^{2} \\  &+& a^{2} \, \left\vert \langle \int_{B} \overset{3}{\mathbb{P}}\left(\tilde{E}_{j_{0}}^{T}\right)(y) \, dy, e^{(3)}_{n} \rangle \right\vert^{2} + \left\vert \langle \int_{B} \Phi_{0}(\cdot,y) \frac{A(\cdot,y) \cdot \overset{3}{\mathbb{P}}\left(\tilde{E}_{j_{0}}^{T}(y)\right)}{\left\Vert \cdot - y \right\Vert^{2}} dy, e^{(3)}_{n} \rangle \right\vert^{2} \\  
		&+& a^4\,\sum_{\ell \geq 1}  \left\vert \langle  \int_{B} \frac{\underset{y}{Hess}\left( \left\vert \cdot - y \right\vert^{\ell} \right)}{(\ell +1 )!}  \cdot \overset{3}{\mathbb{P}}\left(\tilde{E}_{j_{0}}^{T}\right)(y) \, dy , e^{(3)}_{n} \rangle \right\vert^{2} + \left\lvert \langle N^{ka} \left(\overset{3}{\mathbb{P}}\left(\tilde{E}^T_{j_0}\right)\right), e_n^{(3)}\rangle\right\rvert^2  \\ 
		&+& a^{4} \, \aleph \, \sum_{j \neq j_{0}}^{\aleph} \left\vert \langle \int_{0}^{1} \underset{x}{\nabla}\left(\Upsilon_{ka}\right)(z_{j_{0}} + t \, a \,  \cdot ,z_{j}) \cdot \mathcal{P}(x,0) dt \cdot \int_{B} \overset{3}{\mathbb{P}}\left(\tilde{E}_{j}^{T}\right)(y) \, dy, e^{(3)}_{n} \rangle \right\vert^{2} \\ 
		&+& a^{6} \, \aleph \, \sum_{j \neq j_{0}}^{\aleph} \left\vert \langle \int_{B} \int_{0}^{1} \int_{0}^{1} \underset{y}{Hess}\left(\Phi_{ka}\right)(z_{j_{0}}+h \, a \, \cdot ,z_{j}+t \, a \, y) \cdot x \, dh \cdot y \, dt \; \overset{1}{\mathbb{P}}\left(\tilde{E}_{j}^{T}\right)(y) \, dy, e^{(3)}_{n} \rangle \right\vert^{2} \\  
		&+&  a^{4} \, \aleph \, \sum_{j \neq j_{0}}^{\aleph} \left\vert \langle \int_{B} \int_{0}^{1} \underset{y}{\nabla}\left(\Upsilon_{ka}\right)(z_{j_{0}}+a \cdot ,z_{j}+t \, a \, y) \cdot \mathcal{P}(y,0) \, dt \cdot \overset{3}{\mathbb{P}}\left(\tilde{E}_{j}^{T}\right)(y) \, dy, e^{(3)}_{n} \rangle \right\vert^{2}.
	\end{eqnarray*}
	From \eqref{expansion-Nk}, we can know that
		\begin{equation}\label{expan-Nka}
		N^{ka}\left(\overset{3}{\mathbb{P}}\left(\tilde{E}^T_{j_0}\right)\right)=N\left(\overset{3}{\mathbb{P}}\left(\tilde{E}^T_{j_0}\right)\right)+\frac{i k a}{4\pi}\int_{D}\overset{3}{\mathbb{P}}\left(\tilde{E}^T_{j_0}\right)(y)\,dy+\frac{1}{4\pi}\sum_{n \geq 1}\frac{(i k a)^{n+1}}{(n+1)!}\int_D\left\lVert x-y\right\rVert^n\overset{3}{\mathbb{P}}\left(\tilde{E}^T_{j_0}\right)(y)\,dy.
		\end{equation}
		It is obvious that in \eqref{expan-Nka}, $N\left(\overset{3}{\mathbb{P}}\left(\tilde{E}^T_{j_0}\right)\right)$ dominates. Thus by
	taking the series with respect to $n$ index, using the continuity of the Newtonian potential operator, we get\footnote{We neglect the estimation of $
		\underset{n}{\sum} \left\vert \langle \int_{B} \Phi_{0}(\cdot,y) \frac{A(\cdot,y) \cdot \overset{3}{\mathbb{P}}\left(\tilde{E}_{j_{0}}^{T}(y)\right)}{\left\Vert \cdot - y \right\Vert^{2}} dy, e^{(3)}_{n} \rangle \right\vert^{2}
$
		since the singularity of the corresponding kernel is like the one of the Newtonian operator.} 
	\begin{eqnarray*}
		\sum_{n} \left\vert Error(n,j_{0}) \right\vert^{2} & \lesssim &  a^{4}  \left\Vert \overset{1}{\mathbb{P}}\left(\tilde{E}_{j_{0}}^{T}\right) \right\Vert^{2} \sum_{\ell \geq 1}  \int_{B} \int_{B} \frac{\left\vert x - y \right\vert^{2 \, \ell}}{((\ell + 1)!)^{2}}  \, dy \, dx  + \left\Vert  \overset{3}{\mathbb{P}}\left(\tilde{E}_{j_{0}}^{T}  \right) \right\Vert^{2}  \\ 
		&+&  a^4 \left\Vert  \overset{3}{\mathbb{P}}\left(\tilde{E}_{j_{0}}^{T}\right) \right\Vert^{2} \, \sum_{\ell \geq 1}  \int_{B} \int_{B} \frac{ \left\vert \underset{y}{Hess}\left( \left\vert x - y \right\vert^{\ell} \right) \right\vert^{2}}{((\ell +1 )!)^{2}} \, dy \, dx + a^{4} \, \aleph \, d^{-8} \, \max_{j} \left\Vert \overset{3}{\mathbb{P}}\left(\tilde{E}_{j}^{T}\right) \right\Vert^{2} \\ 
		&+& a^{6} \, \aleph \, d^{-6} \, \max_{j} \left\Vert  \overset{1}{\mathbb{P}}\left(\tilde{E}_{j}^{T}\right) \right\Vert^{2}.
	\end{eqnarray*}
	For the convergence of the two series appearing on the right hand side, we refer to $(\ref{Series1})$ and $(\ref{Series2})$. Then, 
	\begin{equation*}
	\sum_{n} \left\vert Error(n,j_{0}) \right\vert^{2} \lesssim  a^{4}  \left\Vert \overset{1}{\mathbb{P}}\left(\tilde{E}_{j_{0}}^{T}\right) \right\Vert^{2}  + \left\Vert  \overset{3}{\mathbb{P}}\left(\tilde{E}_{j_{0}}^{T}  \right) \right\Vert^{2} +  a^{4} \, \aleph \, d^{-8} \, \max_{j} \left\Vert \overset{3}{\mathbb{P}}\left(\tilde{E}_{j}^{T}\right) \right\Vert^{2} + a^{6} \, \aleph \, d^{-6} \, \max_{j} \left\Vert  \overset{1}{\mathbb{P}}\left(\tilde{E}_{j}^{T}\right) \right\Vert^{2}.
	\end{equation*}
	Taking the $\max\left( \cdot \right)$ with respect to $j_{0}$ index, we obtain that
	\begin{equation*}
	\max_{j_{0}} \sum_{n} \left\vert Error(n,j_{0}) \right\vert^{2} \lesssim  \max\left( 1 ;  a^{4} \, \aleph \, d^{-8} \right) \max_{j} \left\Vert \overset{3}{\mathbb{P}}\left(\tilde{E}_{j}^{T}\right) \right\Vert^{2} + \max(a^{4}; a^{6} \, \aleph \, d^{-6}) \max_{j} \left\Vert  \overset{1}{\mathbb{P}}\left(\tilde{E}_{j}^{T}\right) \right\Vert^{2}.
	\end{equation*}
	Going back to $(\ref{3rd-proj=inc+err})$ and plugging the previous result, we can derive that
	\begin{eqnarray*}
		\max_{j_{0}}  \left\Vert  \overset{3}{\mathbb{P}}\left(\tilde{E}_{j_{0}}^{T} \right) \right\Vert^{2}  & \lesssim & \left\vert \eta \right\vert^{-2}  \max_{j_{0}} \left\Vert \tilde{E}^{Inc}_{j_{0}} \right\Vert^{2} \\ &+&  \left\vert \eta \right\vert^{-2} \, \left[ \max\left( 1 ;  a^{4} \, \aleph \, d^{-8} \right) \max_{j} \left\Vert \overset{3}{\mathbb{P}}\left(\tilde{E}_{j}^{T}\right) \right\Vert^{2} + \max(a^{4}; a^{6} \, \aleph \, d^{-6} )\, \max_{j} \left\Vert  \overset{1}{\mathbb{P}}\left(\tilde{E}_{j}^{T}\right) \right\Vert^{2} \right].
	\end{eqnarray*}
	We assume that $\, \left\vert \eta \right\vert^{-2} \, a^{4} \, \aleph \, d^{-8} \,\sim a^{8-8t-s}$ is small enough, which is equivalent to the following condition  
	\begin{equation}\label{Condition-3}
	4-4t-\frac{s}{2} > 0.
	\end{equation} 
	In
	particular, when the distribution of the cluster of the nanoparticles is made over a volume, we have $\aleph \sim d^{-3}$. Therefore, we have $s\leq 3t$. Coupling this with (\ref{Cdt=0}), i.e. $t\leq 1-\frac{h}{3}$, then \eqref{Condition-3} is satisfied only if 
	\begin{equation}\label{condt-h}
	h>\frac{9}{11}.
	\end{equation}

	Then  
	\begin{equation*}
	\max_{j_0} \left\Vert  \overset{3}{\mathbb{P}}\left(\tilde{E}_{j_0}^{T} \right) \right\Vert^{2}  \lesssim  \left\vert \eta \right\vert^{-2}  \max_{j_0} \left\Vert \tilde{E}^{Inc}_{j_0} \right\Vert^{2} +  \left\vert \eta \right\vert^{-2} \; \max(a^{4}; a^{6} \, \aleph \, d^{-6}) \, \max_{j} \left\Vert  \overset{1}{\mathbb{P}}\left(\tilde{E}_{j}^{T}\right) \right\Vert^{2}.
	\end{equation*}
	For fixed $j_0$, we have 
	\begin{eqnarray*}
		\left\Vert \tilde{E}^{Inc}_{j_0} \right\Vert^{2} &=& \sum_{n} \left\vert \langle \tilde{E}^{Inc}_{j_0} , e^{(1)}_{n}  \rangle \right\vert^{2} + \sum_{n} \left\vert \langle \tilde{E}^{Inc}_{j_0} , e^{(3)}_{n}  \rangle \right\vert^{2} \\
		&=& \sum_{n} \left\vert \langle \tilde{E}^{Inc}_{j_0} , Curl\left( \phi_{n} \right) \rangle \right\vert^{2} + \mathcal{O}\left( 1 \right) \\
		&=& k^{2} \, a^{2} \, \sum_{n} \left\vert \langle \tilde{H}^{Inc}_{j_0} , \phi_{n}  \rangle \right\vert^{2} + \mathcal{O}\left( 1 \right) = \mathcal{O}\left( a^{2} \right) + \mathcal{O}\left( 1 \right) = \mathcal{O}\left( 1 \right).
	\end{eqnarray*}
	Then, 
	\begin{equation}\label{MP3<MP1}
	\max_{j_0}  \left\Vert  \overset{3}{\mathbb{P}}\left(\tilde{E}_{j_0}^{T} \right) \right\Vert^{2}  \lesssim   \, a^{4} +  \left\vert \eta \right\vert^{-2} \; \max(a^{4}; a^{6} \, \aleph \, d^{-6}) \, \max_{j} \left\Vert  \overset{1}{\mathbb{P}}\left(\tilde{E}_{j}^{T}\right) \right\Vert^{2}.
	\end{equation}
\end{enumerate}

	Now, from $(\ref{MP1<MP3})$ and $(\ref{MP3<MP1})$, we deduce that
	\begin{align}\notag
	\max_{j}\left\lVert\overset{1}{\mathbb{P}}\left(\tilde{E}^T_j\right)\right\rVert^2&\lesssim a^{2-2h}+a^{4-2h}d^{-8}\left[a^4+a^4\max(a^{4}; a^{6} \, \aleph \, d^{-6}) \, \max_{j} \left\Vert  \overset{1}{\mathbb{P}}\left(\tilde{E}_{j}^{T}\right) \right\Vert^{2}\right]\notag\\
	&\lesssim a^{2-2h}+a^{8-2h}d^{-8}+a^{8-2h}d^{-8}\max(a^{4}; a^{6} \, \aleph \, d^{-6}) \, \max_{j} \left\Vert  \overset{1}{\mathbb{P}}\left(\tilde{E}_{j}^{T}\right) \right\Vert^{2}.\notag
	\end{align}
	If there holds
	\begin{equation}\notag
	2-s-6t<0,
	\end{equation}
	then $\max(a^{4}; a^{6} \, \aleph \, d^{-6})=a^{6} \, \aleph \, d^{-6}$, which indicates that
	\begin{equation}\notag
	\max_{j}\left\lVert\overset{1}{\mathbb{P}}\left(\tilde{E}^T_j\right)\right\rVert^2\lesssim a^{2-2h}+a^{8-2h}d^{-8}+a^{14-2h}d^{-14} \aleph  \max_{j} \left\Vert  \overset{1}{\mathbb{P}}\left(\tilde{E}_{j}^{T}\right) \right\Vert^{2}.
	\end{equation}
	From \eqref{Cdt=0} and \eqref{Condition-3}, it is direct to verify that $7-7t-h-\frac{s}{2}>0$. Hence, we can derive 
	\begin{equation}\label{max-P1}
	\max_{j}\left\lVert\overset{1}{\mathbb{P}}\left(\tilde{E}^T_j\right)\right\rVert^2\lesssim a^{2-2h}+a^{8-2h}d^{-8}.
	\end{equation} 
	Otherwise, if there holds
	\begin{equation}\notag
	2-s-6t\geq 0,
	\end{equation}
	then $(a^{4}; a^{6} \, \aleph \, d^{-6})=a^4$, which indicates that
	\begin{equation}\notag
	\max_{j}\left\lVert\overset{1}{\mathbb{P}}\left(\tilde{E}^T_j\right)\right\rVert^2\lesssim a^{2-2h}+a^{8-2h}d^{-8}+a^{12-2h}d^{-8}   \max_{j} \left\Vert  \overset{1}{\mathbb{P}}\left(\tilde{E}_{j}^{T}\right) \right\Vert^{2}.
	\end{equation}
	By using \eqref{Cdt=0} and the fact that $t<1$, it is automatically fulfilled that $12-2h-8t>0$ and therefore \eqref{max-P1} still holds.
	
	Following a similar argument above, for $\overset{3}{\mathbb{P}}\left(\tilde{E}^T_j\right)$, we can also know from \eqref{MP1<MP3} and \eqref{MP3<MP1} that
	\begin{align}\notag
	\max_{j}\left\lVert\overset{3}{\mathbb{P}}\left(\tilde{E}^T_j\right)\right\rVert^2&\lesssim a^{4}+a^{4}\cdot \max(a^{4}; a^{6} \, \aleph \, d^{-6}) \, \max_{j} \left\Vert  \overset{1}{\mathbb{P}}\left(\tilde{E}_{j}^{T}\right) \right\Vert^{2}\notag\\
	&\lesssim 
	a^{4}+a^{4}\cdot \max(a^{4}; a^{6} \, \aleph \, d^{-6}) \, \left[a^{2-2h} 
	+ a^{4-2h} d^{-8} \max_{j} \left\lVert \overset{3}{\mathbb{P}}\left(\tilde{E}^T_j\right)\right\rVert^2\right]\notag\\
	&\lesssim
	a^4+a^{12-2h} \aleph d^{-6}.\notag
	\end{align}
	Since the sufficient conditions \eqref{cond1}, \eqref{Cdt=0} and \eqref{Condition-3} give us $t<\frac{5}{7}$, we see that
	\begin{equation}\notag
	\max_{j}\left\lVert\overset{1}{\mathbb{P}}\left(\tilde{E}^T_j\right)\right\rVert^2\lesssim a^{2-2h}.
	\end{equation}
	Moreover, 
	\begin{equation}\label{es-P3}
	\max_{j}\left\lVert\overset{3}{\mathbb{P}}\left(\tilde{E}^T_j\right)\right\rVert^2\lesssim
	\begin{cases}
	a^4 & \mbox{if}\quad t<\frac{2}{3},\\
	a^{12-2h} \aleph d^{-6} & \mbox{if}\quad \frac{2}{3}\leq t<\frac{5}{7}.
	\end{cases}
	\end{equation}
	
	Since $\aleph=O(d^{-3})\sim a^{-3t}$ as $d \sim a^t$, by recalling that $t\leq 1-\frac{h}{3}$, then for $\frac{9}{11}<h<1$, we have
	\begin{equation}\label{max-P1P3-2}
	\max_{j}\left\lVert\overset{1}{\mathbb{P}}\left(\tilde{E}^T_j\right)\right\rVert^2\lesssim a^{2-2h}\quad\mbox{and}\quad
	\max_{j}\left\lVert\overset{3}{\mathbb{P}}\left(\tilde{E}^T_j\right)\right\rVert^2\lesssim a^{3+h},
	\end{equation}
	which completes the proof.

	\begin{remark}\label{rem-cond1}
	Recall the sufficient condition
	\eqref{Condition-3}. Particularly, in the regime that 
	\begin{equation}\notag
	t = 1 - \frac{h}{3},\quad\mbox{and}\quad s=3t,
	\end{equation} 
	the expression $3-h-3t$, given on the left hand side of $(\ref{Cdt=0})$, will be vanishing and then the condition will not be accomplished. To be able to handle this situation, we assume that $\eta$ satisfies $\left\vert \eta \right\vert = \varsigma \, a^{-2}$, with $ \varsigma< 1$.
\end{remark}

\subsection{Proof of Proposition \ref{prop-scoeff}.}\label{subsec-53}

Here, we present the estimation for the scattering coefficient.
Similar to Lemma \ref{AI4}, we first give the following formulation before proving Proposition \ref{prop-scoeff}. 

\begin{lemma}\label{AI5}
	Recall that $e_n^{(3)}, n=1, 2...,$ are the eigenfunctions of $\nabla M$ in the subspace $\nabla\mathcal{H}armonic$. Then we have
	\begin{eqnarray}\label{AI3}
	\nonumber
\boldsymbol{T}_{k \, a}^{-1} \left( e^{(3)}_{n} \right) &=& \frac{1}{\left(1 + \eta \, \lambda_{n}^{(3)} \right)} e^{(3)}_{n} \\  &+& \frac{1}{2 \, \lambda^{(1)}_{n_{0}} \left(1 + \eta \lambda_{n}^{(3)} \right)} \Bigg[\boldsymbol{T}_{k \, a}^{-1}N\left(e^{(3)}_{n}\right) + \boldsymbol{T}_{k \, a}^{-1} \int_{B} \Phi_{0}(\cdot,y) \frac{A(\cdot,y) \cdot e^{(3)}_{n}(y)}{\left\Vert \cdot - y \right\Vert} \, dy \Bigg] + \overset{(3)}{R_{n}}.
	\end{eqnarray}
	where
	\begin{eqnarray*}
		\overset{(3)}{R_{n}} &:=& \frac{1}{\left( 1 + \eta \, \lambda^{(3)}_{n} \right)} \boldsymbol{T}_{k \, a}^{-1} \Bigg[ \frac{\pm \, c_{0} \, a^{h}}{2 \, \lambda^{(1)}_{n_{0}}} N\left(e^{(3)}_{n}\right) + \frac{\pm \, c_{0} \, a^{h}}{2 \, \lambda^{(1)}_{n_{0}}} \int_{B} \Phi_{0}(\cdot,y) \frac{A(\cdot,y) \cdot e^{(3)}_{n}(y)}{\left\Vert \cdot - y \right\Vert} \, dy + \frac{i \eta \, k^{3} a^{3}}{6 \pi}  \int_{B}  e^{(3)}_{n}(y) \, dy \\ &+& \frac{\eta}{4 \pi} \sum_{\ell \geq 3} \frac{(ika)^{\ell+1}}{(\ell + 1)!} \int_{B} Hess\left( \left\Vert \cdot - y \right\Vert^{\ell} \right) \cdot   e^{(3)}_{n}(y) \, dy  - \frac{\left( 1 \mp c_{0} \, a^{h} \right)}{4 \pi \, \lambda^{(1)}_{n_{0}}} \sum_{\ell \geq 2} \frac{(ika)^{\ell+1}}{(\ell + 1)!} \int_{B} \left\Vert \cdot - y \right\Vert^{\ell}  e^{(3)}_{n}(y) \, dy \Bigg]. 
	\end{eqnarray*}
\end{lemma}

	\begin{proof}
		Let us compute $\boldsymbol{T}_{k \, a}\left(e_n^{(3)}\right)$ as follows:
		\begin{align}\label{en3-1}
		\boldsymbol{T}_{k \, a}\left(e_n^{(3)}\right)&=\left(I + \eta \, \nabla M\right)\left(e_n^{(3)}\right) + \eta \, \left(\nabla M^{ka}-\nabla M\right)\left(e_n^{(3)}\right) - k^2 \, \eta \, a^2 \, N^{ka}\left(e_n^{(3)}\right)\notag\\
		&=\left(1 + \eta \, \lambda_{n}^{(3)}\right)\left(e_n^{(3)}\right) + \eta \, \left[\left(\nabla M^{ka}-\nabla M\right)-k^2 a^2 N^{ka}\right]\left(e_n^{(3)}\right).
		\end{align}
By combining with \eqref{expansion-gradMk} and \eqref{expansion-Nk}, we deduce that
		\begin{eqnarray*}
		\left[\left(\nabla M^{ka}-\nabla M\right)-k^2a^2N^{ka}\right]\left( e_n^{(3)} \right) &=& -\frac{k^2 a^2}{2}N\left(e_n^{(3)}\right)-\frac{k^2 a^2}{2}\int_{B}\Phi_{0}(\cdot,y)\frac{A(\cdot,y)\cdot e_n^{(3)}(y)}{\lVert \cdot - y\rVert^2}\,dy \\ &-& \frac{i k^3 a^3}{6\pi}\int_{B}e_n^{(3)}(y)\,dy\notag - \frac{1}{4\pi}\sum_{n \geq 3}\frac{(ika)^{n+1}}{(n+1)!}\int_{B} Hess \left(\lVert \cdot - y\rVert^n\right)\cdot e_n^{(3)}(y)\,dy \\ &-& \frac{k^2 a^2}{4\pi}\sum_{n \geq 1} \frac{(ika)^{n+1}}{(n+1)!}\int_{B}\lVert \cdot - y \rVert^ne_n^{(3)}(y)\,dy,\notag
		\end{eqnarray*}
		and therefore we see that \eqref{en3-1} can be further written as
		\begin{eqnarray}\label{en3-2}
\nonumber		
	\boldsymbol{T}_{k \, a}\left(e_n^{(3)}\right) &=& \left(1 + \eta \, \lambda_n^{(3)}\right)\left(e_n^{(3)}\right) - \eta\frac{k^2 a^2}{2}N\left(e_n^{(3)}\right) - \eta \frac{k^2 a^2}{2}\int_{B}\Phi_{0}(\cdot ,y)\frac{A(\cdot,y)\cdot e_n^{(3)}(y)}{\lVert x-y\rVert^2}\,dy \\ \nonumber &-& \frac{i\eta k^3 a^3}{6\pi}\int_{B}e_n^{(3)}(y)\,dy - \frac{\eta}{4\pi}\sum_{n \geq 3}\frac{(ika)^{n+1}}{(n+1)!}\int_{B} \underset{x}{Hess}\left(\lVert \cdot -y\rVert^n\right)\cdot e_n^{(3)}(y)\,dy \\ &-& \frac{\eta k^2 a^2}{4\pi}\sum_{n \geq 1} \frac{(ika)^{n+1}}{(n+1)!}\int_{B}\lVert \cdot -y\rVert^ne_n^{(3)}(y)\,dy.
		\end{eqnarray}
		
		Taking the inverse of $\boldsymbol{T}_{k \, a}$ on the both sides of \eqref{en3-2}, we obtain that
		\begin{eqnarray}\label{en3-3}
		\nonumber
		e_n^{(3)} &=& \left(1 + \eta \, \lambda_n^{(3)}\right)\boldsymbol{T}_{k \, a}^{-1}\left(e_n^{(3)}\right) - \eta \, \frac{k^2 a^2}{2}\boldsymbol{T}_{k \, a}^{-1}N\left(e_n^{(3)}\right)\notag\\
		&-& \eta \frac{k^2 a^2}{2}\boldsymbol{T}_{k \, a}^{-1}\int_{B}\Phi_{0}(\cdot,y)\frac{A(\cdot ,y)\cdot e_n^{(3)}(y)}{\lVert \cdot -y\rVert^2}\,dy - \frac{i\eta k^3 a^3}{6\pi}\int_{B}e_n^{(3)}(y) \, dy  \, \boldsymbol{T}_{k \, a}^{-1} \left( 1 \right)\, \notag \\
&-& \frac{\eta}{4\pi}\boldsymbol{T}_{k \, a}^{-1}\sum_{n \geq 3}\frac{(ika)^{n+1}}{(n+1)!}\int_{B} \underset{y}{Hess}\left(\lVert \cdot - y\rVert^n\right)\cdot e_n^{(3)}(y)\,dy\notag\\
		&-& \frac{\eta k^2 a^2}{4\pi}\boldsymbol{T}_{k \, a}^{-1}\sum_{n \geq 1} \frac{(ika)^{n+1}}{(n+1)!}\int_{B}\lVert \cdot - y \rVert^ne_n^{(3)}(y)\,dy.
		\end{eqnarray}
		
		Dividing by $\left(1+\eta \lambda_{n}^{(3)}\right)$ on the both sides of \eqref{en3-3} and after rearranging terms, we derive that
		\begin{eqnarray}\label{en3-4}
\boldsymbol{T}_{k \, a}^{-1}\left(e_n^{(3)}\right)  &=& \frac{1}{\left(1 + \eta \lambda_{n}^{(3)}\right)}e_n^{(3)} + \frac{1}{\left(1+\eta \lambda_{n}^{(3)}\right)}\eta\frac{k^2 a^2}{2}\boldsymbol{T}_{k \, a}^{-1}N\left(e_n^{(3)}\right)\notag\\
		&+& \frac{1}{\left(1 + \eta \lambda_{n}^{(3)}\right)}\eta \frac{k^2 a^2}{2}\boldsymbol{T}_{k \, a}^{-1}\int_{B}\Phi_{0}(\cdot ,y)\frac{A(\cdot,y)\cdot e_n^{(3)}(y)}{\lVert \cdot -y \rVert^2}\,dy\notag\\
		&+& \frac{1}{\left(1 + \eta \lambda_{n}^{(3)}\right)}\frac{i\eta k^3 a^3}{6\pi}\int_{B} e_n^{(3)}(y) \,dy \, \boldsymbol{T}_{k \, a}^{-1}\left(1\right) \notag\\
		&+& \frac{\eta}{4\pi\left(1 + \eta \lambda_{n}^{(3)}\right)}\boldsymbol{T}_{k \, a}^{-1}\sum_{n \geq 3}\frac{(ika)^{n+1}}{(n+1)!}\int_{B} Hess \left(\lVert \cdot -y\rVert^n\right)\cdot e_n^{(3)}(y)\,dy\notag\\
		&+& \frac{\eta k^2 a^2}{4\pi\left(1 + \eta \lambda_{n}^{(3)}\right)}\boldsymbol{T}_{k \, a}^{-1}\sum_{n \geq 1} \frac{(ika)^{n+1}}{(n+1)!}\int_{B}\lVert \cdot - y\rVert^ne_n^{(3)}(y)\,dy.
		\end{eqnarray}
		
		Since from \eqref{choice-k-1st-regime}, there holds
		\begin{equation}\notag
		\eta \, k^2 \, a^2 = \frac{1 \mp c_{0} \, a^h}{\lambda_{n_{0}}^{(1)}},
		\end{equation}
		then by denoting
		\begin{align}\notag
		R_n^{(3)}:&=
		\frac{\mp \, c_{0} \, a^h}{2\lambda_{n_{0}}^{(1)}\left(1 + \eta \lambda_{n}^{(3)}\right)}\boldsymbol{T}_{k \, a}^{-1}N\left(e_n^{(3)}\right) 
		+\frac{\mp \, c_{0} \, a^h}{2\lambda_{n_{0}}^{(1)}\left(1 + \eta \lambda_{n}^{(3)}\right)}\boldsymbol{T}_{k \, a}^{-1}\int_{B}\Phi_{0}(\cdot ,y)\frac{A(\cdot,y)\cdot e_n^{(3)}(y)}{\lVert \cdot -y\rVert^2}\,dy\notag\\
		&+\frac{i \eta k^3 a^3}{6\pi\left(1+\eta \lambda_{n}^{(3)}\right)}\int_{B}e_n^{(3)}(y) \, dy \,\, \boldsymbol{T}_{k \, a}^{-1}\left(1\right) \notag\\
		&+\frac{\eta}{4\pi\left(1+\eta \lambda_{n}^{(3)}\right)}\boldsymbol{T}_{k \, a}^{-1}\sum_{n \geq 3}\frac{(ika)^{n+1}}{(n+1)!}\int_{B} Hess \left(\lVert \cdot -y\rVert^n\right)\cdot e_n^{(3)}(y)\,dy\notag\\
		&+\frac{\left(1 \mp \, c_{0} \, a^{h} \right)}{4\pi \, \lambda_{n_{0}}^{(1)} \, \left(1+\eta \lambda_{n}^{(3)}\right)}\boldsymbol{T}_{k \, a}^{-1}\sum_{n \geq 1} \frac{(ika)^{n+1}}{(n+1)!}\int_{B}\lVert \cdot -y\rVert^ne_n^{(3)}(y)\,dy,\notag
		\end{align}
		we can deduce \eqref{AI3} from \eqref{en3-4}, which completes the proof.
	\end{proof}

We are now in a position to prove Proposition \ref{prop-scoeff} as follows.

\begin{proof}[Proof of Proposition \ref{prop-scoeff}]
	We first give the estimation for the scattering coefficient $\mathcal{C}$.
	Recall the definition of $\mathcal{C}$ given by \eqref{def-scoeff}. After scaling to the domain $B$, we obtain 
	\begin{equation*}
	\mathcal{C} = a^{4} \, \int_{B} \boldsymbol{T}_{k \, a}^{-1}\left( \mathcal{P}(\cdot,0)\right)(x) \, dx.
	\end{equation*}
	By expanding over the basis and using the fact that $\int_{B} e^{(1,2)}_{n}(x) dx = 0$, we get  
	\begin{equation}\label{DJ.H}
	\mathcal{C} = a^{4} \, \sum_{n} \, \langle \boldsymbol{T}_{k \, a}^{-1}\left( \mathcal{P}(\cdot,0)\right) , e^{(3)}_{n}\rangle \, \int_{B} e^{(3)}_{n}(x) \, dx = a^{4} \, \sum_{n} \, \langle \mathcal{P}(\cdot,0), \boldsymbol{T}_{- \, k \, a}^{-1}\left(e^{(3)}_{n}\right)\rangle \, \int_{B} e^{(3)}_{n}(x) \, dx.
	\end{equation}
	
	From the expression \eqref{AI3} in Lemma \ref{AI5}, in order to extract the dominant term of $(\ref{AI3})$, we need to analyse the term $S := \boldsymbol{T}_{- \, k \, a}^{-1} N\left(e^{(3)}_{n}\right)$. For this we have
	\begin{eqnarray*}
		S &:=& \boldsymbol{T}_{- \, k \, a}^{-1}  \left[\overset{1}{\mathbb{P}}\left( N\left(e^{(3)}_{n}\right) \right) + \overset{2}{\mathbb{P}}\left( N\left(e^{(3)}_{n}\right) \right) + \overset{3}{\mathbb{P}}\left( N\left(e^{(3)}_{n}\right) \right)\right], \\
		S &=& \sum_{\ell} \langle \overset{1}{\mathbb{P}}\left( N\left(e^{(3)}_{n}\right) \right), e^{(1)}_{\ell} \rangle \,\, \boldsymbol{T}_{- \, k \, a}^{-1}  \left[e^{(1)}_{\ell}\right] 
		+ \sum_{\ell} \langle \overset{2}{\mathbb{P}}\left( N\left(e^{(3)}_{n}\right) \right), e^{(2)}_{\ell} \rangle \,\, \boldsymbol{T}_{- \, k \, a}^{-1}  \left[e^{(2)}_{\ell}\right] \\
		&+& \sum_{\ell} \langle \overset{3}{\mathbb{P}}\left( N\left(e^{(3)}_{n}\right) \right), e^{(3)}_{\ell} \rangle \,\, \boldsymbol{T}_{- \, k \, a}^{-1}  \left[e^{(3)}_{\ell}\right].
	\end{eqnarray*}
	Now we use the result given in Lemma \ref{AI4} and the fact that $\boldsymbol{T}_{k \, a}$, when restricted to $\mathbb{H}_{0}\left( Curl = 0 \right)$, is equal to $(1 +\eta) \, I$ to deduce that 
	\begin{eqnarray*}
		S &=& \sum_{\ell} \langle \overset{1}{\mathbb{P}}\left( N\left(e^{(3)}_{n}\right) \right), e^{(1)}_{\ell} \rangle \left( \frac{1}{\left(1 - k^{2} \, \eta \, a^{2} \, \lambda_{\ell}^{(1)} \right)} e_{\ell}^{(1)} + R_{\ell} \right) + \left( 1 + \eta \right)^{-1} \; \overset{2}{\mathbb{P}}\left( N\left(e^{(3)}_{n}\right) \right) \\
		&+& \sum_{\ell} \langle \overset{3}{\mathbb{P}}\left( N\left(e^{(3)}_{n}\right) \right), e^{(3)}_{\ell} \rangle \,\, \boldsymbol{T}_{- \, k \, a}^{-1}  \left[e^{(3)}_{\ell}\right].
	\end{eqnarray*}
	For the last term, we repeat the same computations as  those done in Lemma \ref{AI5} to deduce that
		\begin{equation}\notag
		\sim	\sum_{\ell} \langle \overset{3}{\mathbb{P}}\left( N\left(e^{(3)}_{n}\right) \right), e^{(3)}_{\ell} \rangle\frac{1}{\left(  1 + \eta \lambda_\ell^{(3)} \right)}\, e_\ell^{(3)}.
		\end{equation}
	
	Finally, the dominant term of $S$ is given by 
	\begin{equation}\label{AI6}
	S \sim \sum_{\ell} \langle \overset{1}{\mathbb{P}}\left( N\left(e^{(3)}_{n}\right) \right), e^{(1)}_{\ell} \rangle \frac{1}{\left(1 - k^{2} \, \eta \, a^{2} \, \lambda_{\ell}^{(1)} \right)} e_{\ell}^{(1)}. 
	\end{equation}
	Since 
	\begin{equation*}
	\boldsymbol{T}_{-k \, a}^{-1} \int_{B} \Phi_{0}(\cdot,y) \frac{A(\cdot,y) \cdot e^{(3)}_{n}(y)}{\left\Vert \cdot - y \right\Vert} \, dy 
	\end{equation*}
	has the same behaviour as $S$, the corresponding dominant term is also given by $(\ref{AI6})$.\\
	Using $(\ref{AI3})$ and $(\ref{AI6})$, the equation $(\ref{DJ.H})$ takes the following form
	\begin{eqnarray*}
		\mathcal{C} & \sim & a^{4} \, \sum_{n} \, \frac{1}{\left( 1 + \eta \, \lambda_{n}^{(3)} \right)}  \, \langle \mathcal{P}(\cdot,0), e^{(3)}_{n} \rangle \, \int_{B} e^{(3)}_{n}(x) \, dx \\
		&+& \frac{a^{4}}{2 \, \lambda^{(1)}_{n_{0}}} \, \sum_{n} \, \frac{1}{\left( 1 + \eta \, \lambda_{n}^{(3)} \right)} \,\sum_{\ell} \frac{\langle \overset{1}{\mathbb{P}}\left(N\left(e_{n}^{(3)}\right)\right) , e^{(1)}_{\ell} \rangle}{\left(1 - k^{2} \, a^{2} \, \eta \, \lambda^{(1)}_{\ell} \right)} \langle \mathcal{P}(\cdot,0), e^{(1)}_{\ell} \rangle \, \int_{B} e^{(3)}_{n}(x) \, dx.
	\end{eqnarray*}
	Remark that the second term, when taking the index $\ell = n_{0}$, will dominate the first one. Consequently, 
		\begin{equation}\label{Estimation-Int-New-C}
		\mathcal{C}  \sim a^{4-h} \, \sum_{n} \, \frac{1}{\left( 1 + \eta \, \lambda_{n}^{(3)} \right)} \, \langle \overset{1}{\mathbb{P}}\left(N\left(e_{n}^{(3)}\right)\right) , e^{(1)}_{n_{0}} \rangle  \langle \mathcal{P}(\cdot,0), e^{(1)}_{n_0} \rangle \, \int_{B} e^{(3)}_{n}(x) \, dx = \mathcal{O}\left( a^{6-h} \right).
		\end{equation}  

Next, we prove \eqref{prop-es-W}.
Recall that $W$ solves \eqref{AI1}.  
By scaling \eqref{AI1} from $D$ to $B$, we obtain that
\begin{equation}\label{New-W-tilde}
\boldsymbol{T}_{-k \, a} \, \left( \widetilde{W} \right)(x) = a \, \mathcal{P}(x,0), \quad x \in B. 
\end{equation}
We project $\tilde W$ onto the subspaces introduced in \eqref{project} as $\tilde{W} = \overset{1}{\mathbb{P}}\left(\tilde{W}\right) + \overset{2}{\mathbb{P}}\left(\tilde{W}\right) + \overset{3}{\mathbb{P}}\left(\tilde{W}\right).$
Then $(\ref{New-W-tilde})$, using the fact that $\nabla M^{k} \equiv 0$ on $\mathbb{H}_{0}(\div = 0)$ and $(\ref{grad-M-k-2nd-subspace})$, becomes  
\begin{equation}\label{Equa-Est-W}
\left( I - \, k^{2} \, \eta \, a^{2} N \right) \left( \overset{1}{\mathbb{P}}\left(\tilde{W}\right)\right) (x) + (1 + \eta) \; \overset{2}{\mathbb{P}}\left(\tilde{W}\right)(x) +  \left( I + \eta \, \nabla M \right) \left(  \overset{3}{\mathbb{P}}\left(\tilde{W}\right) \right)(x)  = a \, \mathcal{P}(x,0) - T(x),
\end{equation}
where 
\begin{equation*}
T(x):= \eta \,\left(\nabla M^{-ka} - \nabla M \right) \left(  \overset{3}{\mathbb{P}}\left(\tilde{W}\right) \right)(x) - k^{2} \, \eta \, a^{2} \, \left( N^{-ka} \right)\left(  \overset{3}{\mathbb{P}}\left(\tilde{W}\right) \right)(x) -   k^{2} \, \eta \, a^{2} \, \left( N^{-ka} - N \right)\left( \overset{1}{\mathbb{P}}\left(\tilde{W}\right) \right)(x), 
\end{equation*}
which can be further rewritten as
\begin{eqnarray}\label{T-term}
\nonumber
T(x) &=& - \eta \, \frac{(ka)^{2}}{2} N\left(   \overset{3}{\mathbb{P}}\left(\tilde{W}\right) \right)(x) + i \eta \, \frac{(ka)^{3}}{6 \, \pi} \int_{B}   \overset{3}{\mathbb{P}}\left(\tilde{W}\right)(y) \, dy - \eta \, \frac{(ka)^{2}}{2} \int_{B} \Phi_{0}(x,y) \frac{A(x,y) \cdot   \overset{3}{\mathbb{P}}\left(\tilde{W}\right)(y)}{\left\Vert x - y \right\Vert^{2}} dy \\ \nonumber
& - & \frac{1}{4 \pi} \, \eta \, \sum_{\ell \geq 3} \left( -i  k  a \right)^{\ell+1} \int_{B} \frac{Hess\left( \left\vert x - y \right\vert^{\ell} \right)}{(\ell +1)!} \cdot \overset{3}{\mathbb{P}}\left(\tilde{W}\right)(y) \, dy \\ \nonumber &-&  \frac{k^{2} \, \eta \, a^{2}}{4 \pi} \sum_{\ell \geq 2} \left(- i k a \right)^{\ell + 1} \int_{B} \frac{\left\vert x - y \right\vert^{\ell}}{(\ell + 1)!}   \overset{3}{\mathbb{P}}\left(\tilde{W}\right)(y) dy \\ &-& k^{2} \, \eta \, a^{2} \, \frac{1}{4 \pi} \, \sum_{\ell \geq 2} \left(- i  k  a \right)^{\ell + 1} \int_{B} \frac{\left\vert x - y \right\vert^{\ell}}{(\ell + 1)!} \;  \overset{1}{\mathbb{P}}\left(\tilde{W}\right)(y) dy,
\end{eqnarray}
by using $(\ref{expansion-gradMk})$ and $(\ref{expansion-Nk})$. Taking the inner product of $(\ref{Equa-Est-W})$ with respect to $\left( e_{n}^{(1)} \right)$, we obtain
\begin{equation}\label{Isefra}
\langle \overset{1}{\mathbb{P}}\left( \tilde{W} \right), e_{n}^{(1)} \rangle = \frac{1}{\left( 1 - k^{2} \, \eta \, a^{2} \, \lambda^{(1)}_{n} \right)} \left[ a \, \langle \mathcal{P}(\cdot , 0) , e_{n}^{(1)} \rangle - \langle T, e_{n}^{(1)} \rangle \right]. 
\end{equation}
By taking the squared modulus and summing with respect to the index $n$, we deduce that 
\begin{equation}\label{P1W}
\left\Vert \overset{1}{\mathbb{P}}\left( \tilde{W} \right) \right\Vert^{2}  \lesssim  a^{2} \, \sum_{n} \, \frac{1}{\left\vert 1 - k^{2} \, \eta \, a^{2} \, \lambda^{(1)}_{n} \right\vert^{2}} \left\vert \langle \mathcal{P}(\cdot , 0) , e_{n}^{(1)} \rangle \right\vert^{2} + a^{-2h} \, \sum_{n}  \left\vert \langle T, e_{n}^{(1)} \rangle \right\vert^{2}.  
\end{equation}
We need to compute $\underset{n}{\sum} \left\vert \langle T, e_{n}^{(1)} \rangle \right\vert^{2}$. Indeed,
\begin{eqnarray*}
	\left\vert \langle T, e_{n}^{(1)} \rangle \right\vert^{2} & \lesssim & \left\vert \langle \, \int_{B} \, \Phi_{0}(\cdot,y) \frac{A(\cdot,y) \cdot   \overset{3}{\mathbb{P}}\left(\tilde{W}\right)(y)}{\left\Vert x - y \right\Vert^{2}} dy, e_{n}^{(1)} \rangle \right\vert^{2} +  a^{4} \sum_{\ell \geq 3} \left\vert \langle \int_{B} \frac{Hess\left( \left\vert \cdot - y \right\vert^{\ell} \right)}{(\ell +1)!} \cdot \overset{3}{\mathbb{P}}\left(\tilde{W}\right)(y) \, dy, e_{n}^{(1)} \rangle \right\vert^{2} \\ &+&  a^{6} \sum_{\ell \geq 2}  \left\vert \langle \int_{B} \frac{\left\vert \cdot - y \right\vert^{\ell}}{(\ell + 1)!}   \overset{3}{\mathbb{P}}\left(\tilde{W}\right)(y) dy, e_{n}^{(1)} \rangle \right\vert^{2} + a^{6} \sum_{\ell \geq 2}  \left\vert \langle \int_{B} \frac{\left\vert \cdot - y \right\vert^{\ell}}{(\ell + 1)!} \;  \overset{1}{\mathbb{P}}\left(\tilde{W}\right)(y) dy, e_{n}^{(1)} \rangle \right\vert^{2}. 
\end{eqnarray*}
Then, 
\begin{eqnarray}\label{T.en1.en2}
\nonumber
\sum_{n} \left\vert \langle T, e_{n}^{(1)} \rangle \right\vert^{2} & \lesssim & \left\Vert \int_{B} \, \Phi_{0}(\cdot,y) \frac{A(\cdot,y) \cdot   \overset{3}{\mathbb{P}}\left(\tilde{W}\right)(y)}{\left\Vert x - y \right\Vert^{2}} dy \right\Vert^{2} +  a^{4} \sum_{\ell \geq 3} \left\Vert \int_{B} \frac{Hess\left( \left\vert \cdot - y \right\vert^{\ell} \right)}{(\ell +1)!} \cdot \overset{3}{\mathbb{P}}\left(\tilde{W}\right)(y) \, dy \right\Vert^{2} \\ \nonumber &+&  a^{6} \sum_{\ell \geq 2}  \left\Vert \int_{B} \frac{\left\vert \cdot - y \right\vert^{\ell}}{(\ell + 1)!}   \overset{3}{\mathbb{P}}\left(\tilde{W}\right)(y) dy \right\Vert^{2} + a^{6} \sum_{\ell \geq 2}  \left\Vert  \int_{B} \frac{\left\vert \cdot - y \right\vert^{\ell}}{(\ell + 1)!} \;  \overset{1}{\mathbb{P}}\left(\tilde{W}\right)(y) dy \right\Vert^{2}  \\
& \lesssim & \left\Vert  \overset{3}{\mathbb{P}}\left(\tilde{W}\right) \right\Vert^{2}  + a^{6}  \left\Vert  \overset{1}{\mathbb{P}}\left(\tilde{W}\right) \right\Vert^{2}. 
\end{eqnarray}
Equation $(\ref{P1W})$ takes the following form 
\begin{equation}\label{Equa-226}
\left\Vert \overset{1}{\mathbb{P}}\left( \tilde{W} \right) \right\Vert^{2}  \lesssim  a^{2} \, \sum_{n} \, \frac{1}{\left\vert 1 - k^{2} \, \eta \, a^{2} \,  \lambda^{(1)}_{n} \right\vert^{2}} \left\vert \langle \mathcal{P}(\cdot , 0) , e_{n}^{(1)} \rangle \right\vert^{2} + a^{-2h} \, \left\Vert  \overset{3}{\mathbb{P}}\left(\tilde{W}\right) \right\Vert^{2}.  
\end{equation}
We compute $\left\Vert  \overset{2}{\mathbb{P}}\left(\tilde{W}\right) \right\Vert^{2}$ in a similar manner. First, we take the inner product on the both sides of $(\ref{Equa-Est-W})$ with respect to $\left( e_{n}^{(2)} \right)_{n}$, to obtain $\left( 1 + \eta \right) \,\, \langle \overset{2}{\mathbb{P}}\left( \tilde{W} \right), e_{n}^{(2)} \rangle = \left[ a \, \langle \mathcal{P}(\cdot , 0) , e_{n}^{(2)} \rangle - \langle T, e_{n}^{(2)} \rangle \right]$. After taking the squared modulus and the series with respect to $n$, we obtain:
\begin{equation}\label{P2W}
\left\Vert \overset{2}{\mathbb{P}}\left( \tilde{W} \right) \right\Vert^{2}  \lesssim  \frac{1}{\left\vert 1 + \eta \right\vert^{2}} \; a^{2} \, \sum_{n} \,  \left\vert \langle \mathcal{P}(\cdot , 0) , e_{n}^{(2)} \rangle \right\vert^{2} + a^{4} \, \sum_{n}  \left\vert \langle T, e_{n}^{(2)} \rangle \right\vert^{2}.  
\end{equation}
The computation of $\underset{n}{\sum} \left\vert \langle T, e_{n}^{(2)} \rangle \right\vert^{2}$ can be made similar to the computation of $\underset{n}{\sum} \left\vert \langle T, e_{n}^{(1)} \rangle \right\vert^{2}$. Indeed, from $(\ref{T.en1.en2})$, we deduce  
\begin{equation*}
\sum_{n} \left\vert \langle T, e_{n}^{(2)} \rangle \right\vert^{2}  \lesssim  \left\Vert  \overset{3}{\mathbb{P}}\left(\tilde{W}\right) \right\Vert^{2}  + a^{6}  \left\Vert  \overset{1}{\mathbb{P}}\left(\tilde{W}\right) \right\Vert^{2}.
\end{equation*}
Thus $(\ref{P2W})$ takes the form 
\begin{equation}\label{Equa-228}
\left\Vert \overset{2}{\mathbb{P}}\left( \tilde{W} \right) \right\Vert^{2}  \lesssim  \frac{1}{\left\vert 1 + \eta \right\vert^{2}} \; a^{2} \, \sum_{n} \,  \left\vert \langle \mathcal{P}(\cdot , 0) , e_{n}^{(2)} \rangle \right\vert^{2} + a^{4} \, \left\Vert  \overset{3}{\mathbb{P}}\left(\tilde{W}\right) \right\Vert^{2}  + a^{10}  \left\Vert  \overset{1}{\mathbb{P}}\left(\tilde{W}\right) \right\Vert^{2}.  
\end{equation}
Next, we estimate $\left\Vert \overset{3}{\mathbb{P}}\left( \tilde{W} \right) \right\Vert^{2}$. From \eqref{Equa-Est-W}, we have
\begin{eqnarray}\label{P3W}
\nonumber
\langle \overset{3}{\mathbb{P}}\left( \tilde{W} \right), e_{n}^{(3)} \rangle &=& \frac{1}{\left( 1 + \eta \, \lambda_{n}^{(3)} \right)} \left[ a \, \langle \mathcal{P}(\cdot , 0) , e_{n}^{(3)} \rangle - \langle T, e_{n}^{(3)} \rangle \right], \\ \nonumber
\left\vert \langle \overset{3}{\mathbb{P}}\left( \tilde{W} \right), e_{n}^{(3)} \rangle \right\vert^{2} & \lesssim & \left\vert \eta \right\vert^{-2} \left[ a^{2} \, \left\vert \langle \mathcal{P}(\cdot , 0) , e_{n}^{(3)} \rangle \right\vert^{2} + \left\vert \langle T, e_{n}^{(3)} \rangle \right\vert^{2} \right], \\
\left\Vert \overset{3}{\mathbb{P}}\left( \tilde{W} \right) \right\Vert^{2} & \lesssim & \left\vert \eta \right\vert^{-2} \; a^{2} \, \sum_{n} \,  \left\vert \langle \mathcal{P}(\cdot , 0) , e_{n}^{(3)} \rangle \right\vert^{2} + a^{4} \, \sum_{n}  \left\vert \langle T, e_{n}^{(3)} \rangle \right\vert^{2}.  
\end{eqnarray}
For $\left\vert \langle T, e_{n}^{(3)} \rangle \right\vert$, we have
\begin{eqnarray*}
	\left\vert \langle T, e_{n}^{(3)} \rangle \right\vert^{2} & \lesssim & \left\vert \langle N\left(\overset{3}{\mathbb{P}}\left(\tilde{W}\right) \right), e_{n}^{(3)} \rangle \right\vert^{2} + a^{2} \left\vert \langle \int_{B}   \overset{3}{\mathbb{P}}\left(\tilde{W}\right)(y) \, dy, e_{n}^{(3)} \rangle \right\vert^{2}  \\
	& + & a^{4} \, \sum_{\ell \geq 3} \left\vert \langle \int_{B} \frac{Hess\left( \left\vert \cdot - y \right\vert^{\ell} \right)}{(\ell +1)!} \cdot \overset{3}{\mathbb{P}}\left(\tilde{W}\right)(y) \, dy , e_{n}^{(3)} \rangle \right\vert^{2} + a^{6} \sum_{\ell \geq 2}  \left\vert \langle \int_{B} \frac{\left\vert \cdot - y \right\vert^{\ell}}{(\ell + 1)!}   \overset{3}{\mathbb{P}}\left(\tilde{W}\right)(y) dy, e_{n}^{(3)} \rangle \right\vert^{2} \\ &+& a^{6} \sum_{\ell \geq 2}  \left\vert \langle \int_{B} \frac{\left\vert \cdot - y \right\vert^{\ell}}{(\ell + 1)!} \;  \overset{1}{\mathbb{P}}\left(\tilde{W}\right)(y) dy, e_{n}^{(3)} \rangle \right\vert^{2}. 
\end{eqnarray*}
Taking the series with respect to $n$ index and utilizing the continuity of the Newtonian potential operator, we obtain 
\begin{eqnarray*}
	\sum_{n} \left\vert \langle T, e_{n}^{(3)} \rangle \right\vert^{2} & \lesssim & \left\Vert \overset{3}{\mathbb{P}}\left(\tilde{W}\right) \right\Vert^{2} + a^{2} \left\Vert \overset{3}{\mathbb{P}}\left(\tilde{W}\right) \right\Vert^{2} + a^{4} \left\Vert \overset{3}{\mathbb{P}}\left(\tilde{W}\right) \right\Vert^{2} + a^{6} \left\Vert \overset{3}{\mathbb{P}}\left(\tilde{W}\right) \right\Vert^{2} + a^{6} \left\Vert \overset{1}{\mathbb{P}}\left(\tilde{W}\right) \right\Vert^{2} \\ 
	& \lesssim & \left\Vert \overset{3}{\mathbb{P}}\left(\tilde{W}\right) \right\Vert^{2} + a^{6} \left\Vert \overset{1}{\mathbb{P}}\left(\tilde{W}\right) \right\Vert^{2}, 
\end{eqnarray*}
and therefore $(\ref{P3W})$ becomes 
\begin{equation}\label{Equa-230}
\left\Vert \overset{3}{\mathbb{P}}\left( \tilde{W} \right) \right\Vert^{2}  \lesssim  \left\vert \eta \right\vert^{-2} \; a^{2} \, \sum_{n} \,  \left\vert \langle \mathcal{P}(\cdot , 0) , e_{n}^{(3)} \rangle \right\vert^{2} + a^{10} \,\left\Vert \overset{1}{\mathbb{P}}\left( \tilde{W} \right) \right\Vert^{2}.  
\end{equation}
Plugging $(\ref{Equa-230})$ into $(\ref{Equa-226})$, we can deduce  
\begin{eqnarray*}
	\left\Vert \overset{1}{\mathbb{P}}\left( \tilde{W} \right) \right\Vert^{2}  & \lesssim &   a^{2} \, \sum_{n} \frac{\left\vert \langle \mathcal{P}(\cdot , 0) , e_{n}^{(1)} \rangle \right\vert^{2}}{\left\vert 1 - k^{2} \eta \, a^{2} \, \lambda_{n}^{(1)} \right\vert^{2}} + a^{6-2h} \, \sum_{n} \left\vert \langle \mathcal{P}(\cdot , 0) , e_{n}^{(3)} \rangle \right\vert^{2} \\ 
	& \lesssim &  a^{2} \,\frac{\left\vert \langle \mathcal{P}(\cdot , 0) , e_{n_{0}}^{(1)} \rangle \right\vert^{2}}{\left\vert 1 - k^{2} \eta \, a^{2} \, \lambda_{n_{0}}^{(1)} \right\vert^{2}} + a^{2} \, \sum_{n \neq n_{0}} \frac{\left\vert \langle \mathcal{P}(\cdot , 0) , e_{n}^{(1)} \rangle \right\vert^{2}}{\left\vert 1 - k^{2} \eta \, a^{2} \, \lambda_{n}^{(1)} \right\vert^{2}} + a^{6-2h} \, \sum_{n} \left\vert \langle \mathcal{P}(\cdot , 0) , e_{n}^{(3)} \rangle \right\vert^{2}. 
\end{eqnarray*}
This implies, 
\begin{equation}\label{fin-W-1}
\left\Vert \overset{1}{\mathbb{P}}\left( \tilde{W} \right) \right\Vert = \mathcal{O}\left( a^{1-h} \right). 
\end{equation}
Combining with $(\ref{fin-W-1})$ and $(\ref{Equa-230})$, we deduce that 
\begin{equation}\label{fin-W-3}
\left\Vert \overset{2}{\mathbb{P}}\left( \tilde{W} \right) \right\Vert = \mathcal{O}\left( a^{3} \right). 
\end{equation}
Similarly, from $(\ref{fin-W-1}),(\ref{fin-W-3})$ and $(\ref{Equa-228})$, we derive that
\begin{equation}\label{fin-W-2}
\left\Vert \overset{3}{\mathbb{P}}\left( \tilde{W} \right) \right\Vert = \mathcal{O}\left( a^{3} \right). 
\end{equation}
\end{proof}

\section{Derivation of the linear algebraic systems stated in Section \ref{sec-la-system}.}\label{sec-proof-la}

\subsection{Proof of Proposition \ref{pro-general la}.}
We first give the proof of the invertibility condition for the general form of the linear algebraic systerm as follows.

\begin{proof}[Proof of Proposition \ref{pro-general la}]
	We rewrite the linear algebraic system \eqref{eq-gene-la} in the following matrix form
	\begin{equation}\label{la-trans1}
	\left(\begin{array}{c}
	Q_1\\ 
	Q_2\\ 
	\vdots\\ 
	Q_\aleph
	\end{array}\right) - \eta 
	\left(
	\begin{array}{cccc}
	0& [P_{D_1}]\cdot\Upsilon_{12}^k  &\cdots  & [P_{D_1}]\cdot\Upsilon_{1\aleph}^k  \\ 
	\left[P_{D_2}\right] \cdot \Upsilon_{21}^k& 0  & \cdots  &\left[P_{D_2}\right]\cdot \Upsilon_{2\aleph}^k  \\ 
	\vdots&\vdots  & \ddots &\vdots  \\ 
	\left[P_{D_\aleph}\right]\cdot\Upsilon_{\aleph1}^k&\left[P_{D_\aleph}\right]\cdot\Upsilon_{\aleph2}^k  & \cdots &0 
	\end{array} \right)
	\left(
	\begin{array}{c}
	Q_1\\ 
	Q_2\\ 
	\vdots\\ 
	Q_\aleph
	\end{array}
	\right)
	= \left(
	\begin{array}{c}
	J_1\\ 
	J_2\\ 
	\vdots\\ 
	J_\aleph
	\end{array} \right),
	\end{equation}
	where $\Upsilon_{ij}^k:=\Upsilon_k(z_i,z_j)$ and we set $\left\rVert [P_{0}] \right\rVert := \underset{j=1, ..., \aleph}{\max} \left\lVert [P_{D_j}] \right\rVert_{\mathbb{L}^{\infty}(\Omega)}$. 
	%
	
	Taking the inner product on the both sides of \eqref{la-trans1} with the following vector \newline
$
	\left(
	\left[P_{D_1}\right] \cdot Q_1, 
	\left[P_{D_2}\right] \cdot Q_2, 
	\cdots,\\ 
	\left[P_{D_\aleph}\right] \cdot Q_\aleph
\right)^T,
$
	by direct computations, we obtain from \eqref{la-trans1} that
	\begin{equation}\label{la-trans4}
	\sum_{j=1}^{\aleph}\left\langle Q_j, \ [P_{D_j}] \cdot Q_j \right\rangle-\eta  \sum_{\ell=1}^{\aleph}\left\langle[P_{D_\ell}]\cdot\sum_{j=1 \atop j \neq \ell}^{\aleph}\Upsilon_{\ell j}^k\cdot Q_j,  \ [P_{D_\ell}] \cdot Q_\ell\right\rangle
	= \sum_{j=1}^{\aleph}\left\langle J_j, [P_{D_j}] \cdot Q_j \right\rangle.
	\end{equation}
	We rewrite \eqref{la-trans4} as follows
	\begin{align}\label{la-trans5}
	&\sum_{j=1}^{\aleph}\left\langle Q_j , [P_{D_j}] \cdot Q_j\right\rangle-\eta \sum_{\ell=1}^{\aleph}\left\langle[P_{D_\ell}]\cdot\sum_{j=1 \atop j \neq \ell}^{\aleph}\Upsilon_{\ell j}^0\cdot Q_j, \ [P_{D_\ell}] \cdot Q_\ell\right\rangle\notag\\
	&-\eta  \sum_{\ell=1}^{\aleph}\left\langle[P_{D_\ell}]\cdot\sum_{j=1 \atop j \neq \ell}^{\aleph}(\Upsilon_{\ell j}^k-\Upsilon_{\ell j}^0)\cdot Q_j, \ [P_{D_\ell}] \cdot Q_\ell\right\rangle
	= \sum_{j=1}^{\aleph}\left\langle J_j, [P_{D_j}] \cdot Q_j\right\rangle.
	\end{align}

	Denote
	\begin{equation}\notag
	\Gamma^0:=\sum_{\ell = 1}^\aleph\left\langle [P_{D_\ell}]\cdot\sum_{j=1 \atop j \neq \ell}^\aleph\Upsilon_{\ell j}^0\cdot Q_j, \ [P_{D_\ell}] \cdot Q_\ell\right\rangle.
	\end{equation}
	Since $\Upsilon_{\ell j}^0$ is harmonic, by the Mean Value Theorem we know that
	\begin{equation}\notag
	\Upsilon_{\ell j}^0:=\Upsilon^0(z_\ell, z_j)=\frac{1}{|\mathbf{B}_\ell|}\frac{1}{|\mathbf{B}_j|}\int_{\mathbf{B}_\ell}\int_{\mathbf{B}_j}\Upsilon^0(x, z)\,dx\,dz,
	\end{equation}
	where $\mathbf{B}_\ell$ and $\mathbf{B}_j$ are two balls of radius $\dfrac{d}{2}$, centered at $z_\ell$ and $z_j$, respectively.
	Then 
	\begin{align}\notag
	\Gamma^0&=\sum_{\ell = 1}^\aleph\left\langle[P_{D_\ell}]\cdot\sum_{j=1 \atop j \neq \ell}^\aleph\frac{1}{|\mathbf{B}_\ell|}\frac{1}{|\mathbf{B}_j|}\int_{\mathbf{B}_\ell}\int_{\mathbf{B}_j}\Upsilon^0(x, z)\,dx\,dz\cdot Q_j,\ [P_{D_\ell}] \cdot Q_\ell\right\rangle\notag\\
	&=\frac{36}{\pi^2 d^6}\sum_{\ell = 1}^\aleph\left\langle[P_{D_\ell}]\cdot\sum_{j=1 \atop j \neq \ell}^\aleph\int_{\mathbf{B}_\ell}\int_{\mathbf{B}_j}\Upsilon^0(x, z)\,dx\,dz\cdot Q_j,\ [P_{D_\ell}] \cdot Q_\ell\right\rangle\notag\\
	&=\frac{36}{\pi^2 d^6}\sum_{\ell = 1}^\aleph\sum_{j=1 \atop j \neq \ell}^\aleph\int_{\mathbf{B}_\ell}\int_{\mathbf{B}_j}\left\langle[P_{D_\ell}]\cdot\Upsilon^0(x, z)\cdot Q_j,\ [P_{D_\ell}] \cdot Q_\ell\right\rangle\,dx\,dz.\notag
	\end{align}
	
	We denote
	\begin{equation}\label{nota-V}
	V_m:=[P_{D_m}]\cdot Q_m \quad\mbox{for} \quad m=1, 2, \cdots, \aleph,
	\end{equation}
	and for $m=1,\cdots, \aleph$,
	\begin{equation}\label{nota-piq}
	\Pi(x):=
	\begin{cases}
	V_m&\mbox{in}\quad\mathbf{B}_m \\
	0&\mbox{otherwise}
	\end{cases},\
	\Lambda(x):=
	\begin{cases}
	Q_m& \mbox{in}\quad \mathbf{B}_m\\
	0&\mbox{otherwise}
	\end{cases},\
	[\mathbf{P}]:=
	\begin{cases}
	[P_{D_m}]&\mbox{in}\quad\mathbf{B}_m\\
	0&\mbox{otherwise}
	\end{cases}.
	\end{equation}
	Set $\Omega:=\overset{\aleph}{\underset{m=1}{\cup}}\mathbf{B}_m$, then $\Gamma^0$ becomes
	\begin{equation}\notag
	\Gamma^0=\frac{36}{\pi^2 d^6}\int_{\Omega}\int_{\Omega}\left\langle[\mathbf{P}]\cdot\Upsilon^0(x, z)\cdot \Lambda(x), \Pi(z)\right\rangle\,dx\,dz-\frac{36}{\pi^2 d^6}\sum_{m=1}^{\aleph}\int_{\mathbf{B}_m}\int_{\mathbf{B}_m}\left\langle[P_{D_m}]\cdot\Upsilon^0(x, z)\cdot Q_m,  V_m\right\rangle\,dx\,dz,
	\end{equation}
	It is direct to see that
	\begin{eqnarray*}
	\int_{\mathbf{B}_m}\int_{\mathbf{B}_m}\Upsilon^0(x, z)\,dz\,dx = \int_{\mathbf{B}_m}\int_{\mathbf{B}_m}\underset{z}{Hess} \Phi_{0}(x, z)\,dz\,dx
	& = & -\int_{\mathbf{B}_m}\underset{x}{\nabla}\int_{\mathbf{B}_m}\underset{z}{\nabla} \Phi_{0}(x, z)\,dz\,dx \\
	&=& \int_{\mathbf{B}_m}\underset{x}{Hess}\int_{\mathbf{B}_m} \Phi_{0}(x, z)\,dz\,dx.
	\end{eqnarray*}
	From \cite[eq. (1.12)]{GS}, we know that for $\mathbf{B}_{m}$, there holds:
	\begin{equation}\label{N1}
	N_{\mathbf B_{m}}(1)(x)= \frac{d^{2}}{8} - \frac{\left\vert x \right\vert^{2}}{6} \quad \text{henceforth} \quad Hess \, N_{\mathbf B_{m}}(1)(x)=-\frac{1}{3}I, \quad x \in \mathbf{B_{m}}.
	\end{equation}
Then,
	\begin{equation}\label{calcu-int}
	\int_{\mathbf{B}_m}\int_{\mathbf{B}_m}\Upsilon^0(x, z)\,dz\,dx=\int_{\mathbf{B}_m}\underset{x}{Hess}N(1)(x)\,dx=-\frac{\pi d^3}{18}I.
	\end{equation}
Therefore, we obtain
	\begin{align}\notag
	\Gamma^0&=\frac{36}{\pi^2 d^6}\int_{\Omega}\int_{\Omega}\left\langle[\mathbf{P}]\cdot\Upsilon^0(x, z)\cdot \Lambda(x), \Pi(z)\right\rangle\,dx\,dz-\frac{36}{\pi^2 d^6}\sum_{m=1}^{\aleph}\left\langle[P_{D_m}]\cdot\left(-\frac{\pi d^3}{18}I\right)\cdot Q_m,  V_m\right\rangle,\notag\\
	&=\frac{36}{\pi^2 d^6}\int_{\Omega}\int_{\Omega}\left\langle[\mathbf{P}]\cdot\Upsilon^0(x, z)\cdot \Lambda(x), \Pi(z)\right\rangle\,dx\,dz+\frac{2}{\pi d^3}\sum_{m=1}^{\aleph}|V_m|^2.\notag
	\end{align}
	
	
Now we proceed to estimate the difference term by denoting
	\begin{equation}\notag
	\mathcal{R}^k:=\Gamma^k-\Gamma^0=\sum_{\ell = 1}^\aleph\left\langle[P_{D_\ell}]\cdot\sum_{j=1 \atop j \neq \ell}^\aleph\left(\Upsilon_{\ell j}^k-\Upsilon_{\ell j}^0\right)\cdot Q_j, \ V_\ell\right\rangle.
	\end{equation}
	Combining with \eqref{expansion-of-Hess}, we derive by direct computations that
	\begin{align}\label{diff-U}
	\Upsilon_{\ell j}^k-\Upsilon_{\ell j}^0&=\underset{x}{Hess}\Phi_{k}(z_\ell, z_j)-\underset{x}{Hess}\Phi_{0}(z_\ell, z_j)+k^2\Phi_k(z_\ell, z_j)I\quad (x\neq y)\notag\\
	&=\frac{k^2}{2}\Phi_{0}(z_\ell, z_j)I-\frac{i k^3}{12\pi}I+\frac{k^2}{2}\Phi_{0}(z_\ell, z_j)\frac{(z_\ell-z_j)\otimes(z_\ell-z_j)}{\lVert z_\ell-z_j \rVert^2}\notag\\
	&+\frac{1}{4\pi}\sum_{n\geq3}\frac{(ik)^{n+1}}{(n+1)!}Hess\left(\lVert z_\ell-z_j\rVert^n\right)+k^2\frac{1}{4\pi}\sum_{n\geq0}\frac{(ik)^{n+1}}{(n+1)!}|z_\ell-z_j|^n.
	\end{align}
	Then $\mathcal{R}^k$ becomes
	\begin{align}\notag
	\mathcal{R}^k&=\sum_{\ell = 1}^\aleph \sum_{j=1 \atop j \neq \ell}^\aleph \left\langle[P_{D_\ell}]\cdot\left(\Upsilon_{\ell j}^k-\Upsilon_{\ell j}^0\right)\cdot Q_j, \ V_\ell\right\rangle\notag\\
	&=\sum_{\ell = 1}^\aleph \sum_{j=1 \atop j \neq \ell}^\aleph \left\langle[P_{D_\ell}]\cdot\left(-\frac{i k^3}{12\pi}I\right)\cdot Q_j, \ V_\ell\right\rangle
	+\sum_{\ell = 1}^\aleph \sum_{j=1 \atop j \neq \ell}^\aleph \left\langle[P_{D_\ell}]\cdot\left(\frac{k^2}{2}\frac{1}{4\pi|z_\ell-z_j|}I\right)\cdot Q_j, \ V_\ell\right\rangle\notag\\
	&+\sum_{\ell = 1}^\aleph \sum_{j=1 \atop j \neq \ell}^\aleph \left\langle[P_{D_\ell}]\cdot\left(\frac{k^2}{2}\frac{1}{4\pi|z_\ell-z_j|}\frac{(z_\ell-z_j)}{|z_\ell-z_j|}\otimes\frac{(z_\ell-z_j)}{|z_\ell-z_j|}\right)\cdot Q_j, \ V_\ell\right\rangle\notag\\
	&+\sum_{\ell = 1}^\aleph \sum_{j=1 \atop j \neq \ell}^\aleph \left\langle[P_{D_\ell}]\cdot\left(\frac{1}{4\pi}\sum_{n\geq3}\frac{(ik)^{n+1}}{(n+1)!}Hess\left(\lVert z_\ell-z_j\rVert^n\right)\right)\cdot Q_j, \ V_\ell\right\rangle\notag\\
	&+\sum_{\ell = 1}^\aleph \sum_{j=1 \atop j \neq \ell}^\aleph \left\langle[P_{D_\ell}]\cdot\left(\frac{k^2}{4\pi}\sum_{n\geq0}\frac{(ik)^{n+1}}{(n+1)!}\big| z_\ell-z_j\big|^n\right)\cdot Q_j, \ V_\ell\right\rangle.\notag
	\end{align}
	With the notation \eqref{nota-V}, we further derive
	\begin{align}\notag
	\mathcal{R}^k&=-\sum_{\ell = 1}^\aleph\sum_{j=1 \atop j \neq \ell}^\aleph\frac{i k^3}{12\pi}\langle [P_{D_\ell}] Q_j, V_\ell\rangle+\sum_{\ell = 1}^\aleph\sum_{j=1 \atop j \neq \ell}^\aleph\frac{k^2}{2}\frac{1}{4\pi|z_\ell-z_j|}\langle [P_{D_\ell}] Q_j, V_\ell\rangle\notag\\
	&+\sum_{\ell = 1}^\aleph\sum_{j=1 \atop j \neq \ell}^\aleph\frac{k^2}{2}\frac{1}{4\pi|z_\ell-z_j|}\left\langle \frac{(z_\ell-z_j)}{|z_\ell-z_j|},  Q_j\right\rangle\left\langle [P_{D_\ell}]\cdot\frac{(z_\ell-z_j)}{|z_\ell-z_j|}, V_\ell\right\rangle\notag\\
	&+\sum_{\ell = 1}^\aleph\sum_{j=1 \atop j \neq \ell}^\aleph\frac{1}{4\pi}\sum_{n \geq 3}\frac{(ik)^{n+1}}{(n+1)!}Hess(\lVert z_\ell-z_j\rVert^n)\langle [P_{D_\ell}] Q_j, V_\ell\rangle+\sum_{\ell = 1}^\aleph\sum_{j=1 \atop j \neq \ell}^\aleph\frac{k^2}{4\pi}\sum_{n\geq0}\frac{(ik)^{n+1}}{(n+1)!}|z_\ell-z_j|^n\langle [P_{D_\ell}] Q_j, V_\ell\rangle.\notag
	\end{align}
	Since by the Mean Value Theorem, there holds that
	\begin{equation}\notag
	\frac{1}{4\pi|z_\ell-z_j|}=\frac{1}{|\mathbf{B}_\ell|}\frac{1}{|\mathbf{B}_j|}\int_{\mathbf{B}_\ell}\int_{\mathbf{B}_j}\Phi_{0}(x, y)\,dx\,dy=\frac{36}{\pi^{2} \, d^{6}}\int_{\mathbf{B}_\ell}\int_{\mathbf{B}_j}\Phi_{0}(x, y)\,dx\,dy,
	\end{equation}
	then
	\begin{eqnarray}\label{R^k}
	\mathcal{R}^k &=& -\frac{i k^3}{12\pi}\sum_{\ell = 1}^\aleph\sum_{j=1 \atop j \neq \ell}^\aleph\langle [P_{D_\ell}] Q_j, V_\ell\rangle+\frac{18 k^2}{\pi^2 d^6}\sum_{\ell = 1}^\aleph\sum_{j=1 \atop j \neq \ell}^\aleph\int_{\mathbf{B}_\ell}\int_{\mathbf{B}_j}\Phi_{0}(x, y)\langle [P_{D_\ell}] Q_j, V_\ell\rangle\,dx\,dy\notag\\ 
	&+& \frac{18 k^2}{\pi^2 d^6}\sum_{\ell = 1}^\aleph\sum_{j=1 \atop j \neq \ell}^\aleph\int_{\mathbf{B}_\ell}\int_{\mathbf{B}_j}\Phi_{0}(x, y)\left\langle \frac{(z_\ell-z_j)}{|z_\ell-z_j|}, \  Q_j\right\rangle\left\langle [P_{D_\ell}]\cdot\frac{(z_\ell-z_j)}{|z_\ell-z_j|}, V_\ell\right\rangle\,dx\,dy\notag\\ \nonumber
	&+& \sum_{\ell = 1}^\aleph\sum_{j=1 \atop j \neq \ell}^\aleph\frac{1}{4\pi}\sum_{n \geq 3}\frac{(i k)^{n+1}}{(n+1)!}Hess(\lVert z_\ell-z_j\rVert^n)\langle [P_{D_\ell}] Q_j, V_\ell\rangle \\  &+& \sum_{\ell = 1}^\aleph\sum_{j=1 \atop j \neq \ell}^\aleph\frac{k^2}{4\pi}\sum_{n\geq0} \frac{(i k)^{n+1}}{(n+1)!}|z_\ell-z_j|^n\langle [P_{D_\ell}] Q_j, V_\ell\rangle.
	\end{eqnarray}
	Considering the second term in \eqref{R^k}, we know that
	\begin{eqnarray}\label{RK-mid1}
\nonumber
L_{6} &:=&	\sum_{\ell=1}^\aleph\sum_{j=1 \atop j \neq \ell}^\aleph\int_{\mathbf{B}_\ell}\int_{\mathbf{B}_j}\Phi_{0}(x, y)\langle [P_{D_\ell}] Q_j, V_\ell\rangle\,dx\,dy \\
	&=&\sum_{\ell, j=1}^\aleph\int_{\mathbf{B}_\ell}\int_{\mathbf{B}_j}\Phi_{0}(x, y)\langle [P_{D_\ell}] Q_j, V_\ell\rangle\,dx\,dy-\sum_{m=1}^\aleph\int_{\mathbf{B}_m}\int_{\mathbf{B}_m}\Phi_{0}(x, y)\langle V_m, V_m\rangle\,dx\,dy\notag\\
	&=&\int_{\Omega}\int_{\Omega}\Phi_{0}(x, y)[\mathbf{P}]\Lambda(x)\,dx\Pi(y)\,dy-\sum_{m=1}^\aleph\int_{\mathbf{B}_m}\int_{\mathbf{B}_m}\Phi_{0}(x, y)\langle V_m, V_m\rangle\,dx\,dy.
	\end{eqnarray}
	By utilizing \eqref{N1} again, there holds
	\begin{equation}\notag
	\int_{\mathbf{B}_m}\int_{\mathbf{B}_m}\Phi_{0}(x, y)\,dx\,dy=\int_{\mathbf{B}_m}N(1)(y)\,dy=\int_{\mathbf{B}_m} \, \left( \frac{d^{2}}{8} - \frac{\left\vert y \right\vert^{2}}{6} \right) \, dy=\frac{\pi d^5}{60}.
	\end{equation}
	Thus, \eqref{RK-mid1} can be written as
	\begin{equation}\notag
	L_{6} =\int_{\Omega}\int_{\Omega}\Phi_{0}(x, y)[\mathbf{P}]\Lambda(x)\Pi(y)\,dx\,dy-\sum_{m=1}^\aleph\frac{\pi d^5}{60}|V_m|^2,
	\end{equation}
	which implies that in \eqref{R^k}, we have
	\begin{align}\label{Rk-final}
	\mathcal{R}^k&=-\frac{i k^3}{12\pi}\sum_{\ell = 1}^\aleph\sum_{j\neq\ell}^\aleph\langle [P_{D_\ell}] Q_j, V_\ell\rangle+\frac{18 k^2}{\pi^2 d^6}\int_{\Omega}\int_{\Omega}\Phi_{0}(x, y)[\mathbf{P}]\Lambda(x)\Pi(y)\,dx\,dy-\frac{3k^2}{10\pi d}\sum_{m=1}^\aleph|V_m|^2\notag\\
	&+\frac{18 k^2}{\pi^2 d^6}\sum_{\ell = 1}^\aleph\sum_{j\neq\ell}^\aleph\int_{\mathbf{B}_\ell}\int_{\mathbf{B}_j}\Phi_{0}(x, y)\left\langle \frac{(z_\ell-z_j)}{|z_\ell-z_j|}, \ Q_j\right\rangle\left\langle [P_{D_\ell}]\cdot\frac{(z_\ell-z_j)}{|z_\ell-z_j|}, V_\ell\right\rangle\,dx\,dy\notag\\
	&+\sum_{\ell = 1}^\aleph\sum_{j\neq\ell}^\aleph\frac{1}{4\pi}\sum_{n \geq 3}\frac{(i k)^{n+1}}{(n+1)!}Hess(\lVert z_\ell-z_j\rVert^n)\langle [P_{D_\ell}] Q_j, V_\ell\rangle \notag \\ &+
	\sum_{\ell = 1}^\aleph\sum_{j\neq\ell}^\aleph\frac{k^2}{4\pi}\sum_{n\geq0} \frac{(i k)^{n+1}}{(n+1)!}|z_\ell-z_j|^n\langle [P_{D_\ell}] Q_j, V_\ell\rangle.
	\end{align}
	Now, we give the estimation of $\mathcal{R}^k$ term by term. 
	\begin{itemize}
		\item Estimation of 
		\begin{eqnarray}\label{es-RK1}
		\nonumber
	L_{7,1} & := &	-\frac{i k^3}{12\pi}\overset{\aleph}{\underset{\ell=1}{\sum}}\overset{\aleph}{\underset{j\neq\ell}{\sum}}\langle [P_{D_\ell}] Q_j, V_\ell\rangle, \\
	\left\vert	 L_{7,1} \right\vert & = & \bigg|-\frac{i k^3}{12\pi}\sum_{\ell = 1}^\aleph\sum_{j\neq\ell}^\aleph\langle [P_{D_\ell}] Q_j, V_\ell\rangle\bigg|
		\lesssim k^3\sum_{\ell = 1}^\aleph\sum_{j\neq\ell}^\aleph|\langle [P_{D_\ell}] Q_j, V_\ell\rangle|
		\lesssim k^3\sum_{\ell = 1}^\aleph\sum_{j\neq \ell}^\aleph\bigg|\left\langle [P_{D_\ell}]Q_j, [P_{D_\ell}] Q_\ell\right\rangle\bigg|\notag\\
		& \lesssim &  k^3 \left\lVert[P_0]\right\rVert^2\left(\sum_{j = 1}^\aleph|Q_j|\right)^2\lesssim k^3 \left\lVert[P_0]\right\rVert^2 \, \aleph \,  \sum_{j = 1}^\aleph |Q_j|^2.
		\end{eqnarray}

		\item Estimation of 
\begin{equation*}
L_{7,2} := \frac{18k^2}{\pi^2 d^6}\int_{\Omega}\int_{\Omega}\Phi_{0}(x, y)[\mathbf{P}]\Lambda(x)\,dx\Pi(y)\,dy.
\end{equation*}		
Following the definition of the Newtonian potential operator, we have 
		\begin{eqnarray}\label{es-RK2}
		\nonumber
	\left\vert	L_{7,2} \right\vert &=& \bigg|\frac{18k^2}{\pi^2 d^6}\int_{\Omega}\int_{\Omega}\Phi_{0}(x, y)[\mathbf{P}]\Lambda(x)\,dx\Pi(y)\,dy\bigg|  \lesssim  k^2 d^{-6}\big|\int_{\Omega} N(\Pi)(x)\cdot[\mathbf{P}]\Lambda(x)\,dx\big| \\ & \lesssim &  k^2 d^{-6}\lVert N(\Pi)\rVert_{\mathbb{L}^2(\Omega)}\lVert[\mathbf{P}]\Lambda \rVert_{\mathbb{L}^2(\Omega)}\notag 
		 \lesssim  k^2 d^{-6}\lVert N\rVert_{\mathcal{L}(\mathbb{L}^2(\Omega); \mathbb{L}^2(\Omega))}\lVert \Pi\rVert_{\mathbb{L}^2(\Omega)}\left\lVert[P_0]\right\rVert \, \lVert\Lambda\rVert_{\mathbb{L}^2(\Omega)}
	\\ \nonumber	& \lesssim & k^2 d^{-3}\lVert N\rVert_{\mathcal{L}(\mathbb{L}^2(\Omega); \mathbb{L}^2(\Omega))}\left(\sum_{m=1}^\aleph|V_m|^2\right)^\frac{1}{2}\left\lVert[P_0]\right\rVert \, \left(\sum_{m=1}^\aleph |Q_m|^2\right)^\frac{1}{2}\notag \\ 
		& \lesssim & k^2 d^{-3}\lVert N\rVert_{\mathcal{L}(\mathbb{L}^2(\Omega); \mathbb{L}^2(\Omega))}\left\lVert[P_0]\right\rVert^2 \, \sum_{m=1}^\aleph\big|Q_m\big|^2. 
		\end{eqnarray}
		
		\item Estimation of
\begin{eqnarray}\label{es-RK3}
\nonumber
L_{7,3} &:=& -\frac{3k^2}{10\pi d}\overset{\aleph}{\underset{m=1}{\sum}}|V_m|^2, \\
\left\vert L_{7,3} \right\vert & \lesssim & k^2 d^{-1}\sum_{m=1}^\aleph|V_m|^2 	= k^2 d^{-1}\sum_{m=1}^\aleph\big|[P_{D_m}]Q_m\big|^2 = k^2 d^{-1} \, \left\Vert [P_{0}] \right\Vert^{2} \ \sum_{m=1}^\aleph\big| Q_m \big|^2. 
\end{eqnarray}

		\item Estimation of 
\begin{eqnarray*} 
\nonumber
L_{7,4} & := & \frac{18 k^2}{\pi^2 d^6}\overset{\aleph}{\underset{\ell=1}{\sum}}\overset{\aleph}{\underset{j\neq\ell}{\sum}}\int_{\mathbf{B}_\ell}\int_{\mathbf{B}_j}\Phi_{0}(x, y)\left\langle \frac{(z_\ell-z_j)}{|z_\ell-z_j|}, \ Q_j\right\rangle\left\langle [P_{D_\ell}]\cdot\frac{(z_\ell-z_j)}{|z_\ell-z_j|}, V_\ell\right\rangle\,dx\,dy, \\
\left\vert L_{7,4} \right\vert & = & \bigg| \frac{18 k^2}{\pi^2 d^6}\sum_{\ell = 1}^\aleph\sum_{j\neq\ell}^\aleph\int_{\mathbf{B}_\ell}\int_{\mathbf{B}_j}\Phi_{0}(x, y)\left\langle \frac{(z_\ell-z_j)}{|z_\ell-z_j|}, \ Q_j\right\rangle\left\langle [P_{D_\ell}]\cdot\frac{(z_\ell-z_j)}{|z_\ell-z_j|}, V_\ell\right\rangle\,dx\,dy\bigg|\notag \\
& \lesssim & k^2  \sum_{\ell = 1}^\aleph\sum_{j\neq\ell}^\aleph\frac{1}{|z_\ell-z_j|} \; \left\vert \left\langle \frac{(z_\ell-z_j)}{|z_\ell-z_j|} , \ Q_j\right\rangle \right\vert \left\vert \; \left\langle [P_{D_\ell}]\cdot\frac{(z_\ell-z_j)}{|z_\ell-z_j|}, V_\ell\right\rangle \right\vert \notag \\ 
& \lesssim & k^2 \, \left\lVert[P_0]\right\rVert \, \, \sum_{\ell = 1}^\aleph  \left\vert V_{\ell} \right\vert \, \sum_{j\neq\ell}^\aleph\frac{1}{|z_\ell-z_j|}\big| Q_j\big|.
\end{eqnarray*}
Then, by Cauchy-Schwartz inequality, for both of the indices $\ell$ and $j$, we deduce that 
\begin{equation}\label{es-RK4}
\left\vert L_{7,4} \right\vert \lesssim  k^2 \, \left\lVert[P_0]\right\rVert \,\, d^{-\frac{3}{2}} \, \aleph^{\frac{1}{2}} \, \left( \sum_{\ell = 1}^\aleph|V_\ell|^{2} \right)^{\frac{1}{2}} \, \left( \sum_{j = 1}^\aleph \big| Q_j \big|^{2} \right)^{\frac{1}{2}} = k^2 \, \left\lVert[P_0]\right\rVert^{2} \,\, d^{-\frac{3}{2}} \, \aleph^{\frac{1}{2}} \, \, \sum_{j = 1}^\aleph \big| Q_j \big|^{2}. 
\end{equation}
		
		
		
		\item Estimation of 
\begin{equation*}		
L_{7,5}:= \overset{\aleph}{\underset{\ell=1}{\sum}}\overset{\aleph}{\underset{j\neq\ell}{\sum}}\frac{1}{4\pi}\underset{n\geq3}{\sum}\frac{(i k)^{n+1}}{(n+1)!}Hess(\lVert z_\ell-z_j\rVert^n)\langle [P_{D_\ell}]Q_j, V_\ell\rangle.
\end{equation*}		
Similar to \eqref{Series1}, by direct verification, we can prove that 
		\begin{equation}\notag
		\bigg|\sum_{n \geq 3}\frac{(i k)^{n+1}}{(n+1)!}Hess(\lVert z_\ell-z_j\rVert^n)\bigg|\leq\sum_{n\geq3}\frac{n k^{n+1}\lVert z_\ell-z_j\rVert^{n-2}}{(n+1)!}<+\infty,
		\end{equation}
		which, after taking the modulus, implies 
		\begin{equation}\label{es-RK5}
\left\vert L_{7,5} \right\vert   \lesssim    k^4 \sum_{\ell = 1}^\aleph\sum_{j\neq\ell}^\aleph|[P_{D_\ell}]Q_j||V_\ell|  \lesssim  k^4 \, \left\lVert[P_0]\right\rVert^2 \, \aleph \, \sum_{j = 1}^\aleph \lvert Q_j \rvert^2.
		\end{equation}
		Similarly, we also obtain that
		\begin{eqnarray}\label{es-RK6}
		\nonumber
L_{7,6} & := &	\sum_{\ell = 1}^\aleph\sum_{j\neq\ell}^\aleph\frac{k^2}{4\pi}\sum_{n\geq0}\frac{(ik)^{n+1}}{(n+1)!}|z_\ell-z_j|^n\langle [P_{D_\ell}]Q_j, V_\ell\rangle, \\	
\left\vert L_{7,6} \right\vert &=&		\bigg| \sum_{\ell = 1}^\aleph\sum_{j\neq\ell}^\aleph\frac{k^2}{4\pi}\sum_{n\geq0}\frac{(ik)^{n+1}}{(n+1)!}|z_\ell-z_j|^n\langle [P_{D_\ell}]Q_j, V_\ell\rangle\bigg| =
\mathcal{O}\left(  k^3 \, \left\lVert[P_0]\right\rVert^2 \, \aleph \, \sum_{j = 1}^\aleph \lvert Q_j \rvert^2 \right).
%
%
		\end{eqnarray}
	\end{itemize}
	
	Combining with \eqref{es-RK1}, \eqref{es-RK2}, \eqref{es-RK3}, \eqref{es-RK4}, \eqref{es-RK5} and \eqref{es-RK6}, we derive the following estimate:
\begin{equation}\label{es-Rk-final}
\left\vert \mathcal{R}^k \right\vert  \lesssim  \sum_{j=1}^{6} L_{7,j} = \mathcal{O}\left(k^{2} \,  \left\lVert[P_0]\right\rVert^2 \, \sum_{j = 1}^\aleph \lvert Q_j \rvert^2 \left( \aleph \, k  + d^{-3} \, \lVert N\rVert_{\mathcal{L}(\mathbb{L}^2(\Omega); \mathbb{L}^2(\Omega))} \right)  \right).
\end{equation}
	
	Rewriting the linear algebraic system \eqref{la-trans5} by virtue of the expressions of $\Gamma^0$ and $\mathcal{R}^k$, we get
	\begin{equation}\label{la-trans6}
	\sum_{j = 1}^\aleph\left\langle Q_j, V_j\right\rangle-\eta \frac{36}{\pi^2 d^6}\int_{\Omega}\int_{\Omega}\left\langle  [\mathbf{P}]\cdot\Upsilon^0(x, z)\cdot \Lambda(x), \Pi(z)\right\rangle\,dx\,dz
	-\frac{2\eta }{\pi d^3}\sum_{m=1}^\aleph|V_m|^2-\eta \mathcal{R}^k= \sum_{j = 1}^\aleph\left\langle J_j, V_j\right\rangle.
	\end{equation}
	For the term $\dfrac{2 \, \eta }{\pi \, d^3}\overset{\aleph}{\underset{m=1}{\sum}}|V_m|^2$, it is easy to see that
	\begin{equation}\label{es-mid}
	\bigg|\frac{2\eta}{\pi d^3}\sum_{m=1}^\aleph|V_m|^2\bigg|  \lesssim a^{-2}d^{-3}\sum_{m=1}^\aleph|V_m|^2=a^{-2}d^{-3}\sum_{m=1}^\aleph\big|[P_{D_m}]\cdot Q_{m} \big|^2 = a^{-2} \, d^{-3} \, \left\Vert [P_{0}] \right\Vert^{2} \, \sum_{m=1}^\aleph\big| Q_{m} \big|^2.
	\end{equation} 
	By substituting \eqref{es-Rk-final} and \eqref{es-mid} into \eqref{la-trans6}, we can further deduce that
	\begin{eqnarray}\label{la-trans7}
\nonumber	
	\sum_{j = 1}^\aleph\left\langle Q_j, V_j\right\rangle &-& \eta \frac{36}{\pi^2 d^6}\int_{\Omega}\int_{\Omega}\left\langle  [\mathbf{P}]\cdot\Upsilon^0(x, z)\cdot \Lambda(x), \Pi(z)\right\rangle\,dx\,dz \\
	&=&\sum_{j = 1}^\aleph\left\langle  J_j, \ V_j\right\rangle+\mathcal{O}\left( a^{-2} \, \left\lVert[P_0]\right\rVert^2 \, \sum_{j = 1}^\aleph \, \lvert Q_j \rvert^2 \, \left(k^3 \aleph+d^{-3} \, k^2 \, \lVert N\rVert_{\mathcal{L}(\mathbb{L}^2(\Omega); \mathbb{L}^2(\Omega))}+d^{-3} \right) \right),
	\end{eqnarray}
	since $d^{-3}$ and $\aleph$ are of the same order with respect to $a$.
	
	Recall the notations \eqref{nota-piq}, where we have
	\begin{equation}\notag
	\sum_{j = 1}^\aleph\left\langle V_j, Q_j\right\rangle=\frac{6}{\pi d^3}\langle \Pi, \Lambda\rangle_{\mathbb{L}^2(\Omega)}.
	\end{equation}
	Then we can derive from \eqref{la-trans7} the linear algebraic systerm in terms of $\Pi$ and $\Lambda$ as follows
	\begin{eqnarray}\label{la-trans-final}
	\langle \Pi, \Lambda \rangle_{\mathbb{L}^2(\Omega)} &-& \frac{6\eta}{\pi d^3}\int_{\Omega}\int_{\Omega}\left\langle [\mathbf{P}]\cdot\Upsilon^0(x, z)\cdot \Lambda(x), \Pi(z)\right\rangle\,dx\,dz=\frac{\pi d^3}{6} \sum_{j=1}^\aleph\langle J_j, \ V_j\rangle\notag \\
	&+& \mathcal{O}\left( a^{-2} \, \left\lVert[P_0]\right\rVert^2 \, \sum_{j = 1}^\aleph\lvert Q_j \rvert^2 \, \left(k^3 \, \aleph \, d^3 \, + \, k^2 \, \lVert N\rVert_{\mathcal{L}(\mathbb{L}^2(\Omega); \mathbb{L}^2(\Omega))} \, + \, 1 \right)\right).
	\end{eqnarray}
	
	For the second term on the left hand side of \eqref{la-trans-final}, there holds
	\begin{equation}\label{secondterm}
	\frac{6\eta }{\pi d^3}\int_{\Omega}\int_{\Omega}\left\langle [\mathbf{P}]\cdot\Upsilon^0(x, z)\cdot \Lambda(x), \Pi(z)\right\rangle\,dx\,dz=-\frac{6\eta}{\pi d^3} \int_{\Omega}[\mathbf{P}]\cdot\nabla M(\Lambda)(z)\cdot\Pi(z)\,dz,
	\end{equation}
	which fulfills the estimate
	\begin{align*}\notag
	\bigg|-\frac{6\eta}{\pi d^3}\int_{\Omega}[\mathbf{P}]\cdot\nabla M(\Lambda)(z)\cdot \Pi(z)\,dz\bigg|&\lesssim |\eta| d^{-3}\lVert[\mathbf{P}]\cdot\nabla M(\Lambda)\rVert_{\mathbb{L}^2(\Omega)}\lVert\Pi\rVert_{\mathbb{L}^2(\Omega)}\notag\\
	&\lesssim |\eta| \left\lVert[P_0]\right\rVert \, d^{-3} \, \lVert\nabla M\rVert_{\mathcal{L}(\mathbb{L}^2(\Omega); \mathbb{L}^2(\Omega))}\lVert\Lambda\rVert_{\mathbb{L}^2(\Omega)}\lVert \Pi\rVert_{\mathbb{L}^2(\Omega)}.
	\end{align*}
Knowing that $\lVert \nabla M\rVert_{\mathcal{L}(\mathbb{L}^2(\Omega); \mathbb{L}^2(\Omega))} = 1$, see Theorem 2.1 of \cite{friedman1980mathematical}, 
we deduce that
\begin{equation*}
\bigg|-\frac{6\eta}{\pi d^3}\int_{\Omega}[\mathbf{P}]\cdot\nabla M(\Lambda)(z)\cdot \Pi(z)\,dz\bigg| \lesssim |\eta| \left\lVert[P_0]\right\rVert \, d^{-3} \,\lVert\Lambda\rVert_{\mathbb{L}^2(\Omega)}\lVert \Pi\rVert_{\mathbb{L}^2(\Omega)}.
\end{equation*}
Thus, a sufficient condition for the linear algebraic system to be invertible is that
	\begin{equation}\notag
	|\eta| \, \left\lVert[P_0]\right\rVert \, d^{-3} \, < \, 1,
	\end{equation}
	which completes the proof.
\end{proof}

Now, we construct the linear algebraic system related to $\overset{1}{\mathbb{P}}(E^T_j)$, $j=1, 2, \cdots, \aleph$, in the following subsection.

\subsection{Construction of the linear algebraic system in Proposition \ref{prop-la-1}.}

\begin{proof}[Proof of Proposition \ref{prop-la-1}]
	Recall the Lippmann-Schwinger equation \eqref{LS eq2}
	\begin{equation*}
 	E^{T}(x) + \eta \, \nabla M^{k}\left( E^{T} \right)(x) - k^{2} \, \eta \, N^{k}\left(E^{T}\right)(x) = E^{Inc}(x), \;\quad x \in D = \underset{j=1}{\overset{\aleph}{\cup}} D_{j},
	\end{equation*}
	which can be rewritten, for $x \in D_{j_{0}}$, as  
	\begin{equation}\label{AI0}
	\boldsymbol{T}_{k} \, \left(E^{T}_{j_{0}}\right)(x) - \eta \, \sum_{j=1 \atop j \neq j_{0}}^{\aleph} \int_{D_{j}} \Upsilon_{k}(x,y) \cdot E^{T}_j(y)\, dy = E^{Inc}_{j_{0}}(x),
	\end{equation}
	where $\Upsilon_{k}(\cdot,\cdot)$ is given by $(\ref{dyadicG})$.
	Considering the equation $(\ref{AI0})$, and successively, inverting the operator $\boldsymbol{T}_{k}$, multiplying on the both sides by $\mathcal{P}(x,z_{j_{0}})$ and integrating over $D_{j_{0}}$, then with the help of the adjoint operator of $\boldsymbol{T}_{k}$ and the definition of $W$, seeing $(\ref{AI1})$, we obtain that
	\begin{equation*}
	\int_{D_{j_{0}}} \mathcal{P}(x,z_{j_{0}})  \cdot E^{T}_{j_{0}}(x) \; dx - \eta \, \int_{D_{j_{0}}} W(x) \cdot \sum_{j=1 \atop j \neq j_{0}}^{\aleph} \int_{D_{j}} \Upsilon_{k}(x,y) \cdot E^{T}_j(y)\, dy \, dx = \int_{D_{j_{0}}} W(x) \cdot E^{Inc}_{j_{0}}(x) \, dx.
	\end{equation*}
	For $j=1,\cdots,\aleph$, since from Proposition \ref{es-oneP}, we know that $\overset{2}{\mathbb{P}}\left( E^{T}_j \right) = 0$, we write $E^{T}_j$  as 
	$E^{T}_j  = \overset{1}{\mathbb{P}}\left( E^{T}_j \right) + \overset{3}{\mathbb{P}}\left( E^{T}_j \right)$ 
	and we plug it into the previous equation to get  
	\begin{eqnarray}\label{IA2}
	\nonumber
	\int_{D_{j_{0}}} \mathcal{P}(x,z_{j_{0}})  \cdot \overset{1}{\mathbb{P}}\left( E^{T}_{j_{0}} \right)(x) \; dx &-& \eta \, \int_{D_{j_{0}}} W(x) \cdot \sum_{j=1 \atop j \neq j_{0}}^{\aleph} \int_{D_{j}} \Upsilon_{k}(x,y) \cdot \overset{1}{\mathbb{P}}\left( E^{T}_j\right)(y)\, dy \, dx \\ \nonumber 
	& = & \int_{D_{j_{0}}} W(x) \cdot E^{Inc}_{j_{0}}(x) \, dx - \int_{D_{j_{0}}} \mathcal{P}(x,z_{j_{0}})  \cdot \overset{3}{\mathbb{P}}\left(E^{T}_{j_{0}}\right)(x) \; dx \\ &+& \eta \, \int_{D_{j_{0}}} W(x) \cdot \sum_{j=1 \atop j \neq j_{0}}^{\aleph} \int_{D_{j}} \Upsilon_{k}(x,y) \cdot \overset{3}{\mathbb{P}}\left(E^{T}_j\right)(y)\, dy \, dx.
	\end{eqnarray}
	Next, we estimate the last two terms on the right hand side of \eqref{IA2}. To do this, we have  
	\begin{eqnarray*}
		J_{1} &:=& \int_{D_{j_{0}}} \mathcal{P}(x,z_{j_{0}})  \cdot \overset{3}{\mathbb{P}}\left(E^{T}_{j_{0}}\right)(x) \; dx \\
		J_{1} &=& a^{4} \, \int_{B} \mathcal{P}(x,0)  \cdot \overset{3}{\mathbb{P}}\left(\tilde{E}^{T}_{j_{0}}\right)(x) \; dx \\
		\left\vert J_{1} \right\vert & \leq & a^{4} \; \left\Vert \mathcal{P}(\cdot,0) \right\Vert_{\mathbb{L}^{2}(B)} \; \left\Vert \overset{3}{\mathbb{P}}\left(\tilde{E}^{T}_{j_{0}}\right) \right\Vert_{\mathbb{L}^{2}(B)} \lesssim a^{4}  \; \left\Vert \overset{3}{\mathbb{P}}\left(\tilde{E}^{T}_{j_{0}}\right) \right\Vert_{\mathbb{L}^{2}(B)}.
	\end{eqnarray*}
	Similarly, we set 
	\begin{equation*}
	J_{2} := \eta \,\sum_{j=1 \atop j \neq j_{0}}^{\aleph} \, \int_{D_{j_{0}}} W(x) \cdot  \int_{D_{j}} \Upsilon_{k}(x,y) \cdot \overset{3}{\mathbb{P}}\left(E^{T}_j\right)(y)\, dy \, dx.
	\end{equation*}
	By expanding the fundamental solution $\Upsilon_{k}(\cdot,\cdot)$ near the center $z_{j}$, we obtain 
	\begin{eqnarray}\label{Expansion-Upsilon}
	\nonumber
	\Upsilon_{k}(x,y) = \Upsilon_{k}(z_{j_{0}},z_{j}) &+& \int_{0}^{1} \underset{x}{\nabla} \left(\Upsilon_{k} \right)(z_{j_{0}}+t(x-z_{j_{0}}),z_{j}) \cdot \mathcal{P}(x,z_{j_{0}}) \; dt \\
	&+& \int_{0}^{1} \underset{y}{\nabla} \left(\Upsilon_{k} \right)(x,z_{j}+t(y-z_{j})) \cdot \mathcal{P}(y,z_{j}) \; dt.
	\end{eqnarray}
	Then $J_{2}$ becomes
	\begin{eqnarray*}
		J_{2} &:=& \eta \,\sum_{j=1 \atop j \neq j_{0}}^{\aleph} \, \int_{D_{j_{0}}} W(x) \, dx \cdot \Upsilon_{k}(z_{j_{0}},z_{j}) \cdot \int_{D_{j}}  \overset{3}{\mathbb{P}}\left(E^{T}_j\right)(y)\, dy \\
		&+& \eta \,\sum_{j=1 \atop j \neq j_{0}}^{\aleph} \, \int_{D_{j_{0}}} W(x) \cdot \int_{0}^{1} \underset{x}{\nabla} \left(\Upsilon_{k} \right)(z_{j_{0}}+t(x-z_{j_{0}}),z_{j}) \cdot \mathcal{P}(x,z_{j_{0}}) \; dt \, dx \cdot \int_{D_{j}}  \overset{3}{\mathbb{P}}\left(E^{T}_j\right)(y)\, dy \\
		&+& \eta \,\sum_{j=1 \atop j \neq j_{0}}^{\aleph} \, \int_{D_{j_{0}}} W(x) \cdot  \int_{D_{j}} \int_{0}^{1} \underset{y}{\nabla} \left(\Upsilon_{k} \right)(x,z_{j}+t(y-z_{j})) \cdot \mathcal{P}(y,z_{j}) \; dt \cdot \overset{3}{\mathbb{P}}\left(E^{T}_j\right)(y)\, dy \, dx.
	\end{eqnarray*}
	Keeping only the first dominant term, by using the estimate for the scattering coefficient \eqref{*add0} in Proposition \ref{prop-scoeff}, we deduce that
	\begin{eqnarray*}
		J_{2} & \sim & a^{4-h} \ \,\sum_{j=1 \atop j \neq j_{0}}^{\aleph}  \Upsilon_{k}(z_{j_{0}},z_{j}) \cdot \int_{D_{j}}  \overset{3}{\mathbb{P}}\left(E^{T}_j\right)(y)\, dy \\ 
		\left\vert J_{2} \right\vert & \lesssim & a^{\frac{11}{2}-h} \ \,\sum_{j=1 \atop j \neq j_{0}}^{\aleph}  \frac{1}{\left\vert z_{j_{0}} - z_{j} \right\vert^{3} } \left\Vert \overset{3}{\mathbb{P}}\left(E^{T}_j\right) \right\Vert_{\mathbb{L}^{2}\left(D_{j}\right)} \lesssim a^{7-h} \ d^{-3} \, \left\vert \log (d)\right\vert \, \max_{j}  \left\Vert \overset{3}{\mathbb{P}}\left(\tilde{E}^{T}_j\right) \right\Vert_{\mathbb{L}^{2}\left(B\right)}.  
	\end{eqnarray*}
	Thus,  combining with the estimations of $J_{1}$ and $J_{2}$, we rewrite $(\ref{IA2})$ as 
	\begin{eqnarray*}
		\nonumber
		\int_{D_{j_{0}}} \mathcal{P}(x,z_{j_{0}})  \cdot \overset{1}{\mathbb{P}}\left( E^{T}_{j_{0}} \right)(x) \; dx &-& \eta \, \int_{D_{j_{0}}} W(x) \cdot \sum_{j=1 \atop j \neq j_{0}}^{\aleph} \int_{D_{j}} \Upsilon_{k}(x,y) \cdot \overset{1}{\mathbb{P}}\left( E^{T}_j\right)(y)\, dy \, dx \\ \nonumber & = & \int_{D_{j_{0}}} W(x) \cdot E^{Inc}_{j_{0}}(x) \, dx + \mathcal{O}\left(a^{4} \, \left\Vert \overset{3}{\mathbb{P}}\left(\tilde{E}^{T}_{j_{0}}\right) \right\Vert_{\mathbb{L}^{2}\left(B\right)}\right) \\ &+& \mathcal{O}\left( a^{7-h} \ d^{-3} \, \left\vert \log (d)\right\vert \, \max_{j}  \left\Vert \overset{3}{\mathbb{P}}\left(\tilde{E}^{T}_j\right) \right\Vert_{\mathbb{L}^{2}\left(B\right)}\right). 
	\end{eqnarray*}
	Next, we write $W$ as $W = \overset{1}{\mathbb{P}}\left(W \right) + \overset{2}{\mathbb{P}}\left(W \right) + \overset{3}{\mathbb{P}}\left(W \right)$ and we plug it into the above equation to get 
	\begin{eqnarray}\label{Bf-1-2-3}
	\nonumber
	\int_{D_{j_{0}}} \mathcal{P}(x,z_{j_{0}})  \cdot \overset{1}{\mathbb{P}}\left( E^{T}_{j_{0}} \right)(x) \; dx &-& \eta \, \int_{D_{j_{0}}} \overset{1}{\mathbb{P}}\left(W \right)(x) \cdot \sum_{j=1 \atop j \neq j_{0}}^{\aleph} \int_{D_{j}} \Upsilon_{k}(x,y) \cdot \overset{1}{\mathbb{P}}\left( E^{T}_j\right)(y)\, dy \, dx \\ \nonumber
	&-& \eta \, \int_{D_{j_{0}}} \overset{2}{\mathbb{P}}\left(W \right)(x) \cdot \sum_{j=1 \atop j \neq j_{0}}^{\aleph} \int_{D_{j}} \Upsilon_{k}(x,y) \cdot \overset{1}{\mathbb{P}}\left( E^{T}_j\right)(y)\, dy \, dx \\ \nonumber
	&-& \eta \, \int_{D_{j_{0}}} \overset{3}{\mathbb{P}}\left(W \right)(x) \cdot \sum_{j=1 \atop j \neq j_{0}}^{\aleph} \int_{D_{j}} \Upsilon_{k}(x,y) \cdot \overset{1}{\mathbb{P}}\left( E^{T}_j\right)(y)\, dy \, dx \\ \nonumber & = & \int_{D_{j_{0}}} \overset{1}{\mathbb{P}}\left(W \right)(x) \cdot E^{Inc}_{j_{0}}(x) \, dx \\ \nonumber &+& \int_{D_{j_{0}}} \overset{2}{\mathbb{P}}\left(W \right)(x) \cdot E^{Inc}_{j_{0}}(x) \, dx  + \int_{D_{j_{0}}} \overset{3}{\mathbb{P}}\left(W \right)(x) \cdot E^{Inc}_{j_{0}}(x) \, dx \\  &+& \mathcal{O}\left(a^{4} \, \left\Vert \overset{3}{\mathbb{P}}\left(\tilde{E}^{T}_{j_{0}}\right) \right\Vert_{\mathbb{L}^{2}\left(B\right)}\right) + \mathcal{O}\left( a^{7-h} \ d^{-3} \, \left\vert \log (d)\right\vert \, \max_{j}  \left\Vert \overset{3}{\mathbb{P}}\left(\tilde{E}^{T}_j\right) \right\Vert_{\mathbb{L}^{2}\left(B\right)}\right). 
	\end{eqnarray}
	Since $E^{Inc}_{j_0} \in \mathbb{H}\left(\div = 0 \right) \perp \mathbb{H}_{0}\left(Curl = 0 \right)$ and, by construction, $\overset{2}{\mathbb{P}}\left(W \right) \in \mathbb{H}_{0}\left(Curl = 0 \right)$, we deduce that 
	\begin{equation}\label{J3-int=0}
	J_{3} := \int_{D_{j_{0}}} \overset{2}{\mathbb{P}}\left(W \right)(x) \cdot E^{Inc}_{j_{0}}(x) \, dx = 0.
	\end{equation}
Denote 
	\begin{equation}\label{J4=a6--}
		J_{4} := \int_{D_{j_{0}}} \overset{3}{\mathbb{P}}\left(W \right)(x) \cdot E^{Inc}_{j_{0}}(x) \, dx\notag
		\end{equation}
		then
		\begin{eqnarray}\label{J4=a6}
		\left\vert J_{4} \right\vert & \leq & \left\Vert \overset{3}{\mathbb{P}}\left(W \right) \right\Vert_{\mathbb{L}^{2}(D_{j_{0}})} \, \left\Vert  E^{Inc}_{j_{0}}\right\Vert_{\mathbb{L}^{2}(D_{j_{0}})} \lesssim a^{3} \;  \left\Vert \overset{3}{\mathbb{P}}\left(\tilde{W} \right) \right\Vert_{\mathbb{L}^{2}(B)}= \mathcal{O}\left( a^{6}\right),
	\end{eqnarray}
by utilizing \eqref{fin-W-3}.
	On the left hand side of \eqref{Bf-1-2-3}, we need to estimate 
	\begin{eqnarray*}
		J_{5,\ell} &:=&  \eta \, \int_{D_{j_{0}}} \overset{\ell}{\mathbb{P}}\left(W \right)(x) \cdot \sum_{j=1 \atop j \neq j_{0}}^{\aleph} \int_{D_{j}} \Upsilon_{k}(x,y) \cdot \overset{1}{\mathbb{P}}\left( E^{T}_j\right)(y)\, dy \, dx, \quad \ell = 2,3\\
		&=& - \eta \, \int_{D_{j_{0}}} \overset{\ell}{\mathbb{P}}\left(W \right)(x) \cdot \sum_{j=1 \atop j \neq j_{0}}^{\aleph} \underset{x}{\nabla} \int_{D_{j}} \underset{y}{\nabla} \Phi_{k}(x,y) \cdot \overset{1}{\mathbb{P}}\left( E^{T}_j\right)(y)\, dy \, dx\\
		&+&  \eta \, k^{2} \, \int_{D_{j_{0}}} \overset{\ell}{\mathbb{P}}\left(W \right)(x) \cdot \sum_{j=1 \atop j \neq j_{0}}^{\aleph} \int_{D_{j}} \Phi_{k}(x,y) \; \overset{1}{\mathbb{P}}\left( E^{T}_j\right)(y)\, dy \, dx.
	\end{eqnarray*}
	By using integration by parts and the fact that $\overset{1}{\mathbb{P}}\left( E^{T}_j\right) \in \mathbb{H}_{0}(\div = 0)$, we can prove that 
	\begin{equation}\label{AI9}
	\int_{D_{j}} \underset{y}{\nabla} \Phi_{k}(\cdot,y) \cdot \overset{1}{\mathbb{P}}\left( E^{T}_j\right)(y)\, dy = 0, \quad j=1,\cdots,\aleph.
	\end{equation}
	Then, 
	\begin{equation*}
	J_{5,\ell} =  \eta \, k^{2} \, \int_{D_{j_{0}}} \overset{\ell}{\mathbb{P}}\left(W \right)(x) \cdot \sum_{j=1 \atop j \neq j_{0}}^{\aleph} \int_{D_{j}} \Phi_{k}(x,y) \; \overset{1}{\mathbb{P}}\left( E^{T}_j\right)(y)\, dy \, dx.
	\end{equation*}  
	We rewrite $J_{5, \ell}$, by Taylor expansion and the fact that $\int_{D_{j}} \overset{1}{\mathbb{P}}\left( E^{T}_j\right)(y)\, dy = 0$, as  
	\begin{eqnarray*}
		J_{5,\ell} &=&  \eta \, k^{2} \, \int_{D_{j_{0}}} \overset{\ell}{\mathbb{P}}\left(W \right)(x) \cdot \sum_{j=1 \atop j \neq j_{0}}^{\aleph} \int_{D_{j}}  \underset{y}{\nabla}\Phi_{k}(x,z_{j})\cdot(y-z_{j})  \overset{1}{\mathbb{P}}\left( E^{T}_j\right)(y)\, dy \, dx\\
		&+&  \eta \, k^{2} \, \int_{D_{j_{0}}} \overset{\ell}{\mathbb{P}}\left(W \right)(x) \cdot \sum_{j=1 \atop j \neq j_{0}}^{\aleph} \int_{D_{j}} \int_{0}^{1} (1-t) (y-z_{j})^{\perp} \cdot \\ && \qquad \qquad \qquad \qquad \qquad \qquad \underset{y}{Hess}(\Phi_{k})(z_{j_{0}},z_{j}+t(y-z_{j}))\cdot (y-z_{j}) \, dt  \overset{1}{\mathbb{P}}\left( E^{T}_j\right)(y)\, dy \, dx \\
		&+&  \eta \, k^{2} \, \int_{D_{j_{0}}} \overset{\ell}{\mathbb{P}}\left(W \right)(x) \cdot \sum_{j=1 \atop j \neq j_{0}}^{\aleph} \int_{D_{j}} \int_{0}^{1} (1-t) (y-z_{j})^{\perp} \cdot \\ && \qquad \left[\int_{0}^{1} \underset{x}{\nabla}\left(\underset{y}{Hess}(\Phi_{k})\right)(z_{j_{0}}+s(x-z_{j_{0}}),z_{j}+t(y-z_{j})) \cdot \mathcal{P}(x,z_{j_{0}}) ds \right] \cdot (y-z_{j}) \, dt  \overset{1}{\mathbb{P}}\left( E^{T}_j\right)(y)\, dy \, dx. 
	\end{eqnarray*}
	Technically, the analyses for the case $\ell=2$ and the case $\ell=3$ are a bit different. So, we split the analyses into two parts. 
	\begin{enumerate}
		\item[*/] Case $\ell=2$, \\ 
		Recall that we have  $\int_{D_{j_{0}}} \overset{2}{\mathbb{P}}\left(W \right)(x) dx = 0$ and remark that the second term of $J_{5,2}$ is of the form $\int_{D_{j_{0}}} \overset{2}{\mathbb{P}}\left(W \right)(x) \cdot V dx$, where $V$ is a constant vector. Consequently, the second term is vanishing. Then, 
		\begin{eqnarray*}
			J_{5,2} &=&  \eta \, k^{2} \, \int_{D_{j_{0}}} \overset{2}{\mathbb{P}}\left(W \right)(x) \cdot \sum_{j=1 \atop j \neq j_{0}}^{\aleph} \int_{D_{j}}  \underset{y}{\nabla}\Phi_{k}(x,z_{j})\cdot(y-z_{j})  \overset{1}{\mathbb{P}}\left( E^{T}_j\right)(y)\, dy \, dx\\
			&+&  \eta \, k^{2} \, \int_{D_{j_{0}}} \overset{2}{\mathbb{P}}\left(W \right)(x) \cdot \sum_{j=1 \atop j \neq j_{0}}^{\aleph} \int_{D_{j}} \int_{0}^{1} (1-t) (y-z_{j})^{\perp} \cdot \\ && \qquad \left[\int_{0}^{1} \underset{x}{\nabla}\left(\underset{y}{Hess}(\Phi_{k})\right)(z_{j_{0}}+s(x-z_{j_{0}}),z_{j}+t(y-z_{j})) \cdot \mathcal{P}(x,z_{j_{0}}) ds \right] \cdot (y-z_{j}) \, dt  \overset{1}{\mathbb{P}}\left( E^{T}_j\right)(y)\, dy \, dx. 
		\end{eqnarray*}
		Clearly, for the above equation, the first term is more dominant compared with the second one. Using Taylor expansion near $z_{j_{0}}$, we obtain 
		\begin{eqnarray*}
			J_{5,2} &=& - \eta \, k^{2} \, \int_{D_{j_{0}}} \overset{2}{\mathbb{P}}\left(W \right)(x) \cdot \sum_{j=1 \atop j \neq j_{0}}^{\aleph} \int_{D_{j}}  \left[ \int_{0}^{1} \underset{y}{Hess}\left( \Phi_{k} \right)(z_{j_{0}}+t(x-z_{j_{0}}),z_{j}) \cdot (x-z_{j_{0}}) dt \right] \cdot(y-z_{j}) \\ && \qquad\qquad\qquad\qquad\qquad\qquad \qquad\qquad\qquad \overset{1}{\mathbb{P}}\left( E^{T}_j\right)(y)\, dy \, dx\\
			&+&  \eta \, k^{2} \, \int_{D_{j_{0}}} \overset{2}{\mathbb{P}}\left(W \right)(x) \cdot \sum_{j=1 \atop j \neq j_{0}}^{\aleph} \int_{D_{j}} \int_{0}^{1} (1-t) (y-z_{j})^{\perp} \cdot \\ && \qquad \left[\int_{0}^{1} \underset{x}{\nabla}\left(\underset{y}{Hess}(\Phi_{k})\right)(z_{j_{0}}+s(x-z_{j_{0}}),z_{j}+t(y-z_{j})) \cdot \mathcal{P}(x,z_{j_{0}}) ds \right] \cdot (y-z_{j}) \, dt  \overset{1}{\mathbb{P}}\left( E^{T}_j\right)(y)\, dy \, dx \\
			\left\vert J_{5,2} \right\vert & \lesssim & a^{3} \left\Vert \overset{2}{\mathbb{P}}\left(W \right) \right\Vert_{\mathbb{L}^{2}(D_{j_{0}})} \sum_{j=1 \atop j \neq j_{0}}^{\aleph} \left\vert z_{j} - z_{j_{0}} \right\vert^{-3} \; \left\Vert \overset{1}{\mathbb{P}}\left( E^{T}_j\right) \right\Vert_{\mathbb{L}^{2}(D_{j})} + a^{4} \left\Vert \overset{2}{\mathbb{P}}\left(W \right) \right\Vert_{\mathbb{L}^{2}(D_{j_{0}})} \sum_{j=1 \atop j \neq j_{0}}^{\aleph} \frac{\left\Vert \overset{1}{\mathbb{P}}\left( E^{T}_j\right) \right\Vert_{\mathbb{L}^{2}(D_{j})}}{\left\vert z_{j} - z_{j_{0}} \right\vert^{4}}  \\ 
			& \lesssim & a^{6} \, \left\Vert \overset{2}{\mathbb{P}}\left(\tilde{W} \right) \right\Vert_{\mathbb{L}^{2}(B)} \max_{j} \left\Vert \overset{1}{\mathbb{P}}\left( \tilde{E}^{T}_j\right) \right\Vert_{\mathbb{L}^{2}(B)} d^{-3} \left\vert \log (d)\right\vert + a^7 \left\Vert \overset{2}{\mathbb{P}}\left(\tilde{W} \right) \right\Vert_{\mathbb{L}^{2}(B)} \max_{j} \left\Vert \overset{1}{\mathbb{P}}\left( \tilde{E}^{T}_j\right) \right\Vert_{\mathbb{L}^{2}(B)} d^{-4}\\ 
			& \lesssim & a^{6} \left\Vert \overset{2}{\mathbb{P}}\left(\tilde{W} \right) \right\Vert_{\mathbb{L}^{2}(B)} \max_{j} \left\Vert \overset{1}{\mathbb{P}}\left( \tilde{E}^{T}_j\right) \right\Vert_{\mathbb{L}^{2}(B)} d^{-4}. 
		\end{eqnarray*}
		Knowing the estimation of $\left\Vert \overset{2}{\mathbb{P}}\left( \tilde{W} \right) \right\Vert_{\mathbb{L}^{2}}$, given by $(\ref{fin-W-2})$, we deduce that  
		\begin{equation}\label{J52}
		J_{5,2} = \mathcal{O}\left( a^9 \, d^{-4} \, \max_{j} \left\Vert \overset{1}{\mathbb{P}}\left( \tilde{E}^{T}_j\right) \right\Vert_{\mathbb{L}^{2}(B)} \right).
		\end{equation}  
		\item[*/] Case $\ell=3$, \\
		We have, 
		\begin{eqnarray*}
			J_{5,3} &=&  \eta \, k^{2} \, \int_{D_{j_{0}}} \overset{3}{\mathbb{P}}\left(W \right)(x) \cdot \sum_{j=1 \atop j \neq j_{0}}^{\aleph} \int_{D_{j}}  \underset{y}{\nabla}\Phi_{k}(x,z_{j})\cdot(y-z_{j})  \overset{1}{\mathbb{P}}\left( E^{T}_j\right)(y)\, dy \, dx\\
			&+&  \eta \, k^{2} \, \int_{D_{j_{0}}} \overset{3}{\mathbb{P}}\left(W \right)(x) \cdot \sum_{j=1 \atop j \neq j_{0}}^{\aleph} \int_{D_{j}} \int_{0}^{1} (1-t) (y-z_{j})^{\perp} \cdot \underset{y}{Hess}(\Phi_{k})(z_{j_{0}},z_{j}+t(y-z_{j}))\cdot (y-z_{j}) \, dt  \\ && \qquad\qquad\qquad\qquad\qquad\qquad \, \qquad \overset{1}{\mathbb{P}}\left( E^{T}_j\right)(y)\, dy \, dx \\
			&+&  \eta \, k^{2} \, \int_{D_{j_{0}}} \overset{3}{\mathbb{P}}\left(W \right)(x) \cdot \sum_{j=1 \atop j \neq j_{0}}^{\aleph} \int_{D_{j}} \int_{0}^{1} (1-t) (y-z_{j})^{\perp} \cdot \\ && \qquad \left[\int_{0}^{1} \underset{x}{\nabla}\left(\underset{y}{Hess}(\Phi_{k})\right)(z_{j_{0}}+s(x-z_{j_{0}}),z_{j}+t(y-z_{j})) \cdot \mathcal{P}(x,z_{j_{0}}) ds \right] \cdot (y-z_{j}) \, dt  \overset{1}{\mathbb{P}}\left( E^{T}_j\right)(y)\, dy \, dx. 
		\end{eqnarray*}
		Since $\int_{D_{j_{0}}} \overset{3}{\mathbb{P}}\left(W \right)(x) dx \neq 0$, after expanding the first term near $z_{j_{0}}$ and keeping the dominant term of $J_{5,3}$, we deduce that  
		\begin{eqnarray*}
			J_{5,3} & \simeq & \eta \, k^{2} \, \int_{D_{j_{0}}} \overset{3}{\mathbb{P}}\left(W \right)(x) dx \cdot \sum_{j=1 \atop j \neq j_{0}}^{\aleph} \underset{y}{\nabla}\left( \Phi_{k} I \right) (z_{j_{0}},z_{j})\cdot \int_{D_{j}} \mathcal{P}(y,z_{j}) \cdot \overset{1}{\mathbb{P}}\left( E^{T}_j\right)(y)\, dy \\ 
			\left\vert J_{5,3} \right\vert & \lesssim & \left\vert \eta \right\vert \, \left\Vert 1 \right\Vert_{\mathbb{L}^{2}(D_{j_{0}})} \, \left\Vert \overset{3}{\mathbb{P}}\left(W \right) \right\Vert_{\mathbb{L}^{2}(D_{j_{0}})} \sum_{j=1 \atop j \neq j_{0}}^{\aleph} \left\vert z_{j_{0}} - z_{j} \right\vert^{-2}  \,\left\Vert \mathcal{P}(0,z_{j}) \right\Vert_{\mathbb{L}^{2}(D_{j})} \, \left\Vert \overset{1}{\mathbb{P}}\left( E^{T}_j\right) \right\Vert_{\mathbb{L}^{2}(D_{j})} \\ 
			& \lesssim & a^{2}\cdot a^3 \cdot d^{-3} \; \left\Vert \overset{3}{\mathbb{P}}\left(\tilde{W} \right) \right\Vert_{\mathbb{L}^{2}(B)} \max_{j} \left\Vert \overset{1}{\mathbb{P}}\left( \tilde{E}^{T}_j\right) \right\Vert_{\mathbb{L}^{2}(B)}.
		\end{eqnarray*} 
		With the estimation of $\left\Vert \overset{3}{\mathbb{P}}\left(\tilde{W} \right) \right\Vert_{\mathbb{L}^{2}}$, given by $(\ref{fin-W-3})$, we can deduce that
		\begin{equation}\label{J53}
		J_{5,3} = \mathcal{O}\left( a^{8} \; d^{-3} \max_{j}   \left\Vert \overset{1}{\mathbb{P}}\left( \tilde{E}^{T}_j\right) \right\Vert_{\mathbb{L}^{2}(B)}   \right).
		\end{equation}
	\end{enumerate}
	Finally, combining with  $(\ref{J52})$ and $(\ref{J53})$, we obtain that
	\begin{equation}\label{J5ell}
	J_{5,\ell} = \mathcal{O}\left( a^{8} \; d^{-3} \max_{j}   \left\Vert \overset{1}{\mathbb{P}}\left( \tilde{E}^{T}_j\right) \right\Vert_{\mathbb{L}^{2}(B)}   \right).
	\end{equation}
	The equation $(\ref{Bf-1-2-3})$, using $(\ref{J3-int=0}), (\ref{J4=a6})$ and $(\ref{J5ell})$, will be reduced to   
	\begin{eqnarray}\label{*add7}
		\int_{D_{j_{0}}} \mathcal{P}(x,z_{j_{0}})  \cdot \overset{1}{\mathbb{P}}\left( E^{T}_{j_{0}} \right)(x) \; dx &-& \eta \, \int_{D_{j_{0}}} \overset{1}{\mathbb{P}}\left(W \right)(x) \cdot \sum_{j=1 \atop j \neq j_{0}}^{\aleph} \int_{D_{j}} \Upsilon_{k}(x,y) \cdot \overset{1}{\mathbb{P}}\left( E^{T}_j\right)(y)\, dy \, dx\notag \\ 
		& = & \int_{D_{j_{0}}} \overset{1}{\mathbb{P}}\left(W \right)(x) \cdot E^{Inc}_{j_{0}}(x) \, dx  + \mathcal{O}\left(a^{6}\right) + \mathcal{O}\left(a^{4} \, \left\Vert \overset{3}{\mathbb{P}}\left(\tilde{E}^{T}_{j_{0}}\right) \right\Vert_{\mathbb{L}^{2}\left(B\right)}\right) \notag\\ \nonumber &+& \mathcal{O} \left( a^{7-h} \ d^{-3} \, \left\vert \log (d)\right\vert \, \max_j \left\Vert \overset{3}{\mathbb{P}}\left(\tilde{E}^{T}_j\right) \right\Vert_{\mathbb{L}^{2}\left(B\right)} \right) \\ &+& \mathcal{O}\left( a^{8} \; d^{-3} \max_{j}   \left\Vert \overset{1}{\mathbb{P}}\left( \tilde{E}^{T}_j\right) \right\Vert_{\mathbb{L}^{2}(B)}   \right).
	\end{eqnarray}
	Since $\overset{1}{\mathbb{P}}\left(E^T_j\right), \overset{1}{\mathbb{P}}\left(W\right)\in \mathbb{H}_0(\div=0)$, we denote
	\begin{equation}\label{p1=curlA}
	\overset{1}{\mathbb{P}}\left( E^{T}_j\right) = Curl\left( F_{j} \right) \quad \text{and} \quad \overset{1}{\mathbb{P}}\left( W \right) = Curl\left( \mathcal{A} \right),
	\end{equation}
	with 
	\begin{equation}\label{nu*f=0}
	\nu \times \mathcal{A} = 0, \; \div (\mathcal{A}) = 0, \; \nu \times F_{j} = 0, \quad \text{and} \quad \div(F_{j})=0, \quad j=1,\cdots,\aleph.  
	\end{equation}
	Plugging all these formulas into \eqref{*add7}, using $(\ref{AI9})$ and integration by parts,\footnote{We use, for 2 arbitrary vectors, the following relation $
		Curl(a) \cdot b = Curl(b) \cdot a + \div(a \times b)$
	} we can obtain that  
	\begin{eqnarray}\label{bm-matrix-M}
	\nonumber
	\bm{\mathcal{M}} \cdot \int_{D_{j_{0}}} F_{j_{0}}(x) \; dx &-& \eta \, k^{2} \, \sum_{j=1 \atop j \neq j_{0}}^{\aleph} \int_{D_{j_{0}}}  \mathcal{A}(x) \cdot \underset{x}{Curl} \int_{D_{j}} \Phi_{k}(x,y) \; \underset{y}{Curl}\left( F_{j} \right)(y)\, dy \, dx \\ \nonumber
	& = & \int_{D_{j_{0}}}  \mathcal{A}(x) \cdot Curl\left( E^{Inc}_{j_{0}}\right)(x) \, dx  + \mathcal{O}\left(a^{6}\right) + \mathcal{O}\left(a^{4} \, \left\Vert \overset{3}{\mathbb{P}}\left(\tilde{E}^{T}_{j_{0}}\right) \right\Vert_{\mathbb{L}^{2}\left(B\right)}\right) \\ &+& \mathcal{O} \left( a^{7-h} \, d^{-3} \, \left\vert \log (d)\right\vert \, \max_j \left\Vert \overset{3}{\mathbb{P}}\left(\tilde{E}^{T}_j\right) \right\Vert_{\mathbb{L}^{2}\left(B\right)} \right) + \mathcal{O}\left( a^{8} \; d^{-3} \max_{j}   \left\Vert \overset{1}{\mathbb{P}}\left( \tilde{E}^{T}_j\right) \right\Vert_{\mathbb{L}^{2}(B)}   \right),
	\end{eqnarray} 
	where $\bm{\mathcal{M}}$ is the constant matrix given by 
	\begin{equation}\label{def-bm-mathcal-M}
	\bm{\mathcal{M}} := \underset{x}{Curl}\left( \mathcal{P}(x,z_{j_{0}}) \right) = \begin{pmatrix}
	0 & 0 & 0 \\ 
	0 & 0 & -1 \\ 
	0 & 1 & 0 \\ 
	0 & 0 & 1 \\ 
	0 & 0 & 0 \\ 
	1 & 0 & 0 \\ 
	0 & -1 & 0 \\ 
	1 & 0 & 0 \\ 
	0 & 0 & 0
	\end{pmatrix}. 
	\end{equation}
	For the second term on the left hand side of \eqref{bm-matrix-M}, we have\footnote{For any vector field $a(\cdot)$, we have the following formula
		\begin{equation}\label{Curlx=Curly}
		\underset{x}{Curl}\left( \Phi_{k}(x,y) \, a(y) \right) + \underset{y}{Curl}\left( \Phi_{k}(x,y) \, a(y) \right) = \Phi_{k}(x,y) \, \underset{y}{Curl}\left(a(y)\right).
		\end{equation}
	} 
	\begin{eqnarray}\label{interactionsterms}
	\nonumber
	\underset{x}{Curl} \int_{D_{j}} \Phi_{k}(x,y) \; \underset{y}{Curl}\left( F_{j} \right)(y)\, dy & \overset{(\ref{Curlx=Curly})}{=} & \underset{x}{Curl} \int_{D_{j}} \underset{x}{Curl}\left(\Phi_{k}(x,y) \,  F_{j} (y)\right) \, dy \\ \nonumber &+& \underset{x}{Curl} \int_{\partial D_{j}}  \Phi_{k}(x,y) \nu(y) \times  F_{j}(y) \, dy \\ \nonumber
	& \overset{(\ref{nu*f=0})}{=} &  \int_{D_{j}} \underset{x}{Curl} \circ \underset{x}{Curl}\left(\Phi_{k}(x,y) \,  F_{j} (y)\right) \, dy \\ \nonumber
	& = &  \int_{D_{j}} \left( - \underset{x}{\Delta} + \underset{x}{\nabla} \underset{x}{\div} \right) \left(\Phi_{k}(x,y) \,  F_{j} (y)\right) \, dy\\ \nonumber
	& = &  \int_{D_{j}} \left( k^{2} \, \Phi_{k}(x,y) \,  F_{j} (y) +  \underset{x}{Hess} \Phi_{k}(x,y) \cdot  F_{j} (y) \right)\, dy\\
	& \overset{(\ref{dyadicG})}{=} & \int_{D_{j}} \Upsilon_{k}(x,y) \cdot  F_{j}(y)\, dy.
	\end{eqnarray}
	For the right hand side of \eqref{bm-matrix-M}, we have 
	\begin{equation}\label{right-hand-side-magnetic}
	\int_{D_{j_{0}}}  \mathcal{A}(x) \cdot Curl\left( E^{Inc}_{j_{0}}\right)(x) \, dx = i \, k \, \int_{D_{j_{0}}}  \mathcal{A}(x) \cdot H^{Inc}_{j_{0}}(x) \, dx. 
	\end{equation}
	Then, from $(\ref{interactionsterms})$ and $(\ref{right-hand-side-magnetic})$, the equation $(\ref{bm-matrix-M})$ becomes 
	\begin{eqnarray*}
		\nonumber
		\bm{\mathcal{M}} \cdot \int_{D_{j_{0}}} F_{j_{0}}(x) \; dx &-& \eta \, k^{2} \, \sum_{j=1 \atop j \neq j_{0}}^{\aleph} \int_{D_{j_{0}}}  \mathcal{A}(x) \cdot \int_{D_{j}} \Upsilon_{k}(x,y) \cdot F_{j}(y)\, dy \, dx \\ \nonumber
		& = & i \, k \, \int_{D_{j_{0}}}  \mathcal{A}(x) \cdot  H^{Inc}_{j_{0}}(x) \, dx  + \mathcal{O}\left(a^{6}\right) + \mathcal{O}\left(a^{4} \, \left\Vert \overset{3}{\mathbb{P}}\left(\tilde{E}^{T}_{j_{0}}\right) \right\Vert_{\mathbb{L}^{2}\left(B\right)}\right) \\ &+& \mathcal{O} \left( a^{7-h} \, d^{-3} \, \left\vert \log (d)\right\vert \, \max_j \left\Vert \overset{3}{\mathbb{P}}\left(\tilde{E}^{T}_j\right) \right\Vert_{\mathbb{L}^{2}\left(B\right)} \right) + \mathcal{O}\left( a^{8} \; d^{-3} \max_{j}   \left\Vert \overset{1}{\mathbb{P}}\left( \tilde{E}^{T}_j\right) \right\Vert_{\mathbb{L}^{2}(B)}   \right).
	\end{eqnarray*} 
	Next, expanding $H^{Inc}_{j_{0}}(\cdot)$ near $z_{j_{0}}$ and  $\Upsilon_{k}(\cdot,\cdot)$ near the center $(z_{j_0}, z_j)$, similar to $(\ref{Expansion-Upsilon})$, we obtain
	\begin{eqnarray}\label{SIM}
	\nonumber
	\bm{\mathcal{M}} \cdot \int_{D_{j_{0}}} F_{j_{0}}(x) \; dx &-& \eta \, k^{2} \, \sum_{j=1 \atop j \neq j_{0}}^{\aleph} \int_{D_{j_{0}}}  \mathcal{A}(x) \, dx \cdot  \Upsilon_{k}(z_{j_{0}},z_{j})\cdot \int_{D_{j}} F_{j}(y)\, dy  \\ \nonumber
	& = & i \, k \, \int_{D_{j_{0}}}  \mathcal{A}(x) \, dx \cdot  H^{Inc}_{j_{0}}(z_{j_{0}}) \\ \nonumber
	&+& i \, k \, \int_{D_{j_{0}}}  \mathcal{A}(x) \cdot \int_{0}^{1} \underset{x}{\nabla}\left( H^{Inc}_{j_{0}}\right)(z_{j_{0}}+t(x-z_{j_{0}})) \cdot (x-z_{j_{0}}) \, dt \, dx \\ \nonumber 
	&+& \eta \, k^{2} \, \sum_{j=1 \atop j \neq j_{0}}^{\aleph} \int_{D_{j_{0}}}  \mathcal{A}(x) \cdot  \int_{0}^{1} \underset{x}{\nabla} \left(\Upsilon_{k} \right)(z_{j_{0}}+t(x-z_{j_{0}}),z_{j}) \cdot \mathcal{P}(x,z_{j_{0}}) \; dt \, dx \cdot \int_{D_{j}} F_{j}(y)\, dy  \\ \nonumber
	&+& \eta \, k^{2} \, \sum_{j=1 \atop j \neq j_{0}}^{\aleph} \int_{D_{j_{0}}}  \mathcal{A}(x) \cdot \int_{D_{j}}  \int_{0}^{1} \underset{y}{\nabla} \left(\Upsilon_{k} \right)(x,z_{j}+t(y-z_{j})) \cdot \mathcal{P}(y,z_{j}) \; dt \cdot F_{j}(y)\, dy \, dx \\ \nonumber
	&+& \mathcal{O}\left(a^{6}\right) + \mathcal{O}\left(a^{4} \, \left\Vert \overset{3}{\mathbb{P}}\left(\tilde{E}^{T}_{j_{0}}\right) \right\Vert_{\mathbb{L}^{2}\left(B\right)}\right) \\ 
	&+& \mathcal{O} \left( a^{7-h} \, d^{-3} \, \left\vert \log (d)\right\vert \, \max_j \left\Vert \overset{3}{\mathbb{P}}\left(\tilde{E}^{T}_j\right) \right\Vert_{\mathbb{L}^{2}\left(B\right)} \right) + \mathcal{O}\left( a^{8} \; d^{-3} \max_{j}   \left\Vert \overset{1}{\mathbb{P}}\left( \tilde{E}^{T}_j\right) \right\Vert_{\mathbb{L}^{2}(B)}   \right).
	\end{eqnarray} 
	Now we estimate the error terms appearing on the right hand side of \eqref{SIM} as follows. Firstly, set 
	\begin{eqnarray*}
		J_{6} &:=& i \, k \, \int_{D_{j_{0}}}  \mathcal{A}(x) \cdot \int_{0}^{1} \underset{x}{\nabla}\left( H^{Inc}_{j_{0}}\right)(z_{j_{0}}+t(x-z_{j_{0}})) \cdot (x-z_{j_{0}}) \, dt \, dx, \\
		J_{6} &=& - k^{2} \, \int_{D_{j_{0}}}  \mathcal{A}(x) \cdot \int_{0}^{1} e^{i \, k \, \theta \cdot (z_{j_{0}}+t(x-z_{j_{0}}))} \left( \theta^{\perp} \times \theta \right) \langle \theta , (x-z_{j_{0}}) \rangle \, dt \, dx, \\
		\left\vert J_{6} \right\vert & \lesssim &  \left\Vert  \mathcal{A} \right\Vert_{\mathbb{L}^{2}\left({D_{j_{0}}}\right)} \;  \left\Vert  \cdot - z_{j_{0}} \right\Vert_{\mathbb{L}^{2}\left({D_{j_{0}}}\right)} \leq a^{4} \; \left\Vert \tilde{\mathcal{A}} \right\Vert_{\mathbb{L}^{2}\left(B \right)}.
	\end{eqnarray*}
	Using the variant of Friedrich inequality on the subspace $\mathbb{H}_{0}(Curl) \cap \mathbb{H}(div=0)$, we have
	\begin{equation*}
	\left\vert J_{6} \right\vert  \lesssim   a^{4} \; \left\Vert  Curl\left( \tilde{\mathcal{A}} \right) \right\Vert_{\mathbb{L}^{2}\left(B \right)} = a^{5} \left\Vert \widetilde{Curl\left( \mathcal{A} \right)} \right\Vert_{\mathbb{L}^{2}\left(B \right)} \overset{( \ref{p1=curlA})}{=} a^{5} \left\Vert \overset{1}{\mathbb{P}}\left(\tilde{W} \right) \right\Vert_{\mathbb{L}^{2}\left(B \right)} \overset{(\ref{fin-W-1})}{=} \mathcal{O}\left(a^{6-h}\right).
	\end{equation*}
	Then,  
	\begin{equation}\label{J6-Estimation}
	J_{6} = \mathcal{O}\left(a^{6-h}\right).
	\end{equation}
	Secondly, set
	\begin{eqnarray}\label{Yas-Ra}
	\nonumber
	J_{7} &:=& \eta \, k^{2} \, \sum_{j=1 \atop j \neq j_{0}}^{\aleph} \int_{D_{j_{0}}}  \mathcal{A}(x) \cdot  \int_{0}^{1} \underset{x}{\nabla} \left(\Upsilon_{k} \right)(z_{j_{0}}+t(x-z_{j_{0}}),z_{j}) \cdot \mathcal{P}(x,z_{j_{0}}) \; dt \, dx \cdot \int_{D_{j}} F_{j}(y)\, dy \\ \nonumber
	\left\vert J_{7} \right\vert 
	& \lesssim & a^{-\frac{1}{2}} \, \sum_{j=1 \atop j \neq j_{0}}^{\aleph} \left\Vert   \mathcal{A} \right\Vert_{\mathbb{L}^{2}(D_{j_{0}})} \left\Vert \int_{0}^{1} \underset{x}{\nabla} \left(\Upsilon_{k} \right)(z_{j_{0}}+t(\cdot - z_{j_{0}}),z_{j}) \cdot \mathcal{P}(\cdot,z_{j_{0}}) \; dt \, \right\Vert_{\mathbb{L}^{2}(D_{j_{0}})} \,\left\Vert F_{j} \right\Vert_{\mathbb{L}^{2}(D_{j})} \\
	& \lesssim & a^{2} \, \sum_{j=1 \atop j \neq j_{0}}^{\aleph} \left\Vert   \mathcal{A} \right\Vert_{\mathbb{L}^{2}(D_{j_{0}})} \left\vert z_{j} - z_{j_{0}} \right\vert^{-4}  \,\left\Vert F_{j} \right\Vert_{\mathbb{L}^{2}(D_{j})} \lesssim a^2 \cdot a^3 \sum_{j=1 \atop j \neq j_{0}}^{\aleph} \left\Vert   \tilde{\mathcal{A}} \right\Vert_{\mathbb{L}^{2}(B)} \left\vert z_{j} - z_{j_{0}} \right\vert^{-4}  \,\left\Vert \tilde{F}_{j} \right\Vert_{\mathbb{L}^{2}(B)} \\ \nonumber 
	& \leq &  a^5 \, d^{-4} \, \left\Vert   \tilde{\mathcal{A}} \right\Vert_{\mathbb{L}^{2}(B)} \, \max_{j} \,\left\Vert \tilde{F}_{j} \right\Vert_{\mathbb{L}^{2}(B)} \overset{(\ref{J6-Estimation})}{\lesssim}  a^{7-h} \, d^{-4} \, \max_{j} \,\left\Vert \tilde{F}_{j} \right\Vert_{\mathbb{L}^{2}(B)}.
	\end{eqnarray}
By applying the Friedrich type inequality for  $\tilde{F}_{j}, \, j=1,\cdots,\aleph$, we have: 
\begin{equation*}
	\left\Vert \tilde{F}_j \right\Vert_{\mathbb{L}^{2}\left(B \right)} \lesssim \; \left\Vert Curl\left(\tilde{F}_j \right) \right\Vert_{\mathbb{L}^{2}\left(B\right)}\lesssim a \,\,  \left\lVert\widetilde{Curl(F_j)}\right\rVert_{\mathbb{L}^2(B)} = \mathcal{O}\left( a \, \left\lVert\overset{1}{\mathbb{P}}\left(\tilde{E}^T_j\right)\right\rVert_{\mathbb{L}^2(B)} \right).
\end{equation*}
Then we can further derive that:
	\begin{equation}\label{J7-Estimation}
	J_{7} = \mathcal{O}\left(a^{8-h} \, d^{-4}  \, \max_j  \,\left\Vert \overset{1}{\mathbb{P}}\left( \tilde{E}^{T}_j\right) \right\Vert_{\mathbb{L}^{2}(B)}\right). 
	\end{equation}
	Finally, denote
	\begin{eqnarray*}
		J_{8} &:=& \eta \, k^{2} \, \sum_{j=1 \atop j \neq j_{0}}^{\aleph} \int_{D_{j_{0}}}  \mathcal{A}(x) \cdot \int_{D_{j}}  \int_{0}^{1} \underset{y}{\nabla} \left(\Upsilon_{k} \right)(x,z_{j}+t(y-z_{j})) \cdot \mathcal{P}(y,z_{j}) \; dt \cdot F_{j}(y)\, dy \, dx,\\
		\left\vert J_{8} \right\vert & \leq & \left\vert \eta \right\vert  \sum_{j=1 \atop j \neq j_{0}}^{\aleph}  \left\Vert  \mathcal{A} \right\Vert_{\mathbb{L}^{2}(D_{j_{0}})}  \, \left\Vert \int_{D_{j}}  \int_{0}^{1} \underset{y}{\nabla} \left(\Upsilon_{k} \right)(\cdot ,z_{j}+t(y-z_{j})) \cdot \mathcal{P}(y,z_{j}) \; dt \cdot F_{j}(y)\, dy \right\Vert_{\mathbb{L}^{2}(D_{j_{0}})} \\
		& \leq &  a^{2} \; \left\Vert  \mathcal{A} \right\Vert_{\mathbb{L}^{2}(D_{j_{0}})} \; \sum_{j=1 \atop j \neq j_{0}}^{\aleph} \left\vert z_{j} - z_{j_{0}} \right\vert^{-4}   \; \left\Vert  F_{j} \right\Vert_{\mathbb{L}^{2}(D_{j})}.
	\end{eqnarray*}
	We remark that the previous formula follows the same as $(\ref{Yas-Ra})$, and then, in a straightforward manner, we deduce that  
	\begin{equation}\label{J8-Estimation}
	J_{8} = \mathcal{O}\left(a^{8-h} \, d^{-4}  \, \max_j  \,\left\Vert \overset{1}{\mathbb{P}}\left( \tilde{E}^{T}_j\right) \right\Vert_{\mathbb{L}^{2}(B)}\right). 
	\end{equation} 
	Taking into account the estimations $(\ref{J6-Estimation}),(\ref{J7-Estimation})$ and $(\ref{J8-Estimation})$, the equation $(\ref{SIM})$ becomes, 
	\begin{eqnarray}\label{IdV}
	\nonumber
	\bm{\mathcal{M}} \cdot \int_{D_{j_{0}}} F_{j_{0}}(x) \; dx &-& \eta \, k^{2} \, \sum_{j=1 \atop j \neq j_{0}}^{\aleph} \int_{D_{j_{0}}}  \mathcal{A}(x) \, dx \cdot  \Upsilon_{k}(z_{j_{0}},z_{j})\cdot \int_{D_{j}} F_{j}(y)\, dy  \\ \nonumber
	& = & i \, k \, \int_{D_{j_{0}}}  \mathcal{A}(x) \, dx \cdot  H^{Inc}_{j_{0}}(z_{j_{0}}) +   \mathcal{O}\left(a^{8-h} \, d^{-4}  \, \max_j  \,\left\Vert \overset{1}{\mathbb{P}}\left( \tilde{E}^{T}_j\right) \right\Vert_{\mathbb{L}^{2}(B)}\right) \\ 
	&+& \mathcal{O}\left(a^{6-h}\right) + \mathcal{O}\left(a^{4} \, \left\Vert \overset{3}{\mathbb{P}}\left(\tilde{E}^{T}_{j_{0}}\right) \right\Vert_{\mathbb{L}^{2}\left(B\right)}\right) + \mathcal{O} \left(  a^{7-h} \, d^{-3} \left\vert \log (d)\right\vert \, \max_{j}  \left\Vert \overset{3}{\mathbb{P}}\left(\tilde{E}^{T}_j\right) \right\Vert_{\mathbb{L}^{2}\left(B\right)} \right).
	\end{eqnarray} 
	Remark that $\dfrac{1}{4} \left(\bm{\mathcal{M}}^{Tr} \cdot \bm{\mathcal{M}} \cdot \bm{\mathcal{M}}^{Tr} \right) \cdot \bm{\mathcal{M}} = I$ and denote $\bm{\mathcal{A}_{j_{0}}}$ to be 
	\begin{equation}\label{Def-Mt-M-Mt-int-A}
	\bm{\mathcal{A}_{j_{0}}} := \frac{1}{4} \left(\bm{\mathcal{M}}^{Tr} \cdot \bm{\mathcal{M}} \cdot \bm{\mathcal{M}}^{Tr} \right) \cdot \int_{D_{j_{0}}}  \mathcal{A}(x) \, dx \, \in \mathbb{C}^{3 \times 3}.
	\end{equation}
	Then, the equation $(\ref{IdV})$, after multiplying each side by $\dfrac{1}{4} \left(\bm{\mathcal{M}}^{Tr} \cdot \bm{\mathcal{M}} \cdot \bm{\mathcal{M}}^{Tr} \right)$, becomes 
	\begin{equation}\label{AS-2nd-Term}
	\int_{D_{j_{0}}} F_{j_{0}}(y) \; dy - \eta \, k^{2} \, \sum_{j=1 \atop j \neq j_{0}}^{\aleph} \bm{\mathcal{A}_{j_{0}}} \cdot  \Upsilon_{k}(z_{j_{0}},z_{j})\cdot \int_{D_{j}} F_{j}(y)\, dy  
	=  i \, k \; \bm{\mathcal{A}_{j_{0}}} \cdot  H^{Inc}_{j_{0}}(z_{j_{0}}) + Error,
	\end{equation}
	where  
	\begin{eqnarray}\label{HAPE}
	\nonumber
	Error &:=&    
	\mathcal{O}\left(a^{8-h} \, d^{-4}  \, \max_j  \,\left\Vert \overset{1}{\mathbb{P}}\left( \tilde{E}^{T}_j\right) \right\Vert_{\mathbb{L}^{2}(B)}\right)+\mathcal{O}\left(a^{6-h}\right) \\ 
	&+&  \mathcal{O}\left(a^{4} \, \left\Vert \overset{3}{\mathbb{P}}\left(\tilde{E}^{T}_{j_{0}}\right) \right\Vert_{\mathbb{L}^{2}\left(B\right)}\right) + \mathcal{O} \left(  a^{7-h} \, d^{-3} \, \left\vert \log (d)\right\vert \, \max_{j}  \left\Vert \overset{3}{\mathbb{P}}\left(\tilde{E}^{T}_j\right) \right\Vert_{\mathbb{L}^{2}\left(B\right)} \right)
	\end{eqnarray}   
	Next, we estimate $\bm{\mathcal{A}_{j_{0}}}$. 
	Since $\dfrac{1}{4} \left(\bm{\mathcal{M}}^{Tr} \cdot \bm{\mathcal{M}} \cdot \bm{\mathcal{M}}^{Tr} \right)$ is a constant matrix, the estimation of $\bm{\mathcal{A}_{j_{0}}}$ is the same as that of $\int_{D_{j_{0}}}  \mathcal{A}(x) \, dx$. We have 
	\begin{equation*}
	\int_{D_{j_{0}}}  \mathcal{A}(x) \, dx = a^{3} \, \int_{B}  \tilde{\mathcal{A}}(x) \, dx = a^{3} \, \int_{B}  \tilde{\mathcal{A}}(x) \cdot I \; dx = - \, a^{3} \, \int_{B}  \tilde{\mathcal{A}}(x) \cdot Curl\left(\mathcal{Q}(x) \right) \; dx,
	\end{equation*}
	where $\mathcal{Q}(x)$ is the matrix given by 
	\begin{equation*}
	\mathcal{Q}(x) := \begin{pmatrix}
	0 & x_{3} & 0 \\ 
	0 & 0 & x_{1} \\ 
	x_{2} & 0 & 0
	\end{pmatrix},
	\end{equation*}
	and, with integration by parts, we get
	\begin{equation*}
	\int_{D_{j_{0}}}  \mathcal{A}(x) \, dx =- a^{3} \, \int_{B} Curl\left( \tilde{\mathcal{A}} \right)(x) \cdot \mathcal{Q}(x)  \; dx =- a^{4} \, \int_{B} \widetilde{Curl\left( \mathcal{A} \right)}(x) \cdot \mathcal{Q}(x)  \; dx \overset{( \ref{p1=curlA})}{=} -a^{4} \, \int_{B} \overset{1}{\mathbb{P}}\left(\tilde{W} \right)(x) \cdot \mathcal{Q}(x)  \; dx.
	\end{equation*}
	By expanding $\overset{1}{\mathbb{P}}\left( \tilde{W} \right)$ over the basis $\left(e^{(1)}_{n} \right)_{n}$, we obtain 
	\begin{eqnarray*}
		\int_{D_{j_{0}}}  \mathcal{A}(x) \, dx &=& - a^{4} \, \sum_{n} \, \langle \overset{1}{\mathbb{P}}\left(\tilde{W} \right);e^{(1)}_{n} \rangle \cdot \left( \int_{B} e^{(1)}_{n}(x) \cdot \mathcal{Q}(x)  \; dx \right)^{Tr} \\
		& \overset{(\ref{Isefra})}{=} & - a^{4} \, \sum_{n} \, \frac{1}{\left( 1 - k^{2} \, \eta \, a^{2} \, \lambda^{(1)}_{n} \right)} \left[ a \, \langle \mathcal{P}(\cdot , 0) , e_{n}^{(1)} \rangle - \langle T, e_{n}^{(1)} \rangle \right] \cdot \left( \int_{B} e^{(1)}_{n}(x) \cdot \mathcal{Q}(x)  \; dx \right)^{Tr}.
	\end{eqnarray*}
	We write $e_{n}^{(1)}$ as $e_{n}^{(1)} = Curl\left( \phi_{n} \right)$ and using integration by parts to get
	\begin{equation*}
	\langle \mathcal{P}(\cdot , 0) , e_{n}^{(1)} \rangle =  \langle \mathcal{P}(\cdot , 0) , Curl\left( \phi_{n} \right) \rangle = \langle Curl\left( \mathcal{P}(\cdot , 0) \right),  \phi_{n}  \rangle \overset{(\ref{def-bm-mathcal-M})}{=} \langle \bm{\mathcal{M}} ,  \phi_{n}  \rangle
	\end{equation*}
	and 
	\begin{equation*}
	\int_{B} e^{(1)}_{n}(x) \cdot \mathcal{Q}(x)  \; dx  = \int_{B} Curl\left( \phi_{n} \right)(x) \cdot \mathcal{Q}(x)  \; dx = \int_{B}  \phi_{n}(x) \cdot Curl\left( \mathcal{Q}\right)(x)  \; dx = -\int_{B}  \phi_{n}(x) \; dx.
	\end{equation*}  
	Then 
	\begin{eqnarray}\label{A-expr}
	\int_{D_{j_{0}}}  \mathcal{A}(x) \, dx & = &  a^{5-h}  \, \, \langle \bm{\mathcal{M}} ,  \phi_{n_{0}}  \rangle  \cdot \left( \int_{B}  \phi_{n_{0}}(x) \; dx \right)^{Tr} \notag\\
	& - &  a^{5} \, \sum_{n \neq n_{0}} \, \frac{1}{\left( 1 - k^{2} \, \eta \, a^{2} \, \lambda^{(1)}_{n} \right)} \, \langle \mathcal{P}(\cdot , 0) , e_{n}^{(1)} \rangle  \cdot \left( \int_{B} e^{(1)}_{n}(x) \cdot \mathcal{Q}(x)  \; dx \right)^{Tr} \notag\\
	& + &  a^{4} \, \sum_{n} \, \frac{1}{\left( 1 - k^{2} \, \eta \, a^{2} \, \lambda^{(1)}_{n} \right)}  \langle T, e_{n}^{(1)} \rangle  \cdot \left( \int_{B} e^{(1)}_{n}(x) \cdot \mathcal{Q}(x)  \; dx \right)^{Tr}. 
	\end{eqnarray}
	Set
	\begin{equation*}
	J_{9,1} := a^{5} \, \sum_{n \neq n_{0}} \, \frac{1}{\left( 1 - k^{2} \, \eta \, a^{2} \, \lambda^{(1)}_{n} \right)} \, \langle \mathcal{P}(\cdot , 0) , e_{n}^{(1)} \rangle  \cdot \left( \int_{B} e^{(1)}_{n}(x) \cdot \mathcal{Q}(x)  \; dx \right)^{Tr},
	\end{equation*}
	then
	\begin{equation*}
	\left\vert J_{9,1} \right\vert \leq  a^{5} \; \sum_{n \neq n_{0}} \left\vert \langle \mathcal{P}(\cdot , 0) , e_{n}^{(1)} \rangle \right\vert \; \left\vert \langle \mathcal{Q} , e_{n}^{(1)} \rangle \right\vert \sim a^{5}.
	\end{equation*}
	And for the term $J_{9,2}$ given by 
	\begin{equation*}
	J_{9,2} := a^{4} \, \sum_{n} \, \frac{1}{\left( 1 - k^{2} \, \eta \, a^{2} \, \lambda^{(1)}_{n} \right)}  \langle T, e_{n}^{(1)} \rangle  \cdot \left( \int_{B} e^{(1)}_{n}(x) \cdot \mathcal{Q}(x)  \; dx \right)^{Tr},
	\end{equation*}
	we can check, from the definition of the term $T(\cdot)$, see for instance $(\ref{T-term})$, that 
	\begin{eqnarray*}
		J_{9,2} & \sim & a^{4} \, \sum_{n} \, \frac{1}{\left( 1 - k^{2} \, \eta \, a^{2} \, \lambda^{(1)}_{n} \right)}  \langle \int_{B} \Phi_{0}(\cdot,y) \frac{A(\cdot,y)}{\left\Vert \cdot - y \right\Vert^{2}} \cdot \overset{3}{\mathbb{P}}\left(\tilde{W} \right)(y) dy, e_{n}^{(1)} \rangle  \cdot \left( \int_{B} e^{(1)}_{n}(x) \cdot \mathcal{Q}(x)  \; dx \right)^{Tr}\\
		\left\vert J_{9,2} \right\vert & \lesssim & a^{4-h} \, \sum_{n} \, \left\vert  \langle \int_{B} \Phi_{0}(\cdot,y) \frac{A(\cdot,y)}{\left\Vert \cdot - y \right\Vert^{2}} \cdot \overset{3}{\mathbb{P}}\left(\tilde{W} \right)(y) dy, e_{n}^{(1)} \rangle \right\vert  \left\vert \langle e^{(1)}_{n}, \mathcal{Q} \rangle\right\vert \\
		& \lesssim & a^{4-h} \,  \, \left\Vert   \int_{B} \Phi_{0}(\cdot,y) \frac{A(\cdot,y)}{\left\Vert \cdot - y \right\Vert^{2}} \cdot \overset{3}{\mathbb{P}}\left(\tilde{W} \right)(y) dy \right\Vert \lesssim a^{4-h} \,  \, \left\Vert  \overset{3}{\mathbb{P}}\left(\tilde{W} \right) \right\Vert \overset{(\ref{fin-W-3})}{=} \mathcal{O}\left( a^{7-h} \right). 
	\end{eqnarray*}
	Using the estimations $J_{9,1}$ and $J_{9,2}$, we obtain that
	\begin{equation}\notag
	\int_{D_{j_{0}}}  \mathcal{A}(x) \, dx =  a^{5-h}  \, \, \langle \bm{\mathcal{M}} ,  \phi_{n_{0}}  \rangle  \cdot \left( \int_{B}  \phi_{n_{0}}(x) \; dx \right)^{Tr} + \mathcal{O}\left( a^{5} \right). 
	\end{equation}  
	Seeing the definition of $\bm{\mathcal{A}_{j_{0}}}$, in $(\ref{Def-Mt-M-Mt-int-A})$, by a straightforward computation, we deduce that 
	\begin{equation}\label{int-phi-n0}
	\bm{\mathcal{A}_{j_{0}}}  =  a^{5-h} \, \left( \left( \int_{B} \, \phi_{n_{0}}(y) \, dy \right) \otimes \left(\int_{B} \phi_{n_{0}}(y) dy \right) \right) + \mathcal{O}\left( a^{5} \right). 
	\end{equation} 
	By substituting \eqref{int-phi-n0} into \eqref{AS-2nd-Term}, after rearranging the terms, we obtain the linear algebraic with the form \eqref{la-1}. Moreover, from the a-priori estimates for $\overset{1}{\mathbb{P}}(\tilde{E}^T_j)$ and $\overset{3}{\mathbb{P}}(\tilde{E}^T_j)$ given by \eqref{max-P1P3} in Proposition \ref{lem-es-multi}, the error term in \eqref{HAPE} can be simplified as \eqref{la-error}.
	
	As for the invertibility condition \eqref{inver-con-1}, indeed, utilizing the precise representation \eqref{A-expr} of the tensor $\bm{\mathcal{A}_{j_{0}}}$ and combining with \eqref{int-phi-n0}, it is direct to verify that $\left\lVert[P_0]\right\rVert$ of the invertibility condition \eqref{invert-condi-gene} in Proposition \ref{pro-general la} can be replaced by
	\begin{equation}\notag
		\frac{k^2 a^5}{|1 - k^2 \, \eta \, a^2 \, \lambda_{n_0}^{(1)}|}\left\lVert\int_{B}\phi_{n_{0}}(y)\,dy\otimes\int_{B}\phi_{n_{0}}(y)\,dy\right\rVert,
	\end{equation}
	which yields \eqref{inver-con-1}.
\end{proof}

\subsection{Construction of the linear algebraic systerm in Proposition \ref{prop-la-2}.}

Similar to Proposition \ref{prop-la-1}, we construct the precise form of the linear algebraic system related to $\overset{3}{\mathbb{P}}(E^T_j)$, $j=1, 2, \cdots, \aleph$, in Proposition \ref{prop-la-2} as follows.

\begin{proof}[Proof of Proposition \ref{prop-la-2}.]
	Recall again the Lippmann-Schwinger system of the equation  
	\begin{equation*}
\boldsymbol{T}_{k} \, \left(E^{T}_{j_{0}}\right)(x) - \eta \, \sum_{j=1 \atop j \neq j_{0}}^{\aleph} \int_{D_{j}} \Upsilon_{k}(x,y) \cdot E^{T}_j(y)\, dy = E^{Inc}_{j_{0}}(x), \mbox{ in } D_{j_0}, \; j_0=1, ..., \aleph.
	\end{equation*}
	Define $V$ as
	\begin{equation}\label{def-W}
	V := \boldsymbol{T}_{- \, k}^{-1}\left( I \right),
	\end{equation}
Inverting the operator $\boldsymbol{T}_{k}$, integrating over $D_{j_{0}}$ and using the definition of $V$, see for instance $(\ref{def-W})$, we obtain that
	\begin{equation*}
	\int_{D_{j_{0}}} E^{T}_{j_{0}}(x) \, dx - \eta \, \sum_{j=1 \atop j \neq j_{0}}^{\aleph} \int_{D_{j_{0}}} V(x) \cdot \int_{D_{j}} \Upsilon_{k}(x,y) \cdot E^{T}_j(y)\, dy \, dx = \int_{D_{j_{0}}} V(x) \cdot E^{Inc}_{j_{0}}(x) dx,
	\end{equation*}
	or in the following form 
	\begin{eqnarray*}
		\int_{D_{j_{0}}} \overset{3}{\mathbb{P}}\left(E^{T}_{j_{0}}\right)(x) \, dx &-& \eta \, \sum_{j=1 \atop j \neq j_{0}}^{\aleph} \int_{D_{j_{0}}} V(x) \cdot \int_{D_{j}} \Upsilon_{k}(x,y) \cdot \overset{3}{\mathbb{P}}\left(E^{T}_j\right)(y)\, dy \, dx \\ &-& \eta \,k^{2} \sum_{j=1 \atop j \neq j_{0}}^{\aleph} \int_{D_{j_{0}}} V(x) \cdot \int_{D_{j}} \Phi_{k}(x,y) \overset{1}{\mathbb{P}}\left(E^{T}_j\right)(y)\, dy \, dx = \int_{D_{j_{0}}} V(x) \cdot E^{Inc}_{j_{0}}(x) dx.
	\end{eqnarray*}
We rewrite the previous equation, with the help of \eqref{Expansion-Upsilon}, as 
	\begin{eqnarray}\label{equa-M-L}
	\nonumber
	\int_{D_{j_{0}}} \overset{3}{\mathbb{P}}\left(E^{T}_{j_{0}}\right)(x) \, dx &-& \eta \, \sum_{j=1 \atop j \neq j_{0}}^{\aleph} \int_{D_{j_{0}}} V(x) \, dx \cdot  \Upsilon_{k}(z_{j_{0}},z_{j}) \cdot \int_{D_{j}} \overset{3}{\mathbb{P}}\left(E^{T}_j\right)(y)\, dy = \int_{D_{j_{0}}} V(x) \cdot E^{Inc}_{j_{0}}(x) dx  \\ \nonumber
	&+&  \eta \, \sum_{j=1 \atop j \neq j_{0}}^{\aleph} \int_{D_{j_{0}}} V(x) \cdot \int_{D_{j}} \int_{0}^{1} \underset{y}{\nabla} \left(\Upsilon_{k} \right)(x,z_{j}+t(y-z_{j})) \cdot \mathcal{P}(y,z_{j}) \; dt \cdot \overset{3}{\mathbb{P}}\left(E^{T}_j\right)(y)\, dy \, dx \\ \nonumber
	&+&  \eta \, \sum_{j=1 \atop j \neq j_{0}}^{\aleph} \int_{D_{j_{0}}} V(x) \cdot  \int_{0}^{1} \underset{x}{\nabla} \left(\Upsilon_{k} \right)(z_{j_{0}}+t(x-z_{j_{0}}),z_{j}) \cdot \mathcal{P}(x,z_{j_{0}}) \; dt \cdot \int_{D_{j}} \overset{3}{\mathbb{P}}\left(E^{T}_j\right)(y)\, dy \, dx\\
	&+& \eta \,k^{2} \sum_{j=1 \atop j \neq j_{0}}^{\aleph} \int_{D_{j_{0}}} V(x) \cdot \int_{D_{j}}\Phi_{k}(x, y)\overset{1}{\mathbb{P}}\left(E^{T}_j\right)(y)\,dy\,dx
	\end{eqnarray}
	Then, we estimate the last three terms on the right hand side of \eqref{equa-M-L}. Set, 
	\begin{eqnarray*}
		J_{10} & := &\eta \, \sum_{j=1 \atop j \neq j_{0}}^{\aleph} \int_{D_{j_{0}}} V(x) \cdot \int_{D_{j}} \int_{0}^{1} \underset{y}{\nabla} \left(\Upsilon_{k} \right)(x,z_{j}+t(y-z_{j})) \cdot \mathcal{P}(y,z_{j}) \; dt \cdot \overset{3}{\mathbb{P}}\left(E^{T}_j\right)(y)\, dy \, dx \\
		\left\vert J_{10} \right\vert & \leq & \left\vert \eta \right\vert \; \left\Vert V \right\Vert_{\mathbb{L}^{2}(D_{j_{0}})} \, \sum_{j=1 \atop j \neq j_{0}}^{\aleph} \left[ \int_{D_{j_{0}}} \int_{D_{j}} \; \left\vert \int_{0}^{1} \underset{y}{\nabla} \left(\Upsilon_{k} \right)(x, z_{j}+t(y-z_{j})) \cdot \mathcal{P}(y,z_{j}) \; dt \right\vert^{2} dy \, dx \right]^{\frac{1}{2}} \; \left\Vert \overset{3}{\mathbb{P}}\left(E^{T}_j\right) \right\Vert_{\mathbb{L}^{2}(D_{j})} \\
		& \lesssim & a^{2} \; \left\Vert V \right\Vert_{\mathbb{L}^{2}(D_{j_{0}})} \, \sum_{j=1 \atop j \neq j_{0}}^{\aleph}  \left\vert z_{j} - z_{j_{0}} \right\vert^{-4} \; \left\Vert \overset{3}{\mathbb{P}}\left(E^{T}_j\right) \right\Vert_{\mathbb{L}^{2}(D_{j})}.
	\end{eqnarray*}
	Then,  
	\begin{equation}\label{Estimation-J-10}
	J_{10} = \mathcal{O}\left( a^{2} \; \left\Vert V \right\Vert_{\mathbb{L}^{2}(D_{j_{0}})} \; d^{-4}  \; \max_{j} \left\Vert \overset{3}{\mathbb{P}}\left(E^{T}_j\right) \right\Vert_{\mathbb{L}^{2}(D_j)} \right) \overset{(\ref{max-P1P3})}{=} \mathcal{O}\left( \left\Vert V \right\Vert_{\mathbb{L}^{2}(D_{j_{0}})} \; d^{-4}  \; a^{\frac{10+h}{2}} \right).
	\end{equation} 
	Denote
	\begin{eqnarray*}
		J_{11} &:=& \eta \, \sum_{j=1 \atop j \neq j_{0}}^{\aleph} \int_{D_{j_{0}}} V(x) \cdot  \int_{0}^{1} \underset{x}{\nabla} \left(\Upsilon_{k} \right)(z_{j_{0}}+t(x-z_{j_{0}}),z_{j}) \cdot \mathcal{P}(x,z_{j_{0}}) \; dt \cdot \int_{D_{j}} \overset{3}{\mathbb{P}}\left(E^{T}_j\right)(y)\, dy \, dx \\
		J_{11} &=& \eta \, \sum_{j=1 \atop j \neq j_{0}}^{\aleph} \int_{D_{j_{0}}} V(x) \cdot  \int_{0}^{1} \underset{x}{\nabla} \left(\Upsilon_{k} \right)(z_{j_{0}}+t(x-z_{j_{0}}),z_{j}) \cdot \mathcal{P}(x,z_{j_{0}}) \; dt \, dx \cdot \int_{D_{j}} \overset{3}{\mathbb{P}}\left(E^{T}_j\right)(y)\, dy  \\ 
		\left\vert J_{11} \right\vert & \leq & a^{-\frac{1}{2}} \; \left\Vert V \right\Vert_{\mathbb{L}^{2}(D_{j_{0}})}   \;  \sum_{j=1 \atop j \neq j_{0}}^{\aleph} \left\Vert \int_{0}^{1} \underset{x}{\nabla} \left(\Upsilon_{k} \right)(z_{j_{0}}+t( \cdot -z_{j_{0}}),z_{j}) \cdot \mathcal{P}(\cdot ,z_{j_{0}}) \; dt \, \right\Vert_{\mathbb{L}^{2}(D_{j_{0}})}  \left\Vert \overset{3}{\mathbb{P}}\left(E^{T}_j\right) \right\Vert_{\mathbb{L}^{2}(D_{j})}  \\
		& \lesssim & a^{2} \; \left\Vert V \right\Vert_{\mathbb{L}^{2}(D_{j_{0}})}   \;  \sum_{j=1 \atop j \neq j_{0}}^{\aleph} \left\vert z_{j} - z_{j_{0}} \right\vert^{-4}  \left\Vert \overset{3}{\mathbb{P}}\left(E^{T}_j\right) \right\Vert_{\mathbb{L}^{2}(D_{j})}. 
	\end{eqnarray*}
	Then,   
	\begin{equation}\label{Estimation-J-11}
	J_{11} = \mathcal{O}\left( a^{2} \; \left\Vert V \right\Vert_{\mathbb{L}^{2}(D_{j_{0}})} \; d^{-4}   \; \max_{j} \left\Vert \overset{3}{\mathbb{P}}\left(E^{T}_j\right) \right\Vert_{\mathbb{L}^{2}(D_j)}  \right)\overset{(\ref{max-P1P3})}{=} \mathcal{O}\left( \left\Vert V \right\Vert_{\mathbb{L}^{2}(D_{j_{0}})} \; d^{-4}  \; a^{\frac{10+h}{2}} \right).
	\end{equation} 
	Also, we denote $J_{12}$ to be 
	\begin{align*}
	J_{12} &:=\eta \,k^{2} \sum_{j=1 \atop j \neq j_{0}}^{\aleph} \int_{D_{j_{0}}} V(x) \cdot \int_{D_{j}}\Phi_{k}(x, y)\overset{1}{\mathbb{P}}\left(E^{T}_j\right)(y)\,dy\,dx\notag\\
	&=\eta \, k^2\sum_{j=1 \atop j \neq j_{0}}^{\aleph} \int_{D_{j_{0}}} V(x) \cdot \int_{D_{j}}\bigg[\Phi_{k}(x, z_j)+\underset{y}{\nabla}\Phi_{k}(x, z_j)\cdot(y-z_j)\notag\\
	&+\frac{1}{2}\int_{0}^1(y-z_j)^\perp\cdot\underset{y}{Hess}\left(\Phi_{k}\right)(x, z_j+t(y-z_j))\cdot (y-z_j)\,dt\bigg] \, \overset{1}{\mathbb{P}}\left(E^{T}_j\right)(y)\,dy\,dx.
	\end{align*}
For the first term on the right hand side of the above equation, it is clear that: 	
\begin{equation*}
 \int_{D_{j_{0}}} V(x) \cdot \int_{D_{j}} \Phi_{k}(x, z_j) \, \overset{1}{\mathbb{P}}\left(E^{T}_j\right)(y)\,dy\,dx = \int_{D_{j_{0}}} V(x) \cdot \Phi_{k}(x, z_j) \, \int_{D_{j}}  \overset{1}{\mathbb{P}}\left(E^{T}_j\right)(y)\,dy\,dx = 0.
\end{equation*}
And for the last two terms, there holds
\begin{eqnarray*}
L_{4,1} &:=& \sum_{j = 1\atop j\neq j_0}^\aleph\int_{D_{j_{0}}}V(x)\int_{D_{j}}\underset{y}{\nabla }\Phi_{k}(x, z_j)\cdot(y-z_j)\cdot\overset{1}{\mathbb{P}}\left(E^T_j\right)(y)\,dy \, dx, \\
\left\vert L_{4,1} \right\vert &=& \left\vert \sum_{j = 1\atop j\neq j_0}^\aleph\int_{D_{j_{0}}}V(x)\int_{D_{j}}\underset{y}{\nabla }\Phi_{k}(x, z_j)\cdot(y-z_j)\cdot\overset{1}{\mathbb{P}}\left(E^T_j\right)(y) \, dy \, dx \right\vert \\
& \lesssim & a^4 \, \lVert V \rVert_{\mathbb{L}^2(D_{j_0})} \, \sum_{j = 1\atop j\neq j_0}^\aleph \frac{1}{|z_{j_0}-z_j|^2} \, \left\lVert\overset{1}{\mathbb{P}}\left(E^T_j\right)\right\rVert_{\mathbb{L}^2(D_{j})} \\ &=& \mathcal{O}\left(a^4 \, d^{-3} \, \lVert V\rVert_{\mathbb{L}^2(D_{j_0})} \, \max_{j}\left\lVert\overset{1}{\mathbb{P}}\left(E^T_j\right)\right\rVert_{\mathbb{L}^2(D_{j})} \right) \overset{(\ref{max-P1P3})}{=} \mathcal{O}\left(a^{\frac{13}{2}-h} \, d^{-3} \, \lVert V\rVert_{\mathbb{L}^2(D_{j_0})} \right),
\end{eqnarray*}	

	and
	\begin{eqnarray*}
L_{4,2} & := & \sum_{j = 1\atop j\neq j_0}^\aleph\int_{D_{j_{0}}}V(x)\int_{D_{j}}\int_{0}^1(y-z_j)^\perp\cdot\underset{y}{Hess}\left(\Phi_{k}\right)(x, z_j+t(y-z_j))\cdot (y-z_j)\,dt\cdot \overset{1}{\mathbb{P}}\left(E^{T}_j\right)(y)\,dy\,dx \\
\left\vert L_{4,2} \right\vert &\lesssim& 
a^5 \lVert V\rVert_{\mathbb{L}^2(D_{j_0})}\sum_{j = 1\atop j\neq j_0}^\aleph \frac{1}{|z_{j_0}-z_j|^3}\lVert \overset{1}{\mathbb{P}}\left(E^T_{j}\right)\rVert_{\mathbb{L}(D_j)}
\notag\\
&\lesssim & a^5 \, d^{-3} \, \left\vert \log (d)\right\vert \, \lVert V \rVert_{\mathbb{L}^2(D_{j_0})}\max_{j}\left\lVert\overset{1}{\mathbb{P}}\left({E}^{T}_j\right)\right\rVert_{\mathbb{L}^2(D_{j})}\overset{(\ref{max-P1P3})}{=} \mathcal{O}\left( a^{\frac{15}{2}-h} \, d^{-3} \, \left\vert \log (d)\right\vert \, \lVert V \rVert_{\mathbb{L}^2(D_{j_0})}\right).
	\end{eqnarray*}
Hence, $J_{12}$ can be estimated as 
\begin{equation}\label{Estimation-J-12}
J_{12} = \mathcal{O}\left(a^{\frac{9}{2}-h} \, d^{-3} \, \lVert V\rVert_{\mathbb{L}^2(D_{j_0})} \right).
\end{equation}
Recall $(\ref{equa-M-L})$. Using $(\ref{Estimation-J-10}),(\ref{Estimation-J-11})$ and $(\ref{Estimation-J-12})$, we deduce that	
	\begin{eqnarray*}
	\nonumber
	\int_{D_{j_{0}}} \overset{3}{\mathbb{P}}\left(E^{T}_{j_{0}}\right)(x) \, dx &-& \eta \, \sum_{j=1 \atop j \neq j_{0}}^{\aleph} \int_{D_{j_{0}}} V(x) \, dx \cdot  \Upsilon_{k}(z_{j_{0}},z_{j}) \cdot \int_{D_{j}} \overset{3}{\mathbb{P}}\left(E^{T}_j\right)(y)\, dy \\ &=& \int_{D_{j_{0}}} V(x) \cdot E^{Inc}_{j_{0}}(x) dx  
 +\mathcal{O}\left(a^{5+\frac{h}{2}} d^{-4} \lVert V\rVert_{\mathbb{L}^2(D_{j_0})}\right)+\mathcal{O}\left(a^{\frac{9}{2}-h}d^{-3}\lVert V\rVert_{\mathbb{L}^2(D_{j_0})}\right), 
	\end{eqnarray*}
which can be reduced, using the conditions on $t$ and $h$ given by $(\ref{conditions-t-h})$, to:  	
	\begin{eqnarray}\label{final-1st}
	\nonumber
	\int_{D_{j_{0}}} \overset{3}{\mathbb{P}}\left(E^{T}_{j_{0}}\right)(x) \, dx - \eta \, \sum_{j=1 \atop j \neq j_{0}}^{\aleph} \mathcal{C}_{j_{0}} \cdot  \Upsilon_{k}(z_{j_{0}},z_{j}) \cdot \int_{D_{j}} \overset{3}{\mathbb{P}}\left(E^{T}_j\right)(y)\, dy &=& \int_{D_{j_{0}}} V(x) \cdot E^{Inc}_{j_{0}}(x) dx  
\\ &+& \mathcal{O}\left(a^{\frac{9}{2}-h}d^{-3}\lVert V\rVert_{\mathbb{L}^2(D_{j_0})}\right),
	\end{eqnarray}
where $\mathcal{C}_{j_{0}} := \int_{D_{j_{0}}} V(x) dx. $
	Moreover,  
	\begin{equation*}
		\int_{D_{j_{0}}} V(x) \cdot E^{Inc}_{j_{0}}(x) dx = \mathcal{C}_{j_{0}} \cdot E^{Inc}_{j_{0}}(z_{j_{0}}) + \int_{D_{j_{0}}} V(x) \cdot \int_{0}^{1} \underset{x}{\nabla}\left( E^{Inc}_{j_{0}}\right)(z_{j_{0}}+t(x-z_{j_{0}})) \cdot (x-z_{j_{0}}) dt dx,  
	\end{equation*}
	and we estimate the last term as 
	\begin{equation*}
	\left\vert \int_{D_{j_{0}}} V(x) \cdot \int_{0}^{1} \underset{x}{\nabla}\left( E^{Inc}_{j_{0}}\right)(z_{j_{0}}+t(x-z_{j_{0}})) \cdot (x-z_{j_{0}}) dt dx \right\vert \lesssim a^{\frac{5}{2}} \;  \left\Vert V \right\Vert_{\mathbb{L}^{2}(D_{j_{0}})}.
	\end{equation*}
	Therefore, equation $(\ref{final-1st})$ takes the following form 
	\begin{eqnarray}\label{Dj.E.B.C}
	\nonumber
	\int_{D_{j_{0}}} \overset{3}{\mathbb{P}}\left(E^{T}_{j_{0}}\right)(x) \, dx &-& \eta \, \sum_{j=1 \atop j \neq j_{0}}^{\aleph} \mathcal{C}_{j_0} \cdot  \Upsilon_{k}(z_{j_{0}},z_{j}) \cdot \int_{D_{j}} \overset{3}{\mathbb{P}}\left(E^{T}_j\right)(y)\, dy \\ &=& \mathcal{C}_{j_0}  \cdot E^{Inc}_{j_{0}}(z_{j_{0}})   + \mathcal{O}\left(a^{\frac{5}{2}} \lVert V\rVert_{\mathbb{L}^2(D_{j_0})}\right)
	+ \mathcal{O}\left(a^{\frac{9}{2}-h}d^{-3}\lVert V\rVert_{\mathbb{L}^2(D_{j_0})}\right).  
	\end{eqnarray}  
	Next, we compute the estimation of  $\mathcal{C}_{j_0}$.
	Recalling the definition of the matrix $V$ in $(\ref{def-W})$, i.e. $V := \boldsymbol{T}_{- \, k}^{-1}\left( I \right)$, and scaling to the domain $B$, we obtain 
	\begin{equation}\label{scale-W}
	\overset{\sim}{V}  := \boldsymbol{T}_{- \, k \, a}^{-1}\left( I \right).
	\end{equation} 
Now, to give the estimation of $\mathcal{C}_{j_0}$, we compute $\langle \overset{\sim}{V}; e^{(j)}_{n} \rangle, j=1,2,3$. Indeed, we have  	
	\begin{eqnarray*}
		0 = \int_{B} e^{(1)}_{n}(x) \, dx & = & \int_{B} I \cdot e^{(1)}_{n}(x) \, dx \overset{(\ref{scale-W})}{=} \int_{B} \boldsymbol{T}_{- \, k \, a} \left( \overset{\sim}{V} \right)(x) \cdot e^{(1)}_{n}(x) \, dx =  \int_{B}  \overset{\sim}{V}(x) \cdot \boldsymbol{T}_{k \, a} \left(e^{(1)}_{n}\right)(x) \, dx.
\end{eqnarray*}
Knowing that $\nabla M^{ka}\left(e^{(1)}_{n}\right) = 0$, combining with $(\ref{expansion-Nk})$, we deduce from the above equation that
	\begin{equation}\label{w-en1}
	\langle \overset{\sim}{V}, e_{n}^{(1)} \rangle = \frac{k^{2} \, \eta \, a^{2}}{4 \, \pi \, \left( 1 - k^{2} \, \eta \, a^{2} \, \lambda_{n}^{(1)} \right)} \; \sum_{\ell \geq 1} \frac{(ika)^{\ell+1}}{(\ell+1)!} \; \int_{B} \overset{\sim}{V}(x) \cdot \int_{B} \; \left\Vert x - y \right\Vert^{\ell}  e_{n}^{(1)}(y) \, dy \, dx.
	\end{equation}
	Similarly, we have 
	\begin{equation}\label{W-en2}
	0 = \int_{B} e^{(2)}_{n}(x) \, dx  =  \int_{B} I \cdot e^{(2)}_{n}(x) \, dx  \overset{(\ref{scale-W})}{=}  \int_{B} \boldsymbol{T}_{- \, k \, a} \left( \overset{\sim}{V} \right) \cdot e^{(2)}_{n}(x) \, dx  \overset{(\ref{grad-M-k-2nd-subspace})}{=}  (1 + \eta) \; \langle \overset{\sim}{V}; e_{n}^{(2)} \rangle. 
	\end{equation}
	In the subspace $\nabla \mathcal{H}armonic$, we have
	\begin{eqnarray*}
		\int_{B} e^{(3)}_{n}(x) \, dx & = & \int_{B} I \cdot e^{(3)}_{n}(x) \, dx \overset{(\ref{scale-W})}{=} \int_{B} \boldsymbol{T}_{- \, k \, a} \left( \overset{\sim}{V} \right)(x) \cdot e^{(3)}_{n}(x) \, dx 
		 =  \int_{B}  \overset{\sim}{V}(x) \cdot \boldsymbol{T}_{k \, a} \left(e^{(3)}_{n}\right)(x) \, dx \\
		& = & \left( 1 + \eta \lambda_{n}^{(3)} \right) \langle \overset{\sim}{V}, e^{(3)}_{n} \rangle +   \eta  \, \langle \overset{\sim}{V}, \left( - k^{2} \, a^{2} \, N^{ka} + \left( \nabla M^{ka} - \nabla M \right)\right) \left(e^{(3)}_{n}\right) \rangle.
	\end{eqnarray*}
	Then, 
	\begin{eqnarray}\label{W,e3}
	\nonumber
	\langle \overset{\sim}{V}, e^{(3)}_{n} \rangle & = & \left( 1 + \eta \lambda_{n}^{(3)} \right)^{-1} \left[ \int_{B} e^{(3)}_{n}(x) \, dx + \eta  \, \langle \overset{\sim}{V}, \left(  k^{2} \, a^{2} \, N^{ka} - \left( \nabla M^{ka} - \nabla M \right)\right) \left(e^{(3)}_{n}\right) \rangle \right] \\
	& = &  \frac{1}{1 + \eta \lambda_{n}^{(3)}}  \int_{B} e^{(3)}_{n}(x) \, dx + \frac{\eta }{1 + \eta \lambda_{n}^{(3)}}  \, \langle \overset{\sim}{V}, \mathcal{S} \left(e^{(3)}_{n}\right) \rangle,
	\end{eqnarray}
	where $\mathcal{S}$ is given by \eqref{DefOperatorS}.
	Now,  
	\begin{eqnarray}\label{beforJ12}
	\nonumber
	\mathcal{C}_{j_0} &=& a^{3} \,  \int_{B}  \, \overset{\sim}{V}(x) \, dx = a^{3} \; \sum_{n} \langle \overset{\sim}{V}, e^{(3)}_{n} \rangle \otimes \int_{B} e^{(3)}_{n}(x) \, dx \\ \nonumber
	& \overset{(\ref{W,e3})}{=} & a^{3} \; \sum_{n} \left[ \frac{1}{1 + \eta \lambda_{n}^{(3)}}  \int_{B} e^{(3)}_{n}(x) \, dx + \frac{\eta }{1 + \eta \lambda_{n}^{(3)}}  \, \langle \overset{\sim}{V}, \mathcal{S} \left(e^{(3)}_{n}\right) \rangle \right] \otimes \int_{B} e^{(3)}_{n}(x) \, dx \\ \nonumber 
	& = & a^{3} \; \sum_{n} \left[ \frac{1}{1 + \eta \lambda_{n}^{(3)}}  \int_{B} e^{(3)}_{n}(x) \, dx + \frac{\eta }{1 + \eta \lambda_{n}^{(3)}}  \, \langle \overset{\sim}{V}, \mathcal{S}_{d} \left(e^{(3)}_{n}\right) \rangle \right] \otimes \int_{B} e^{(3)}_{n}(x) \, dx \\
	& + & a^{3} \, \eta \; \sum_{n}  \frac{1}{1 + \eta \lambda_{n}^{(3)}}  \, \langle \overset{\sim}{V}, \mathcal{S}_{r} \left(e^{(3)}_{n}\right) \rangle \otimes \int_{B} e^{(3)}_{n}(x) \, dx,
	\end{eqnarray}
	where, from $(\ref{DefOperatorS})$, we have  
	\begin{equation}\label{SdW}
	\mathcal{S}_{d}(e^{(3)}_{n})(x)  =  \frac{a^{2} \, k^{2}}{2} \, N(e^{(3)}_{n})(x)  + \frac{a^{2} k^{2}}{2} \int_{B} \Phi_{0}(x,y) \frac{A(x,y)\cdot e^{(3)}_{n}(y)}{\Vert x -y \Vert^{2}}  dy 
	\end{equation}
	and
	\begin{eqnarray}\label{SrW}
	\nonumber
	\mathcal{S}_{r} \left(e^{(3)}_{n}\right)(x) &:=& \frac{i (ak)^{3}}{6 \pi} \int_{B} e^{(3)}_{n}(y) dy + \frac{1}{4 \pi} \; \sum_{\ell \geq 3} \frac{(ika)^{\ell +1}}{(\ell +1)!} \; \int_{B} \; \underset{x}{Hess} \left( \left\Vert x - y \right\Vert^{\ell} \right) \cdot e^{(3)}_{n}(y) \, dy \\
	&+&   \frac{k^{2} \, a^{2}}{4 \pi} \; \sum_{\ell \geq 1} \frac{(ika)^{\ell +1}}{(\ell +1)!} \; \int_{B} \; \left\Vert x - y \right\Vert^{\ell}  e^{(3)}_{n}(y) \, dy.
	\end{eqnarray} 
	First, we estimate 
	\begin{equation*}
		J_{13} := a^{3} \eta \; \sum_{n}  \frac{1}{1 + \eta \lambda_{n}^{(3)}}  \, \langle \overset{\sim}{V}, \mathcal{S}_{r} \left(e^{(3)}_{n}\right) \rangle \otimes \int_{B} e^{(3)}_{n}(x) \, dx  \simeq  a^{3}  \; \sum_{n}  \langle \overset{\sim}{V}, \mathcal{S}_{r} \left(e^{(3)}_{n}\right) \rangle \otimes \int_{B} e^{(3)}_{n}(x) \, dx.
	\end{equation*}
	By keeping the dominant term of $(\ref{SrW})$, we see that $J_{13}$ fulfills 
	\begin{eqnarray*}
		J_{13} & \simeq & a^{6}  \; \sum_{n}  \langle \overset{\sim}{V}, \int_{B} e^{(3)}_{n}(x) \, dx \rangle \otimes \int_{B} e^{(3)}_{n}(x) \, dx \\
		\left\vert J_{13} \right\vert & \lesssim & a^{6}  \; \left( \sum_{n}  \left\vert \langle \overset{\sim}{V}, \int_{B} e^{(3)}_{n}(x) \, dx \rangle \right\vert^{2} \right)^{\frac{1}{2}} \, \left( \sum_{n}  \left\vert \int_{B} e^{(3)}_{n}(x) \, dx\right\vert^{2} \right)^{\frac{1}{2}} =   \mathcal{O}\left(a^{6} \; \left\Vert \overset{\sim}{V} \right\Vert_{\mathbb{L}^{2}(B)} \right).
	\end{eqnarray*}
	Then equation $(\ref{beforJ12})$ becomes
	\begin{equation*}
	\mathcal{C}_{j_0}  =  a^{3} \; \sum_{n} \frac{1}{1 + \eta \lambda_{n}^{(3)}} \left[   \int_{B} e^{(3)}_{n}(x) \, dx + \eta  \, \langle \overset{\sim}{V}, \mathcal{S}_{d} \left(e^{(3)}_{n}\right) \rangle \right] \otimes \int_{B} e^{(3)}_{n}(x) \, dx + \mathcal{O}\left(a^{6} \; \left\Vert \overset{\sim}{V} \right\Vert_{\mathbb{L}^{2}(B)} \right). 
\end{equation*}		
With the help of the expression of\footnote{To write short, in $(\ref{SdW})$, the estimation of the second term of the right side is neglected because the singularity of the corresponding kernel behaves like the one of the Newtonian operator.} $\mathcal{S}_{d} \left(e^{(3)}_{n}\right)$ in $(\ref{SdW})$, we have
\begin{equation*}		
\mathcal{C}_{j_0}  =  a^{3} \; \sum_{n} \frac{1}{1 + \eta \lambda_{n}^{(3)}} \,  \left[  \int_{B} e^{(3)}_{n}(x) \, dx + \eta \, \frac{k^{2} \, a^{2}}{2} \langle \overset{\sim}{V}, N \left(e^{(3)}_{n}\right) \rangle \right] \otimes \int_{B} e^{(3)}_{n}(x) \, dx + \mathcal{O}\left(a^{6} \; \left\Vert \overset{\sim}{V} \right\Vert_{\mathbb{L}^{2}(B)} \right).
\end{equation*}
Then, using the fact that $k^{2} \, \eta \, a^{2} = \dfrac{1 \mp c_{0} \, a^{h}}{\lambda^{(1)}_{n_{0}}}$ given by  $(\ref{choice-k-1st-regime})$, there holds
\begin{eqnarray*}		
\mathcal{C}_{j_0} & = & a^{3} \; \sum_{n} \frac{1}{1 + \eta \lambda_{n}^{(3)}} \,  \left[  \int_{B} e^{(3)}_{n}(x) \, dx +  \frac{1}{2 \, \lambda^{(1)}_{n_{0}}} \,  \langle \overset{\sim}{V}, N \left(e^{(3)}_{n}\right) \rangle \right] \otimes \int_{B} e^{(3)}_{n}(x) \, dx  \\
& \mp & a^{3} \; \sum_{n} \frac{1}{1 + \eta \lambda_{n}^{(3)}} \,   \frac{ c_{0} \, a^{h}}{2 \, \lambda^{(1)}_{n_{0}}} \,  \langle \overset{\sim}{V}, N \left(e^{(3)}_{n}\right) \rangle  \otimes \int_{B} e^{(3)}_{n}(x) \, dx + \mathcal{O}\left(a^{6} \; \left\Vert \overset{\sim}{V} \right\Vert_{\mathbb{L}^{2}(B)} \right).
\end{eqnarray*}
It is easy to verify that for 
\begin{eqnarray*}
L_{5} &:=& a^{3} \; \sum_{n} \frac{1}{1 + \eta \lambda_{n}^{(3)}} \,   \frac{ c_{0} \, a^{h}}{2 \, \lambda^{(1)}_{n_{0}}} \,  \langle \overset{\sim}{V}, N \left(e^{(3)}_{n}\right) \rangle  \otimes \int_{B} e^{(3)}_{n}(x) \, dx, \\
\left\vert L_{5} \right\vert & \lesssim & a^{5+h} \; \sum_{n} \left\vert  \langle N \left( \overset{\sim}{V}\right), e^{(3)}_{n} \rangle  \otimes \int_{B} e^{(3)}_{n}(x) \, dx \right\vert \\
 & \lesssim & a^{5+h} \;  \left\Vert   N \left( \overset{\sim}{V}\right) \right\Vert_{\mathbb{L}^{2}(B)} = \mathcal{O}\left(  a^{5+h} \;  \left\Vert  \overset{\sim}{V} \right\Vert_{\mathbb{L}^{2}(B)} \right).
\end{eqnarray*}
Hence, 
\begin{equation*}		
\mathcal{C}_{j_0} =  a^{3} \; \sum_{n} \frac{1}{1 + \eta \lambda_{n}^{(3)}} \,  \left[  \int_{B} e^{(3)}_{n}(x) \, dx +  \frac{1}{2 \, \lambda^{(1)}_{n_{0}}} \,  \langle \overset{\sim}{V}, N \left(e^{(3)}_{n}\right) \rangle \right] \otimes \int_{B} e^{(3)}_{n}(x) \, dx + \mathcal{O}\left(a^{5+h} \; \left\Vert \overset{\sim}{V} \right\Vert_{\mathbb{L}^{2}(B)} \right),
\end{equation*}
	where we deduce that\footnote{The smallness of the term $a^{3} \; \underset{n}{\sum} \dfrac{1}{1 + \eta \lambda_{n}^{(3)}} \,  \dfrac{1}{2 \, \lambda^{(1)}_{n_{0}}} \,  \langle \overset{\sim}{V}, N \left(e^{(3)}_{n}\right) \rangle  \otimes \int_{B} e^{(3)}_{n}(x) \, dx$, can be proved using the estimation of $\left\Vert \overset{\sim}{V} \right\Vert_{\mathbb{L}^{2}(B)}$, given by $(\ref{Norm-Scale-W})$.} 
	\begin{equation}\label{int-W}
\mathcal{C}_{j_0} \sim  a^{3} \; \sum_{n} \frac{1}{1 + \eta \lambda_{n}^{(3)}} \,    \int_{B} e^{(3)}_{n}(x) \, dx  \otimes \int_{B} e^{(3)}_{n}(x) \, dx \sim a^{5}. 
	\end{equation}
	Next, we estimate the $\mathbb{L}^{2}(B)$-norm of $\overset{\sim}{V}$. For this, using $(\ref{w-en1}),(\ref{W-en2})$ and $(\ref{W,e3})$, we obtain that 
	\begin{eqnarray*}
		\left\Vert \overset{\sim}{V} \right\Vert^{2}_{\mathbb{L}^{2}(B)} &=& \sum_{n} \left\vert \langle \overset{\sim}{V}; e_{n}^{(1)} \rangle \right\vert^{2} + \sum_{n} \left\vert \langle \overset{\sim}{V}; e_{n}^{(3)} \rangle \right\vert^{2} \\
		& \simeq & \sum_{n}\left\vert \frac{1}{\left( 1 - k^{2} \, \eta \, a^{2} \, \lambda_{n}^{(1)} \right)} \; \sum_{\ell \geq 1} \frac{(ika)^{\ell+1}}{(\ell+1)!} \; \int_{B} \overset{\sim}{V}(x) \cdot \int_{B} \; \left\Vert x - y \right\Vert^{\ell}  e_{n}^{(1)}(y) \, dy \, dx \right\vert^{2} \\
		&+& \sum_{n} \left\vert  \frac{1}{1 + \eta \lambda_{n}^{(3)}}  \int_{B} e^{(3)}_{n}(x) \, dx + \frac{\eta }{ 1+ \eta \lambda_{n}^{(3)}}  \, \langle \overset{\sim}{V}, \mathcal{S}_{d}\left(e^{(3)}_{n}\right) \rangle \right\vert^{2} \\ 
		& \lesssim & a^{4-2h} \, \sum_{n}  \sum_{\ell \geq 2}  \left\vert \int_{B} \overset{\sim}{V}(x) \cdot \int_{B} \; \frac{\left\Vert x - y \right\Vert^{\ell}}{(\ell+1)!}  e_{n}^{(1)}(y) \, dy \, dx \right\vert^{2} \\
		&+& a^{4} \sum_{n} \left\vert  \langle I , e^{(3)}_{n} \rangle \right\vert^{2} +  \sum_{n} \left\vert   \, \langle\mathcal{S}^{\star}_{d}\left(\overset{\sim}{V}\right), e^{(3)}_{n} \rangle \right\vert^{2}.
	\end{eqnarray*}
	Since the kernel of $\mathcal{S}_{d}$ behaves like the one of the Newtonian operator, in terms of singularity, the corresponding of the adjoint operator $\mathcal{S}^*_{d}$ also behaves like the one of the Newtonian operator. Then for the last term above, we get  
	\begin{equation*}
	\sum_{n} \left\vert   \, \langle\mathcal{S}^{\star}_{d}\left(\overset{\sim}{V}\right), e^{(3)}_{n} \rangle \right\vert^{2} \overset{(\ref{SdW})}{\sim} a^{4} \,  \sum_{n} \left\vert   \,\langle N\left(\overset{\sim}{V}\right), e^{(3)}_{n} \rangle \right\vert^{2} \lesssim a^{4} \; \left\Vert N\left(\overset{\sim}{V}\right) \right\Vert^{2} \lesssim a^{4} \; \left\Vert \overset{\sim}{V} \right\Vert^{2}.
	\end{equation*}
In addition,	 we have 
	\begin{eqnarray*}
		\sum_{n}  \sum_{\ell \geq 1}  \left\vert \int_{B} \overset{\sim}{V}(x) \cdot \int_{B} \; \frac{\left\Vert x - y \right\Vert^{\ell}}{(\ell+1)!}  e_{n}^{(1)}(y) \, dy \, dx \right\vert^{2} & = &    \sum_{\ell \geq 1}  \sum_{n} \left\vert \int_{B} \int_{B} \overset{\sim}{V}(x) \frac{\left\Vert x - y \right\Vert^{\ell}}{(\ell+1)!} \, dx \cdot e_{n}^{(1)}(y) \, dy \right\vert^{2} \\
		& \lesssim & \sum_{\ell \geq 1} \left\Vert \int_{B}  \overset{\sim}{V}(x) \frac{\left\Vert x - \cdot \right\Vert^{\ell}}{(\ell+1)!} \, dx \right\Vert^{2} \lesssim  \left\Vert  \overset{\sim}{V} \right\Vert^{2}.
	\end{eqnarray*}
	We use all these estimates to deduce that
	\begin{equation}\label{Norm-Scale-W}
	\left\Vert \overset{\sim}{V} \right\Vert^{2}_{\mathbb{L}^{2}(B)} 
	\lesssim   a^{4} \sum_{n} \left\vert  \langle I , e^{(3)}_{n} \rangle \right\vert^{2} \sim a^{4}.
	\end{equation} 
	Going back to $(\ref{Dj.E.B.C})$, by using $(\ref{Norm-Scale-W})$, we get 
	\begin{equation}\label{la-1-mid}
	\int_{D_{j_{0}}} \overset{3}{\mathbb{P}}\left(E^{T}_{j_{0}}\right)(x) \, dx - \eta \, \sum_{j=1 \atop j \neq j_{0}}^{\aleph} \mathcal{C}_{j_0} \cdot  \Upsilon_{k}(z_{j_{0}},z_{j}) \cdot \int_{D_{j}} \overset{3}{\mathbb{P}}\left(E^{T}_j\right)(y)\, dy = \mathcal{C}_{j_0}  \cdot E^{Inc}_{j_{0}}(z_{j_{0}})  + \mathcal{O}\left(a^{\min(6;8-h-3t)}\right).  
	\end{equation} 
	By injecting the expression $(\ref{int-W})$, of $\mathcal{C}_{j_0}$, into \eqref{la-1-mid}, after rearranging terms, we derive the linear algebraic systerm of the form \eqref{la-2}. Finally, replacing $\left\lVert[P_0]\right\rVert$ in \eqref{invert-condi-gene} by 
	\begin{equation}\notag
a^3 \, \left\lVert \sum_{n} \, \frac{1}{1+\eta \, \lambda_{n}^{(3)}} \, \int_{B}e_n^{(3)}(y)\,dy\otimes\int_{B}e_n^{(3)}(y)\,dy\right\rVert,
	\end{equation}
	 we deduce the invertibility condition \eqref{inver-con-2} for the linear algebraic systerm \eqref{la-2}.
\end{proof}

\end{document}